\documentclass[leqno]{article}

\usepackage[frenchb,english]{babel}
\usepackage[utf8]{inputenc}
\usepackage{amsmath}
\usepackage{amssymb}
\usepackage{amsfonts}
\usepackage{enumerate}
\usepackage{vmargin}
\usepackage[all]{xy}
\usepackage{mathrsfs}
\usepackage{mathtools}
\usepackage{esint}
\usepackage{lmodern}
\usepackage{slashed}
\usepackage{esint}
\usepackage{ulem}
\usepackage{stmaryrd} 
\usepackage[colorlinks=true,linkcolor=blue,pagebackref=true]{hyperref}%
\setmarginsrb{3cm}{3cm}{3.5cm}{3cm}{0cm}{0cm}{1.5cm}{3cm}
\usepackage{comment}
\usepackage{tikz}
\usetikzlibrary{patterns}

\newcommand{\A}{\ensuremath{\mathcal{A}}}

\newcommand{\C}{\mathrm{C}}
\newcommand{\D}{\ensuremath{\mathcal{D}}}
\newcommand{\E}{\mathrm{E}}

\newcommand{\N}{\ensuremath{\mathbb{N}}}
\newcommand{\B}{\mathrm{B}} 
\let\H\relax 
\newcommand{\H}{\mathrm{H}}
\newcommand{\I}{\mathrm{I}} 
\newcommand{\K}{\mathrm{K}}

\let\L\relax 
\newcommand{\L}{\mathrm{L}}
\newcommand{\J}{\mathcal{J}} 

\newcommand{\BMO}{\mathrm{BMO}} 
\newcommand{\VMO}{\mathrm{VMO}}

\newcommand{\UMD}{\mathrm{UMD}} 

\newcommand{\scr}{\mathscr}

\let\top\relax 

\newcommand{\top}{\mathrm{top}} 

\newcommand{\red}{\mathrm{red}} 

\newcommand{\M}{\mathrm{M}} 
\newcommand{\symb}{\mathrm{symb}} 
\newcommand{\Bott}{\mathrm{Bott}}
\newcommand{\bi}{\mathrm{bi}} 
\newcommand{\Ch}{\mathrm{Ch}}

\newcommand{\GL}{\mathrm{GL}}
\newcommand{\Ell}{\mathrm{Ell}} 

\def\qd{\,{\raisebox{-0.2mm}{$\mathchar'26$}\mkern-12mu\mathrm{d}}}
\newcommand{\weyl}{\mathrm{weyl}}
\newcommand{\app}{\mathrm{app}}

\newcommand{\per}{\mathrm{per}} 
\newcommand{\V}{\mathrm{V}}

\let\div\relax
\newcommand{\div}{\mathrm{div}} 
\let\cal\relax
\newcommand{\cal}{\mathcal}
\newcommand{\Z}{\ensuremath{\mathbb{Z}}}
\newcommand{\R}{\ensuremath{\mathbb{R}}}
\newcommand{\T}{\ensuremath{\mathbb{T}}}

\newcommand{\W}{\mathrm{W}}

\newcommand{\Id}{\mathrm{Id}} 
\newcommand{\SL}{\mathscr{S}\mathscr{L}} 
\newcommand{\SQL}{\mathscr{S}\mathscr{Q}\mathscr{L}} 
\newcommand{\VN}{\mathrm{VN}} 
\newcommand{\nc}{\mathrm{nc}} 
\newcommand{\la}{\langle}
\newcommand{\ra}{\rangle}

\renewcommand{\leq}{\ensuremath{\leqslant}}
\renewcommand{\geq}{\ensuremath{\geqslant}}
\newcommand{\qed}{\hfill \vrule height6pt  width6pt depth0pt}
\newcommand{\bnorm}[1]{ \big\| #1  \big\|}

\newcommand{\Bgnorm}[1]{ \Bigg\| #1  \Bigg\|}
\newcommand{\norm}[1]{\left\Vert#1\right\Vert}

\newcommand{\xra}{\xrightarrow}
\newcommand{\co}{\colon}

\newcommand{\ot}{\otimes}
\newcommand{\ovl}{\overline}

\newcommand{\sign}{\mathrm{sign}} 

\newcommand{\Fred}{\mathrm{Fred}} 

\newcommand{\dsp}{\displaystyle}
\let\i\relax 
\newcommand{\i}{\mathrm{i}} 
\let\Im\relax 
\newcommand{\Im}{\mathrm{Im}}

\newcommand{\ov}{\overset}

\newcommand{\w}{\mathrm{w}} 

\newcommand{\epsi}{\varepsilon}
\renewcommand{\d}{\mathop{}\mathopen{}\mathrm{d}} 
\newcommand{\e}{\mathrm{e}} 
\renewcommand{\d}{\mathop{}\mathopen{}\mathrm{d}}

\let\div\relax
\newcommand{\div}{\mathrm{div}}

\DeclareMathOperator{\Span}{span} 
\DeclareMathOperator{\Lip}{\mathrm{Lip}} 
\DeclareMathOperator{\supp}{supp} 
\DeclareMathOperator{\Index}{Index} 
\DeclareMathOperator{\sgn}{\mathrm{sgn}} 
\DeclareMathOperator{\tr}{Tr} 
\let\ker\relax 
\DeclareMathOperator{\ker}{Ker} 
\DeclareMathOperator{\Sp}{Sp} 
\DeclareMathOperator{\Ran}{Ran} 
\DeclareMathOperator{\diag}{diag} 
\DeclareMathOperator{\dom}{dom} 
\let\Re\relax 
\DeclareMathOperator{\Re}{Re} 
\DeclareMathOperator{\Ind}{Ind} 
\DeclareMathOperator*{\esssup}{esssup}

\DeclareMathOperator{\diam}{diam} 
\DeclareMathOperator{\wind}{wind} 
\DeclareMathOperator{\ch}{ch} 
\DeclareMathOperator{\pv}{p.v.} 
\DeclareMathOperator{\rank}{rank} 

\newtheorem{thm}{Theorem}[section]
\newtheorem{defi}[thm]{Definition}
\newtheorem{prop}[thm]{Proposition}

\newtheorem{cor}[thm]{Corollary}
\newtheorem{lemma}[thm]{Lemma}

\newtheorem{remark}[thm]{Remark}
\newtheorem{example}[thm]{Example}

\newenvironment{proof}[1][]{\noindent {\it Proof #1} : }{\hbox{~}\qed
\smallskip
}

\usepackage{tocloft}
\numberwithin{equation}{section}
\usepackage[nottoc,notlot,notlof]{tocbibind}

\let\OLDthebibliography\thebibliography
\renewcommand\thebibliography[1]{
  \OLDthebibliography{#1}
  \setlength{\parskip}{0pt}
  \setlength{\itemsep}{0pt plus 0.3ex}
}

\newcommand\reallywidehat[1]{\arraycolsep=0pt\relax%
\begin{array}{c}
\stretchto{
  \scaleto{
    \scalerel*[\widthof{\ensuremath{#1}}]{\kern-.5pt\bigwedge\kern-.5pt}
    {\rule[-\textheight/2]{1ex}{\textheight}} 
  }{\textheight} %
}{0.5ex}\\           
#1\\                 
\rule{-1ex}{0ex}
\end{array}
}

\begin{document}
\selectlanguage{english}
\title{\bfseries{Classical harmonic analysis viewed through the prism of noncommutative geometry}}
\date{}
\author{\bfseries{C\'edric Arhancet}}
\maketitle


\begin{abstract}
The aim of this paper is to bridge noncommutative geometry with classical harmonic analysis on Banach spaces, focusing primarily on both classical and noncommutative $\L^p$ spaces. Introducing a notion of Banach Fredholm module, we define new abelian groups, $\K^{0}(\cal{A},\scr{B})$ and $\K^{1}(\cal{A},\scr{B})$, of $\K$-homology associated with an algebra $\cal{A}$ and a suitable class $\scr{B}$ of Banach spaces, such as the class of $\L^p$-spaces. We establish index pairings of these groups with the $\K$-theory groups of the algebra $\cal{A}$. Subsequently, by considering (noncommutative) Hardy spaces, we uncover the natural emergence of Hilbert transforms, leading to Banach Fredholm modules and culminating in new index theorems.  Moreover, by associating each reasonable sub-Markovian semigroup of operators with a <<Banach noncommutative manifold>>, we explain how this leads to (possibly kernel-degenerate) Banach Fredholm modules, thereby revealing the role of vectorial Riesz transforms in this context. Overall, our approach significantly integrates the analysis of operators on $\L^p$-spaces into the expansive framework of noncommutative geometry, offering new perspectives.
\end{abstract}

%
%
%


\makeatletter
 \renewcommand{\@makefntext}[1]{#1}
 \makeatother

{
  \hypersetup{linkcolor=blue}
 \tableofcontents
}

\section{Introduction}
\label{sec:Introduction}

The aim of this paper is to provide a coherent framework that encompasses various aspects of harmonic analysis on $\L^p$-spaces within the context of noncommutative geometry. We seek to elucidate connections that appear sporadically throughout the literature, clarifying these overlaps and coincidences. This work can be seen as the next step in our ongoing research program, which began in our previous studies \cite{ArK22}, \cite{Arh24a}, \cite{Arh24b}, and \cite{Arh24c}.

It is crucial to emphasize that noncommutative geometry is not merely a <<generalization>> of classical geometry. While it subsumes known spaces as particular cases, it offers a radically different approach to some classical spaces, such as fractals or leaf  spaces of foliations, by introducing powerful analytical points of view and tools. A concrete and recent example is provided by the phenomenon of <<hidden noncommutative geometry>>, which arises in the study of some (Markovian) semigroups of operators acting on classical $\L^p$-spaces, such as the Poisson semigroup on the torus $\mathbb{T}$, described in \cite{JMP18} and \cite{ArK22}. To understand the analogue of the vectorial Riesz equivalence 
\begin{equation}
\label{}
\bnorm{(-\Delta)^{\frac12}f}_{\L^p(\R^n)} 
\approx_{p} \norm{\nabla f}_{\L^p(\R^n,\ell^2_n)},\quad 1 < p < \infty,
\end{equation}
for the Poisson semigroup on $\T$, where the Laplacian $-\Delta$ on $\R^n$ is replaced by the infinitesimal generator $A$ of the Poisson semigroup $(\e^{-tA})_{t \geq 0}$, it is necessary to introduce a gradient $\partial$ taking values in a closed subspace of a suitable noncommutative $\L^p$-space $\L^p(\cal{M})$ associated to a specific noncommutative von Neumann algebra $\cal{M}$. This leads to the equivalence
\begin{equation} 
\label{Riesz-Parcet-commutatif}
\bnorm{A^{\frac12}f}_{\L^p(\mathbb{T})} 
\approx_{p} \norm{\partial f}_{\L^p(\cal{M})},\quad 1 < p < \infty,
\end{equation}
and to the introduction of some <<noncommutative manifold>> defined essentially with the operator introduced in \eqref{Hodge-Dirac-I} below. 

In noncommutative geometry, the starting point is an algebra $\cal{A}$ (which may be commutative or not), representing the space, with its elements acting as bounded operators on a complex Hilbert space $H$ via a representation $\pi \co \cal{A} \to \B(H)$. This algebra $\cal{A}$ replaces the algebra $\C(\mathcal{X})$ of continuous functions on a classical compact Hausdorff space $\mathcal{X}$ endowed with a finite regular Borel measure, which acts on the complex Hilbert space $\L^2(\mathcal{X})$ by multiplication operators.

Recall that in this context, the K-theory groups $\K_0(\cal{A})$ and $\K_1(\cal{A})$ provide homotopic invariants for each $\C^*$-algebra $\cal{A}$ (and even for more general algebras). These groups are abelian and are furthermore countable if $\cal{A}$ is separable. For example, it is known \cite[Corollary 3.2 p.~152]{PiV82} that
$$
\K_0(\C^*_\red(\mathbb{F}_n)) 
= \mathbb{Z} 
\quad \text{and} \quad 
\K_1(\C^*_\red(\mathbb{F}_n)) 
= \mathbb{Z}^n, \quad n \geq 1,
$$
which makes it possible to distinguish the reduced $\C^*$-algebras of free groups. If $\mathcal{X}$ is a compact Hausdorff topological space, it is worth noting that the group $\K_0(\C(\mathcal{X}))$ coincide with the group $\K_{0,\top}(\cal{X})$, which is the Grothendieck group of the abelian monoid of isomorphism classes of complex vector bundles on $\cal{X}$. Finally, if $M$ is a closed manifold (smooth compact without boundary), it is well-known \cite[Theorem 3.8.13 p.~107]{NSSS06} \cite[p.~277]{BlB13} that there exists an isomorphism between the group $\Ell(M)$ of stable homotopy classes of elliptic pseudo-differential operators between complex vector bundles over $M$ and the K-theory group $\K_{0,\top}(\mathrm{T}^*M)$ of the cotangent bundle $\mathrm{T}^*M$.

%
%
%
%
%
%

A fundamental concept in this framework is that of a Fredholm module. If $\cal{A}$ is unital, a Fredholm module over $\cal{A}$ is a bounded operator $F \co H \to H$ satisfying the three relations $F^2=\Id_{H}$, $F=F^*$ and $[F,\pi(a)]=0$ for any $a \in \cal{A}$ \textit{modulo a compact operator}. The last equality is a form of <<pseudolocality>>. This notion of Fredholm module is a <<normalized>> variant of the notion of Atiyah cycle\footnote{\thefootnote. An Atiyah cycle on $\mathcal{X}$ is a Fredholm operator
$T \co H_1 \to H_2$ in Hilbert spaces $H_1$ and $H_2$ equipped with the structure of $*$-modules over the algebra $\C(\cal{X})$ such that the commutator $[T,f]$ is a compact operator for any function $f \in \C(\cal{X})$. By Kasparov’s lemma (Higson and Roe 2000), if we assume that the module structures are nondegenerate (i.e., $\C(\cal{X})H_{1,2} = H_{1,2}$), then condition
(10.6) is equivalent to pseudolocality: the operator $fDg$ is compact for any continuous functions $f$, $g$ on $X$ with disjoint supports.} (or generalized elliptic operator) introduced in the paper \cite{Ati70} (see also \cite[Definition 7.6 p.~317]{BDF77}, \cite[Definition 6 p.~165]{Dou86} and \cite[Definition 10.9 p.~274]{NSSS06}). If $D$ is the Dirac operator on a compact Riemannian spin manifold $M$, the Fredholm module $\sgn D$ canonically associated to $M$ encodes the conformal structure of the manifold, as discussed in \cite[Theorem 3.1 p.~388]{Bar07}. These operators enable the definition of $\K$-homology groups $\K^0(\cal{A})$ and $\K^1(\cal{A})$, whose elements are equivalence classes of Fredholm modules, defined under appropriate notions of unitary equivalence and homotopy equivalence, and a natural operation of addition. The distinction between the two groups lies in the fact that, in the case of $\K^0(\cal{A})$, only \textit{even} Fredholm modules are considered, provided an appropriate grading is introduced. These groups are linked to the $\K$-theory groups $\K_0(\cal{A})$ and $\K_1(\cal{A})$ of the algebra $\cal{A}$ through pairings 
\begin{equation}
\label{pairings-intro}
\K_0(\cal{A}) \times \K^0(\cal{A}) \to \Z
\qquad \text{and} \qquad
\K_1(\cal{A}) \times \K^1(\cal{A}) \to \Z,
\end{equation}
that lead to index theorems. In general, these pairings are challenging to compute, and this is precisely the focus of index theory. 

If $E$ is a Hermitian complex vector bundle over a closed smooth manifold $M$ endowed with a Lebesgue measure, it is interesting to observe that if $D \co \scr{C}^\infty(M,E) \to \scr{C}^\infty(M,E)$ is a Hermitian elliptic pseudo-differential operator of order $m > 0$ acting on the space $\scr{C}^\infty(M,E)$ of smooth sections of $E$ then it induces a selfadjoint operator on the Hilbert space $\L^2(M,E)$. Moreover, the operator $\sign D$ induces a pseudo-differential operator of order 0 acting on $\scr{C}^\infty(M,E)$, see \cite[Proposition 2.4, p.~143]{BaG82b} and its proof. Recall that the multiplication operator $\pi(f) \co \L^2(M,E) \to \L^2(M,E)$, defined by any function $f\in \C(M)$ acting on the Hilbert space, satisfies a standard property: the commutator $[P,\pi(f)]$ of any pseudo-differential operator $P$ of order 0 is a compact operator. This entails that any Hermitian elliptic pseudo-differential operator defines a K-homology class $[\sign D]$, often denoted $[D]$, in the K-homology group $\K^1(\C(M))$. Using the operator $\sign \begin{bmatrix}
   0  & D^*  \\
   D  & 0  \\
\end{bmatrix}$ for an elliptic pseudo-differential operator $D$ of order $m > 0$, instead of $\sign D$, we can also obtain a K-homology class $[D]$ in the even K-homology group $\K^0(\C(M))$.

If $M$ is a spin closed manifold\footnote{\thefootnote. More generally, it is true for spin closed manifolds.} of dimension $2n$ and if $D$ is the associated spin Dirac operator, then we have essentially by \cite[Theorem 9.6.11 p.~388]{WiY20} and \cite[Proposition 11.3.7]{Eme24} a group isomorphism $\K_{0,\top}(M) \to \K^0(\C(M))$, $[E] \mapsto [D_E]$, where $D_E$ denotes the Dirac operator $D$ twisted by the complex vector bundle $E$ over $M$. This is an analogue of the classical Poincar\'e duality on the structure of the homology and cohomology groups of orientable closed manifolds. In this context, the twisted Dirac operator $D_E$ is a Hermitian elliptic first order differential operator by \cite[p.~530]{BlB13} and can be expressed as $D_E=\begin{bmatrix}
   0  & D_E^-  \\
   D_E^+  & 0  \\
\end{bmatrix}$. The operator $D_E^+$ induces a Fredholm operator between suitable Sobolev spaces. A particular case of the first pairing described in \eqref{pairings-intro} can be written by \cite[Proposition 11.3.7 p.~512]{Eme24}
\begin{equation}
\label{}
\la [D], [E] \ra_{\K_0(\C(M)),\K_{0,\mathrm{top}}(M)}
= \Index D_E^+.
\end{equation}
A particular case of the Atiyah-Singer index theorem describes this index and consequently this pairing explicitly in terms of de Rham cohomology of the manifold $M$ by providing the formula $\Index D_E = (2\pi \i)^{-n} \int_M \ch(E) \frown \hat{\mathrm{A}}(M)$, where $\Ch(E)$ is the Chern character of the vector bundle $E$ and where $\hat{\mathrm{A}}(M)$ is the total $\hat{\mathrm{A}}$-class of the manifold $M$. 

The culmination of this theory is perhaps the index theorem of Connes-Moscovici itself \cite{CoM95}, \cite{Hig02}, and its generalization to the locally compact case \cite{CGRS14}, drawing inspiration from the Atiyah-Singer index theorem, whose applications are thoroughly covered in \cite{BlB13} and in the survey paper  \cite{Fre21}. For more information on noncommutative geometry, we refer to the books \cite{GVF01}, \cite{EcI18} and the survey article \cite{CPR11}, as well as specific applications to solid state physics in \cite{PSB16} and related works on $\K$-theory and $\K$-homology \cite{Pus11}, \cite{Had03}, \cite{Had04}, \cite{NeT11}, \cite{EmN18}, \cite{FGMR19}, \cite{Ger22} all of which are grounded in the classical text \cite{HiR00}.

In this paper, we define a notion of \textit{Banach} Fredholm module $F \co X \to X$ on an arbitrary Banach space $X$ over an algebra $\cal{A}$, where elements act on the space $X$ via a representation $\pi \co \cal{A} \to \B(X)$. In Connes' foundational work \cite{Con94}, an illustrative example of a Fredholm module is given by $F \ov{\mathrm{def}}{=}  \i\cal{H}$ where $\cal{H} \co \L^2(\R) \to \L^2(\R)$ is the Hilbert transform. This example is sometimes seen as an <<exotic>> Fredholm module. Here, we show that various Hilbert transforms on $\L^p$-spaces satisfying $F^2 = \Id_X$ \textit{modulo a compact operator} naturally arise as examples. Actually, we introduce \textit{huge} classes of Banach Fredholm modules over group $\mathrm{C}^*$-algebras, which have not been previously considered \textit{even} in the Hilbertian case. If $M_f \co \L^p(\R^n)\to \L^p(\R^n)$, $g \mapsto fg$ is the multiplication operator by a function $f$, the commutators $[\cal{H},M_f]$ of the Hilbert transform $\cal{H}$ (with $n=1$) and more generally commutators $[\cal{R}_j,M_f]$ of Riesz transforms $\cal{R}_j \co \L^p(\R^n) \to \L^p(\R^n)$ where $1 \leq j \leq n$, on $\L^p$-spaces or even similar operators in other contexts is a classical topic in analysis, see the survey \cite{Wic20} and references therein, connected to Hardy spaces, spaces of functions of bounded/vanishing mean oscillation and factorization of functions, and we demonstrate that they also emerge naturally in the commutators of Banach Fredholm modules. Such a commutator is given by $[F,\pi(a)]$ for some $a \in \cal{A}$. In this framework, we can introduce the Connes quantized differential
\begin{equation}
\label{quantized-differential-intro}
\qd a
\ov{\mathrm{def}}{=} \i [F,\pi(a)], \quad a \in \cal{A}
\end{equation}
of $a$. In the case $F=\i \cal{H}$ and $a=f \in \C_0^\infty(\R)$,  we have
\begin{equation}
\label{commutator-Hilbert-R}
\big([F,M_f]g\big)(x)
=\pv\int_{-\infty}^{\infty}\frac{f(x)-f(y)}{x-y}g(y) \d y.
\end{equation}
A well-known phenomenon in the Hilbertian context is the interplay between the differentiability of $a$ and the <<degree of compactness>> of the quantized differential $\qd a$. In the Banach space context, we use some suitable quasi-Banach ideal $S^q_\app(X)$, where $0 < q < \infty$, relying on the concept of $s$-numbers introduced by Pietsch, as a substitute of the Schatten space $S^q(H) \ov{\mathrm{def}}{=} \{T \co H \to H : \tr |T|^q < \infty\}$, allowing us to define a notion of finitely summable Banach Fredholm module. The notion of $s$-numbers is a generalization of the notion of singular value of operators acting on Hilbert spaces.

We equally introduce $\K$-homology groups $\K^{0}(\cal{A},\scr{B})$ and $\K^{1}(\cal{A},\scr{B})$ associated to the algebra $\cal{A}$ and with a class $\scr{B}$ of Banach spaces, as the class of (noncommutative) $\L^p$-spaces. We explore their pairings with the $\K$-theory groups $\K_0(\cal{A})$ and $\K_1(\cal{A})$ of the algebra $\cal{A}$. In the simplest case, the pairing gives the index theorem of Gohberg-Krein for a Toeplitz operator $T_a \co \H^p(\T) \to \H^p(\T)$ acting on the classical Hardy space $\H^p(\T)$ on the circle for any $1 < p < \infty$, generalizing the well-known case $p=2$.

Connes demonstrated how to associate a canonical cyclic cocycle, known as the Chern character, with a \textit{finitely summable} Fredholm module, and how this character can be used to compute the index pairing between the 
$\K$-theory of $\cal{A}$ and the $\K$-homology class of the Fredholm module. We show that our framework admits a similar Chern character. The (odd) Chern character in K-homology is defined from the quantized calculus \eqref{quantized-differential-intro} associated to a Banach Fredholm module by
$$
\Ch_{n}^F(a_0,a_1,\dots,a_{n})
\ov{\mathrm{def}}{=} c_n \tr(F \qd a_0 \qd a_1 \cdots \qd a_n),\qquad a_0,a_1,\dots,a_n \in \cal{A},
$$
for some sufficiently large odd integer $n$ and where $c_n$ is some suitable constant. Here, we use the trace $\tr$ on the space $S^1_\app(X)$ which is the unique continuous extension of the trace defined on the space of finite-rank operators acting on the Banach space $X$.

Within the Hilbertian setting of Connes, it is well-known that a Fredholm module can be constructed from a compact spectral triple (i.e.~a noncommutative compact Riemannian spin manifold). Such a compact spectral triple $(\cal{A},H,D)$ consists of a selfadjoint operator $D$, defined on a dense subspace of the Hilbert space $H$, and satisfying certain axioms. Specifically, $D^{-1}$ must be compact on $\ker D^\perp$, often referred to as the <<unit length>> or <<line element>>, denoted by $\d s$ since we can introduce the noncommutative integral $\int$ with a Dixmier trace.

The Dirac operator $D$ on a spin Riemannian compact manifold $M$ can be used for constructing a classical example of spectral triple. Several examples in different contexts are discussed in the survey \cite{CoM08}. 
The Fredholm module obtained from a spectral triple $(\cal{A},H,D)$ is built using the bounded operator 
\begin{equation}
\label{sgn-Hilbert}
\sgn D \co H \to H,
\end{equation}
constructed with the spectral theorem. 

In this paper, we replace Connes' spectral triples by the Banach compact spectral triples of \cite{ArK22}. Such a triple $(\cal{A},X,D)$ is composed of a Banach space $X$, a representation $\pi \co \cal{A} \to \B(X)$ and a bisectorial operator $D$ on a Banach space $X$ that admits a bounded $\H^\infty(\Sigma_\theta^\bi)$ functional calculus on the open bisector $\Sigma_\theta^\bi \ov{\mathrm{def}}{=} \Sigma_\theta \cup (-\Sigma_\theta)$ where $\Sigma_{\theta} \ov{\mathrm{def}}{=} \big\{ z \in \mathbb{C} \backslash \{ 0 \} : \: | \arg z | < \theta \big\}$, with $0 < \theta <\frac{\pi}{2}$. We refer to Definition \ref{Def-Banach-spectral-triple} for a precise definition with an assumption on some commutators. Roughly speaking, this means that the spectrum $\sigma(D)$ of $D$ is a subset of the closed bisector $\ovl{\Sigma^\bi_{\omega}}$ for some $\omega \in (0,\theta)$ as in Figure \ref{Figure-1}, that we have an appropriate <<resolvent estimate>> and that
\begin{equation}
\label{}
\norm{f(D)}_{X \to X}
\lesssim \norm{f}_{\H^\infty(\Sigma^\bi_\theta)}
\end{equation} 
for any \textit{suitable} function $f$ of the algebra $\H^\infty(\Sigma^\bi_\theta)$ of all bounded holomorphic functions defined on the bisector $\Sigma^\bi_\sigma$. Here <<suitable>> means regularly decaying at 0 and at $\infty$. In broad terms, the operator $f(D)$ is defined by a <<Cauchy integral>> 
\begin{equation}
\label{}
f(D)
=\frac{1}{2\pi \i} \int_{\partial \Sigma^\bi_\nu} f(z)R(z,D) \d z
\end{equation}
by integrating over the boundary of a larger bisector $\Sigma^\bi_\nu$ using the resolvent operator $R(z,D)\ov{\mathrm{def}}{=} (z-D)^{-1}$, where $\omega < \nu <\theta$. Note that the boundedness of such a functional calculus is not free, contrary to the case of the functional calculus of a selfadjoint operator. We refer to our paper \cite{Arh24c} for concrete examples of operator with such a bounded functional calculus, where we use a notion of curvature for obtaining it. Using the function $\sgn$ defined by $\sgn(z) \ov{\mathrm{def}}{=} 1_{\Sigma_\theta}(z)-1_{-\Sigma_\theta}(z)$, we show as the hilbertian case that $\sign D \co X \to X$ is a Fredholm module. This notion of functional calculus was popularized in the paper \cite{AKM06}, which contains a (second) solution to famous  Kato's square root problem solved in \cite{AHLMT02} and in \cite{AKM06} (see also \cite{HLM02} and \cite{Tch01} and the Bourbaki seminar \cite{Mey03}). 

Let us explain how this approach integrates with ours, starting with the simplest case, the one-dimensional scenario. Consider a function $a \in \L^\infty(\R)$ such that $\Re a(x) \geq \kappa > 0$ for almost all $x \in \R$ and the multiplication operator $M_a \co \L^2(\R) \to \L^2(\R)$, $f \mapsto a f$. Following the approach of \cite{AKM06}, we can consider the unbounded operator
\begin{equation}
\label{Dirac-perturbed-intro}
D
\ov{\mathrm{def}}{=}\begin{bmatrix} 
0 & -\frac{\d }{\d x} M_a \\ 
\frac{\d }{\d x} & 0 
\end{bmatrix}, 
\end{equation}
acting on the complex Hilbert space $\L^2(\R) \oplus \L^2(\R)$, where $\frac{\d }{\d x} M_a$ denotes the composition $\frac{\d }{\d x} \circ M_a$. Note that the operator $D$ is not selfadjoint in general. However, by \cite[Theorem 3.1 (i) p.~465]{AKM06}, $D$ admits a bounded $\H^\infty(\Sigma_\theta^\bi)$ functional calculus for some angle $0 < \theta <\frac{\pi}{2}$. Using $D^2=\begin{bmatrix} 
-\frac{\d }{\d x} M_a \frac{\d }{\d x}& 0 \\ 
0 & -\frac{\d^2 }{\d x^2} M_a
\end{bmatrix}$, we see that formally we have a bounded operator
\begin{equation}
\label{formally-bis-2}
\sign D
=D (D^2)^{-\frac{1}{2}}
=\begin{bmatrix} 
0 & * \\ 
\frac{\d }{\d x}( -\frac{\d }{\d x} M_a \frac{\d }{\d x})^{-\frac{1}{2}} & 0 
\end{bmatrix}.
\end{equation}
It is immediate to obtain the obtain the estimate $\norm{\frac{\d f}{\d x}}_{\L^2(\R)} \lesssim \norm{\big(-\frac{\d }{\d x} M_a \frac{\d }{\d x}\big)^{\frac{1}{2}}f}_{\L^2(\R)}$. Actually, a slightly more elaborate argument gives the Kato square root estimate in one dimension
\begin{equation}
\label{Kato-on-dimension}
\norm{\bigg(-\frac{\d }{\d x} M_a \frac{\d }{\d x}\bigg)^{\frac{1}{2}}f}_{\L^2(\R)}
\approx \norm{\frac{\d f}{\d x}}_{\L^2(\R)}, \quad f \in \W^{1,2}(\R).
\end{equation}
Using the homomorphism $\pi \co \C_c^\infty(\R) \mapsto \B(\L^2(\R) \oplus \L^2(\R))$, $a \mapsto M_a \oplus M_a$. We will easily show that $(\C_c^\infty(\R),\L^2(\R) \oplus \L^2(\R),D)$ is a \textit{Banach} locally compact spectral triple, which is not a locally compact spectral triple in the classical hilbertian sense.

It is important to realize that all the theory of sub-Markovian semigroups acting on $\L^p$-spaces can be integrated into the notion of Banach spectral triples.  
Indeed the $\L^2$-generator $-A_2$ of such semigroup $(T_t)_{t \geq 0}$ acting on the Hilbert space $\L^2(\Omega)$ can be written $A_2=\partial^*\partial$ where $\partial$ is a (unbounded) closed derivation defined on a dense subspace of $\L^2(\Omega)$ with values in a Hilbert $\L^\infty(\Omega)$-bimodule $\cal{H}$. Here $T_t=\e^{-tA_2}$ for any $t \geq 0$. The map $\partial$ can be seen as an <<abstract>> analogue of the gradient operator $\nabla$ of a smooth Riemannian manifold $M$, which is a closed operator defined on a subspace of $\L^2(M)$ into the space $\L^2(M,\mathrm{T} M)$ satisfying the relation $-\Delta=\nabla^*\nabla$ where $\Delta$ is the Laplace-Beltrami operator and where $\nabla^*=-\div$.

This fundamental result allows anyone to introduce a triple $(\L^\infty(\Omega),\L^2(\Omega) \oplus_2 \cal{H},D)$ associated to the semigroup in the spirit of the previous Banach spectral triples. Here $D$ is the unbounded selfadjoint operator acting on a dense subspace of the Hilbert space $\L^2(\Omega) \oplus_2 \cal{H}$ defined by 
\begin{equation}
\label{Hodge-Dirac-I}
D
\ov{\mathrm{def}}{=}
\begin{bmatrix} 
0 & \partial^* \\ 
\partial & 0 
\end{bmatrix}.
\end{equation}
It is possible in this context to introduce a homomorphism $\pi \co \L^\infty(\Omega) \to \B(\L^2(\Omega) \oplus_2 \cal{H})$, see \eqref{Def-pi-a}. Now, suppose that $1 < p < \infty$. Sometimes, the map $\partial$ induces a closable unbounded operator $\partial \co \dom \partial \subset \L^p(\cal{M}) \to \cal{X}_p$ for some Banach space $\cal{X}_p$. So we can consider the $\L^p$-realization of the previous operator $D$
as acting on a dense subspace of the Banach space $\L^p(\Omega) \oplus_p \cal{X}_p$. If $D$ is bisectorial and admits a bounded $\H^\infty(\Sigma_\theta^\bi)$ functional calculus then using the equalities $D^2=\begin{bmatrix} 
\partial^* \partial & 0 \\ 
0 & \partial \partial^*
\end{bmatrix}$ and $A_2=\partial^*\partial$, we see that formally we have a bounded operator
\begin{equation}
\label{formally-bis}
\sign D
=D (D^2)^{-\frac{1}{2}}
=\begin{bmatrix} 
0 & * \\ 
\partial A^{-\frac{1}{2}} & 0 
\end{bmatrix}.
\end{equation}
From this, it is apparent that the vectorial Riesz transform $\partial A^{-\frac{1}{2}} \co \L^p(\Omega) \to \cal{X}_p$ is bounded and appears in the operator $F \ov{\mathrm{def}}{=}\sign D$ and particularly with the possible pairing with the group $\K_0(\cal{A})$ of $\K$-theory for a suitable subalgebra $\cal{A}$ of the algebra $\L^\infty(\Omega)$, as we will see. Of course, we can consider more generally sub-Markovian semigroups acting on the noncommutative $\L^p$-space $\L^p(\cal{M})$ of a von Neumann algebra $\cal{M}$ in the previous discussion. 

\paragraph{Structure of the Paper}
This paper is structured as follows. Section \ref{sec-preliminaries} provides the necessary background and revisits key notations, as well as essential results required for our work. In Section \ref{sec-K-homology}, we introduce the concept of (odd or even) Banach Fredholm modules and define new groups, $\K^{0}(\cal{A},\scr{B})$ and $\K^{1}(\cal{A},\scr{B})$, in $\K$-homology, associated with an algebra $\cal{A}$ and a suitable class $\scr{B}$ of Banach spaces. Section \ref{sec-coupling} focuses on establishing pairings with the $\K$-theory groups of the algebra $\cal{A}$, which is central to our approach. In Section \ref{sec-summability}, we introduce a suitable notion of summability for Banach Fredholm modules, revisiting key concepts related to approximation numbers of operators acting on Banach spaces. Section \ref{sec-Chern} is dedicated to the Chern character associated with finitely summable Banach Fredholm modules. Following the classical Hilbert space approach, we demonstrate how the pairing between $\K$-homology and $\K$-theory groups can be described using the Chern character. In Section \ref{sec-from}, we revisit the notion of Banach spectral triples and explain how Banach Fredholm modules can be constructed from these triples. Section \ref{Examples-1} explores concrete examples of triples constructed using Dirac operators, leading to Banach Fredholm modules, where we compute the pairings with $\K$-theory in specific cases. We also reveal the connection between the summability of Banach spectral triples and Sobolev embedding theorems. In Section \ref{Examples-Hardy}, we present examples of Banach Fredholm modules in the context of noncommutative Hardy spaces, leading to new Banach Fredholm modules on group $\C^*$-algebras, even within the Hilbert space framework. Section \ref{sec-free-Hilbert} present a Banach Fredholm module on the reduced $\mathrm{C}^*$-algebra of the free group $\mathbb{F}_\infty$. Finally, in Section \ref{sec-triples} discusses how to associate a Banach spectral triple with each reasonable sub-Markovian semigroup. We also observe that the corresponding Banach Fredholm module, which may be kernel-degenerate, is closely linked to the vectorial Riesz transform associated to the semigroup.

\section{Preliminaries}
\label{sec-preliminaries}

\paragraph{Fredholm operators}
Following \cite[Definition 4.37 p.~156]{AbA02}, we say that a bounded operator $T \co X \to Y$, acting between complex Banach spaces $X$ and $Y$, is a Fredholm operator if the subspaces $\ker T$ and $Y/\Ran T$ are finite-dimensional. In this case, we introduce the index
\begin{equation}
\label{Fredholm-Index}
\Index T 
\ov{\mathrm{def}}{=} \dim \ker T-\dim Y/\Ran T.
\end{equation}
Every Fredholm operator has a closed range by \cite[Lemma 4.38 p.~156]{AbA02}. Recall the Banach version of Atkinson's theorem \cite[Theorem 4.46 p.~161]{AbA02}.

\begin{thm}
\label{th-Atkinson}
A bounded operator $T \co X \to Y$ between Banach spaces is a Fredholm operator if and only if there exists a bounded operator $R \co Y \to X$ such that $RT-\Id_X$ and $TR-\Id_Y$ are compact operators. Moreover, we can replace <<compact>> by <<finite-rank projection>> in this assertion.
\end{thm}

According to \cite[Theorem 4.48 p.~163]{AbA02}, the set $\Fred(X,Y)$ of all Fredholm operators from $X$ into $Y$ is an open subset of $\B(X,Y)$ and the index function $\Index \co \Fred(X,Y) \to \Z$ is continuous (hence locally constant). By \cite[Corollary 4.47 p.~162]{AbA02}, if $T \co X \to Y$ is a Fredholm operator and if $K \co X \to Y$ is a compact operator, then $T+K$ is a Fredholm operator and 
\begin{equation}
\label{invariance-under-compact-perturbations}
\Ind T 
= \Ind(T+K)
\end{equation}
(invariance under compact perturbations). If $T \co X \to Y$ and $S \co Y \to Z$ are Fredholm operators then the composition $ST$ is also a Fredholm operator and $\Index ST=\Index S +\Index T$, see \cite[Theorem 4.43 p.~158]{AbA02}. Moreover, by \cite[Theorem 4.42 p.~157]{AbA02}, a bounded operator $T \co X \to Y$ is a Fredholm operator if and only if its adjoint $T^* \co X^* \to Y^*$ is a Fredholm operator. In this case, we have 
\begin{equation}
\label{Index-adjoint}
\Index T^*
=-\Index T.
\end{equation}

Recall that by \cite[Theorem 4.54 p.~167]{AbA02} any Fredholm operator $T \co X \to Y$ admits a generalized inverse (or pseudo-inverse), i.e., an operator $U \co Y \to X$ such that
\begin{equation}
\label{generalized-inverse}
TUT=T
.
\end{equation}
The same result says that every generalized inverse $U$ of $T$ is also a Fredholm operator and satisfies $\Index U=-\Index T$. In this case, the proof of \cite[Theorem 4.52 p.~166]{AbA02} 
shows that the operators $TU \co Y \to Y$ and $\Id-UT \co X \to X$ are bounded projections on the subspaces $\Ran T$ and $\ker T$.

Two bounded operators $T_0, T_1 \co X \to Y$ are called homotopic (through Fredholm operators) if there exists a continuous map $f \co [0,1] \to \Fred(X,Y)$ with $f(0)=T_0$ and $f(1)=T_1$. In this case, we have $\Index T_0 =\Index T_1$ by \cite[Proposition 4 p.~15]{ReS82}.

\paragraph{Unbounded operators}
The following result is \cite[Theorem 1.2.4 p.~15]{HvNVW16}, which allows swapping a closed operator and an integral when some conditions are satisfied.

\begin{prop}
\label{prop-Hille}
Let $f \co \Omega \to X$ be a Bochner integrable function and let $T$ be a closed linear operator with domain $\dom T$ in $X$ and with values in a Banach space $Y$. Suppose that $f$ takes its values in $\dom T$ almost everywhere and that the almost everywhere defined function $T \circ f \co \Omega \to Y$ is Bochner integrable. Then $f$ is Bochner integrable as a $\dom(T)$-valued function, $\int_\Omega f \d\mu$ belongs to $\dom(T)$ and
\begin{equation}
\label{unbounded-Hille}
T\bigg(\int_\Omega f \d\mu\bigg)
=\int_\Omega T \circ f \d\mu.
\end{equation}
\end{prop}

\paragraph{$R$-boundedness} Suppose that $1 \leq p < \infty$. Following \cite[Definition 8.1.1, Remark 8.1.2 p.~165]{HvNVW18}, we say that a set $\cal{F}$ of bounded operators on a Banach space $X$ is $R$-bounded \index{$R$-boundedness} if there exists a constant $C \geq 0$ such that for any integer $n \geq 1$, any $T_1,\ldots, T_n$ in $\cal{F}$ and any $x_1,\ldots,x_n$ in $X$, we have
\begin{equation}
\label{R-boundedness}
\Bgnorm{\sum_{k=1}^{n} \epsi_k \ot T_k (x_k)}_{\L^p(\Omega,X)}
\leq C \Bgnorm{\sum_{k=1}^{n} \epsi_k \ot x_k}_{\L^p(\Omega,X)},
\end{equation}
where $(\epsi_{k})_{k \geq 1}$ is a sequence of independent Rademacher variables on some probability space $\Omega$. As explained in \cite[Remark 8.1.2 p.~165]{HvNVW18}, this property is independent of $p$. By \cite[Theorem 8.1.3 p.~166]{HvNVW18}, an $R$-bounded subset of operators is bounded. Finally, a singleton $\{T\}$ is automatically $R$-bounded by \cite[Example 8.1.7 p.~170]{HvNVW18}. Actually, the notion of $R$-boundedness can be viewed as a natural and powerful substitute for the usual notion of boundedness, particularly useful when extending results that hold on Hilbert spaces to the more general setting of Banach spaces.

\paragraph{Sectorial operators}
For any angle $\omega \in (0,\pi)$, we introduce the open sector symmetric around the positive real half-axis with opening angle $2\omega$
\begin{equation}
\label{def-sigma-omega}
\Sigma_{\omega} 
\ov{\mathrm{def}}{=} \big\{ z \in \mathbb{C} \backslash \{ 0 \} : \: | \arg z | < \omega\big\}.
\end{equation}
See Figure 1. It will be useful to put $\Sigma_0 \ov{\mathrm{def}}{=} (0,\infty)$. 

\begin{figure}[ht]
\begin{center}
\includegraphics[scale=0.4]{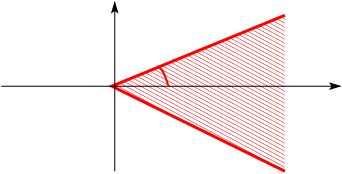}
\begin{picture}(0,0)

\put(-70,38){ $\omega$}

\put(-45,20){$\Sigma_\omega$}
\end{picture}

Figure 1: open sector $\Sigma_\omega$ of angle $\omega$
\end{center}
\label{figure-sector}
\end{figure}
We refer to the books \cite{Haa06} and \cite{HvNVW18} for background on sectorial operators and their $\H^\infty$ functional calculus, introduced in the seminal papers \cite{CDMY96} and \cite{McI86}. Let $A \co \dom A \subset X \to X$ be a closed densely defined linear operator acting on a Banach space $X$. We say that $A$ is a sectorial operator of type $\omega \in (0,\pi)$ if its spectrum $\sigma(A)$ is a subset of the closed sector $\ovl{\Sigma_\omega}$ and if the set  
$\big\{zR(z,A) : z \in \mathbb{C} - \overline{\Sigma_\omega}\big\}$ 
is bounded in the algebra $\B(X)$ of bounded operators acting on $X$\label{bounded-BX}, where $R(z,A) \ov{\mathrm{def}}{=} (z-A)^{-1}$\label{resolvent-operator} is the resolvent operator. The operator $A$ is said to be sectorial if it sectorial operator of type $\omega$ for some angle $\omega \in (0,\pi)$. In this case, we can introduce the angle of sectoriality
\begin{equation}
\label{equ-sectorial-operator}
\omega_{\sec}(A) 
\ov{\mathrm{def}}{=} \inf\{ \omega \in (0,\pi) : A \textrm{ is sectorial of type $\omega$} \}.
\end{equation}

\begin{example} \normalfont
\label{prop-sec-pi-2}
If $-A$ is the generator of a bounded strongly continuous semigroup $(T_t)_{t \geq 0}$ on a Banach space $X$ then by \cite[Example 10.1.2 p.~362]{HvNVW18} the operator $A$ is sectorial of type $\frac{\pi}{2}$, i.e.~$\omega_{\sec}(A) \leq \frac{\pi}{2}$.
\end{example}

Moreover, by \cite[Example 10.1.3 p.~362]{HvNVW18}, $A$ is sectorial of type $< \frac{\pi}{2}$ if and only if $-A$ generates a bounded holomorphic strongly continuous semigroup. If $A$ is a sectorial operator on a \textit{reflexive} Banach space $X$, we have by \cite[Proposition 2.1.1 (h) p.~21]{Haa06} or \cite[Proposition 10.1.9 p.~367]{HvNVW18} a topological decomposition
\begin{equation}
\label{decompo-reflexive}
X
=\ker A \oplus \ovl{\Ran A}.
\end{equation}

For any angle $\theta \in (0,\pi)$, we consider the algebra $\H^{\infty}(\Sigma_\theta)$\label{algebra-Hinfty} of all bounded analytic functions $f \co \Sigma_\theta \to \mathbb{C}$, equipped with the supremum norm 
\begin{equation}
\label{norm-Hinfty}
\norm{f}_{\H^{\infty}(\Sigma_\theta)}
\ov{\mathrm{def}}{=} \sup\bigl\{\vert f(z)\vert \, :\, z\in \Sigma_\theta\bigr\}.
\end{equation}
Let $\H^{\infty}_{0}(\Sigma_\theta)$\label{algebra-Hinfty0} be the subalgebra of bounded analytic functions $f \co \Sigma_\theta \to \mathbb{C}$ for which there exist $s,c>0$ such that 
\begin{equation*}
\label{ine-Hinfty0}
\vert f(z)\vert
\leq C\min\{|z|^s,|z|^{-s}\}, \quad z \in \Sigma_\theta,
\end{equation*} 
as discussed in \cite[Section 2.2]{Haa06}

\begin{example} \normalfont
\label{Example-Haase}
Let $\alpha$ and $\beta$ be complex numbers such that $0 < \Re \beta < \Re \alpha$. Then for any angle $\theta \in(0,\pi)$ the function $z \mapsto \frac{z^\beta}{(1+z)^\alpha}$ belongs  to the algebra $\H^{\infty}_{0}(\Sigma_\theta)$, as observed in \cite[Example 2.2.5 p.~29]{Haa06}.
\end{example}

Let $A$ be a sectorial operator acting on a Banach space $X$. Consider some angle $\theta \in (\omega_{\sec}(A), \pi)$ and a function $f \in \H^\infty_0(\Sigma_\theta)$. Following \cite[p.~30]{Haa06}, \cite[p.~5]{LM99} (see also \cite[p.~369]{HvNVW18}), for any angle $\nu \in (\omega_{\sec}(A),\theta)$ 
 we introduce the operator
\begin{equation}
\label{2CauchySec}
f(A)
\ov{\mathrm{def}}{=} \frac{1}{2\pi \i}\int_{\partial\Sigma_\nu} f(z) R(z,A) \d z,
\end{equation}
acting on $X$, with a Cauchy integral, where the boundary $\partial\Sigma_\nu$ is oriented counterclockwise. The sectoriality condition ensures that this integral is absolutely convergent and defines a bounded operator on the Banach space $X$. Using Cauchy's theorem, it is possible to show that this definition does not depend on the choice of the chosen angle $\nu$. The resulting map $\H^\infty_0(\Sigma_\theta) \to \B(X)$, $f \mapsto f(A)$ is an algebra homomorphism.

Following \cite[Definition 2.6 p.~6]{LM99} (see also \cite[p.~114]{Haa06}), we say that the operator $A$ admits a bounded $\H^\infty(\Sigma_\theta)$ functional calculus if the latter homomorphism is bounded, i.e., if there exists a constant $C \geq 0$ such that 
\begin{equation}
\label{Def-functional-calculus}
\norm{f(A)}_{X \to X} 
\leq C\norm{f}_{\H^\infty(\Sigma_\theta)},\quad f \in \H^\infty_0(\Sigma_\theta).
\end{equation}
In this context, we can introduce the $\H^\infty$-angle
\begin{equation*}
\label{angle-Hinfty}
\omega_{\H^\infty}(A) 
\ov{\mathrm{def}}{=} \inf\{\theta \in (\omega_{\sec}(A),\pi) : A \text{ admits a bounded $\H^\infty(\Sigma_\theta)$ functional calculus} \}.
\end{equation*}
If the operator $A$ has dense range and admits a bounded $\H^\infty(\Sigma_\theta)$ functional calculus, then the previous homomorphism naturally extends to a bounded homomorphism $f \mapsto f(A)$ from the algebra $\H^\infty(\Sigma_\theta)$ into the algebra $\B(X)$ of bounded linear operators on $X$. 

\begin{example} \normalfont
\label{Laplacian-funct}
If $1 < p < \infty$ and if $X$ is a $\UMD$ Banach space, then by \cite[Theorem 10.2.25 p.~391]{HvNVW18}, the Laplacian $-\Delta$ admits a bounded $\H^\infty(\Sigma_\theta)$ functional calculus on the Bochner space $\L^p(\R^n,X)$ for any angle $\theta >0$, i.e.~$\omega_{\H^\infty}(-\Delta)=0$.
\end{example}

\begin{example} \normalfont
\label{analytic-funct}
If $1 < p < \infty$ and if $\Omega$ is a measure space, then by \cite[Theorem 10.7.13 p.~462]{HvNVW18}, the generator $A$ of any bounded holomorphic strongly continuous semigroup $(e^{-tA})_{t \geq 0}$ of positive contractions on $\L^p(\Omega)$ admits a bounded $\H^\infty(\Sigma_\theta)$ functional calculus for some $\theta \in (0,\frac{\pi}{2})$.
\end{example}

\paragraph{Fractional powers}
See \cite{Haa06}, \cite{HvNVW23} and \cite{KuW04} for more information on fractional powers. Let $A$ be an injective sectorial operator on a Banach space $X$.  
Consider a complex number $\alpha \in \mathbb{C}$ such that $0 < \Re \alpha < 1$. By \cite[Proposition 3.2.1 p.~70]{Haa06} or \cite[p.~449]{HvNVW23}, we have
\begin{equation}
\label{for-frac-powers}
A^{-\alpha}
=\frac{\sin \pi \alpha}{\pi} \int_{0}^{\infty} t^{-\alpha} (t+A)^{-1}x \d t, \quad x \in \Ran A, 
\end{equation}
where the integral is an improper Bochner integral.

\paragraph{Bisectorial operators}
We refer to \cite{Ege15} and to the books \cite{HvNVW18} and \cite{HvNVW23} for more information on bisectorial operators, which can be seen as generalizations of unbounded selfadjoint operators in the framework of Banach spaces. For any angle $\omega \in (0,\frac{\pi}{2})$, we consider the open bisector $\Sigma_\omega^\bi \ov{\mathrm{def}}{=} \Sigma_\omega  \cup (-\Sigma_\omega)$ where the sector $\Sigma_{\omega}$ is defined in \eqref{def-sigma-omega} and $\Sigma_0^\bi \ov{\mathrm{def}}{=} (-\infty,\infty)$. Following \cite[Definition 10.6.1 p.~447]{HvNVW18}
, we say that a closed densely defined operator $D$ on a Banach space $X$ is bisectorial of type $\omega \in (0,\frac{\pi}{2})$ if its spectrum $\sigma(D)$ is a subset of the closed bisector $\ovl{\Sigma^\bi_{\omega}}$ and if the subset $\big\{z R(z,D) : z \not\in \ovl{\Sigma_{\omega}^\bi} \big\}$ is bounded in the space $\B(X)$ of bounded operators acting on $X$, where $R(z,D) \ov{\mathrm{def}}{=} (z-D)^{-1}$ denotes the resolvent operator. See Figure 1. The infimum of all $\omega \in (0,\frac{\pi}{2})$ such that $D$ is bisectorial is called the angle of bisectoriality of $D$ and denoted by $\omega_{\bi}(D)$. The definition of an $R$-bisectorial operator is obtained by replacing <<bounded>> by <<$R$-bounded>>.

\begin{center}
\begin{center}
\begin{tikzpicture}[scale=0.55]
\clip (-4.5,-2) rectangle (4.5,2);

\fill[pattern=north east lines, pattern color=blue!40, opacity=0.4]
  (0,0) -- (5,0) -- (5,5) -- (0,0);
\fill[pattern=north east lines, pattern color=blue!40, opacity=0.4]
  (0,0) -- (5,0) -- (5,-5) -- (0,0);

\fill[pattern=north east lines, pattern color=blue!40, opacity=0.4]
  (0,0) -- (-5,0) -- (-5,5) -- (0,0);
\fill[pattern=north east lines, pattern color=blue!40, opacity=0.4]
  (0,0) -- (-5,0) -- (-5,-5) -- (0,0);

\filldraw[fill=blue!20, draw=black, domain=0:2*pi, samples=200] 
  plot ({4*sqrt(2)*cos(\x r)/(1+sin(\x r)^2)}, 
        {1.5*sqrt(2)*cos(\x r)*sin(\x r)/(1+sin(\x r)^2)}) -- cycle;

\draw (1,0) arc (0:45:1);
\node at (1,0.8) [right] {$\omega$};

\draw (-4.5,0) -- (4.5,0);
\draw (0,-2) -- (0,2);
\draw (-2,2) -- (2,-2);
\draw (-2,-2) -- (2,2);

\node[right] at (-4.5,0.25) {$\sigma(D)$};

\node[right] at (2.5,1.3) {$\ovl{\Sigma^\bi_{\omega}}$};
\end{tikzpicture}
\end{center}

\label{Figure-1}

Figure 1: the spectrum $\sigma(D)$ of a bisectorial operator $D$
\end{center}
By \cite[p.~447]{HvNVW18}, a linear operator $D$ is bisectorial (resp.~$R$-bisectorial) if and only if 
\begin{equation}
\label{Def-R-bisectorial}
\i \R^* \subset \rho(D)
\quad \quad \text{and if the set} \quad
 \{t R(\i t, D) : t \in \R_+^* \} \text{ is bounded (resp.~$R$-bounded)}.
\end{equation}

\begin{example} \normalfont
\label{generators-bisectorial}
Let $D$ be a linear operator $D$ on a Banach space $X$. If the operator $\i D$ generates a strongly continuous group of operators on $X$ then by \cite[Example 10.6.3 p.~448]{HvNVW18} the operator $D$ is bisectorial of angle 0, i.e.~$\omega_{\bi}(D)=0$. Combined with Stone's theorem on one-parameter unitary groups \cite[Theorem G.6.2 p.~544]{HvNVW18}, we see in particular that a unbounded selfadjoint operator with dense domain on a Hilbert space is bisectorial of angle 0.
\end{example}

If $D$ is a bisectorial operator of type $\omega$ on a Banach space $X$ then by \cite[Proposition 10.6.2 (2) p.~448]{HvNVW18} its square $D^2$ is sectorial of type $2\omega$ and we have
\begin{equation}
\label{Bisec-Ran-Ker}
\ovl{\Ran D^2}
=\ovl{\Ran D}
\quad \text{and} \quad
\ker D^2
=\ker D.
\end{equation}
\paragraph{Functional calculus} Consider a bisectorial operator $D$ on a Banach space $X$ of angle $\omega_\bi(D)$. For any angle $\theta \in (\omega_\bi(D), \frac{\pi}{2})$ and any function $f$ in the space 
$$
\H^{\infty}_0(\Sigma_\theta^\bi) 
\ov{\mathrm{def}}{=}  \left\{ f \in \H^\infty(\Sigma_\theta^\bi) :\: \exists C,s > 0 \: \forall \: z \in \Sigma_\theta^\bi : \: |f(z)| \leq C \min\{|z|^s, |z|^{-s} \}  \right\},
$$
we can define a bounded operator $f(D)$ acting on the space $X$ by integrating on the boundary $\partial \Sigma^\bi_{\nu}$ of the bisector $\Sigma^\bi_{\nu}$ for some angle $\nu \in (\omega_\bi(D),\theta)$ using a Cauchy integral
\begin{equation}
\label{def-f(D)}
f(D)
\ov{\mathrm{def}}{=} \frac{1}{2\pi \i}\int_{\partial \Sigma^\bi_{\nu}} f(z)R(z,D) \d z.
\end{equation}
The integration contour is oriented counterclockwise, so that the interiors of the two sectors $\Sigma_\nu$ and $-\Sigma_\nu$ are always to its left. The integral in \eqref{def-f(D)} converges absolutely thanks to the decay of the function $f$ and is independent of the particular choice of the angle $\nu$ by Cauchy's integral theorem. We refer to \cite[Section 3.2.1]{Ege15} and \cite[Theorem 10.7.10 p.~449]{HvNVW18} for more complete explanations. 

The operator $D$ is said to admit a bounded $\H^\infty(\Sigma_\theta^\bi)$ functional calculus, if there exists a constant $C \geq 0$ such that 
\begin{equation}
\label{funct-cal-bisector}
\bnorm{f(D)}_{X \to X} 
\leq C \norm{f}_{\H^\infty(\Sigma_\theta^\bi)}, \quad f \in \H^{\infty}_0(\Sigma_\theta^\bi).
\end{equation}
The infimum  of all $\theta \in (\omega_\bi(D),\frac{\pi}{2})$ such that $D$ admits a bounded $\H^\infty(\Sigma_\theta^\bi)$ functional calculus is denoted by $\omega_{\H^\infty}(D)$.

\begin{example} \normalfont
\label{ex-bound-Hinfty}
By \cite[Theorem 10.7.10 p.~461]{HvNVW18}, if the operator $\i D$ is the generator of a bounded strongly continuous group on a $\UMD$ Banach space $X$ then the operator $D$ admits a bounded $\H^{\infty}(\Sigma_\theta^\bi)$ functional calculus for any angle $\theta > 0$, i.e.~$\omega_{\H^\infty}(D)=0$.
\end{example}

\begin{example} \normalfont
\label{ex-signe}
With the function $\sgn\ov{\mathrm{def}}{=} 1_{\Sigma_\theta}-1_{-\Sigma_\theta}$, we can define the bounded operator $\sgn D \co X \to X$.
\end{example}

The following result combines \cite[Proposition 2.3 p.~6]{NeV17} and \cite[Proposition 10.6.2 p.~448]{HvNVW18}. It establishes a connection between sectorial and bisectorial operators and their functional calculi.

\begin{prop}
\label{prop-bisectorial-to-sectorial}
Let $\omega \in (0,\frac{\pi}{2})$. 
\begin{enumerate}
	\item If $A$ is a bisectorial operator of type $\omega$ then the operator $A^2$ is sectorial of type $2\omega$.
	
	\item Suppose that $A$ is an $R$-bisectorial operator on a Banach space $X$ of finite cotype. Then $A^2$ is $R$-sectorial and the operator $A$ admits a bounded $\H^\infty(\Sigma_\omega^\bi)$-calculus if and only if $A^2$ admits a bounded $\H^\infty(\Sigma_{2\omega}^+)$-calculus.
\end{enumerate}
\end{prop}


\paragraph{$\K$-theory} 
%
%
%
%
%

We refer to the books \cite{Bl98}, \cite{CMR07}, \cite{Eme24}, \cite{GVF01}, \cite{RLL0} and \cite{WeO93} for more information and to \cite{BaS11} for a short introduction. We briefly recall the algebraic construction of the $\K_0$-group, which applies to Banach algebras as well as to more general rings.

\paragraph{$\mathrm{K}_0$-group}
Let $\cal{A}$ be a unital ring. We denote by $\M_{\infty}(\cal{A})$ the algebraic direct limit of the algebras $\M_n(\cal{A})$ under the embeddings $x \mapsto \diag(x,0)$. The direct limit $\M_{\infty}(\cal{A})$ can be thought of as the algebra of infinite-dimensional matrice with entries in $\cal{A}$, all but finitely many of which are zero. Two idempotents $e,f \in \M_{\infty}(\cal{A})$ are said to be (Murray--von Neumann) equivalent if there exist elements $v,w \in \M_{\infty}(\cal{A})$ such that $e=v w$ and $f=w v$. The equivalence class of an idempotent $e$ is denoted by $[e]$ and we introduce the set $\V(\cal{A}) \ov{\mathrm{def}}{=} \{ [e] : e \in \M_{\infty}(\cal{A}), e^2=e \}$ of all equivalence classes of idempotents in $\M_\infty(\cal{A})$.

According to \cite[pp.~3-4]{CMR07}, 
the set $\V(\cal{A})$ has a canonical structure of abelian semigroup, with the addition operation defined by the formula $[e]+[f] \ov{\mathrm{def}}{=} [\diag(e,f)]$. The identity of the semigroup $\V(\cal{A})$ is given by $[0]$, where $0$ is the zero idempotent. The group $\K_0(\cal{A})$ of the unital ring $\cal{A}$ is the Grothendieck group of the abelian semigroup $\V(\cal{A})$. In particular, each element of $\K_0(\cal{A})$ identifies with a formal difference $[e]-[f]$. Two such formal differences $[e_1] - [f_1]$ and $[e_2] - [f_2]$ are equal in $\K_0(\cal{A})$ if there exists $g \in \V(\cal{A})$ such that $[e_1] + [f_2] +  [g] = [f_1] + [e_2] +  [g]$.

Recall that by \cite[p.~4 and p.~9]{CMR07}, 
$\K_0$ is functorial in a natural way: if $\phi \co \cal{A} \to \cal{B}$ is a unital ring homomorphism, and $e \in \M_\infty(\cal{A})$ is an idempotent, then $\phi(e)$ is an idempotent in $\M_\infty(\cal{B})$, and the map $\V(\cal{A}) \to \V(\cal{B})$, $[e] \mapsto [\phi(e)]$ is well-defined. Hence it induces a group homomorphism $\phi_* \co \K_0(\cal{A}) \to \K_0(\cal{B})$ by the universal property of the Grothendieck group.

\begin{example} \normalfont 
If $\cal{A} = \C(\cal{X})$ is the algebra of continuous functions on a second countable compact Hausdorff topological space $\cal{X}$, then by \cite[Proposition 8.1.7 p.~305]{Eme24} we have an isomorphism $\K_0(\C(\cal{X})) \cong \K^0(\cal{X})$ where $\K^0(\cal{X})$ is the topological $\K$-theory group defined by vector bundles.
\end{example}

If $\cal{A}$ is a (not necessarily unital) ring, let $\tilde{\cal{A}}=\cal{A} \oplus \Z$
be its unitization ring, with the canonical unital homomorphism $\epsi \co \tilde{\cal{A}} \to \Z$,
$(a,\lambda) \mapsto \lambda$ and kernel $\cal{A}$. Following \cite[p.~11]{CMR07}, we set $\K_0(\cal{A}) \ov{\mathrm{def}}{=} \ker \epsi_* $ where $\epsi_* \co \K_0(\tilde{\cal{A}}) \to \K_0(\Z) \approx \Z$.

\paragraph{$\mathrm{K}_1$-group} We now recall the topological definition of $\K_1$ for Banach algebras. Let $\cal{A}$ be a unital Banach algebra. We denote by $\GL_n(\cal{A})$ the set of invertible matrices with entries in $\cal{A}$. It is a topological group. Using the embeddings $x \mapsto \diag(x,1)$, we can introduce the direct limit $\GL_\infty(\cal{A}) \ov{\mathrm{def}}{=} \varinjlim \GL_n(\cal{A})$, which is also a topological group. The group $\K_1(\cal{A}) \ov{\mathrm{def}}{=} \pi_0(\GL_\infty(\cal{A}))$ is defined\footnote{\thefootnote. If $X$ is a topological group, then $\pi_0(X) \cong X/X_0$, where $X_0$ denotes the path-connected component of the identity element, and $\pi_0(X)$ inherits a group structure as a quotient.} as the $0$-th homotopy group of the topological group $\GL_\infty(\cal{A})$.



%
%


\section{Banach Fredholm modules and Banach $\K$-homology}
\label{sec-K-homology}






In this section, we introduce a Banach space variant of the theory of $\K$-homology, relying on the notion of Fredholm module. Our aim is to replace Hilbert spaces by Banach spaces. 
Note that the classical notion of Fredholm module admits several variations in the literature (compare the references \cite[Definition 8.1.1 p.~199]{HiR00}, \cite[Definition 1 p.~293]{Con94} and \cite[Definition 2.2]{CGIS14}). After the first version of this work appeared, we realized that Lafforgue had implicitly introduced a similar notion in \cite[Definition 1.2.1 p.~16]{Laf02}, although with a different purpose, namely in connection with the Baum-Connes conjecture (see \cite{GJV19} for a survey). 
We begin with the following definition. 

\begin{defi}
\label{def-Fredholm-module-odd}
Let $\cal{A}$ be an algebra. An odd Banach Fredholm module $(X,\pi,F)$ over $\cal{A}$ on $X$ consists of a Banach space $X$ endowed with a representation $\pi \co \cal{A} \to \B(X)$, and a bounded operator $F \co X \to X$ such that
\begin{enumerate}
\item $F^2-\Id_X$ is a compact operator on $X$,
\item for any $a \in \cal{A}$ the commutator $[F, \pi(a)]$ is a compact operator on $X$.
\end{enumerate}
\end{defi}
Sometimes, we will use the notation $[F, a]$ for $[F, \pi(a)]$ and $\sim$ for the equality up to a compact operator., i.e., equality in the Calkin algebra. We also introduce a natural notion of \textit{even} Fredholm module.

\begin{defi}
\label{def-Fredholm-module-even}
Let $\cal{A}$ be an algebra. An even Banach Fredholm module $(X,\pi,F,\gamma)$ over $\cal{A}$ on $X$ consists of a Fredholm module $(X,\pi,F)$ endowed with a bounded operator $\gamma \co X \to X$ with $\gamma^2=\Id_X$ and such that
\begin{equation}
\label{commuting-rules}
F\gamma
=-\gamma F
\quad \text{and} \quad  
[\pi(a),\gamma]=0
\end{equation}
for any $a \in \cal{A}$. We say that $\gamma$ is the grading operator.
\end{defi}
In this case, $P \ov{\mathrm{def}}{=} \frac{\Id_X+\gamma}{2}$ is a bounded projection and we can write $X=X_+\oplus X_-$ where $X_+ \ov{\mathrm{def}}{=} \Ran P$ and $X_- \ov{\mathrm{def}}{=}\Ran (\Id_X-P)$. With respect to this decomposition, the equations of \eqref{commuting-rules} implies that we can write
\begin{equation}
\label{Fredholm-even}
\pi(a)
=\begin{bmatrix}
 \pi_+(a)    &  0 \\
  0   & \pi_-(a)  \\
\end{bmatrix}
\quad \text{and} \quad 
F=\begin{bmatrix}
  0   & F_-  \\
  F_+   & 0  \\
\end{bmatrix},
\end{equation}
where $\pi_+ \co \cal{A} \to \B(X_+)$ and $\pi_- \co \cal{A} \to \B(X_-)$ are representations of $\cal{A}$. In particular, we have
\begin{equation}
\label{commutator-12}
\left[F,\pi(a)\right]
=\begin{bmatrix}
  0   &  F_-\pi_-(a) -\pi_+(a)F_+ \\
  F_+\pi_+(a)-\pi_-(a)F_+   &  0 \\
\end{bmatrix}.
\end{equation}

\begin{remark} \normalfont
Note that if $X$ is a Hilbert space, these definitions are weaker generalizations of the notions of Fredholm modules of the previous references since the assumption of selfadjointness is not required.
\end{remark}

\begin{example}[from odd to even Fredholm Banach modules] \normalfont
Let $(X,\pi,F)$ be an odd Banach Fredholm module over $\cal{A}$. It is possible to construct an \textit{even} Banach Fredholm module $(X',\pi',F',\gamma)$ by letting $X' \ov{\mathrm{def}}{=}X \oplus X$, $\pi' \ov{\mathrm{def}}{=} \pi \oplus \pi$, $F' \ov{\mathrm{def}}{=} \begin{bmatrix}
   0  & F  \\
   F  &  0 \\
\end{bmatrix}$ and $\gamma\ov{\mathrm{def}}{=} \begin{bmatrix}
  -\Id_X   & 0  \\
  0   & \Id_X  \\
\end{bmatrix}$. Indeed, we have $(F')^2=\begin{bmatrix}
   F^2  & 0  \\
   0  &  F^2 \\
\end{bmatrix} \sim \begin{bmatrix}
   \Id_X  & 0  \\
   0  &  \Id_X \\
\end{bmatrix}$ (equality up to a compact operator). Moreover, the commutator
$$
\left[F',\pi'(a)\right]
=F'\pi'(a)-\pi'(a)F'
=\begin{bmatrix}
   0  & [F,\pi(a)]  \\ 
  [F,\pi(a)]   &  0 \\
\end{bmatrix}
$$
is compact. Furthermore, we observe that
$\gamma^2=\begin{bmatrix}
   \Id_X  & 0  \\
   0  &  \Id_X \\
\end{bmatrix}$, 
$$
F'\gamma
=\begin{bmatrix}
   0  & F  \\
   F  &  0 \\
\end{bmatrix}\begin{bmatrix}
  -\Id_X   & 0  \\
  0   & \Id_X  \\
\end{bmatrix}
=\begin{bmatrix}
   0  &  F \\
  -F   & 0  \\
\end{bmatrix}
=-\begin{bmatrix}
  -\Id_X   & 0  \\
  0   & \Id_X  \\
\end{bmatrix}\begin{bmatrix}
   0  & F  \\
   F  &  0 \\
\end{bmatrix}
=-\gamma F'
$$ 
and finally for any $a \in \cal{A}$
$$
\left[\pi'(a),\gamma\right]
=\begin{bmatrix}
   \pi(a)  & 0  \\
   0  &  \pi(a) \\
\end{bmatrix}\begin{bmatrix}
  -\Id_X   & 0  \\
  0   & \Id_X  \\
\end{bmatrix}-\begin{bmatrix}
  -\Id_X   & 0  \\
  0   & \Id_X  \\
\end{bmatrix}\begin{bmatrix}
   \pi(a)  & 0  \\
   0  &  \pi(a) \\
\end{bmatrix}
=0.
$$
\end{example}

The next definition is straightforward variation of \cite[Definition 8.2.1 p.~204]{HiR00}. It is important to note that in the next statement the isomorphism is not necessarily isometric.

\begin{defi}[isomorphic Banach Fredholm modules]
\label{}
Let $(X,\pi,F)$ be an odd Banach Fredholm module and let $U \co Y \to X$ be an isomorphism. Then $(Y,U^{-1}\pi U,U^{-1}FU)$ is also a Banach Fredholm module, and we say that it is isomorphic to $(X,\pi,F)$.
\end{defi}

There is a natural notion of direct sum for Banach Fredholm modules. One takes the direct sum $\oplus_2$ of the Banach spaces, of the representations, and of the operators. Actually, we may use any direct sum $\oplus_p$ with $1 \leq p \leq \infty$, since for any Banach spaces $X$ and $Y$, the spaces $X \oplus_p Y$ and $X \oplus_q Y$ are isomorphic for all $1 \leq p,q \leq \infty$. We also define the zero Banach Fredholm module with the zero Banach space, the zero representation and the zero operator.

Similarly to \cite[Definition 8.2.2 p.~204]{HiR00}, we introduce the following notion of equivalence.

\begin{defi}[operator homotopic Banach Fredholm modules]
\label{}
Suppose that $(X,\pi,F_t)$ is a family of odd Banach Fredholm modules parametrized by $t \in [0,1]$, in which the representation and the Banach space remain constant but the operator $F_t$ varies with $t$. If the function $[0,1] \to \B(X)$, $t \mapsto F_t$ is norm continuous, then we say that the family defines an operator homotopy between the odd Banach Fredholm modules $(X,\pi,F_0)$ and $(X,\pi,F_1)$, and that these two odd Fredholm modules are operator homotopic.
\end{defi}
 
In addition to operator homotopy, we shall also use the following definition.

\begin{defi}[compact perturbations]
Suppose that $(X,\pi,F)$ and $(X,\pi,F')$ are Banach Fredholm modules on the same Banach space $X$, and that $(F-F')\pi(a)$ is compact for all $a \in \cal{A}$  and $F-F'$ is compact. In this case, we say that $(X,\pi,F)$ is a compact perturbation of $(X,\pi,F')$.
\end{defi}

Compact perturbation implies operator homotopy since the linear path $t \mapsto (1-t)F+t F'$ from $F$ to $F'$ defines an operator homotopy. Finally, we need the following definition.

\begin{defi}[degenerated Banach Fredholm module]
\label{def:degenerate-Banach-Fredholm}
Let $\cal{A}$ be an algebra. A odd Fredholm module $(X,\pi,F)$ over $\cal{A}$ on $X$ is said to be degenerate if $F^2=\Id_X$ and if for any $a \in \cal{A}$ we have $[F, \pi(a)]=0$.
\end{defi}

Similar definitions for the even case are left to the reader. In the spirit of \cite[Definition 8.2.5 p.~205]{HiR00}, we introduce the next definition.


\begin{defi}[odd Banach $\K$-homology group]
Let $\scr{B}$ be a class of Banach spaces containing the zero space stable under finite sums and let $\cal{A}$ be an algebra. The Banach $\K$-homology group $\K^{1}(\cal{A},\scr{B})$ is the abelian group with one generator for each isomorphic equivalence class $[(X,\pi,F)]$ of Banach Fredholm modules over $\cal{A}$ on any $X$ of $\scr{B}$ subject only to the relations:
\begin{enumerate}
	\item if $(X,\pi,F_0)$ and $(X,\pi,F_1)$ are operator homotopic Fredholm modules then
$[(X,\pi,F_0)] =[(X,\pi,F_1)]$ in $\K^{1}(\cal{A},\scr{B})$,
	\item if $(X_0,\pi_0,F_0)$ and $(X_1,\pi_1,F_1)$ are any two Fredholm modules then 
	$$
	[(X_0,\pi_0,F_0) \oplus (X_1,\pi_1,F_1)] 
	=[(X_0,\pi_0,F_0)] + [(X_1,\pi_1,F_1)]
	$$
	in $\K^{1}(\cal{A},\scr{B})$.
	
	\item If $(X,\pi,F)$ is degenerate in the sense of Definition~\ref{def:degenerate-Banach-Fredholm}, then $[(X,\pi,F)]=0$ in $\K^{1}(\cal A,\scr B)$.
\end{enumerate}
\end{defi}

Similarly, we introduce a group $\K^{0}(\cal{A},\scr{B})$ in the even case. At the present time of the work, it is not clear if we need to require that a compact perturbation of a Fredholm module $(X,\pi,F)$  must give the same class than the one of $(X,\pi,F)$.

\begin{example} \normalfont
If $\scr{B}$ is the class $\scr{H}$ of Hilbert spaces, the previous groups $\K^{0}(\cal{A},\scr{H})$ and $\K^{1}(\cal{A},\scr{H})$ seems different from the K-homology groups $\K^0(\cal{A})$ and $\K^1(\cal{A})$ of \cite[Definition 8.2.5 p.~205]{HiR00} (defined for a separable $\C^*$-algebra $\cal{A}$) since the bounded operator $F$ of a Banach Fredholm module $(H,\pi,F)$ is not necessarily selfadjoint. 
\end{example}

\begin{example} \normalfont
For any $1 \leq p < \infty$, it is natural to consider the class $\scr{L}^p$ of spaces isomorphic to $\L^p$-spaces and the class $\mathscr{L}^p_{\nc}$ of spaces isomorphic to noncommutative $\L^p$-spaces. We could also consider the class $\SL_p$ of saces isomorphoc to subspaces of $\L^p$-spaces and the class $\SQL_p$ of spaces isomorphic to subspaces of quotients of $\L^p$-spaces. Finally, we can use the class $\mathrm{Ban}$ of all Banach spaces.
\end{example}

Identifying the two graded components through an isomorphism that intertwines both the representation and the even Fredholm operator allows us to obtain an odd Banach Fredholm module from an even module.
 
\begin{prop}[even to odd Banach Fredholm module]
\label{prop-even-to-odd-Fredholm}
Let $\cal{A}$ be an algebra. Consider an even Banach Fredholm module $(X,\pi,F,\gamma)$. Let $V \co X_- \to X_+$ be a surjective isomorphism such that
\begin{equation}
\label{div-986}
V F_+V 
=F_-,
\quad 
V\pi_-(a)
=\pi_+(a)V, \quad a \in \cal{A}.
\end{equation}
With the notations of \eqref{Fredholm-even} and $G \ov{\mathrm{def}}{=} V F_+ \co X_+ \to X_+$, the triple $(X_+,\pi_+,G)$ is an odd Banach Fredholm module on $\cal{A}$, where the Banach space $X_+$ is endowed with the homomorphism $\pi_+ \co \cal{A} \to \B(X_+)$.
\end{prop}

\begin{proof}
Since $F^2 \sim \Id_X$, we have $F_-F_+ \sim \Id_{X_+}$ and $F_+F_- \sim\Id_{X_-}$. We obtain that the operator
$$
G^2-\Id_{X_+}
=V F_+ V F_+-\Id_{X_+}
\ov{\eqref{div-986}}{=} F_- F_+-\Id_{X_+}
$$ 
is compact. Moreover, we have
\begin{align*}
\MoveEqLeft
\big[VF_+,\pi_+(a)\big]         
=VF_+ \pi_+(a)-\pi_+(a)VF_+ 
\ov{\eqref{div-986}}{=} VF_+ \pi_+(a)VV^{-1}-\pi_+(a)F_- V^{-1} \\
&\ov{\eqref{div-986}}{=} VF_+ V\pi_-(a)V^{-1}-\pi_+(a)F_- V^{-1} 
\ov{\eqref{div-986}}{=} F_-\pi_-(a)V^{-1}-\pi_+(a)F_- V^{-1}\\
&=(F_-\pi_-(a)-\pi_+(a)F_- )V^{-1},
\end{align*}
which is a compact operator by \eqref{commutator-12} and since the space of compact operators is an ideal.
\end{proof}

For the proof of Theorem \ref{th-Chern-coupling-odd}, we introduce the notation
\begin{equation}
\label{quantized-differential}
\qd x 
\ov{\mathrm{def}}{=}[F,x]
\end{equation}
(note that $\qd x 
\ov{\mathrm{def}}{=} \i[F,x]$ or $\qd x 
\ov{\mathrm{def}}{=} 2\pi \i[F,x]$ is better for some topics but we do not need it in this paper). A simple computation shows that
\begin{equation}
\label{prop-quantized-deriv}
\qd(ab)
=a\qd b+\qd a \cdot b 
\quad\text{and} \quad F\qd a \sim -\qd a F.
\end{equation}

We also need a notion of Banach Fredholm module such that the relation $F^2-\Id$ is compact is <<almost satisfied>> inspired by \cite[Definition 2.2 p.~4818]{CGIS14}.

\begin{defi}[possibly kernel-degenerate Banach Fredholm module]
\label{def-Fredholm-module}
Let $\cal{A}$ be an algebra. A possibly kernel-degenerate Banach Fredholm module $(X,\pi,F)$ over $\cal{A}$ on $X$ consists of a Banach space $X$ endowed with a representation $\pi \co \cal{A} \to \B(X)$, and a bounded operator $F \co X \to X$ such that
\begin{enumerate}
\item $X$ can be written $X = Y \oplus \ker F$ for some Banach space $Y$, i.e., the subspace $\ker F$ is complemented in $X$ by a bounded projection $Q \co X \to X$,
\item $Q(F^2-\Id_X)|_Y$ is compact on $Y$,
\item for any $a \in \cal{A}$ the commutator $Q[F, \pi(a)]|_Y$ is a compact operator on $Y$.
\end{enumerate}
\end{defi}

\section{Coupling between Banach K-homology with K-theory}
\label{sec-coupling}


This section establishes an explicit analytic index pairing between $\K$-theory and our Banach $\K$-homology groups defined by Banach Fredholm modules. In spirit it is the Banach-space analogue of index pairing in the Hilbert setting \cite[Ch.~IV.2]{Con94}. Given a Banach Fredholm module $(X,\pi,F)$, we show that the Toeplitz-type operators $P u P-(\Id-P)$ in the odd case and $e F_+e$ in the even case are Fredholm and yield a well-defined integer that depends only on the classes $[u]\in \K_1(\A)$ or $[e]\in \K_0(\A)$ and $[F] \in \K^0(\A,\scr B)$ (or $[F]\in \K^1(\A,\scr B)$. 


While Lafforgue's $\K\K^{\mathrm{ban}}$ provides a powerful bivariant framework and an abstract product formalism \cite{Laf02}, it does not supply a concrete Toeplitz-type formula for the pairing. The novelty here is precisely this explicit, calculable index formula in a non-Hilbert Banach context. The construction applies uniformly to classes $\scr{B}$ stable finite sums (e.g., the class $\scr{L}^p$ of $\L^p$-spaces up to isomorphism or the class $\mathrm{Ban}$ of all Banach spaces), and it will serve as the analytic backbone for the Banach-space Chern characters developed later in the paper.

We start with the odd case.

\begin{thm}[coupling, odd case]
\label{th-pairing-odd}
Let $\cal{A}$ be a unital algebra and let $\scr{B}$ be a class of Banach spaces stable under finite sums and containing the zero space. Consider an invertible $u$ of the matrix algebra $\M_n(\cal{A})$ and an odd Banach Fredholm module $(X,\pi,F)$ over $\cal{A}$ with $\pi(1)=\Id_X$. We introduce the bounded operators $P_n \ov{\mathrm{def}}{=} \Id \ot \frac{\Id+F}{2} \co X^{\oplus n} \to X^{\oplus n}$ and $u_n \ov{\mathrm{def}}{=} (\Id \ot \pi)(u) \co X^{\oplus n} \to X^{\oplus n}$. Then the bounded operator
\begin{equation}
\label{PuP-98}
P_n u_n P_n -(\Id-P_n) \co X^{\oplus n} \to X^{\oplus n}
\end{equation}
is Fredholm. If $X$ is in $\scr{B}$, its Fredholm index depends only on $[u] \in \K_1(\cal{A})$ and on $[F] \in \K^{1}(\cal{A},\scr{B})$.
\end{thm}

\begin{proof}
We only do the case $n=1$ and we use the shorthand notation $P$ for the map  $P_1=\frac{\Id+F}{2}$. The case $n > 1$ is left to the reader. Since 
\begin{equation}
\label{P2=P}
P^2
=\frac{\Id+2F+F^2}{4}
=\frac{\Id+2F}{4}+\frac{1}{4}F^2 
\sim \frac{\Id+2F}{4}+\frac{1}{4}\Id
=\frac{\Id+F}{2} 
=P.
\end{equation} 
the map $P$ is a projection modulo a compact operator. If $u \in \cal{A}$ is invertible, we have
\begin{align*}
\MoveEqLeft
\big[P \pi(u) P-(\Id-P)\big]\big[P \pi(u^{-1}) P-(\Id-P)\big]   \\
&=P \pi(u) \underbrace{P^2}_{\sim P} \pi(u^{-1})P
-P \pi(u) \underbrace{P(\Id-P)}_{\sim 0}
-\underbrace{(\Id-P)P}_{\sim 0} \pi(u^{-1}) P
+\underbrace{(\Id-P)^2}_{\sim \I-P} \\
&\sim P \pi(u) P \pi(u^{-1})P+(\Id-P)
=P \pi(u) \big[P \pi(u^{-1})-\pi(u^{-1})P+\pi(u^{-1})P\big] P+(\Id-P) \\
&=P \pi(u)\big([P, \pi(u^{-1})]+\pi(u^{-1})P\big) P+(\Id-P) \\
&=P \pi(u)\big(\tfrac{1}{2}[F, \pi(u^{-1})]+\pi(u^{-1})P\big) P +(\Id-P) \\
&\ov{\eqref{P2=P}}{\sim} \tfrac{1}{2}P \pi(u)[F, \pi(u^{-1})]P+P \pi(u u^{-1})P+(\Id-P) \\
&= \tfrac{1}{2}P \pi(u)\underbrace{[F, \pi(u^{-1})]}_{\textrm{compact}}P+P^2+(\Id-P)
\sim \Id.
\end{align*}
Similarly, we have $\big[P \pi(u^{-1}) P-(\Id-P)\big]\big[P \pi(u) P-(\Id-P)\big] \sim \Id$. By Atkinson's theorem (Theorem \ref{th-Atkinson}), we deduce that the bounded operator $P \pi(u) P-(\Id-P)$ is Fredholm. The other verifications are routine.
\end{proof}

So we have a pairing $\la \cdot, \cdot \ra_{\K_1(\cal{A}),\K^1(\cal{A},\scr{B})} \co \K_1(\cal{A}) \times \K^1(\cal{A},\scr{B}) \to \Z$ defined by
\begin{equation}
\label{pairing-odd-1}
\big\la [u], [(X,\pi,F)] \big\ra_{\K_1(\cal{A}),\K^1(\cal{A},\scr{B})} 
=\Index P_n u_n P_n -(\Id-P_n)
=\Index PuP-(\Id-P).
\end{equation}
If $P^2=P$, a similar proof shows that
\begin{equation}
\label{pairing-odd-2}
\big\la [u], [(X,\pi,F)] \big\ra_{\K_1(\cal{A}),\K^1(\cal{A},\scr{B})} 
=\Index P_n u_n P_n \co P_n(X^{\oplus n}) \to P_n(X^{\oplus n}).
\end{equation}

 Similarly, we prove the even case.

\begin{thm}[coupling, even case]
\label{th-pairing-even}
Let $\cal{A}$ be a unital algebra and let $\scr{B}$ be a class of Banach spaces stable under finite sums and containing the zero space. Consider a projection $e$ of the matrix algebra $\M_n(\cal{A})$ and an even Banach Fredholm module $\bigg(X_+ \oplus X_-,\pi,\begin{bmatrix}
  0   & F_-  \\
   F_+  &  0 \\
\end{bmatrix},\gamma\bigg)$ over $\cal{A}$ with $\pi(1)=\Id_X$. We introduce the bounded operator $e_n \ov{\mathrm{def}}{=} (\Id \ot \pi)(e) \co X^{\oplus n} \to X^{\oplus n}$. Then the bounded operator
\begin{equation}
\label{eFe-90}
e_n (\Id \ot F_+) e_n \co e_n(X_+^{\oplus n}) \to e_n(X_-^{\oplus n}) 
\end{equation}
is Fredholm. If $X$ is in $\scr{B}$, its Fredholm index depends only on $[e] \in \K_0(\cal{A})$ and on $[F] \in \K^{0}(\cal{A},\scr{B})$.
\end{thm}

\begin{proof}
We only do the case $n=1$ and we use the shorthand notation $e=e_1$. The case $n > 1$ is left to the reader. Note that the operators $F_-F_+-\Id_{X_+}$ and $F_+F_- - \Id_{X_-}$ are compact since $\begin{bmatrix}
   F_-F_+  &  0 \\
   0  &  F_+F_- \\
\end{bmatrix}
-\Id_X=F^2-\Id_X$ is compact. We have
\begin{align*}
\MoveEqLeft
(e F_+ e)(e F_- e) 
= e F_+ e F_- e 
=e (F_+e-eF_+ + e F_+)F_- e \\
&=e ([F_+,e] + e F_+)F_- e
=e \underbrace{[F_+,e]}_{\textrm{compact}}F_-e + e F_+F_- e
\sim  e.
\end{align*}
Similarly, we have $(e F_- e)(e F_+ e) \sim e$. By Atkinson's theorem (Theorem \ref{th-Atkinson}), we deduce that the bounded operator $eF_+e \co e(X_+) \to e(X_-)$ is Fredholm. The proof is complete.
\end{proof}

Consequently, we have a pairing $\la \cdot, \cdot \ra_{\K_0(\cal{A}),\K^0(\cal{A},\scr{B})} \co \K_0(\cal{A}) \times \K^0(\cal{A},\scr{B}) \to \Z$ defined by
\begin{equation}
\label{pairing-even-1}
\big\la [e], [(X,\pi,F,\gamma)] \big\ra_{\K_0(\cal{A}),\K^0(\cal{A},\scr{B})} 
=\Index e_n (\Id \ot F_+) e_n.
\end{equation}

\begin{remark} \normalfont
\label{non-unital-case}
Similarly to \cite[Proposition 2 p.~289]{Con94}, it is easy to extend these result to the non-unital case. For example, in the odd case, it suffices to replace $\M_n(\cal{A})$ by $\M_n(\tilde{\cal{A}})$ where $\tilde{\cal{A}}$ is the unitization of the non-unital algebra $\cal{A}$. 
\end{remark}

\section{Summability of Banach Fredholm modules}
\label{sec-summability}

\paragraph{Approximation numbers} In the classical Hilbert-space setting of noncommutative geometry, the notions of summability and dimension are defined in terms of singular values. In our Banach-space framework, however, one must rely on suitable generalizations of these quantities for bounded operators acting on Banach spaces. Following \cite[Definition p.~68]{Kon86}, we say that a map assigning to any bounded operator $T$ between Banach spaces a sequence $((s_n(T))_{n \geq 1}$ of real numbers is an $s$-number function if the following conditions are satisfied.
\begin{enumerate}
\item $\norm{T} \geq s_1(T) \geq s_2(T) \geq \cdots \geq 0$.
\item If $R,T \co X \to Y$ and if $n,m \geq 1$ then $s_{n+m-1}(R+T) \leq s_n(R)+s_n(T)$.
\item If $T \co X_0 \to X$, $S \co X \to Y$ and $R \co Y \to Y_0$ are bounded operators then we have
\begin{equation}
\label{majo-sn-2}
s_n(RST) 
\leq \norm{R} s_n(S) \norm{T}.
\end{equation}
\item We have $s_n(T)=0$ if $\rank T <n$ and $s_n(\Id_{\ell^2_n})=1$.
\end{enumerate}
This concept was introduced by Pietsch in \cite{Pie74}. If $X$ and $Y$ are Hilbert spaces and if $T \co X \to Y$ is a compact operator then by \cite[p.~69]{Kon01} these numbers coincide with the singular numbers of the operator $T$. We refer to the books \cite{Pie80}, \cite{Pie87}, \cite{Kon01} and to the survey papers \cite{Kon01}, \cite{Ger25} and \cite{Pie23} for more information. The history of this topic is equally described in \cite[Chapter 6]{Pie07}. The $s$-numbers are continuous functions since $|s_n(S)-s_n(T)| \leq \norm{S-T}$ for any bounded operators $S,T \co X \to Y$ by \cite[Theorem 1 p.~203]{Pie74}. Finally, when $Y$ is reflexive we have by \cite[Remark 7 p.~237]{EdT86} the equality
\begin{equation}
\label{duality-an}
a_n(T)
=a_n(T^*), \quad n \geq 1.
\end{equation}

Let us now describe a few examples.

\begin{example} \normalfont
The approximation numbers \cite[Definition I.d.14 p.~69]{Kon86} of a bounded linear operator $T \co X \to Y$ are defined by 
\begin{equation}
\label{def-approximation-numbers}
a_n(T) 
\ov{\mathrm{def}}{=} \inf\{\norm{T-R} : R \co X \to Y,\,  \rank R < n \}, \quad n \geq 1.
\end{equation}
These numbers quantify how well the map can be approximated by finite-dimensional linear operators. The Weyl numbers are defined by
\begin{equation}
\label{}
x_n(T)
\ov{\mathrm{def}}{=} \sup \{a_n(TA): A \co \ell^2\to X,\, \norm{A}=1\}, \quad n \geq 1.
\end{equation}
These sequences are example of $s$-number sequences by \cite[Lemma p.~69]{Kon86}. Moreover, we have 
\begin{equation}
\label{order-xn-an}
x_n(T) \leq a_n(T), \quad n \geq 1.
\end{equation}
Actually, the approximation numbers are the largest $s$-numbers by \cite[Lemma p.~69]{Kon86} or  \cite[Theorem 11.2.3 p.~146]{Pie80}. 
\end{example} 

We refer to \cite[Section 6.2]{Pie07} for more examples of $s$-number sequences. 
In Section \ref{sec-Dirac-Rn}, we will use the following observation.
\begin{lemma}
\label{lemma:approx-finite-direct-sum}
Suppose that $1 \leq p \leq \infty$. Consider two $\ell^p$-direct sums $
X \ov{\mathrm{def}}{=} \bigoplus_{k=1}^N X_k$ and $Y \ov{\mathrm{def}}{=} \bigoplus_{k=1}^N Y_k$ of Banach spaces and a block diagonal operator $T \ov{\mathrm{def}}{=} \bigoplus_{k=1}^N T_k \co X \to Y$, where each $T_k \co X_k \to Y_k$ is a bounded linear operator. Then for every integer $n \geq 1$ we have
\begin{equation}
\label{eq:an-direct-sum}
a_{Nn}(T)
\leq \max_{1 \leq k \leq N} a_n(T_k).
\end{equation}
\end{lemma}

\begin{proof}
Let $n \geq 1$ and $\epsi > 0$. For each $k \in \{1,\dots,N\}$ choose a finite-rank operator $S_k \co X_k \to Y_k$ with
$
\rank S_k < n
$ and $
\norm{T_k - S_k}
\leq a_n(T_k) + \epsi$. 
Consider the block diagonal operator $
S \ov{\mathrm{def}}{=} \bigoplus_{k=1}^N S_k \co X \to Y$. Then $\rank S \leq \sum_{k=1}^N \rank S_k < Nn$. Moreover, by the definition of the operator norm on the direct sum,
\[
\norm{T-S}
= \sup_{1 \leq k \leq N} \norm{T_k - S_k}
\leq \sup_{1 \leq k \leq N} \big(a_n(T_k) + \epsi\big)
= \max_{1 \leq k \leq N} a_n(T_k) + \epsi.
\]
Taking the infimum over all such families $(S_k)_{1 \leq k \leq N}$ and letting $\epsi \to 0$ yields \eqref{eq:an-direct-sum}. 
\end{proof}

\paragraph{Quasi-Banach ideals} 
The notion of operator ideals plays a central role in the study of compact operators. In the Banach space setting, using the notion of $s$-numbers, we introduce some (quasi-)ideals for measuring the <<degree of compactness>> of operators and for defining summability conditions analogous to those used in the classical noncommutative geometry. Following \cite[Definition I.d.1 p.~56]{Kon86}, a quasi-Banach ideal of operators $(\mathfrak{A},\alpha)$ is a subclass $\mathfrak{A}$ of all bounded linear operators between Banach spaces together with $\alpha \co \mathfrak{A} \to \R^+$ such that for all Banach spaces $X,Y$ the sets $\mathfrak{A}(X,Y) \ov{\mathrm{def}}{=} \mathfrak{A} \cap \B(X,Y)$ satisfy:
\begin{enumerate}
\item $\mathfrak{A}(X,Y)$ contains all finite-rank operators from $X$ to $Y$ and $\alpha(\Id_\mathbb{C}) = 1$.
\item $(\mathfrak{A}(X,Y),\alpha)$ is a quasi-Banach space with quasi-triangle constant $K$ independent of $X$ and $Y$, i.e.,
$$
\alpha(R + T)
\leq K (\alpha(R) + \alpha(T)), \quad R, T \in \mathfrak{A}(X,Y) 
$$
\item If $R \in \B(X_0,X)$, $T \in \mathfrak{A}(X,Y)$, $S \in \B(Y,Y_0)$ for some Banach spaces $X_0$, $Y_0$, then
\begin{equation}
\label{ideal}
STR \in \mathfrak{A}(X_0,Y_0)
\quad \text{and} \quad 
\alpha(STR) \leq 
\norm{S} \alpha(T)\norm{R}. 
\end{equation}
\end{enumerate}
If $\alpha$ is a norm on each $\mathfrak{A}(X,Y)$, i.e., $K = 1$, then we say that $(\mathfrak{A},\alpha)$ is a Banach ideal of operators and that $\alpha$ is the ideal norm. 
We also define $\mathfrak{A}(X) \ov{\mathrm{def}}{=} \mathfrak{A}(X,X)$.

\paragraph{Quasi-Banach ideals associated to $s$-numbers} Suppose that $0 < q < \infty$ and $0 < r \leq \infty$. Let $T \mapsto (s_n(T))_{n \geq 0}$ be an $s$-number sequence. For any Banach spaces $X$ and $Y$, following \cite[Definition 1.d.18 p.~72]{Kon86} we define the classes
\begin{equation}
\label{def-Sqs}
S^{q,r}_s(X,Y)
\ov{\mathrm{def}}{=} \big\{T \co X \to Y : (s_n(T)) \in \ell^{q,r} \big\}
\end{equation}
In the case $p=q$, we let
\begin{equation}
\label{def-Sqs-bis}
S^{q}_s(X,Y)
\ov{\mathrm{def}}{=} S^{q,q}_s(X,Y)=\big\{T \co X \to Y : (s_n(T)) \in \ell^q \big\}.
\end{equation}
Moreover, if $T \in S^{q,r}_s(X,Y)$, we let 
\begin{equation}
\label{def-quasi-norm-Spq}
\norm{T}_{S^{q,r}_s(X,Y)} 
\ov{\mathrm{def}}{=} \norm{(s_n(T))}_{\ell^{q,r}}
=\begin{cases}
\dsp\bigg(\sum_{n=1}^\infty n^{\frac{r}{q}-1}s_n(T)^r\bigg)^{\frac{1}{r}} & \text{if } 0 < r < \infty \\
\dsp \sup_{n \geq 1} n^{\frac{1}{q}} s_n(T)& \text{if } r =\infty \\
\end{cases}.
\end{equation}
By \cite[Lemma p.~72]{Kon86}, $(S^{q,r}_s,\norm{\cdot}_{S^{q,r}_s})$ is a quasi-Banach ideal. 
For the cases $s_n=a_n$ or if $s_n=x_n$ we will use the notations $S^{q,r}_\app$ and $S^{q,r}_\weyl$. If $X=Y$, we use the notation $S^{q,r}_s(X)$. If, in addition $X$ is equal to a Hilbert space $H$, we recover the Schatten class $S^q(H)$ with the space $S^q(X)$. 

If $0 < p < q < \infty $, we have obviously the contractive inclusion $S^p_s(X) \subset S^q_s(X)$ and $S^q_s(X) \subset S^{q,\infty}_s(X)$. Furthermore, if $\frac{1}{r}=\frac{1}{p}+\frac{1}{q}$, we have by \cite[Proposition 1.d.19 p.~73]{Kon86} the inclusion 
\begin{equation}
\label{composition-Schatten}
S^p_s(X) \circ S^q_s(X)
\subset S^r_s(X).
\end{equation}
Finally, if the Banach space $X$ has the approximation property, then by \cite[]{Pie80} we have $s_n(T) \to 0$ if and only if $T$ is compact.

\paragraph{Links with the trace} By \cite[p.~220]{Kon86}, the trace $\tr$ on the space of finite-rank operators acting on a Banach space $X$ admits a unique continuous extension on the space $S^1_\app(X)$, again denoted $\tr$. The same thing is true for $S^1_\weyl$ for the class of Banach spaces with the bounded approximation property. Moreover, by \cite[pp.~224-225]{Kon86}, the previously trace coincide with the sum of eigenvalues of $T$, i.e., $\tr T=\sum_{n=0}^{\infty} \lambda_n(T)$. Furthermore, if $T \in S^1_\app(X)$ and if $R \co X \to X$ is a bounded operator then by \cite[Corollary 2 p.~228]{Kon86}
\begin{equation}
\label{trace-prop}
\tr(TR)
=\tr(RT).
\end{equation}
It is worth noting that by \cite[Theorem 4.b.12 p.~245]{Kon86} a Banach space $X$ is isomorphic to a Hilbert space if and only if the space $\mathcal{N}(X)$ of nuclear operators acting on $X$ coincide with the space $S^1_\app(X)$.

\paragraph{Weyl's inequality} Recall Weyl's inequality for operators acting on Banach spaces. For that For this purpose, we first recall that a bounded operator $T \co X \to X$ is a Riesz operator \cite[Definition 1.a.3 p.~19]{Kon86} if
\begin{enumerate}
\item for all $\lambda \in \mathbb{C}- \{0\}$, $T- \lambda \Id_X$ is a Fredholm operator and has finite ascent and finite descent,

\item all non-zero spectral values $\lambda \in \sigma(T)$ are eigenvalues of finite multiplicity and have no accumulation point except possibly zero.
\end{enumerate}
Then by \cite[Theorem 2.a.6 p.~85]{Kon86} any bounded operator $T \in S^q_\weyl(X)$ for some $0 < q < \infty$ is a Riesz operator such that its sequence $(\lambda_n)$ of eigenvalues belongs to the Banach space $\ell^q$ and we have the inequality
\begin{equation}
\label{Weyl-inequality}
\norm{(\lambda_n)}_{\ell^q}
\leq 2^{\frac{1}{q}}\sqrt{2\e} \norm{T}_{S^q_\weyl(X)}.
\end{equation}

\paragraph{Summability of Banach Fredholm modules}

Here, we give a generalization of \cite[Definition 3 p.~290]{Con94}.

\begin{defi}
Suppose that $0 < q < \infty$. We say an odd Banach Fredholm module $(X,\pi,F)$ or an even Banach Fredhlom $(X,\pi,F,\gamma)$ over an algeba $\cal{A}$ is $q$-summable if the commutator $[F, \pi(a)]$ belongs to the space $S^q_\app(X)$ for any $a \in \cal{A}$. 
\end{defi}

By \eqref{composition-Schatten}, this implies that if $n+1 \geq q$ every product $[F,a_0][F,a_1] \cdots [F,a_n]$ of $n+1$ commutators belongs to the space $S^1_\app(X)$.

\begin{remark} \normalfont
Puschnigg demonstrated in \cite[Corollary 3.4 p.~230]{Pus11} that the reduced $\C^*$-algebras of lattices in higher-rank Lie groups do not admit any Fredholm modules that are $q$-summable for some $0 < q <\infty$ over them and which define non-zero classes of $\K$-homology.
\end{remark}

\section{The Chern character of a $q$-summable Banach Fredhlom module}
\label{sec-Chern}

In this section, we extend the notion of Chern character introduced by Connes to our setting. In the spirit of \cite[Definition 3 p.~295]{Con94} (see also \cite[Definition 4.13 p.~34]{CPR11}), we introduce the following definition.

\begin{defi}[odd case]
Consider a $q$-summable odd Banach Fredholm module $(X,\pi,F)$ for some $q > 0$ over an algebra $\cal{A}$. Let $n$ be an odd integer with $n+1 \geq q$. We define the Chern character by the formula
\begin{equation}
\label{Chern-odd}
\Ch_{n}^F(a_0,a_1,\dots,a_{n})
\ov{\mathrm{def}}{=} c_n \tr(F[F,a_0][F,a_1]\cdots[F,a_n]),\qquad a_0,a_1,\dots,a_n \in \cal{A},
\end{equation}
where $c_n$ is a constant (but the exact value is not useful for this paper).
\end{defi}

In the even case, the definition is slightly different.

\begin{defi}[even case]
Consider a $q$-summable even Banach Fredholm module $(X,\pi,F,\gamma)$ for some $q > 0$ over an algebra $\cal{A}$. Let $n$ be an even integer with $n+1 \geq q$. We define the Chern character by the formula
\begin{equation}
\label{}
\Ch_{n}^F(a_0,a_1,\dots,a_{n})
\ov{\mathrm{def}}{=} c_n \tr(\gamma F[F,a_0][F,a_1]\cdots[F,a_n]),\qquad a_0,a_1,\dots,a_n \in \cal{A},
\end{equation}
where we use a constant $c_n$ (but the exact value is not useful for this paper).
\end{defi}



We need the following Banach space generalization of \cite[p.~143]{GVF01}. 

\begin{prop}
\label{prop-Calderon}
Let $X$ be a Banach space. Let $T \co X \mapsto X$ be a Fredholm operator and let $R \co X \mapsto X$ be a bounded operator such that the operators $\Id-TR$ and $\Id-RT$ belong to the space $S^1_\app(X)$. Then
$$
\Ind T
=\tr(\Id-TR)-\tr(\Id-RT).
$$
\end{prop}

\begin{proof} 
Let $U$ be a generalized inverse of $T$ as in \eqref{generalized-inverse}. Then 
$$
R
=U+\underbrace{(\Id- UT)R -U(\Id - TR)}_{\ov{\mathrm{def}}{=} G}.
$$
Since $\Id-UT$ has finite rank and $\Id - TR$ belongs to the space $S^1_\app(X)$, it follows that $G$ also belongs to the space $S^1_\app(X)$. Observe that
$$
RT-TR
=(U+G)T-T(U+G)
= UT - TU + GT - TG.
$$
We note that $\tr GT 
\ov{\eqref{trace-prop}}{=} \tr TG$. Consequently, it suffices to prove the result for $R = U$.

Recall that $\Id-TU$ and $\Id-UT$ are projections of finite rank. Thus $\Id-TU$ and $\Id -UT$ belong to $S^1_\app(X)$. Hence $TU - UT$ belongs to $S^1_\app(X)$ and recalling that the trace of a finite rank projection is equal to the rank of the projection, we obtain
\begin{align*}
\MoveEqLeft
\tr(\Id - UT)- \tr(\Id-TU) 
=\rank(\Id - UT)- \rank (\Id - TU) \\
&=\dim \ker T-\dim( Y / \Ran T) 
=\Ind T.         
\end{align*}
\end{proof} 

\begin{remark} \normalfont
Let $T \co X \to X$ be a bounded operator. At the time of writing, it is unclear if the existence of a bounded operator $R \co X \mapsto X$ be a bounded operator such that the operators $\Id-TR$ and $\Id-RT$ belong to the space $S^1_\app(X)$ implies that $T$ is a Fredholm operator.
\end{remark}

The following is a Banach space generalization of \cite[Proposition 4.2 p.~144]{GVF01}.

\begin{cor}
\label{cor-useful-23}
Let $X$ be a Banach space. Let $T \co X \mapsto X$ be a Fredholm operator and let $R \co X \mapsto X$ be a bounded operator such that the operators $(\Id - TR)^n$ and $(\Id - RT)^n$ belong to the space $S^1_\app(X)$ for some integer $n \geq 1$. Then the index of the Fredholm operator $T$ is given by
\begin{equation}
\label{index-power-N}
\Ind T
= \tr(\Id-RT)^n-\tr (\Id-TR)^n.
\end{equation}
\end{cor}

\begin{proof}
We let $K \ov{\mathrm{def}}{=} \Id-RT$ and $L \ov{\mathrm{def}}{=} \Id-TR$. Note that $TK=T-TRT=LT$. We introduce the sum $S_n=\sum_{j=0}^{n-1}K^j R$. We have
$$
S_nT
=\sum_{j=0}^{n-1} K^j RT
=\sum_{j=0}^{n-1} K^j (\Id-K)
=\Id-K^n.
$$
Moreover, using $TK=LT$ in the second equality, we see that
$$
TS_n
=T\sum_{j=0}^{n-1}K^j R
=\sum_{j=0}^{n-1}L^j T R 
=\sum_{j=0}^{n-1}L^j (\Id-L)
=\Id-L^n.
$$
Hence the operators $\Id-S_nT=K^n=(\Id - RT)^n$ and $\Id-TS_n=L^n=(\Id-TR)^n$ belong to the space $S^1_\app$. Consequently, using Proposition \ref{prop-Calderon} in the first equality, we obtain
$$
\Ind T
=\tr(\Id-S_nT) -\tr(\Id-TS_n)
=\tr K^n- \tr L^n
= \tr(\Id-RT)^n-\tr (\Id-TR)^n.
$$
\end{proof}

Now we describe the pairing \eqref{pairing-odd-2} in the odd case.

\begin{thm}[odd case]
\label{th-Chern-coupling-odd}
Let $(X,\pi,F)$ be a $q$-summable odd Banach Fredholm module over the algebra $\cal{A}$ for some $q > 0$ with $F^2=\Id$. Consider a class $\scr{B}$ of Banach spaces stable under countable sums and containing the zero space. For any element $[u]$ of the group $\K_1(\cal{A})$ and any odd integer $n$ with $n+1 \geq q$, we have
\begin{equation}
\label{Chern-odd-123}
\big\la [u], [(X,\pi, F)] \big\ra_{\K_1(\cal{A}),\K^1(\cal{A},\scr{B})} 
=\frac{(-1)^{\frac{n+1}{2}}}{ 4^{\frac{n+1}{2}}c_n}  \Ch_n^F(u^{-1},u,\ldots,u^{-1},u).
\end{equation}
\end{thm}

\begin{proof} 
We let $N \ov{\mathrm{def}}{=} \frac{n+1}{2}$ and $P \ov{\mathrm{def}}{=} \frac{\I+F}{2}$. 
Note that
\begin{align*}
\MoveEqLeft
P[P,u^{-1}]P
=P(Pu^{-1}-u^{-1}P)P 
=P(Pu^{-1}+Pu^{-1}P-Pu^{-1}P-u^{-1}P)P \\
&=P(Pu^{-1}+u^{-1}P-Pu^{-1}-u^{-1}P)P 
=P[P,u^{-1}]P+P[P,u^{-1}]P.  
\end{align*}
Hence $P[P,u^{-1}]P=0$ which is equivalent to $P \d u^{-1}P=0$. Using this equality, in the last equality, we obtain
\begin{align}
\MoveEqLeft  
\label{div96}       
P - Pu^{-1}P u P
=Pu^{-1}uP - Pu^{-1}Pu P 
=-P(u^{-1}P-Pu^{-1})uP\\
&=P([P,u^{-1}])uP
=\frac{1}{2} P \d u^{-1} uP 
=\frac{1}{2} P \d u^{-1} (u P-Pu+Pu) \nonumber \\
&=-\frac{1}{4} P \d u^{-1} \d u+\frac{1}{2} P \d u^{-1} P u 
= -\frac{1}{4} P \d u^{-1} \d u. \nonumber
\end{align}
Since the the Banach Fredholm module is $q$-summable, the elements $\d u$ and $\d u^{-1}$ belongs to the space $S^q_\app(X)$. We deduce that $P - Pu^{-1}P u P$ belongs to $S^{\frac{q}{2}}_\app(X)$, hence to $S^{N}_\app(X)$ since $n+1 \geq q$. Similarly, we can prove that
\begin{equation}
\label{div90}
P - P u P u^{-1}P
=-\frac{1}{4} P \d u \d u^{-1}.
\end{equation}
Consequently, this operator also belongs to the space $S^{N}_\app(X)$. So we use Corollary \ref{cor-useful-23} with the operators $PuP \co P(X) \to P(X)$ and $Pu^{-1}P \co P(X) \to P(X)$ replacing $T$ and $R$.  Using the equalities $\tr (\d u^{-1} \d u)^N=\tr (\d u \d u^{-1})^{N}$ and $\tr F(\d u \d u^{-1})^N \ov{\eqref{prop-quantized-deriv}}{=} -\tr F (\d u^{-1} \d u)^{N}$, we obtain
\begin{align*}
\MoveEqLeft
\Index(PuP) 
\ov{\eqref{index-power-N}}{=} \tr(P - Pu^{-1}P u P)^{N} - \tr (P - P u P u^{-1}P)^{N} \\
&\ov{\eqref{div96} \eqref{div90}}{=} \tr(-\tfrac{1}{4} P \d u^{-1} \d u)^{N} - \tr(-\tfrac{1}{4} P \d u \d u^{-1})^{N} \\
&=\frac{(-1)^{N}}{4^{N}}\big[\tr( P \d u^{-1} \d u)^N - \tr( P \d u \d u^{-1})^{N}\big] 
=\frac{(-1)^{N}}{4^{N}}\big[\tr P (\d u^{-1} \d u)^N - \tr P (\d u \d u^{-1})^{N}\big] \\
&=\frac{(-1)^{N}}{4^{N}}\bigg[\tr \tfrac{\I+F}{2} (\d u^{-1} \d u)^{N} - \tr \tfrac{\I+F}{2} (\d u \d u^{-1})^{N}\bigg]\\
&=\frac{(-1)^{N}}{2 \cdot 4^n}\bigg[\tr (\d u^{-1} \d u)^{N} +\tr F(\d u^{-1} \d u)^{N}-\tr (\d u \d u^{-1})^{N}- \tr F (\d u \d u^{-1})^{N}\bigg] \\
&=\frac{(-1)^{N}}{4^{N}} \tr \big(F (\d u^{-1} \d u)^{N}\big) 
=\frac{(-1)^{N}}{4^{N}}\tr\big(F\underbrace{[F,u^{-1}][F,u]\cdots[F,u^{-1}]}_{2N=n+1}\big) \\
&\ov{\eqref{Chern-odd}}{=} \frac{(-1)^{\frac{n+1}{2}}}{ 4^{\frac{n+1}{2}}c_n} \Ch_{n}^F(u^{-1},u,\ldots,u^{-1},u).   
\end{align*}
%
\end{proof}

Now, we describe the pairing \eqref{pairing-even-1} in the even case. The proof is left to the reader as an easy exercise.

\begin{thm}[even case]
\label{Th-even-chern}
Let $(X,\pi,F,\gamma)$ be a $q$-summable even Banach Fredholm module over the algebra $\cal{A}$ for some $q > 0$ with $F^2=\Id$. Consider a class $\scr{B}$ of Banach spaces stable under countable sums and containing the zero space. For any element $[e]$ of the group $\K_0(\cal{A})$ and any even integer $n \geq q$, we have
\begin{equation}
\label{Chern-even-123}
\big\la [e], [(X,\pi, F,\gamma)] \big\ra_{\K_0(\cal{A}),\K^0(\cal{A},\scr{B})} 
=\frac{(-1)^{\frac{n}{2}}}{2 c_n} \Ch_n^F(e,e,\ldots,e).
\end{equation}
\end{thm}

\section{From Banach spectral triples to Banach Fredhlom modules}
\label{sec-from}


The following definition is extracted of \cite{ArK22}. 

\begin{defi}[Lipschitz algebra]
\label{Def-Lipschitz-algebra}
Consider a triple $(\cal{A},X,D)$ constituted of the following data: a Banach space $X$, a closed unbounded operator $D$ on $X$ with dense domain $\dom D \subset X$, and an algebra $\cal{A}$ endowed with a homomorphism $\pi \co \cal{A} \to \B(X)$. In this case, we define the Lipschitz algebra
\begin{align}
\label{Lipschitz-algebra-def}
\MoveEqLeft
\Lip_D(\cal{A}) 
\ov{\mathrm{def}}{=} \big\{a \in \cal{A} : \pi(a) \cdot \dom D \subset \dom D 
\text{ and the unbounded operator } \\
&\qquad \qquad  [D,\pi(a)] \co \dom D \subset X \to X \text{ extends to an element of } \B(X)\big\}. \nonumber
\end{align} 
\end{defi}

We refer to \cite[Section 5.7]{ArK22} for some properties of $\Lip_D(\cal{A})$. If $\cal{B}$ is a subalgebra of $\cal{A}$, observe that we have $\Lip_D(\cal{B})= \Lip_D(\cal{A}) \cap \cal{B}$. 

The following definition is essentially \cite[Definition 5.10 p.~218]{ArK22}. Here $\rho(D) \ov{\mathrm{def}}{=} \mathbb{C}\backslash \sigma(D)$ is the resolvent set of the operator $D$. 

\begin{defi}[compact Banach spectral triple]
\label{Def-Banach-spectral-triple}
Consider a triple $(\cal{A},X,D)$ constituted of the following data: a reflexive\footnote{\thefootnote. It may perhaps be possible to replace the reflexivity by an assumption of weak compactness, see \cite[p.~361]{HvNVW18}.} Banach space $X$, a bisectorial operator $D$ on $X$ with dense domain $\dom D \subset X$, and a Banach algebra $\cal{A}$ equipped with a continuous homomorphism $\pi \co \cal{A} \to \B(X)$. We say that $(\cal{A},X,D)$ is a compact Banach spectral triple if
\begin{enumerate}
\item{} $D$ admits a bounded $\H^\infty(\Sigma^\bi_\theta)$ functional calculus on a bisector $\Sigma^\bi_\theta$ for some angle $\theta \in (\omega_{\bi}(D),\frac{\pi}{2})$,
\item{} $D$ has compact resolvent, i.e.~there exists $\lambda \in \rho(D)$ such that $R(\lambda,D)$ is compact,
\item{} the subset $\Lip_D(\cal{A})$ is dense in the algebra $\cal{A}$.
\end{enumerate}
If the second condition is replaced by the weaker assumption <<$D^{-1}$ is a compact operator on $\ovl{\Ran D}$>>, the corresponding structure is referred to as a possibly kernel-degenerate compact Banach spectral triple.
If in addition, there exists an isomorphism $\gamma \co X \to X$, with $\gamma^2=\Id$, 
\begin{equation}
\label{grading-ST}
D\gamma=-\gamma D 
\quad \text{and} \quad 
\gamma \pi(a)=\pi(a)\gamma, \quad a \in \cal{A},
\end{equation}
the spectral triple is said to be even or graded. Otherwise, it is said to be odd or ungraded. 
\end{defi}

Consider a compact Banach spectral triple $(\cal{A},X,D)$. In this situation, since $X$ is reflexive, \cite[p.~448]{HvNVW18} yields the direct sum decomposition 
\begin{equation}
\label{decompo-space-bisect}
X
=\ovl{\Ran D} \oplus \ker D.
\end{equation} 
Moreover, by \cite[Theorem 6.29 p.~187]{Kat76}, the spectrum of $D$ consists entirely of isolated eigenvalues with finite multiplicities and $R(\lambda,D)$ is compact for any $\lambda \in \rho(D)$.

Now, we introduce of summability similar to the hilbertian context. Recall that for any $0<q<\infty$ the spaces $S^q_\app(X)$ and $S^{q,\infty}_\app(X)$ are defined in \eqref{def-Sqs}. The operator $|D|^{-1}$ is well-defined on $\ovl{\Ran D}$. Furthermore, we can extend it by letting $|D|^{-1}=0$ on $\ker D$. 

\begin{defi}[summability]
\label{summability-spectral-triple}
Suppose that $0 < q <\infty$. We say that a possibly kernel-degenerate compact spectral triple $(\cal{A},X,D)$ is $q$-summable if the operator $|D|^{-1}$ belongs to $S^q_\app(X)$. The triple is said to be $q^+$-summable if $|D|^{-1}$ belongs to $S^{q,\infty}_\app(X)$. In this case, the spectral dimension of the spectral triple is defined by
\begin{equation}
\label{Def-spectral-dimension}
\dim(\cal{A},X,D)  
\ov{\mathrm{def}}{=} \inf\big\{ q > 0 : |D|^{-1} \in S^{q,\infty}_\app(X) \big\}.
\end{equation}
\end{defi}

The following definition is essentially \cite[Definition 5.11 p.~2123]{ArK22}. Compare to \cite[Definition 2.1 p.~33]{CGRS14}. We can also introduce variants for algebras which are not Banach algebras. These extensions are left to the reader. 

\begin{defi}[locally compact Banach spectral triple]
\label{Def-Banach-spectral-triple}
Consider a triple $(\cal{A},X,D)$ constituted of the following data: a reflexive Banach space $X$, a bisectorial operator $D$ on $X$ with dense domain $\dom D \subset X$, and a Banach algebra $\cal{A}$ equipped with a continuous homomorphism $\pi \co \cal{A} \to \B(X)$. We say that $(\cal{A},X,D)$ is a locally compact Banach spectral triple if
\begin{enumerate}
\item{} $D$ admits a bounded $\H^\infty(\Sigma^\bi_\theta)$ functional calculus on a bisector $\Sigma^\bi_\theta$ for some angle $\theta \in (\omega_{\bi}(D),\frac{\pi}{2})$,
\item{} $\pi(a)(D+\i)^{-1}$ is a compact operator.
\item{} The subset $\Lip_D(\cal{A})$ is dense in the algebra $\cal{A}$.
\end{enumerate}
If the second condition is replaced by the weaker assumption <<$\pi(a)|D|^{-1}$ is a compact operator>>, the corresponding structure is referred to as a possibly kernel-degenerate locally compact Banach spectral triple.
\end{defi}
In particular, a locally compact Banach spectral triple where the algebra $\cal{A}$ is unital with $\pi(1)=\Id_X$ is a \textit{compact} Banach spectral triple.

In this case, the spectral dimension of the locally compact Banach spectral triple is defined by
\begin{equation}
\label{Def-spectral-dimension}
\dim(\cal{A},X,D)  
\ov{\mathrm{def}}{=} \inf\big\{ q > 0 : \text{for any $a \in \Lip_D(\cal{A}) $ we have }\pi(a)(D+\i)^{-1} \in S^{q,\infty}_\app(X) \big\}.
\end{equation}
In the degenerated situation, we need to replace $(D+\i)^{-1}$ by $|D|^{-1}$.


\paragraph{Strong convex compactness property} Recall that compactness is preserved under strong integrals by \cite[Theorem 1.3 and Remark 1.2 b) p.~260]{Voi92}. More precisely, for any measure space $(\Omega,\mu)$ and any strongly measurable function $T \co \Omega \to \cal{K}(X,Y)$, with values in the space $\cal{K}(X,Y)$ of compact operators from $X$ into $Y$, such that the integral $\int_\Omega \norm{T(\omega)} \d \mu(\omega)$ is finite, the strong integral $\int_\Omega T(\omega) \d \mu(\omega)$ is a compact operator. 

Proposition \ref{prop-triple-to-Fredholm} is a variant of \cite{BaJ83}, \cite[Proposition 2.4 p.~4818]{CGIS14} and \cite[Proposition 4.4]{CPR11}, unfortunately, with a level of rigor that remains somewhat incomplete. The proof is in the same spirit as the one in the Hilbert space setting. We need the two following elementary lemma. The first lemma is proved with \cite[Lemma~5.1.2]{Haa06}.

\begin{lemma}
\label{lemma-limit-sectorial}
Let $A$ be a sectorial operator on a reflexive Banach space $X$ such that $0$ is an isolated spectral point. Assume that $A$ admits a bounded $\H^\infty(\Sigma_\theta)$ functional calculus for some $\theta\in(\omega_{\sec}(A),\pi)$. Then in the operator norm we have
\begin{equation}
\label{eq:limit-sectorial-Hinfty}
s^2(s^2+A)^{-1}
\xrightarrow[s \to 0]{} P_{\ker A},
\end{equation}
where $P_{\ker A}$ is the projection onto $\ker A$ with respect to the decomposition \eqref{decompo-reflexive}.
\end{lemma}

The second lemma is proved with a bisectorial version of \cite[Lemma~5.1.2]{Haa06}.

\begin{lemma}
\label{lemma-limit-bisector-D}
Let $D$ be bisectorial operator on a reflexive Banach space $X$ admitting a bounded $\H^\infty(\Sigma_\theta^\bi)$ functional calculus with compact resolvent and let $P_{\ker D}$ be the projection onto $\ker D$ with respect to the decomposition \eqref{decompo-space-bisect}. Then in operator norm we have
\begin{equation}
\label{eq:bis-limit}
D(s^2+D^2)^{-1}
\xrightarrow[s\to 0]{} \sgn(D) |D|^{-1},
\end{equation}
where $|D|^{-1}$ is extended by $0$ on $\ker D$.
\end{lemma}

Now, we can prove the main result of this section. 

\begin{prop} 
\label{prop-triple-to-Fredholm}
Let $(\cal{A},X, D)$ be a (possibly kernel-degenerate) compact Banach spectral triple over an algebra $\cal{A}$ via a homomorphism $\pi \co \cal{A} \to \B(X)$. Then $(X,\pi,\sgn D)$ is a (possibly kernel-degenerate) Banach Fredholm module over $\cal{A}$. 
\end{prop}

\begin{proof}
First, assume that $(\cal{A},X, D)$ is compact Banach spectral triple. Using the boundedness of the functional calculus of the operator $D$, we can introduce the bounded operator $F \ov{\mathrm{def}}{=} \sgn D$ with Example \ref{ex-signe}. Since $F^2 -\Id_X=-P_{\ker D}$, with $P_{\ker D}$ the projection onto $\ker D$ vanishes on $\ovl{\Ran D}$ by construction, the first requirement in the Definition \ref{def-Fredholm-module-odd} of a Banach Fredholm module hold true.

For any $a \in \Lip_D(\cal{A})$, we have
\begin{align}
\MoveEqLeft
\label{inter-33R}
\left[ F , a \right]       
=[D|D|^{-1},a] =D|D|^{-1} a - a D|D|^{-1} \\
&= (Da-aD)|D|^{-1}+D(|D|^{-1}a-a|D|^{-1}) 
=[D,a]|D|^{-1}+D[|D|^{-1},a]. \nonumber
\end{align}
Observe that the operator $[D,a]|D|^{-1}$ is compact. It suffices to prove that the operator $D[|D|^{-1},a]$ is compact. 

The operator $D$ is bisectorial and admits a bounded $\H^\infty(\Sigma_\theta^\bi)$ functional calculus. Consequently, by Proposition \ref{prop-bisectorial-to-sectorial} the operator $D^2$ is sectorial and admits a bounded $\H^\infty(\Sigma_\theta)$ functional calculus. Recall that the functional calculus of the operators $D$ and $D^2$ are compatible. So using the change of variables $t=s^2$ in the last equality, we obtain for any $x \in \Ran D$ 
\begin{align}
\MoveEqLeft
\label{div-890}
|D|^{-1}x
=(D^2)^{-\frac{1}{2}}x 
\ov{\eqref{for-frac-powers}}{=} \frac{1}{\pi} \int_{0}^{\infty} (t+D^2)^{-1}x \frac{\d t}{\sqrt{t}} 
=\frac{2}{\pi}\int_0^\infty (s^2+D^2)^{-1}\d s.         
\end{align}
Consequently, we have
$$
D[|D|^{-1},a]
\ov{\eqref{div-890}}{=} D\bigg(\frac{2}{\pi} \int_{0}^{\infty} (s^2+D^2)^{-1}\d s,a\bigg]
=D\bigg(\frac{2}{\pi} \int_{0}^{\infty} [(s^2+D^2)^{-1},a] x\d s\bigg).
$$
We want pass the operator $D$ through the left integral by using Proposition \ref{prop-Hille}. Note that the same proof than the one of \cite[Lemma 2.3 p.~685]{CaP98} provides the formula
\begin{equation}
\label{formule-Carey}
[(s^2+D^2)^{-1},a]
=D(s^2+D^2)^{-1}[a,D](s^2+D^2)^{-1}+(s^2+D^2)^{-1}[a,D]D(s^2+D^2)^{-1}.
\end{equation}
Consequently, the integral $\int_{0}^{\infty} D[(s^2+D^2)^{-1},a] \d s$ identifies to the integral
\begin{equation}
\label{integral-cv}
\int_{0}^{\infty} \bigg(\underbrace{D^2(s^2+D^2)^{-1}[a,D](s^2+D^2)^{-1}}_{K_1(s)}
+\underbrace{D(s^2+D^2)^{-1}[a,D]D(s^2+D^2)^{-1}}_{K_2(s)}\bigg) \d s,
\end{equation}
Using the computation 
\begin{equation}
\label{inter-6778}
\Id -D^2(s^2+D^2)^{-1}
=(s^2+D^2)(s^2+D^2)^{-1}-D^2(s^2+D^2)^{-1}
=s^2(s^2+D^2)^{-1},
\end{equation}
we see that the last integral is equal to
$$
\int_{0}^{\infty} \bigg(\big(\Id-s^2(s^2+D^2)^{-1}\big)[a,D](s^2+D^2)^{-1}
+D(s^2+D^2)^{-1}[a,D]D(s^2+D^2)^{-1}\bigg) \d s.
$$
Note that the operator $D(s^2+D^2)^{-1}$ identifies to $f_s(D^2)$ where
$$
f_s(z)
\ov{\mathrm{def}}{=} \frac{z^{\frac{1}{2}}}{s^2+z}
=\frac{1}{s} \frac{(\frac{z}{s^2})^{\frac{1}{2}}}{1+\frac{z}{s^2}}, \quad z \in \Sigma_\theta. 
$$
By Example \ref{Example-Haase}, this function belongs to the algebra $\H^{\infty}_{0}(\Sigma_\theta)$ with $\norm{f_s}_{\H^\infty(\Sigma_\theta)} \lesssim \frac{1}{s}$. Using the $\H^\infty$ functional calculus of the sectorial operator $D^2$, we infer that the operator $D(s^2+D^2)^{-1}$ is bounded with 
$$
\norm{D(s^2+D^2)^{-1}}
\lesssim \frac{1}{s}.
$$ 
Similarly, we have $\norm{D^2(s^2+D^2)^{-1}} \lesssim \frac{1}{s^2}$. 

\textit{Norm–integrability at $\infty$.}
For any $s \geq 1$, we have
\[
\|K_1(s)\|
\leq \norm{s^2(s^2+D^2)^{-1}} \norm{[a,D]} \norm{(s^2+D^2)^{-1}}
\lesssim \frac{1}{s^2},
\]
and similarly
\[
\norm{K_2(s)}
\leq \norm{D(s^2+D^2)^{-1}} \norm{[a,D]} \norm{D(s^2+D^2)^{-1}}
\lesssim \frac{1}{s^2}.
\]
Hence $\int_1^\infty \norm{K_1(s)}+\norm{K_2(s)} \d s 
<\infty$.

\textit{Behaviour at $s=0$.} Using the projection $P_{\ker D}$, we deduce by \eqref{Bisec-Ran-Ker}, Lemma \ref{lemma-limit-sectorial} and Lemma \ref{lemma-limit-bisector-D} the norm limits in $\B(X)$
\begin{equation*}
\label{eq:limits-at-0}
s^2(s^2+D^2)^{-1}
\xrightarrow[s\to 0]{} P_{\ker D}
\quad \text{and} \quad
D(s^2+D^2)^{-1}
\xrightarrow[s\to 0]{} \sgn(D) |D|^{-1},
\end{equation*}
where $|D|^{-1}$ is extended by $0$ on $\ker D$.

We conclude that the integral \eqref{integral-cv} is norm convergent. Using Proposition \ref{prop-Hille}, we conclude that
$$
D[|D|^{-1},a]
=\int_{0}^{\infty} \bigg(D^2(s^2+D^2)^{-1}[a,D](s^2+D^2)^{-1}
+D(s^2+D^2)^{-1}[a,D]D(s^2+D^2)^{-1}\bigg) \d s.
$$
Note that the operator $[a,D]$ is bounded and that the operator $(s^2+D^2)^{-1}=-R(\i s,D)R(-\i s,D)$ is compact. Since the space $\cal{K}(X)$ of compact operators acting on $X$ has the strong convex compactness property, we conclude that the operator $[F,a]$ is compact. Let $a \in \cal{A}$. There exists a sequence $(a_k)$ of elements of $\Lip_D(\cal{A})$ such that $\pi(a_k) \to \pi(a)$. Each operator $[\sgn D,a_k]$ is compact by the first part of the proof. Since the operator product is continuous in $\B(X)$, we have $[\sgn D,a_k] \to \left[\sgn D, a\right]$  We conclude that the operator $\left[\sgn D, a\right]$ is also compact.

If $(\cal{A},X, D)$ is kernel-degenerate, replacing $X$ by $\ovl{\Ran D}$, we obtain the conclusion.
\end{proof}

\begin{remark} \normalfont
A more simple proof can be given with the additional assumption that $[|D|, a]$ defines a bounded operator on the Banach space $X$ for all $a \in \Lip_D(\cal{A})$. Indeed, 
for any $a \in \Lip_D(\cal{A})$, we have 
\begin{align}
\MoveEqLeft
\label{formula-simple}
\left[\sgn D, a\right]
=D|D|^{-1}a - a D|D|^{-1} 
=(Da-aD)|D|^{-1}+D(|D|^{-1}a- a|D|^{-1}) \\
&=[D,a]|D|^{-1}+D|D|^{-1} (a|D|-|D|a)|D|^{-1} 
=[D,a]|D|^{-1}-\sign(D) [|D|,a]|D|^{-1} \nonumber \\
&=\big([D,a]-\sgn(D)[|D|,a]|\big)|D|^{-1}. \nonumber
\end{align}
This operator is compact as a product of a bounded operator and a compact operator.
\end{remark}

The next result is a variant of \cite[Lemma 2.4 p.~35]{CGRS14}, \cite[Lemma 10.18 p.~462]{GVF01}, \cite[Proposition 5.9]{CPR11} and \cite[Proposition 2.4 p.~686]{CaP98} \cite{SWW98}.

\begin{prop}
\label{prop-q-summable-Fredhom-123}
Let $q > 0$. Suppose that $(\cal{A},X, D)$ is a $q$-summable Banach spectral triple over an algebra $\cal{A}$ via a homomorphism $\pi \co \cal{A} \to \B(X)$. Assume that $[|D|, a]$ induces a well-defined bounded operator for any $a \in \cal{A}$. Then $(X,\pi, \sgn D)$ is a $q$-summable Banach Fredholm module over $\cal{A}$.
\end{prop}

\begin{proof}
Let $a \in \cal{A}$. If $R \ov{\mathrm{def}}{=} [D,a]-\sgn(D)[|D|,a]|$, we have
$$
\left[\sgn D, a\right]
\ov{\eqref{formula-simple}}{=} R|D|^{-1}. \nonumber
$$
By the ideal property \eqref{ideal}, the operator $R|D|^{-1}$ belongs to the space $S^{q}_\app(X)$.
\end{proof}

\begin{remark} \normalfont
Let $(\cal{A},X, D)$ be a $q$-summable compact Banach spectral triple for some $q > 0$. If $S^{q}_\app(X)$ has the strong convex compactness property then the proof of Proposition \ref{prop-triple-to-Fredholm} shows that the Banach Fredholm module $(X,\pi, \sgn D)$ is $q$-summable if we can extend the Bochner integral to quasi-Banach spaces.
\end{remark} 
\section{Examples of triples and Banach Fredholm modules}
\label{Examples-1}

We describe in details some examples. We also refer to \cite{Arh26} for the case of Hodge-Dirac operators on oriented compact Riemannian manifolds.

\subsection{The Dirac operator $\frac{1}{\i}\frac{\d}{\d \theta}$ on the space $\L^p(\T)$ and the periodic Hilbert transform}
\label{sec-dtheta-T}


Suppose that $1 < p < \infty$. Here $\T \ov{\mathrm{def}}{=} \{z \in \mathbb{C}: |z|=1\}$ is the one-dimensional torus. The functions defined on $\T$ can be identified with periodic functions on $\R$ with period $2\pi$. In particular, the algebra of continuous complex functions on $\T$ can be written as
$$
\C(\T)
= \{ f \in \C(\R) : f(\theta + 2\pi) = f(\theta) \}.
$$
Recall that by \cite[Proposition 4.10 p.~31]{EnN00} any multiplication operator $M_f \co \L^p(\T) \to \L^p(\T)$, $g \mapsto fg$ by a measurable function $f \co \T \to \mathbb{C}$ is bounded if and only if $f$ is essentially bounded. In this case, we have $\norm{M_f}_{\L^p(\T) \to \L^p(\T)}=\norm{f}_{\L^\infty(\T)}$. Consequently, we can consider the isometric homomorphism $\pi \co \L^\infty(\T) \to \B(\L^p(\T))$, $f \mapsto M_f$. It is well known, essentially from \cite[Exercise 4.2.6 p.~194]{Kha13} and \cite[p.~390]{GVF01}, that $(\C(\T),\L^2(\T),\frac{1}{\i}\frac{\d}{\d \theta})$ is a compact spectral triple, where the used homomorphism is the restriction of $\pi$ on the algebra $\C(\T)$. First, we prove an $\L^p$-generalization of this classical fact. Here, we consider the closure of the closable\footnote{\thefootnote. Let $(f_n)_{n\geq1}$ be a sequence of functions belonging to $\C^\infty(\T)$ such that $f_n \to 0$ in $\L^p(\T)$ and $f_n' \to g$ in $\L^p(\T)$ for some $g \in \L^p(\T)$. For any $h \in \C^\infty(\T)$, integration by parts gives
$\int_0^{2\pi} f_n'(\theta)h(\theta) \d\theta
= -\int_0^{2\pi} f_n(\theta)h'(\theta)\d\theta$. Passing to the limit (using H\"older's inequality), we obtain $\int_0^{2\pi} g(\theta) h(\theta)\d\theta 
= 0$ for all $h \in \C^\infty(\T)$. Hence $g=0$ in $\L^p(\T)$ by \cite[Proposition 1.1.5 p.~4]{PRR19}.} unbounded operator $\frac{\d}{\d \theta} \co \C^\infty(\T) \subset\L^p(\T) \to \L^p(\T)$ and we denote again $\frac{\d}{\d \theta}$ its closure whose the domain $\dom \frac{\d}{\d \theta}$ is the Sobolev space
$$
\W^{1,p}(\T)
\ov{\mathrm{def}}{=}\{f \in \L^p(\R) : f(\theta + 2\pi) = f(\theta), f \text{ is absolutely continuous}, f' \in \L^p(\R) \}.
$$ 
Finally, note that the operator $\frac{\d}{\d \theta}$ is an unbounded Fourier multiplier with symbol $(\i \,n)_{n \in \Z}$.

Suppose $1 < p < q < \infty$ and consider some $\alpha > 0$ such that $\frac{\alpha}{d} \geq \frac{1}{p}-\frac{1}{q}$. We will use the Sobolev inequality 
\begin{equation}
\label{Sobolev-inequality-torus}
\norm{f}_{\L^q(\T^d)}
\lesssim \norm{f}_{\W^{\alpha,p}(\T^d)}.
\end{equation}
of \cite[Corollary 1.2]{BeO13}, 
which gives a bounded embedding $J \co \W^{\alpha,p}(\T^d) \to \L^q(\T^d)$. Note that Sobolev embeddings are also available for compact Riemannian manifolds, see e.g., \cite[Proposition 2.2]{ReS24}. 
We need the following estimates on the approximation numbers.

\begin{prop}
\label{prop:Koenig-on-torus}
Suppose that $1 < p \leq q \leq 2$. Consider some $\alpha > 0$ with $\frac{\alpha}{d} > \frac1p-\frac1q$. If $J \co \W^{\alpha,p}(\T^d) \to \L^q(\T^d)$ is the canonical Sobolev embedding then
\begin{equation}
\label{approx-tori-Sobolev}
a_n(J) 
\lesssim \frac{1}{n^{\frac{\alpha}{d}}}, \qquad n \geq 1.
\end{equation}
\end{prop}

\begin{proof}
Consider the restriction map $R_p \co \L^p_\alpha(\T^d) \to \W^{\alpha,p}((0,2\pi)^d)$, $f \mapsto f|_{(0,2\pi)^d}$ and the periodization $\cal{P}_q \co \L^q((0,2\pi)^d) \to \L^q(\T^d)$. 
Both maps are bounded. The range $X_p\ov{\mathrm{def}}{=} \Ran R_p$ is a closed subspace of $\W^{\alpha,p}((0,2\pi)^d)$.
Consider the inclusion $j \co X_p \hookrightarrow \L^q((0,2\pi)^d)$. By construction, the Sobolev embedding on $\T^d$ factorizes as
\[
J 
= \cal{P}_q \circ j \circ R_p.
\]
The ideal (composition) property of approximation numbers yields
\[
a_n(J)
=a_n(\cal{P}_q \circ j \circ R_p)
\ov{\eqref{majo-sn-2}}{\leq} \norm{\cal{P}_q} \norm{R_p} a_n(j),\quad n \geq 1.
\]
Since $X_p$ is a closed subspace of $\W^{\alpha,p}((0,2\pi)^d)$ and restriction of the domain can only decrease approximation numbers, we have
\[
a_n(j)
\leq a_n(J_0),
\]
where $J_0 \co \W^{\alpha,p}((0,2\pi)^d) \hookrightarrow \L^q((0,2\pi)^d)$ denotes the Sobolev embedding. By  \cite[Theorem pp.~186–187]{Kon86}, we have $a_n(J_0) \lesssim \frac{1}{n^{\frac{\alpha}{d}}}$. We conclude that \eqref{approx-tori-Sobolev} is true.
\end{proof}

%

We also need the next result.

\begin{prop}
\label{prop-group-isometries}
Let $(T_t)_{t \in \R}$ be a strongly continuous group of isometries on a Banach space $X$ with generator $-\i D$, i.e., $T_t =\e^{\i t D}$ for any $t \in \R$. Consider a bounded operator $R \co X \to X$ satisfying $R(\dom D) \subset \dom D$ such that $[D,R]$ induces a bounded operator on $X$. Then the function $F \co \R \to \B(X)$ defined by 
\[
F(t)
= T_t R T_{-t},\quad t \in \R.
\]
is strongly $\C^1$ and
\begin{equation}
\label{eq:deriv}
F'(t)
=\i T_t [D,R] T_{-t}\qquad\text{(strongly on $X$)}
\end{equation}
and
\begin{equation}
\label{inter-formula-complex}
\frac{T_t R T_{-t}-R}{t} 
= \int_0^1 T_{s t} (\i [D,R]) T_{-s t} \d s.
\end{equation}
\end{prop}

\begin{proof}
Fix $x \in \dom D$. Using the generator identities \cite[Lemma 1.3 p.~50]{EnN00}
\[
\frac{\d}{\d t}T_t x
=\i D T_t x,\qquad 
\frac{\d}{\d t}T_{-t}x
=-\i T_{-t}D x
\quad(\text{derivatives in $X$}),
\]
and the product rule in the strong sense, we get
\begin{equation}
\label{inter-654}
\frac{\d}{\d t}\big(T_t R T_{-t}x\big)
= \i D T_t R T_{-t} x + T_t R(-\i)T_{-t}Dx
= \i T_t\big(DR-RD\big)T_{-t}x
= \i T_t [D,R] T_{-t}x.
\end{equation}
By density of $\dom D$ and boundedness of $[D,R]$, this extends to all $x \in X$, proving \eqref{eq:deriv}.

Now fix $t\in\R$ and $x \in X$. The map $u \mapsto T_u [D,R] T_{-u}x$ is continuous (as a vector-valued function) and bounded by $\norm{[D,R]} \norm{x}$, hence Bochner integrable on $[0,t]$. The Banach-valued fundamental theorem of calculus yields
\[
T_t R T_{-t}x - Rx 
= \int_0^t \frac{\d}{\d u}\big(T_u R T_{-u}x\big) \d u
\ov{\eqref{inter-654}}{=}  \int_0^t \i T_u [D,R] T_{-u}x \d u.
\]
Dividing by $t \neq 0$ and using the change variables $u= st$, we obtain
\[
\frac{T_t R T_{-t}-R}{t} x
= \int_0^1 T_{s t} (\i [D,R]) T_{-s t}x \d s.
\]
Since this holds for every $x \in X$, we have in the strong operator topology the operator identity
\[
\frac{T_t R T_{-t}-R}{t} 
= \int_0^1 T_{s t} (\i [D,R]) T_{-s t} \d s.
\]
\end{proof}

Now, we can determine the Lipschitz algebra of the operator $\frac{1}{\i}\frac{\d}{\d \theta}$. This result is new even for $p=2$.

\begin{prop}
\label{prop:Lip-alg-T}
Suppose that $1 < p <\infty$. Consider the operator $D \ov{\mathrm{def}}{=} \frac{1}{\i}\frac{\d}{\d \theta}$ with domain  $\dom D=\W^{1,p}(\T)$. Then
\begin{equation}
\label{eq:Lip-equals-W1infty}
\W^{1,\infty}(\T) 
= \Lip_{D}(\L^\infty(\T)).
\end{equation}
Moreover, for any function $f \in \W^{1,\infty}(\T)$ we have
\begin{equation}
\label{eq:commutator-formula-R}
[D,M_{f}]
=\frac{1}{\i} M_{f'} 
\quad \text{and} \quad
\norm{[D,M_f]}_{\L^p(\T) \to \L^p(\T)}
=\norm{f'}_{\L^\infty(\T)}.
\end{equation}
\end{prop}

\begin{proof}
If $f \in \W^{1,\infty}(\T)$ then $M_f(\W^{1,p}(\T))\subset \W^{1,p}(\T)$. Moreover, for any function $u \in \C^\infty(\T)$ we have
\[
[D,M_f]u
=(DM_f-M_fD)u
=\tfrac{1}{\i} \big((fu)'-f u'\big)
=\tfrac{1}{\i}f' u.
\]
Hence $[D,M_f]=\tfrac{1}{\i}M_{f'}$ extends boundedly on $\L^p(\T)$ with $\norm{[D,M_f]}_{\L^p(\T) \to \L^p(\T)} =\norm{f'}_{\L^\infty(\T)}$. In particular, we obtained $\W^{1,\infty}(\T) \subset \Lip_{D}(\L^\infty(\T))$. 

Note that it is well known \cite[pp.~10-11]{AGGG86} that the operator $\frac{\d}{\d \theta}$ generates a strongly continuous group $(T_t)_{t \in \R}$ of operators acting on $\L^p(\T)$, namely the group of translations. For any $t \in \R$ we have $T_t=\e^{\i t D}$
and
\begin{equation}
\label{eq:conj-torus}
T_t M_f T_{-t}
=M_{f(\cdot+t)}, \qquad t \in \R.
\end{equation}
Let $f \in \Lip_{D}(\L^\infty(\T))$. Since $[D,M_f] \in \B(\L^p(\T))$ by Proposition \ref{prop-group-isometries} we deduce that
\begin{equation}
\label{eq:FTC}
M_{\Delta_t f}
\ov{\mathrm{def}}{=} \frac{M_{f(\cdot+t)}-M_f}{t}
\ov{\eqref{eq:conj-torus}}{=} \frac{T_t M_f T_{-t}-M_f}{t}
\ov{\eqref{inter-formula-complex}}{=} \int_0^1 T_{st} (\i [D,M_f]) T_{-st} \d s,
\quad t\neq 0.
\end{equation}
If $|t|$ is small, we infer that
\begin{align*}
\MoveEqLeft
\label{eq:diffquot-bound}
\norm{\Delta_t f}_{\L^\infty(\T)}
=\norm{M_{\Delta_t f}}_{\B(\L^p(\T))} 
\ov{\eqref{eq:FTC}}{=} \norm{\int_0^1 T_{st} (\i [D,M_f]) T_{-st} \d s}_{\B(\L^p(\T))} \\
&\leq \int_0^1 \norm{T_{st} (\i [D,M_f]) T_{-st}}_{\B(\L^p(\T))} \d s
\leq \norm{[D,M_f]}_{\B(\L^p(\T))}.         
\end{align*}
By Banach–Alaoglu theorem, there exist $t_\alpha \to 0$ and $g \in \L^\infty(\T)$ such that $\Delta_{t_\alpha} f \overset{}{\longrightarrow} g$ in the weak* topology of the dual space $\L^\infty(\T)$, with $\norm{g}_{\L^\infty(\T)} \leq \norm{[D,M_f]}_{\B(\L^p(\T))}$. For any function $\varphi \in \C^\infty(\T)$, we have with a change of variables
\[
\int_\T (\Delta_{t_\alpha} f) \varphi
=\int_\T f \frac{\varphi(\cdot-t_\alpha)-\varphi}{t_\alpha}
\xra[\alpha]{} -\int_\T f\,\varphi'
\]
since $\frac{\varphi(\cdot-t)-\varphi}{t} \to -\varphi'$ in $\L^1(\T)$ as $t \to 0$ (since $\i D$ is the generator of $(T_t)_{t \in \R}$ on $\L^1(\T)$). Passing to the weak* limit gives
\[
\int_\R g \varphi
=-\int_\R f \varphi', \quad \varphi \in \C^\infty(\T),
\]
i.e., $g$ is the weak derivative of $f$. Hence $f$ belongs to $\W^{1,\infty}(\T)$.
\end{proof}

The key point of the following result is the novel connection established in the proof between Sobolev embeddings and the summability of spectral triples.

\begin{thm}
\label{Th-dtheta}
Suppose that $1 < p < \infty$. Then $(\C(\T),\L^p(\T),\frac{1}{\i}\frac{\d}{\d \theta})$ is a Banach compact spectral triple which is $1^+$-summable, in particular $q$-summable for any $q > 1$. More precisely, we have
$$
\dim \big(\C(\T),\L^p(\T),\tfrac{1}{\i}\tfrac{\d}{\d \theta}\big) 
= 1. 
$$
\end{thm}

\begin{proof}
As we said the operator $\frac{\d}{\d \theta}$ generates a strongly continuous group of operators acting on $\L^p(\T)$, namely the group of translations. By Example \ref{generators-bisectorial} and Example \ref{ex-bound-Hinfty}, we deduce that the unbounded operator $D \ov{\mathrm{def}}{=} \frac{1}{\i}\frac{\d}{\d \theta}$ is bisectorial with $\omega_\bi(D)=0$ and admits a bounded $\H^\infty(\Sigma_\omega^\bi)$ functional calculus for any angle $\omega > 0$ on the Banach space $\L^p(\T)$. By Proposition \ref{prop:Lip-alg-T}, we have $\Lip_{D}(\C(\T)) = \Lip_{D}(\L^\infty(\T)) \cap \C(\T) =\W^{1,\infty}(\T) \cap \C(\T)$. Since this last space contains $\C^\infty(\T)$, it is dense in $\C(\T)$. 

Finally, it is stated without proof in \cite[Example 6.31 p.~187]{Kat76} that the operator $D=\frac{1}{\i}\tfrac{\d}{\d \theta}$ has compact resolvent. Actually, below we prove a stronger fact than the stated result in \cite[Example 6.31 p.~187]{Kat76}.

For any function $f \in \L^p(\T)$, according to \cite[Lemma 2 p.~24]{BeS22} we have
$$
\norm{(\Id+\Delta)^{-1}(f)}_{\W^{2,p}(\T)}
\lesssim \norm{f}_{\L^p(\T)}.
$$
Note that $(D+\i)^{-1}=(D-\i)(\Id+\Delta)^{-1}$. If $J \co \W^{1,p}(\T) \to \L^p(\T)$ is the canonical embedding then
\begin{align}
\MoveEqLeft
\label{an-inter-AZETY}
a_{n}((D+\i)^{-1})         
\leq a_{n}\big(J \circ (D-\i) \circ (\Id+\Delta)^{-1}\big) \\
&\ov{\eqref{majo-sn-2}}{\leq} \norm{D-\i}_{\W^{2,p}(\T) \to \W^{1,p}(\T)} \norm{(\Id+\Delta)^{-1}}_{\L^p(\T) \to \W^{2,p}(\T)} a_n(J)
\ov{\eqref{approx-tori-Sobolev}}{\lesssim} \frac{1}{n}. \nonumber
\end{align}
So, the quasi-norm can be estimated by
\[
\norm{(D+\i)^{-1}}_{p,\infty}
\ov{\eqref{def-quasi-norm-Spq}}{=} \sup_{n \geq 1}\, n^{\frac{1}{p}} s_n((D+\i)^{-1})
\ov{\eqref{an-inter-AZETY}}{\leq} \sup_{n \geq 1}\, n^{\frac{1}{p}-1}.
\]
Hence $\norm{(D+\i)^{-1}}_{p,\infty} < \infty$ for every $p \geq 1$. 
We deduce that the resolvent operator $(D+\i)^{-1} \co \L^p(\T) \to \L^p(\T)$ belongs to the space $S^{1,\infty}_\app(\L^p(\T))$, defined in \eqref{def-Sqs-bis}. 
Thus the triple $(\C(\T),\L^p(\T),\frac{1}{\i}\frac{\d}{\d \theta})$ is $1^+$-summable.

Now, we prove that this result is optimal. When $p = 2$, the statement is well known. The Banach compact spectral triple $(\C(\T),\L^p(\T),\frac{1}{\i}\frac{\d}{\d \theta})$ is not $1$-summable. Indeed, by \cite[Remark B.7 p.~127]{AmG16}, the spectrum of the Dirac operator $D \ov{\mathrm{def}}{=} \frac{1}{\i}\frac{\d}{\d \theta}$ is independent of $p$. 
In the case $p=2$, it is known that $\Sp D=\Z$ and that each eigenvalue is simple, see e.g., \cite[Exercise 8.1]{Var10} or \cite[Notes 2.1.2 p.~31]{Gin09}. Using Weyl's inequality \eqref{Weyl-inequality} and the spectral theorem \cite[Proposition A.3.1 p.~276]{Haa06} for the resolvent operator $R_z(D) \ov{\mathrm{def}}{=} (z-D)^{-1}$ , we see that for any $z \not\in \Sp D$ we have
\begin{align*}
\MoveEqLeft
\infty
=\sum_{n \in \Z} \frac{1}{|z-n|}
=\sum_{n \in \Z} \lambda_n(R_z(D))
=\norm{\big(\lambda_n(R_z(D))\big)}_{\ell^1} \\
&\ov{\eqref{Weyl-inequality}}{\lesssim}  \norm{R_z(D)}_{S^1_\weyl(\L^p(\T))} 
\ov{\eqref{order-xn-an}\eqref{def-Sqs-bis}}{\leq} \norm{R_z(D)}_{S^1_\app(\L^p(\T))}.
\end{align*}
Consequently the resolvent operator $R_z(D)$ does not belong to the space $S^1_\app(\L^p(\T))$. We conclude that the triple is not $1$-summable according to Definition \ref{summability-spectral-triple}. Since $S^{q,\infty}_\app(\L^p(\T))\subset S^1_\app(\L^p(\T))$ for any $0 < q < 1$ the proof is complete.
\end{proof}

Following \cite[(3.285) p.~132]{Kin09a} (see also \cite[p.~181]{Cas22}), recall that the periodic Hilbert transform $\cal{H}_\per \co \L^p(\T) \to \L^p(\T)$ is defined by the principal value 
\begin{equation}
\label{Hilbert-transform-torus}
(\mathcal{H}_\per g)(\theta)
\ov{\mathrm{def}}{=} \frac{1}{2\pi} \lim_{\epsi \to 0^+} \int_{\epsi \leq |s| \leq \pi} g(\theta-s)\cot(\tfrac{s}{2}) \d s
=\frac{1}{2\pi} \mathrm{p.v.}\int_{-\pi}^{\pi} g(\varphi) \cot\Big(\frac{\theta-\varphi}{2}\Big)\d\varphi.
\end{equation}
It can also be viewed as a Fourier multiplier with symbol $-\i\,\sign(n)$, i.e., we have
\begin{equation}
\label{symbole-Hper}
\widehat{\cal{H}_\per(g)}(n) 
= -\i \sgn(n)\hat{g}(n), \quad n \in \Z,
\end{equation}
where $\sgn(0) \ov{\mathrm{def}}{=} 0$. The boundedness of the periodic Hilbert transform $\cal{H}_\per$ on the space $\L^p(\T)$ is proved, for instance, in \cite[Section 6.17]{Kin09a} with an approach of Calder\'on and in \cite[Proposition 5.2.5 p.~391]{HvNVW16} by transference from the Hilbert transform $\cal{H}$ on $\L^p(\R)$. 
Another proof is provided by the next result, which simultaneously determines the index pairing of the associated Banach Fredholm module provided by Proposition \ref{prop-triple-to-Fredholm}. The case $p=2$ is folklore.

For this, we need the notion of the winding number of a continuous function $f \co \T \to \mathbb{C}-\{0\}$, which, roughly speaking, is the number of turns of the point $f(\e^{\i t})$ around the origin when $t$ runs from $0$ to $2\pi$. More precisely, consider a continuous branch of the argument $\arg_f$ of the function $[0,2\pi] \to \mathbb{C}-\{0\}$, $t \mapsto f(\e^{\i t})$, i.e., $\arg_f$ is continuous on $[0,2\pi]$ and
$$
\frac{f(\e^{\i t})}{|f(\e^{\i t})|}
=\e^{\i\arg_f(t)}, \quad t \in [0,2\pi].
$$
The winding number of $f$ is defined by
\begin{equation*}
\label{}
\wind f
\ov{\mathrm{def}}{=} \frac{1}{2\pi}(\arg_f(2\pi)-\arg_f(0)).
\end{equation*}
It is worth noting that if $f \in \C^1(\T)$ and does not vanish on $\T$ then
\begin{equation}
\label{wind-second-def}
\wind f
=\frac{1}{2\pi \i} \int_{0}^{2\pi} \frac{f'(\e^{\i \theta})}{f(\e^{\i \theta})} \d \theta.
\end{equation}
Now, we can prove the following result.
\begin{prop}
\label{prop-pairing-T}
Suppose that $1 < p < \infty$. The odd Banach Fredholm module over the algebra $\C(\T)$ associated to the Banach compact spectral triple $(\C(\T),\L^p(\T),\frac{1}{\i}\frac{\d}{\d \theta})$ is $(\L^p(\T),\pi,\i\cal{H}_\per)$. For any function $f \in \C(\T)$ which does not vanish on $\T$, we have
 \begin{equation}
\label{}
\big\la [f], (\L^p(\T),\pi,F') \big\ra_{\K_1(\C(\T)),\K^1(\C(\T),\scr{L}^p)}
=-\wind f,
\end{equation}
where $F' \ov{\mathrm{def}}{=} F+Q$ using the map $Q \co \L^p(\T) \to \L^p(\T)$, $\e^{\i n\theta} \to \delta_{n=0}$.
\end{prop}

\begin{proof}
The assertion is clear once the operators are viewed as Fourier multipliers. We let $F \ov{\mathrm{def}}{=} \i\cal{H}_\per$. Note that the map $Q \co \L^p(\T) \to \L^p(\T)$, $\e^{\i n\theta} \to \delta_{n=0}$ (which maps a function on the constant term in its Fourier series) is a finite-rank contractive map (it is the $\L^p$-extension of a conditional expectation). So $F' \ov{\mathrm{def}}{=} F+Q$ is clearly a Banach Fredholm module since $(F')^2=\Id_{\L^p(\T)}$ and a compact perturbation of $F$. Moreover, observe that the map $P \ov{\mathrm{def}}{=} \frac{\Id+F'}{2}$ identifies to the Riesz projection $\L^p(\T) \to \L^p(\T)$, $\e^{\i n\theta} \mapsto \delta_{n \geq 0} \e^{\i n\theta}$. The range of this map is the Hardy space $\H^p(\T)$. Let $f \in \C(\T)$  which does not vanish on $\T$. This means that $f$ is an invertible element of the unital algebra $\C(\T)$.  The operator $PM_f P \co P(\L^p(\T)) \to P(\L^p(\T))$ of \eqref{pairing-odd-2} (with $n=1$) identifies to the Toeplitz operator $T_f \co \H^p(\T) \to \H^p(\T)$, $g \mapsto P(fg)$ with symbol $f$, where we identify the map $P$ with its corestriction $P|^{\H^p(\T)} \co \L^p(\T)\to \H^p(\T)$. So the operator $T_f$ is Fredholm and
\begin{equation}
\label{inter-AETY}
\big\la [f], (\L^p(\T),\pi,F') \big\ra_{\K_1(\C(\T)),\K^1(\C(\T),\scr{L}^p)} 
=\Index PM_f P
= \Index T_f
\end{equation}
Using the classical Gohberg-Krein index theorem \cite[Theorem 2.42 p.~74]{BoS06} in the last equality, we deduce that
\begin{align*}
\MoveEqLeft
\big\la [f], (\L^p(\T),\pi,F') \big\ra_{\K_1(\C(\T)),\K^1(\C(\T),\scr{L}^p)}      
\ov{\eqref{inter-AETY}}{=} \Index T_f 
=-\wind f.
\end{align*}
\end{proof}

\begin{remark} \normalfont
Recall that the first group of $\K$-theory of the algebra $\C(\T)$ of continuous functions on the torus $\T=S^1$ is given by $\K_1(\C(\T))=\Z$. By the way, we refer to~\cite{ScS23} for a nice study of the groups of K-theory of spheres $S^n$ for $n \geq 1$. 
\end{remark}

\begin{remark} \normalfont
In the case $p = 2$, it is known that for any function $f \in \L^\infty(\T)$, the commutator $[F,f]$ is compact on $\L^2(\T)$ if and only if $f$ belongs to the space $\VMO(\T)$ of functions of vanishing mean oscillation.
This follows directly from \cite[Corollary 2.3.3, p.~215 and Exercises 2.5.2, p.~220]{Nik02}, which describe the compactness of the commutators $[P,f]$ with the Riesz projection $P$.
Moreover, the commutator $[F,f]$ belongs to the Schatten class $S^q(\L^2(\T))$ if and only if $f$ lies in the Besov space $\B^{\frac{1}{q}}_{q,q}(\T)$.
This is a direct consequence of \cite[Theorem 7.3, p.~276]{Pel03}, which characterizes the commutators $[P,f]$ with the Riesz projection $P$ belonging to $S^q(\L^2(\T))$, together with the inclusion $\L^\infty(\T)\subset \BMO(\T)$. 
\end{remark}

\begin{prop}
\label{prop-|D|-T}
Suppose that $1<p<\infty$. Consider the unbounded operator $D=\frac{1}{\i}\frac{\d}{\d \theta}$ acting on the space $\L^p(\mathbb{T})$. We have $|D|=\cal{H}_\per \frac{\d}{\d \theta}$. Moreover, for any function $f \in \C^\infty(\T)$, the commutator $[|D|,f]$ induces a bounded operator on the Banach space $\L^p(\mathbb{T})$ with
\begin{equation}
\label{estim-commutator-|D|}
\bnorm{[|D|,f]}_{\L^p(\T) \to \L^p(\T)}
\lesssim_p \norm{f'}_{\L^\infty(\T)}.
\end{equation}
\end{prop}

\begin{proof} 
\textit{First proof.} The operator $|D|$ is the unbounded Fourier multiplier defined by the symbol $(|n|)_{n \in \Z}$ and $D$ by $(n)_{n\in \Z}$. The first assertion is then a consequence of \eqref{symbole-Hper}. Hence, for any function $g \in \C^\infty(\T)$, we have
\begin{align*}
\label{}
\MoveEqLeft
\big[|D|,M_f\big]g
= \mathcal{H}_{\per} \tfrac{\d}{\d \theta}(fg)-f\cal{H}_\per \tfrac{\d}{\d \theta} g
= \cal{H}_\per(f'g)+\cal{H}_\per(fg')-f\cal{H}_\per g' \\
&= \cal{H}_\per(f'g)+[\cal{H}_\per,M_f] g'.
\end{align*}
Observe that we have
\begin{align*}
\MoveEqLeft
\big([\cal{H}_\per,M_f]g'\big)(\theta)
=\big(\cal{H}_\per (fg')\big)(\theta)-f(\theta)(\cal{H}_\per g')(\theta) \\
&\ov{\eqref{Hilbert-transform-torus}}{=}\frac{1}{2\pi} \mathrm{p.v.}\int_0^{2\pi} \big(f(\varphi)-f(\theta)\big) g'(\varphi) \cot\Big(\frac{\theta-\varphi}{2}\Big) \d\varphi.
\end{align*}
Integrating by parts on the circle (no boundary term) and using the equality (recall that $\cot' x = -\csc^2 x$ for any $x \not\in \pi\Z$, where $\csc x=\frac{1}{\sin x}$)
\begin{equation}
\label{inter-L987}
\frac{\d}{\d \varphi}\cot\Big(\frac{\theta-\varphi}{2}\Big)
=\frac{1}{2}\csc^2\Big(\frac{\theta-\varphi}{2}\Big),
\end{equation}
we obtain
\begin{align*}
\big([\cal{H}_\per,M_f]g'\big)(\theta)
&= -\frac{1}{2\pi} \mathrm{p.v.}\int_0^{2\pi} \frac{\d}{\d \varphi}\Big[\big(f(\varphi)-f(\theta)\big)\cot\Big(\frac{\theta-\varphi}{2}\Big)\Big] g(\varphi) \d\varphi\\
&\ov{\eqref{inter-L987}}{=} -\frac{1}{2\pi} \mathrm{p.v.}\int_0^{2\pi} \Big(f'(\varphi) \cot\Big(\frac{\theta-\varphi}{2}\Big)
+\frac{1}{2}\big(f(\varphi)-f(\theta)\big)\csc^2\Big(\frac{\theta-\varphi}{2}\Big)\Big) g(\varphi)\d\varphi.
\end{align*}
On the other hand, we have
\[
\big(\cal{H}_\per(f'g)\big)(\theta)
\ov{\eqref{Hilbert-transform-torus}}{=} \frac{1}{2\pi} \mathrm{p.v.}\int_0^{2\pi} f'(\varphi)g(\varphi)\cot\Big(\frac{\theta-\varphi}{2}\Big) \d\varphi.
\]
Adding the two identities, the terms with $f'(\varphi)\cot(\frac{\theta-\varphi}{2})$ cancel, and we arrive at the formula
\begin{equation}
\label{eq:kernel}
[|D|,M_f]g(\theta)
=\frac{1}{4\pi} \mathrm{p.v.}\int_0^{2\pi}
\frac{f(\theta)-f(\varphi)}{\sin^2\big(\frac{\theta-\varphi}{2}\big)} g(\varphi) \d\varphi.
\end{equation}
Thus $[|D|,M_f]$ is an integral operator with kernel
\[
K_f(\theta,\varphi)
\ov{\mathrm{def}}{=} \frac{1}{4\pi}\,\frac{f(\theta)-f(\varphi)}{\sin^2\big(\frac{\theta-\varphi}{2}\big)}
\]
in the principal value sense. By the mean value theorem, we have $
|f(\theta)-f(\varphi)|
\leq \norm{f'}_{\L^\infty(\T)} |\theta-\varphi|$. 
Using $\sin t\sim t$ as $t \to 0$, there exists a constant $c > 0$ such that
\[
|K_f(\theta,\varphi)|
\leq C\frac{\norm{f'}_{\L^\infty(\T)}}{|\theta-\varphi|}, \quad \theta \neq \varphi.
\]
Set $s=\frac{\theta-\varphi}{2}$. Write
\[
K_f(\theta,\varphi)=\frac{1}{4\pi}\,(f(\theta)-f(\varphi))\,\sin^{-2}(s).
\]
We first estimate the derivatives of $K_f$. Using the chain rule and
\[
\frac{\d}{\d s}\sin^{-2}(s)
=-2\sin^{-3}(s)\cos(s),\qquad 
\frac{\partial}{\partial \theta} s
=\tfrac12,\quad \frac{\partial}{\partial \varphi} s
=-\tfrac12,
\]
we get
\[
\frac{\partial K_f}{\partial \theta} (\theta,\varphi)=\frac{1}{4\pi}\Big(\frac{f'(\theta)}{\sin^{2}(s)}-(f(\theta)-f(\varphi))\frac{\cos(s)}{\sin^{3}(s)}\Big),
\]
and similarly
\[
\frac{\partial K_f}{\partial \varphi} (\theta,\varphi)=\frac{1}{4\pi}\Big(-\frac{f'(\varphi)}{\sin^{2}(s)}-(f(\theta)-f(\varphi))\frac{\cos(s)}{\sin^{3}(s)}\Big).
\]
By the mean value theorem, we have $|f(\theta)-f(\varphi)|\leq \norm{f'}_{\L^\infty(\T)}|\theta-\varphi|=2\norm{f'}_{\L^\infty(\T)}|s|$. Since $|\sin s| \sim_0 |s|$ and $|\cos s| \leq 1$ for $|s| \leq \frac{\pi}{2}$, there exists constant $C$ such that
\[
\Big|\frac{1}{\sin^{2}(s)}\Big|
\leq \frac{C}{|s|^{2}},\qquad \Big|\frac{\cos(s)}{\sin^{3}(s)}\Big|
\leq \frac{C}{|s|^{3}}.
\]
Hence
\[
\left|\tfrac{\partial K_f}{\partial \theta}(\theta,\varphi)\right|
\leq C \norm{f'}_{\L^\infty(\T)}\Big(\frac{1}{|s|^{2}}+\frac{|s|}{|s|^{3}}\Big)
\lesssim \frac{ \norm{f'}_{\L^\infty(\T)}}{|s|^{2}}
= \frac{4 \norm{f'}_{\L^\infty(\T)}}{|\theta-\varphi|^{2}}.
\]
The same bound holds for $\tfrac{\partial K_f}{\partial \theta} K_f(\theta,\varphi)$.

We now estimate the differences. For the first one, by the fundamental theorem of calculus,
\[
K_f(\theta,\varphi)-K_f(\theta',\varphi)
=\int_{0}^{1}\partial_\theta K_f\big(\theta+t(\theta'-\theta),\varphi\big)\,(\theta-\theta')\,dt.
\]
If $|\theta-\theta'|\leq \tfrac12|\theta-\varphi|$, then for all $t\in[0,1]$,
\[
|\theta+t(\theta'-\theta)-\varphi|
\geq |\theta-\varphi|-|\theta-\theta'|
\geq \tfrac12|\theta-\varphi|.
\]
Therefore,
\[
|K_f(\theta,\varphi)-K_f(\theta',\varphi)|
\leq |\theta-\theta'|\sup_{t\in[0,1]}\frac{C \norm{f'}_{\L^\infty(\T)}}{|\theta+t(\theta'-\theta)-\varphi|^{2}}
\leq C \norm{f'}_{\L^\infty(\T)}\frac{|\theta-\theta'|}{|\theta-\varphi|^{2}}.
\]
The estimate
\[
|K_f(\theta,\varphi)-K_f(\theta,\varphi')|
\leq C \|f'\|_{\L^\infty(\T)}\frac{|\varphi-\varphi'|}{|\theta-\varphi|^{2}}
\]
is obtained in exactly the same way from the bound on $\partial_\varphi K_f$, under the condition $|\varphi-\varphi'|\leq \tfrac12|\theta-\varphi|$. Combining the two bounds completes the proof.
Therefore $K_f$ is a standard Calder\'on–Zygmund kernel on the circle, with constant controlled by $\norm{f'}_{\L^\infty(\T)}$. 
By the Calderon–Zygmund theory on the circle, we conclude that the operator extends to a bounded operator on $\L^p(\T)$, with the estimate \eqref{estim-commutator-|D|}.

\textit{Second proof of the first assertion.} Observe that $|D|$ is a pseudo-differential operator of order $1$. Moreover, recall that if $f \in \C^\infty(\T)$ then $M_f$ is a differential operator of order 0, hence a pseudo-differential operator of order 0 by \cite[Remark 23.26.5]{Dieu88}. Consequently, according to \cite[Remark 5.3.3 p.~423]{RuT10} the commutator $[|D|,M_f]$ is a pseudo-differential operator of order $0$. Such operator induces a bounded operator on the Banach space $\L^p(\mathbb{T})$ by \cite[Theorem 5.2.23 p.~420]{RuT10}.
\end{proof}

\begin{remark} \normalfont
We can probably replace $\C^\infty(\T)$ by the space $\W^{1,\infty}(\T)$ in the previous result.
\end{remark}

So, we can use Proposition \ref{prop-q-summable-Fredhom-123} which provides the next result.

\begin{cor}
The Fredholm module $(\L^p(\T),\pi,\i\cal{H}_\per)$ over the algebra $\C^\infty(\T)$ is $q$-summable for any $q > 1$.
\end{cor}

The previous result allows us to define and compute the Chern character. The case $p=2$ is folklore.

\begin{prop}
Consider the Fredholm module $(\L^p(\T),\pi,F')$ where $F' \ov{\mathrm{def}}{=}\i\cal{H}_\per+ Q$ where $Q \co \L^p(\T) \to \L^p(\T)$, $\e^{\i n\theta} \to \delta_{n=0}$. We have
\begin{equation}
\label{descrip-Chern-Torus}
\Ch_{1}^{F'}(f_0,f_1)
=\frac{2}{\pi \i} \int_\T f_0 \d f_1, \quad f_0,f_1 \in \C^\infty(\T).
\end{equation}
\end{prop}

\begin{proof}
For any functions $f_0,f_1 \in \C^\infty(\T)$, we have
\begin{align*}
\MoveEqLeft
\Ch_{1}^{F'}(f_0,f_1)
\ov{\eqref{Chern-odd}}{=} \tr(F'[F',f_0][F',f_1]).
\end{align*}
Using the computation of \cite[p.~183]{Kha13}, we obtain that $\tr(F'[F',f_0][F',f_1])=\frac{2}{\pi \i} \int_\T f_0 \d f_1$.
\end{proof}

Using Theorem \ref{th-Chern-coupling-odd}, we obtain a new proof of the Gohberg-Krein index theorem in the case $p \not=2$  since for any function $f \in \C^\infty(\T)$ which does not vanish on $\T$, we have
\begin{align}
\MoveEqLeft
\label{Gohberg-Krein}
\Index T_f \ov{\eqref{inter-AETY}}{=} \big\la [f], [(\L^p(\T),\pi, F')] \big\ra_{\K_1(\cal{A}),\K^1(\cal{A},\scr{B})} 
\ov{\eqref{Chern-odd-123}}{=}-\frac{1}{4c_1}  \Ch_1^{F'}(f^{-1},f) \\      
&\ov{\eqref{descrip-Chern-Torus}}{=} -\frac{1}{4c_1} \frac{2}{\pi \i} \int_\T f \d f
=-\frac{1}{2\pi \i} \int_{\T} \frac{f'}{f} 
\ov{\eqref{wind-second-def}}{=} -\wind f.
\end{align}
This method for the case $p=2$ was already known.

\subsection{The Dirac operator on $\ell^p_N(\L^p(\R^d))$}
\label{sec-Dirac-Rn}

The Dirac operator $D$ is defined in terms of $\gamma$ matrices. Define $N \ov{\mathrm{def}}{=} 2^{\lfloor \frac{d}{2}\rfloor}$ and select $N \times N$ complex selfadjoint matrices $(\gamma_1,\ldots,\gamma_d)$ satisfying 
\begin{equation}
\label{Clifford-Rn}
\gamma_j\gamma_k +\gamma_k\gamma_j 
= 2\delta_{j,k}1.
\end{equation}
Following \cite[p.~XXXI]{LMSZ23} (see also \cite[p.~118]{LaM89}), we define the unbounded densely defined linear operator
\begin{equation}
\label{Dirac-operator-Rd}
D 
\ov{\mathrm{def}}{=} -\i\sum_{j=1}^d \gamma_j \ot \partial_j
\end{equation}
with domain $\ell^p_N(\W^{1,p}(\R^d))$ acting on Banach space $\ell^p_N(\L^p(\R^d))$. The operator $D$ is selfadjoint if $p=2$. By \cite[p.~118]{LaM89}, we have
\begin{equation}
\label{square-of-Dirac}
D^2
=-\Id_{\ell^p_{N}} \ot \Delta.
\end{equation}
Recall that by \cite[Proposition 4.10 p.~31]{EnN00} any multiplication operator $M_f \co \L^p(\R^d) \to \L^p(\R^d)$, $g \mapsto fg$ by a measurable function $f \co \R^d \to \mathbb{C}$ is bounded if and only if $f$ is essentially bounded. In this case, we have $\norm{M_f}_{\L^p(\R^d) \to \L^p(\R^d)}=\norm{f}_{\L^\infty(\R^d)}$. Consequently, we can consider the isometric homomorphism 
\begin{equation}
\label{isometric-homomorphism}
\pi \co \L^\infty(\R^d) \to \B(\ell^p_N(\L^p(\R^d))), f \mapsto M_f \oplus \cdots \oplus M_f=\Id_{\ell^p_N} \ot M_f.
\end{equation}

\begin{lemma}
\label{lem:matrix-multiplication-Lp}
Let $(\Omega,\mu)$ be a measure space and let $1 \leq p < \infty$.  
Let $A \colon \Omega \to \M_N$ be a strongly measurable function such that
\[
\alpha
\ov{\mathrm{def}}{=}
\esssup_{x \in \Omega} \norm{A(x)}_{\M_N}
< \infty.
\]
Define the multiplication operator
\[
M_A \colon \L^p(\Omega,\mathbb{C}^N) \to \L^p(\Omega,\mathbb{C}^N), 
\qquad
(M_A u)(x) \ov{\mathrm{def}}{=} A(x)u(x).
\]
Then $M_A$ is bounded and its operator norm is given by
\begin{equation}
\label{eq:norm-matrix-mult}
\norm{M_A}_{\L^p(\Omega,\mathbb{C}^N)\to \L^p(\Omega,\mathbb{C}^N)}
=
\esssup_{x \in \Omega} \norm{A(x)}_{\M_N}.
\end{equation}
\end{lemma}

\begin{proof}
If $1 \leq p < \infty$, we obtain
\begin{align*}
\MoveEqLeft
\norm{M_A u}_{\L^p(\Omega,\mathbb{C}^N)}
=\bigg(\int_X \norm{A(x)u(x)}_{\mathbb{C}^N}^p \d\mu(x)\bigg)^{\frac{1}{p}} \\
&\leq \bigg(\int_X \norm{A(x)}_{\M_N}^p \norm{u(x)}_{\mathbb{C}^N}^p \d\mu(x) \bigg)^{\frac{1}{p}}
\leq
\alpha \int_X \norm{u(x)}_{\mathbb{C}^N} \d\mu(x)
=\alpha \norm{u}_{\L^p(\Omega,\mathbb{C}^N)}.
\end{align*}
Hence
\[
\norm{M_A}_{\L^p(\Omega,\mathbb{C}^N) \to \L^p(\Omega,\mathbb{C}^N)}
\leq \alpha
=\esssup_{x \in \Omega} \norm{A(x)}_{\M_N}.
\]
For the reverse inequality, let $\epsi > 0$. By definition of the essential supremum, there exists a measurable set
$E \subset \Omega$ with $\mu(E) > 0$ such that $
\norm{A(x)}_{\M_N}
\geq
\alpha - \epsi$ for almost all $x \in E$. Let $S^{N-1} \subset \mathbb{C}^N$ be the unit sphere and choose a countable dense subset $\{w_j\}_{j \geq 1} \subset S^{N-1}$ for example vectors with rational coordinates. For each $j \geq 1$, the map
\[
x \mapsto \norm{A(x)w_j}_{\mathbb{C}^N}
\]
is measurable, since $A$ is strongly measurable and the matrix norm is continuous.

For each $x \in \Omega$ we know that
\[
\norm{A(x)}_{\M_N}
= \sup_{\norm{v}_{\mathbb{C}^N} = 1} \norm{A(x)v}_{\mathbb{C}^N}
= \sup_{j \geq 1} \norm{A(x)w_j}_{\mathbb{C}^N},
\]
because $\{w_j\}$ is dense in $S^{N-1}$. Fix $\epsi > 0$ as above. For each $x \in E$, there exists some $j$ such that
\[
\norm{A(x)w_j}_{\mathbb{C}^N} \geq \norm{A(x)}_{\M_N} - \epsi.
\]
Define measurable sets $E_j \subset E$ by
\[
E_1 \ov{\mathrm{def}}{=} \bigl\{x \in E : \norm{A(x)w_1}_{\mathbb{C}^N} \geq \norm{A(x)}_{\M_N} - \epsi\bigr\},
\]
and for $j \geq 2$,
\[
E_j 
\ov{\mathrm{def}}{=} \bigl\{x \in E : \norm{A(x)w_j}_{\mathbb{C}^N} \geq \norm{A(x)}_{\M_N} - \epsi
\text{ and } \norm{A(x)w_i}_{\mathbb{C}^N} < \norm{A(x)}_{\M_N} - \epsi \text{ for all } 1 \leq i < j\bigr\}.
\]
Each $E_j$ is measurable since it is defined by inequalities involving measurable functions. By construction, the sets $(E_j)_{j \geq 1}$ are pairwise disjoint, and their union is exactly $E$.

Now define a vector field $v \co \Omega \to \mathbb{C}^N$ by
\[
v(x) 
\ov{\mathrm{def}}{=} 
\begin{cases}
w_j & \text{if } x \in E_j \text{ for some } j \geq 1,\\
0   & \text{if } x \in \Omega \setminus E.
\end{cases}
\]
This is a simple function, hence strongly measurable. Moreover, for almost all $x \in E$ we have
\[
\norm{v(x)}_{\mathbb{C}^N} = 1
\quad \text{and} \quad
\norm{A(x)v(x)}_{\mathbb{C}^N} \geq \norm{A(x)}_{\M_N} - \epsi.
\]
Define
\[
u(x) \ov{\mathrm{def}}{=} v(x).
\]
Then $u \in \L^p(\Omega,\mathbb{C}^N)$ and
\[
\norm{u}_{\L^p(\Omega,\mathbb{C}^N)}^p
= \int_E \norm{v(x)}_{\mathbb{C}^N}^p \d\mu(x)
= \mu(E).
\]
On the other hand,
\begin{align*}
\MoveEqLeft
\norm{M_A u}_{\L^p(\Omega,\mathbb{C}^N)}^p
= \int_E \norm{A(x)v(x)}_{\mathbb{C}^N}^p \d\mu(x) 
\geq \int_E \bigl(\norm{A(x)}_{\M_N} - \epsi\bigr)^p \d\mu(x) 
\geq (\alpha - 2\epsi)^p \mu(E)
\end{align*}
for $\epsi$ small enough, using that $\norm{A(x)}_{\M_N} \geq \alpha - \epsi$ almost everywhere on $E$. Therefore
\[
\norm{M_A u}_{\L^p(\Omega,\mathbb{C}^N)}
\geq (\alpha - 2\epsi) \mu(E)^{1/p}
= (\alpha - 2\epsi) \norm{u}_{\L^p(\Omega,\mathbb{C}^N)}.
\]
Hence
\[
\frac{\norm{M_A u}_{\L^p(\Omega,\mathbb{C}^N)}}{\norm{u}_{\L^p(\Omega,\mathbb{C}^N)}}
\geq \alpha - 2\epsi.
\]
Taking the supremum over all nonzero $u \in \L^p(\Omega,\mathbb{C}^N)$ shows
\[
\norm{M_A}_{\L^p(\Omega,\mathbb{C}^N)\to \L^p(\Omega,\mathbb{C}^N)}
\geq \alpha - 2\epsi.
\]
Since $\epsi > 0$ is arbitrary, we conclude
\[
\norm{M_A}_{\L^p(\Omega,\mathbb{C}^N)\to \L^p(\Omega,\mathbb{C}^N)} 
\geq \alpha.
\]
Combining this with the upper bound obtained at the beginning of the proof gives \eqref{eq:norm-matrix-mult}.

Using a standard measurable selection argument (or by first treating simple matrix-valued functions and then approximating $A$ in $\L^\infty$), we may assume that $x \mapsto v(x)$ is measurable on $E$. Define
\[
u(x)
\ov{\mathrm{def}}{=}
\mathbf 1_E(x) v(x).
\]
Then $u \in \L^p(\Omega,\mathbb{C}^N)$ and $
\norm{u}_{\L^p(\Omega,\mathbb{C}^N)}
=
\mu(E)^{1/p}$. Moreover,
\begin{align*}
\MoveEqLeft
\norm{M_A u}_{\L^p(\Omega,\mathbb{C}^N)}^p
=\int_E \norm{A(x)v(x)}_{\mathbb{C}^N}^p \d\mu(x) 
\geq \int_E \big(\norm{A(x)}_{\M_N}-\epsi\big)^p \d\mu(x) 
\geq (\alpha - 2\epsi)^p \mu(E),
\end{align*}
for $\epsi$ small enough. Hence
\[
\norm{M_A u}_{\L^p(\Omega,\mathbb{C}^N)}
\geq
(\alpha - 2\epsi) \mu(E)^{1/p}.
\]
Dividing by $\norm{u}_{\L^p(\Omega,\mathbb{C}^N)} = \mu(E)^{1/p}$, we obtain
\[
\frac{\norm{M_A u}_{\L^p(\Omega,\mathbb{C}^N)}}{\norm{u}_{\L^p(\Omega,\mathbb{C}^N)}}
\geq
\alpha - 2\epsi.
\]
Taking the supremum over all nonzero $u$ shows $
\norm{M_A}_{\L^p(\Omega,\mathbb{C}^N)\to \L^p(\Omega,\mathbb{C}^N)}
\geq
\alpha - 2\epsi$. Since $\epsi>0$ is arbitrary, we get
\[
\norm{M_A}_{\L^p(\Omega,\mathbb{C}^N)\to \L^p(\Omega,\mathbb{C}^N)}
\geq
\alpha.
\]
Together with the upper bound this proves \eqref{eq:norm-matrix-mult}.
\end{proof}

\begin{prop}
\label{prop:Lip-alg-Rd-Dirac}
Suppose that $1<p<\infty$. Then
\begin{equation}
\label{eq:Lip-equals-W1infty-Rd}
\Lip_{D}(\L^\infty(\R^d))
=\W^{1,\infty}(\R^d).
\end{equation}
Moreover, for any function $f \in \W^{1,\infty}(\R^d)$ we have, on the core $\C_c^\infty(\R^d,\mathbb{C}^N)$,
\begin{equation}
\label{eq:commutator-Rd-Dirac}
[D,f]
=-\i\sum_{j=1}^d \gamma_j \ot M_{\partial_{x_j} f}
=\i \d f \cdot,
\end{equation}
where $\d f \cdot$ denotes Clifford multiplication by the $1$-form $\d f$, and
\begin{equation}
\label{eq:norm-comm-Rd-Dirac}
\bnorm{[D,\pi(f)]}_{\ell^p_N(\L^p(\R^d)) \to \ell^p_N(\L^p(\R^d))}
=\esssup_{x \in \R^d} \Bgnorm{\sum_{j=1}^d \gamma_j \partial_{x_j} f(x)}_{\M_N}
\approx \norm{\nabla f}_{\L^\infty(\R^d)},
\end{equation}
where the equivalence constants depend only on the family $(\gamma_j)_{1\leq j\leq d}$.
\end{prop}

\begin{proof}
Let $f \in \W^{1,\infty}(\R^d)$. By a standard multiplication result for Sobolev spaces on $\R^d$ (see \cite[p.~Corollary 4.1.18 p.~24]{Web18}), the operator of multiplication $M_f\co \W^{1,p}(\R^d)\to \W^{1,p}(\R^d)$ is bounded. Hence $\Id_N \ot M_f$ leaves $\ell^p_N(\W^{1,p}(\R^d))$ invariant. Thus $\pi(f)\cdot \dom D \subset \dom D$. Let $u \in \C_c^\infty(\R^d,\mathbb{C}^N)$. Writing $u(x) \in \mathbb{C}^N$ and using the Leibniz rule, we compute for any $x \in \R^d$
\begin{align*}
\MoveEqLeft
\big[ D,\pi(f)\big]u
=D(\pi(f)u)-\pi(f)Du 
\ov{\eqref{Dirac-operator-Rd}}{=} -\i\sum_{j=1}^d \gamma_j \partial_j(f u)-\i\sum_{j=1}^d f\gamma_j \partial_j u \\
&=-\i\sum_{j=1}^d \gamma_j\big( (\partial_j f) u+ f\partial_j u\big)
-\i\sum_{j=1}^d f\gamma_j \partial_j u
=-\i\sum_{j=1}^d \gamma_j (\partial_j f) u.        
\end{align*}
In other words, we have
\begin{equation}
\label{eq:comm-explicit-pointwise}
[D,f]
=-\i\sum_{j=1}^d \gamma_j \ot M_{\partial_j f},
\end{equation}
which is exactly \eqref{eq:commutator-Rd-Dirac} on the core, if we define Clifford multiplication of the $1$-form $\d f=\sum_j (\partial_j f) \d x_j$ by
\[
\d f \cdot u(x)
\ov{\mathrm{def}}{=}
-\sum_{j=1}^d \gamma_j (\partial_j f)(x) u(x).
\]
Since $\partial_j f \in \L^\infty(\R^d)$ for each $j$, the operator $M_{\partial_j f}$ is bounded on $\L^p(\R^d)$ with
\[
\norm{M_{\partial_j f}}_{\L^p(\R^d)\to \L^p(\R^d)}
=\norm{\partial_j f}_{\L^\infty(\R^d)}.
\]
Consequently, each $\gamma_j \ot M_{\partial_j f}$ is bounded on $\ell^p_N(\L^p(\R^d))$, hence the right hand side of \eqref{eq:commutator-Rd-Dirac} extends to a bounded operator on $\ell^p_N(\L^p(\R^d))$. By density of the core $\C_c^\infty(\R^d,\mathbb{C}^N)$ in $\dom D$ and continuity, \eqref{eq:commutator-Rd-Dirac} holds on all of $\ell^p_N(\L^p(\R^d))$, and we deduce that $f \in \Lip_D(\L^\infty(\R^d))$.

We still assume $f \in \W^{1,\infty}(\R^d)$. By \eqref{eq:commutator-Rd-Dirac}, the commutator $[D,\pi(f)]$ is the multiplication operator on $\L^p(\R^d,\mathbb{C}^N)$ given by
\[
([D,\pi(f)]u)(x)
=A(x)u(x),
\qquad
A(x)
\ov{\mathrm{def}}{=} -\i\sum_{j=1}^d \gamma_j \partial_j f(x)
\in \M_N.
\]
It is standard that for such a multiplication operator we have
\[
\bnorm{[D,\pi(f)]}_{\ell^p_N(\L^p(\R^d)) \to \ell^p_N(\L^p(\R^d))}
\ov{\eqref{eq:norm-matrix-mult}}{=} \esssup_{x \in \R^d} \norm{A(x)}_{\M_N}
=\esssup_{x \in \R^d} \norm{ \sum_{j=1}^d \gamma_j \partial_j f(x) }_{\M_N},
\]
which is the first equality in \eqref{eq:norm-comm-Rd-Dirac}. 

Next, since the matrices $\gamma_j$ satisfy the Clifford relations, the linear map $
\Lambda \co \R^d \to \M_N$, $a \mapsto \sum_{j=1}^d a_j \gamma_j$ is an injective linear map from the finite-dimensional space $(\R^d,\lvert\cdot\rvert)$ into $(\M_N,\norm{\cdot}_{\M_N})$. Hence there exist constants $c_1,c_2>0$ depending only on the family $(\gamma_j)$ such that
\[
c_1 \lvert a\rvert
\leq \norm{\Lambda(a)}_{\M_N}
\leq c_2 \lvert a\rvert,
\qquad a\in\R^d.
\]
Applying this pointwise with $a=\nabla f(x)$ and taking the essential supremum over $x\in\R^d$, we obtain
\[
c_1 \norm{\nabla f}_{\L^\infty(\R^d)}
\leq \bnorm{[D,\pi(f)]}
\leq c_2 \norm{\nabla f}_{\L^\infty(\R^d)}.
\]
This proves the equivalence stated in \eqref{eq:norm-comm-Rd-Dirac}. Now suppose that $f \in \Lip_D(\L^\infty(\R^d))$. By definition, the operator $\pi(f)$ leaves $\dom D=\ell^p_N(\W^{1,p}(\R^d))$ invariant and the unbounded commutator
\[
[D,\pi(f)] \co \dom D \subset \ell^p_N(\L^p(\R^d)) \to \ell^p_N(\L^p(\R^d))
\]
extends to a bounded operator on $\ell^p_N(\L^p(\R^d))$. On the core $\C_c^\infty(\R^d,\mathbb{C}^N)$ we have, as in \eqref{eq:comm-explicit-pointwise},
\begin{equation}
\label{eq:comm-splitting-Rd}
[D,\pi(f)]
=-\i\sum_{j=1}^d \gamma_j \ot [\partial_j,M_f].
\end{equation}
Since the sum on the right-hand side is finite and $[D,\pi(f)]$ is bounded on $\ell^p_N(\L^p(\R^d))$, we will show that each commutator $[\partial_j,M_f]$ extends to a bounded operator on $\L^p(\R^d)$.

For each $k\in\{1,\dots,d\}$, choose a bounded linear functional $\omega_k$ on $\M_N$ such that $\omega_k(\gamma_\ell)=\delta_{k\ell}$ for all $\ell$. This is possible since the $\gamma_\ell$ are linearly independent. Consider the bounded operator
\[
T_k \ov{\mathrm{def}}{=}(\omega_k\ot \Id_{\L^p(\R^d)})([D,\pi(f)])\co \L^p(\R^d)\to \L^p(\R^d),
\]
where we use the identification $\ell^p_N(\L^p(\R^d))\simeq \L^p(\R^d,\mathbb{C}^N)$. Applying $\omega_k\ot \Id$ to \eqref{eq:comm-splitting-Rd} we obtain, on the core,
\[
T_k
=(\omega_k\ot \Id)\Big(-\i\sum_{j=1}^d \gamma_j \ot [\partial_j,M_f]\Big)
=-\i [\partial_k,M_f].
\]
Since the core is dense in $\L^p(\R^d)$ and both sides extend boundedly, we deduce that $[\partial_k,M_f]$ extends to a bounded operator on $\L^p(\R^d)$ and
\[
\norm{[\partial_k,M_f]}_{\L^p(\R^d)\to \L^p(\R^d)}
\leq \norm{\omega_k}\bnorm{[D,\pi(f)]}_{\ell^p_N(\L^p)\to \ell^p_N(\L^p)}.
\]
Fix $j\in\{1,\dots,d\}$. Let $(\tau_{t e_j})_{t\in\R}$ be the translation group on $\L^p(\R^d)$ in the $e_j$-direction, given by
\[
(\tau_{t e_j}g)(x)
\ov{\mathrm{def}}{=}g(x+t e_j),
\qquad g\in \L^p(\R^d),\ t\in\R.
\]
It is well known that $(\tau_{t e_j})_{t\in\R}$ is a strongly continuous group of isometries on $\L^p(\R^d)$ with generator $\partial_j$. Moreover, for any $t\in\R$ we have
\[
\tau_{t e_j} M_f \tau_{-t e_j}
=M_{f(\cdot+ t e_j)}.
\]
Let $t \ne 0$ and define the difference quotient
\[
\Delta_t^{(j)} f
\ov{\mathrm{def}}{=} \frac{f(\cdot+ t e_j)-f}{t} \in \L^\infty(\R^d).
\]
By Proposition~\ref{prop-group-isometries} (applied with $X=\L^p(\R^d)$, $U_t=\tau_{t e_j}$, $A=\partial_j$ and $R=M_f$), since $[\partial_j,M_f]$ is bounded on $\L^p(\R^d)$, we obtain for all small $t\ne 0$,
\begin{equation}
\label{eq:FTC-Rd}
M_{\Delta_t^{(j)} f}
\ov{\mathrm{def}}{=}\frac{M_{f(\cdot+t e_j)}-M_f}{t}
=\frac{\tau_{t e_j} M_f \tau_{-t e_j}-M_f}{t}
=\int_0^1 \tau_{s t e_j} [\partial_j,M_f] \tau_{-s t e_j} \d s.
\end{equation}
Since each $\tau_{s t e_j}$ is an isometry, we obtain that
\begin{align*}
\MoveEqLeft
\bnorm{\Delta_t^{(j)} f}_{\L^\infty(\R^d)}
=\bnorm{M_{\Delta_t^{(j)} f}}_{\L^p(\R^d)\to \L^p(\R^d)} 
=\norm{\int_0^1 \tau_{s t e_j} [\partial_j,M_f] \tau_{-s t e_j} \d s}_{\L^p(\R^d)\to \L^p(\R^d)} \\
&\leq \int_0^1 \bnorm{\tau_{s t e_j} [\partial_j,M_f] \tau_{-s t e_j}}_{\L^p(\R^d)\to \L^p(\R^d)} \d s 
=\bnorm{[\partial_j,M_f]}_{\L^p(\R^d)\to \L^p(\R^d)}.
\end{align*}
Thus, for each fixed $j$, the family $(\Delta_t^{(j)} f)_{t\ne 0}$ is bounded in $\L^\infty(\R^d)$. By the Banach–Alaoglu theorem, there exist $t_\alpha\to 0$ and $g_j\in \L^\infty(\R^d)$ such that $
\Delta_{t_\alpha}^{(j)} f
\overset{w^\ast}{\longrightarrow}
g_j$ in $\L^\infty(\R^d)$. Let $\varphi\in \C_c^\infty(\R^d)$. Then
\[
\int_{\R^d} (\Delta_t^{(j)} f)(x)\varphi(x)\,\d x
=\int_{\R^d} f(x) \frac{\varphi(x- t e_j)-\varphi(x)}{t}\,\d x.
\]
Since $\frac{\varphi(\cdot- t e_j)-\varphi}{t}\to -\partial_j\varphi$ in $\L^1(\R^d)$ as $t\to 0$, we deduce
\[
\int_{\R^d} (\Delta_t^{(j)} f)\,\varphi
\xra[t \to 0]{}  -\int_{\R^d} f\,\partial_j\varphi.
\]
Passing to the weak* limit along $(t_\alpha)$ yields
\[
\int_{\R^d} g_j\,\varphi
=-\int_{\R^d} f\,\partial_j\varphi,
\qquad \varphi\in \C_c^\infty(\R^d),
\]
that is, $g_j$ is the distributional derivative $\partial_j f$. Since $g_j\in \L^\infty(\R^d)$, we conclude that $\partial_j f\in \L^\infty(\R^d)$ for every $j\in\{1,\dots,d\}$. Hence $f\in \W^{1,\infty}(\R^d)$.
\end{proof}

\begin{prop}
\label{prop-fuctional-calculus-Rd}
Suppose that $1 < p < \infty$. The operator $D$ is bisectorial of angle $0$, that is $\omega_\bi(D)=0$, and admits a bounded $\H^\infty(\Sigma_\theta^\bi)$ functional calculus on $\ell^p_N(\L^p(\R^d))$ for every angle $\theta>0$.
\end{prop}

\begin{proof}
The symbol in the sense of \cite[p.~75]{HMP11} of the differential operator $D$ defined in \eqref{Dirac-operator-Rd} is
\begin{equation}
\label{symbol-Dirac}
\symb_D(\xi)
=\sum_{j=1}^d \gamma_j \xi_j,
\qquad \xi \in \R^d.
\end{equation}
A straightforward computation using the Clifford relations gives
\begin{align}
\MoveEqLeft
\label{eq:Dirac-symbol-square}
\symb_D(\xi)^2
\ov{\eqref{symbol-Dirac}}{=} \sum_{j,k=1}^d \gamma_j\gamma_k\,\xi_j\xi_k
=\frac12\sum_{j,k=1}^d (\gamma_j\gamma_k+\gamma_k\gamma_j) \xi_j\xi_k
\ov{\eqref{Clifford-Rn}}{=} \frac12\sum_{j,k=1}^d 2\delta_{j,k}\Id_{\mathbb{C}^N} \xi_j\xi_k \\
&=\sum_{j=1}^d \xi_j^2 \Id_{\mathbb{C}^N}
=\norm{\xi}_{\ell^2_d}^2 \Id_{\mathbb{C}^N}. \nonumber
\end{align}
For each fixed $\xi \in \R^d$, the matrix $\symb_D(\xi)$ is selfadjoint, so its spectrum is contained in $\R$. 
This shows that the condition (D2) of \cite[p.~75]{HMP11} holds with the angle $\omega=0$. Moreover, for any $\xi \in \R^d$ and $u \in \mathbb{C}^N$, we have
\begin{align*}
\norm{\symb_D(\xi) u}_{\ell^2_N}^2
&=\big\la \symb_D(\xi)u,\symb_D(\xi)u \big\ra_{\ell^2_N}
\ov{\eqref{eq:Dirac-symbol-square}}{=} \big\la \symb_D(\xi)^2 u,u \big\ra_{\ell^2_N} \\
&=\big\la \norm{\xi}_{\ell^2_d}^2 \Id_{\mathbb{C}^N} u,u \big\ra_{\ell^2_N}
=\norm{\xi}_{\ell^2_d}^2 \norm{u}_{\ell^2_N}^2.
\end{align*}
Taking square roots, we obtain the exact identity
\[
\bnorm{\symb_D(\xi)u}_{\ell^2_N}
=\norm{\xi}_{\ell^2_d} \norm{u}_{\ell^2_N},
\qquad \xi \in \R^d, u \in \mathbb{C}^N.
\]
Hence the condition (D1) of \cite[p.~75]{HMP11} holds with $\kappa=1$. We conclude with \cite[Theorem 3.2 p.~75]{HMP11} that the Dirac operator $D$ is bisectorial of angle $0$ and that it admits a bounded $\H^\infty(\Sigma_\theta^\bi)$ functional calculus on $\L^p(\R^d,\mathbb{C}^N)$ for every $\theta > 0$.
\end{proof}

\begin{prop}
\label{prop:Ma-iD-compact}
Suppose $1 < p < \infty$. For any function $f \in \C_c^\infty(\R^d)$ the operator $M_f(\i + D)^{-1} \co \ell^p_N(\W^{1,p}(\R^d)) \to \ell^p_N(\W^{1,p}(\R^d))$ is compact. 
\end{prop}

\begin{proof}
The resolvent operator $(\i+D)^{-1} \co \ell^p_N(\L^{p}(\R^d)) \to \ell^p_N(\L^{p}(\R^d))$ is bounded by definition. Let $f \in \C_c^\infty(\R)$. Consider a bounded open subset $\Omega \subset \R^d$ such that $\supp f\subset \Omega$. Note that we have a bounded operator $M_f \co \W^{1,p}(\R^d) \to \W^{1,p}(\Omega)$. We denote by $J \co \W^{1,p}(\Omega) \to \L^p(\Omega)$ the canonical injection, which is compact by \cite[Remark p.~289]{Eva10} or \cite[Theorem 9.16 p.~285]{Bre11}. If $\cal{E} \co \L^p(\Omega) \to \L^p(\R^d)$ is the extension map by 0, then the operator $M_f(\i+D)^{-1}$ admits the factorization
\[
M_f(\i+D)^{-1}
= E\circ J \circ M_f \circ (\i+D)^{-1},
\]
Since this factorization contains a compact operator, it is also a compact operator.
\end{proof}

We will give a quantitative compactness result in Proposition \ref{prop:Sappq-Dirac-Rd}. We need the following estimate on approximation numbers of the Sobolev embedding is a particular case of general results by K\"onig on $s$-numbers of Sobolev embeddings. For the reader's convenience, we give a short self-contained proof based only on a simple dyadic partition argument and Poincar\'e inequality \cite[Theorem 2.4 p.~21]{EdE23}
\begin{equation}
\label{Poincare-inequality}
\norm{u-\frac{1}{|\Omega|}\int_\Omega u}_{\L^p(\Omega)}
\leq \bigg(\frac{\omega_d}{|\Omega|}\bigg)^{1-\frac{1}{d}} \diam(\Omega)^d\norm{\nabla u}_{\L^p(\Omega,\ell^2_d)}
\end{equation}
for a bounded convex subset $\Omega$ of $\R^d$.

\begin{prop}
\label{prop:approx-numbers-W1p-into-Lp-d}
Suppose that $1 < p < \infty$. Let $d\geq 1$. Consider a closed cube $Q \subset \R^d$. Denote by $J \co \W^{1,p}(Q) \to \L^p(Q)$ the canonical embedding. For any integer $n \geq 1$, we have
\begin{equation}
\label{eq:an-general-d}
a_{n}(J)
\lesssim_d |Q|^{\frac{1}{d}} n^{-\frac{1}{d}}.
\end{equation}
\end{prop}

\begin{proof}
Fix an integer $m \in \N$. Partition $Q$ into $m^d$ pairwise disjoint closed subcubes $Q_1,\dots,Q_{m^d}$ of equal measure $|Q_k|=\frac{|Q|}{m^d}$ and side length
\begin{equation}
\label{def-ell-m-Rd}
\ell_m 
\ov{\mathrm{def}}{=} \Big(\frac{|Q|}{m^d}\Big)^{\frac{1}{d}}
= |Q|^{\frac{1}{d}} m^{-1}.
\end{equation}
For any function $g \in \W^{1,p}(Q)$, define the average
\[
(P_m g)(x) 
\ov{\mathrm{def}}{=} \sum_{k=1}^{m^d}
\Big(\frac{1}{|Q_k|}\int_{Q_k} g\Big)\, 1_{Q_k}(x).
\]
Clearly, we have $\rank P_m \leq m^d$.  On each $Q_k$, the Poincar\'e inequality \eqref{Poincare-inequality} applied to $\Omega=Q_k$ yields
\begin{equation}
\label{eq:poincare-local-d}
\norm{g|_{Q_k} - g_{Q_k}}_{\L^p(Q_k)}
\lesssim_d \ell_m \norm{\nabla g}_{\L^p(Q_k)},
\qquad 
g_{Q_k}\ov{\mathrm{def}}{=} \frac{1}{|Q_k|} \int_{Q_k} g.
\end{equation}
Summing over $k$ (the $Q_k$ are disjoint), we obtain
\begin{align*}
\norm{g-P_m(g)}_{\L^p(Q)}^p
&= \sum_{k=1}^{m^d} \norm{g|_{Q_k}-g_{Q_k}}_{\L^p(Q_k)}^p 
\ov{\eqref{eq:poincare-local-d}}{\lesssim_d} \sum_{k=1}^{m^d}(\ell_m\norm{\nabla g}_{\L^p(Q_k)})^p
= \ell_m^p \norm{\nabla g}_{\L^p(Q)}^p.
\end{align*}
Taking the $p$-th root gives
\[
\norm{g-P_m g}_{\L^p(Q)}
\lesssim_d \ell_m \norm{\nabla g}_{\L^p(Q)}
\lesssim_d \ell_m \norm{g}_{\W^{1,p}(Q)}.
\]
Hence
\begin{equation}
\label{eq:J-Pm-bound-d}
\norm{J-P_m}_{\W^{1,p}(Q)\to \L^p(Q)}
\lesssim_d \ell_m
\ov{\eqref{def-ell-m-Rd}}{=} |Q|^{\frac{1}{d}} m^{-1}.
\end{equation}
Since $\rank P_m \leq m^d$, the definition of approximation numbers gives
\[
a_{m^d+1}(J)
\ov{\eqref{def-approximation-numbers}}{\leq} \norm{J-P_m}
\ov{\eqref{eq:J-Pm-bound-d}}{\lesssim_d}
|Q|^{\frac{1}{d}} m^{-1}.
\]
For an arbitrary integer $n \geq 1$, choose $m=\lfloor (n-1)^{\frac{1}{d}}\rfloor$, so that $m^d+1 \leq n$. By monotonicity of the sequence $(a_n(J))$, we have
\[
a_n(J) 
\leq a_{m^d+1}(J)
\lesssim_d |Q|^{\frac{1}{d}} m^{-1}
\lesssim_d |Q|^{\frac{1}{d}} n^{-\frac{1}{d}}.
\]
\end{proof}

\begin{prop} 
\label{prop:Sappq-Dirac-Rd}
Suppose that $1 < p < \infty$. For any function $f \in \C_c^\infty(\R^d)$ the operator $\pi(f)(D+\i)^{-1}$ belongs to $S_{\app}^{d,\infty}\big(\ell^p_N(\L^{p}(\R^d))\big)$.
\end{prop}

\begin{proof}
As in \eqref{square-of-Dirac}, the Clifford relations imply that $
D^2
=-\Id_{\ell^p_N} \ot \Delta$, 
where $\Delta$ denotes the Laplacian on $\R^d$. Hence we have the algebraic identity
\[
(\i + D)^{-1}
=(\i - D)(1+D^2)^{-1}
=(\i - D)\big(1-\Delta\big)^{-1} \ot \I_N.
\]
The Bessel potential operator $(1-\Delta)^{-1} \co \L^p(\R^d) \to \W^{2,p}(\R^d)$ is bounded, and the first order operator $\i - D$ maps $\ell^p_N(\W^{2,p}(\R^d))$ boundedly into $\ell^p_N(\W^{1,p}(\R^d))$. Set
\[
S \ov{\mathrm{def}}{=} (\i - D)\big(1-\Delta\big)^{-1} \ot \I_N
\co \ell^p_N(\L^p(\R^d)) \to \ell^p_N(\W^{1,p}(\R^d)).
\]
Let $f \in \C_c^\infty(\R^d)$. Fix a closed cube $Q \subset \R^d$ such that $\supp f \subset Q$. We still write $M_f$ for the diagonal multiplication operator on $\ell^p_N(\L^p(\R^d))$ and on $\ell^p_N(\L^p(Q))$. Consider the restriction operator $
R_Q \co \ell^p_N(\W^{1,p}(\R^d)) \to \ell^p_N(\W^{1,p}(Q))$ and the extension by zero operator $
E_Q \co \ell^p_N(\L^p(Q)) \to \ell^p_N(\L^p(\R^d))$. 
Let $
J^{\oplus N} \co \ell^p_N(\W^{1,p}(Q)) \hookrightarrow \ell^p_N(\L^p(Q))$ 
be the canonical Sobolev embedding on the vector-valued space, that is, the direct sum of $N$ copies of the embedding $
J \co \W^{1,p}(Q) \hookrightarrow \L^p(Q)$. 
Then we have the factorization
\begin{equation}
\label{eq:factorization-dD}
\pi(f)(\i + D)^{-1}
= E_Q \underbrace{\pi(f)|_{\ell^p_N(\L^p(Q))}}_{M_f^Q} J^{\oplus N} \underbrace{R_Q S}_{S_Q}.
\end{equation}
We now estimate the approximation numbers of $J^{\oplus N}$. Applying Lemma \ref{lemma:approx-finite-direct-sum} with $T_k=J$ for all $k \in \{1,,\ldots,N\}$ yields for any integer $n \geq 1$,
\begin{equation}
\label{eq:an-JN-step1}
a_{Nn}(J^{\oplus N})
\leq \max_{1 \leq k \leq N} a_n(J)
= a_n(J).
\end{equation}
Let $m \geq 1$. Choose $n \geq 1$ such that $
N(n-1) < m \leq Nn$. Since $(a_m(J^{\oplus N}))_{m \geq 1}$ is decreasing, we deduce that
\begin{equation}
\label{eq:an-JN-step2}
a_m(J^{\oplus N})
\leq a_{Nn}(J^{\oplus N})
\ov{\eqref{eq:an-JN-step1}}{\leq} a_n(J)
\ov{\eqref{eq:an-general-d}}{\lesssim} C_Q n^{-\frac{1}{d}}, 
\end{equation}
where we use Proposition \ref{prop:approx-numbers-W1p-into-Lp-d}. Using $n \leq \frac{m}{N}+1$, we obtain a constant $C'_Q>0$ with
\begin{equation}
\label{eq:an-vector-Sobolev}
a_m(J^{\oplus N})
\leq C'_Q m^{-\frac{1}{d}},
\qquad m \geq 1.
\end{equation}
Finally, we estimate the approximation numbers of $\pi(f)(\i + D)^{-1}$. Using the ideal property of approximation numbers together with the factorization \eqref{eq:factorization-dD}, we get for all integers $m \geq 1$,
\begin{align*}
a_m\big(\pi(f)(\i + D)^{-1}\big)
&\leq \norm{E_Q} \bnorm{M_f^Q} \norm{S_Q} a_m(J^{\oplus N}) 
\ov{\eqref{eq:an-vector-Sobolev}}{\lesssim} m^{-\frac{1}{d}}.
\end{align*}
The implicit constant depends only on $p$, $d$, $N$, $f$ and $Q$. Therefore
\[
\sup_{m \geq 1} m^{\frac{1}{d}} a_m\big(\pi(f)(\i + D)^{-1}\big)
< \infty,
\]
which means exactly that $\pi(f)(\i + D)^{-1} \in S_{\app}^{d,\infty}\big(\ell^p_N(\L^p(\R^d))\big)$ in the sense of \eqref{def-quasi-norm-Spq}.
\end{proof}

The previous results allows us to state the following result.

\begin{thm}
\label{thm-spectral-triple-Rd}
Suppose that $1 < p < \infty$. Then $(\C_0(\R^d),\ell^p_N(\W^{1,p}(\R^d)),D)$ is a locally compact spectral triple which is $d^+$-summable, hence $q$-summable for any $q > d$. 
\end{thm}

\newpage

\subsection{The Dirac operator $\frac{1}{\i}\frac{\d}{\d x}$ on the space $\L^p(\R)$ and the Hilbert transform}

Assume that $d=1$. We have $N=2^{\lfloor \frac{d}{2}\rfloor}=1$. Then the Dirac operator is given by
\[
D
=-\i \gamma_1 \ot \partial_{x}
=\frac{1}{\i} \frac{\d}{\d x},
\]
with domain $\W^{1,p}(\R)$ acting on $\L^p(\R))$. Here \(\gamma_1=[1]\) is an \(N\times N\) selfadjoint matrix satisfying the Clifford relation \(\gamma_1^2=\Id_{\mathbb{C}^N}\). In the scalar one-dimensional case we take \(N=1\) and choose the representation of the Clifford algebra given by
\[
\gamma_1=1 \in \M_1(\C).
\]

Suppose that $1 < p < \infty$. Note that by \cite[Proposition 4.10 p.~31]{EnN00} any multiplication operator $M_f \co \L^p(\R) \to \L^p(\R)$, $g \mapsto fg$ by a measurable function $f \co \R \to \mathbb{C}$ is bounded if and only if $f$ is essentially bounded. In this case, we have $\norm{M_f}_{\L^p(\R) \to \L^p(\R)}=\norm{f}_{\L^\infty(\R)}$. Consequently, we can consider the isometric homomorphism $\pi \co \L^\infty(\R) \to \B(\L^p(\R))$, $f \mapsto M_f$. It is well known \cite[Corollary 14 p.~92]{Ren04} \cite[Remark 7.14]{Var10} that $(\C_0(\R),\L^2(\R),\frac{1}{\i}\frac{\d}{\d x})$ is a locally compact spectral triple, where the used homomorphism is the restriction of $\pi$ on the $\C^*$-algebra $\C_0(\R)$. We begin to prove a an $\L^p$-variant of this classical fact. Here, we consider the closure of the closable unbounded operator $\frac{\d}{\d x} \co \C^\infty_0(\R) \subset\L^p(\R) \to \L^p(\R)$. We denote again $\frac{\d}{\d x}$ its closure whose the domain $\dom \frac{\d}{\d x}$ is the Sobolev space 
$$
\W^{1,p}(\R)
\ov{\mathrm{def}}{=}\{f \in \L^p(\R) : f \text{ is absolutely continuous}, f' \in \L^p(\R) \},
$$  
see, e.g., \cite[Proposition 8.4.1 p.~228]{Haa06}. The following result is a consequence of Proposition \ref{prop:Lip-alg-Rd-Dirac}.

\begin{prop}
\label{prop:Lip-alg-R}
Suppose that $1 < p < \infty$. Let $D \ov{\mathrm{def}}{=} \frac{1}{\i}\frac{\d}{\d x}$ with domain  $\dom D=\W^{1,p}(\R)$. Then
\begin{equation}
\label{eq:Lip-equals-W1infty}
\W^{1,\infty}(\R) 
= \Lip_{D}(\L^\infty(\R)).
\end{equation}
Moreover, for any function $f \in \W^{1,\infty}(\R)$ we have
\begin{equation}
\label{eq:commutator-formula-R}
[D,M_{f}]
=\frac{1}{\i} M_{f'} 
\quad \text{and} \quad
\norm{[D,M_f]}_{\L^p(\R) \to \L^p(\R)}
=\norm{f'}_{\L^\infty(\R)}.
\end{equation}
\end{prop}

The following result is a particular case of Proposition \ref{prop-fuctional-calculus-Rd}. We give an alternative proof.

\begin{prop}
\label{prop-fuctional-calculus-R}
Suppose $1 < p < \infty$. Then the operator $D=\frac{1}{\i}\frac{\d}{\d x}$ on $\L^p(\R)$ with domain $\W^{1,p}(\R)$ is bisectorial of angle 0, i.e., $\omega_\bi(D)=0$, and admits a bounded $\H^\infty(\Sigma_\theta^\bi)$ functional calculus for any angle $\theta > 0$ on the Banach space $\L^p(\R)$. 
\end{prop}

\begin{proof}
The operator $\frac{\d}{\d x}$ generates a strongly continuous group of operators acting on the Banach space $\L^p(\R)$, namely the left translation group, defined by $(T_t g)(x)=g(x+t)$ where $t,x \in \R$ (see \cite[Proposition 1 p.~66]{EnN00} or \cite[Proposition 8.4.2 p.~229]{Haa06}). By Example \ref{generators-bisectorial} and Example \ref{ex-bound-Hinfty}, we deduce that the unbounded operator $D\ov{\mathrm{def}}{=}\frac{1}{\i}\frac{\d}{\d x}$ is bisectorial with $\omega_\bi(D)=0$ and admits a bounded $\H^\infty(\Sigma_\theta^\bi)$ functional calculus for any angle $\theta > 0$ on the Banach space $\L^p(\R)$. 
\end{proof}

The next result is a particular case of Proposition \ref{prop:Ma-iD-compact}. We give another argument.

\begin{prop}
\label{prop:Ma-iD-compact}
Suppose $1 < p < \infty$. Consider the operator $D=\frac{1}{\i}\frac{\d}{\d x}$ on $\L^p(\R)$ with domain $\W^{1,p}(\R)$. For any function $f \in \C_c^\infty(\R)$ the operator $M_f(\i + D)^{-1} \co \L^p(\R)\to \L^p(\R)$ is compact. 
\end{prop}

\begin{proof}
Consider the infinitesimal generator $A \ov{\mathrm{def}}{=} \frac{\d}{\d x}$ of the left translation group $(T_t)_{t \in \R}$ on the Banach space $\L^p(\R)$. By the Laplace formula \cite[Proposition G.4.1 p.~532]{HvNVW18} for the resolvent, we have for $\Re\lambda > 0$ the equality
\begin{equation}
\label{Laplace-formula}
(\lambda - A)^{-1} g 
= \int_0^\infty \e^{-\lambda t} T_t g \d t, \quad g \in \L^p(\R).
\end{equation}
Taking $\lambda=1$ and using $\i + D=\i(1 - A)$, we obtain for almost all $x \in \R$
\begin{align}
\MoveEqLeft
\label{inter-AZERTY-55}
\big((\i + D)^{-1}g\big)(x) 
=-\i\big((1-A)^{-1}g\big)(x)
\ov{\eqref{Laplace-formula}}{=} -\i\int_0^\infty \e^{-t} g(x+t) \d t \\
&\ov{t=y-x}{=}-\i \int_x^\infty \e^{x-y} g(y) \d y 
=-\i \int_\R k(x-y) g(y) \d y, \nonumber
\end{align}
with $k(t) \ov{\mathrm{def}}{=} \e^{t}1_{(-\infty,0]}(t)$. Hence $M_f(\i+D)^{-1}$ is an integral operator with kernel
\[
K(x,y)
\ov{\mathrm{def}}{=} -\i f(x) k(x-y).
\]
Note that $k \in \L^{p^*}(\R)$, where $p^\ast$ is the conjugate exponent defined by $\frac{1}{p}+\frac{1}{p^*}=1$. 
For almost all $x \in \R$, we have
\begin{equation}
\label{inter-3567}
\norm{y \mapsto K(x,y)}_{\L^{p^*}(\R)} 
=|f(x)| \norm{y \mapsto k(x-y)}_{\L^{p^*}(\R)}
= |f(x)| \norm{k}_{\L^{p^*}(\R)},
\end{equation}
Consequently, we have
\[
\int_{\R} \norm{y \mapsto K(x,y)}_{\L^{p^*}(\R)}^{p} \d x
\ov{\eqref{inter-3567}}{=} \norm{k}_{\L^{p^*}}^{p} \int_{\R} |f(x)|^p \d x
= \norm{k}_{\L^{p^*}(\R)}^{p} \norm{f}_{\L^p(\R)}^{p}
< \infty,
\]
since $f \in \C_c^\infty(\R)$. The condition $K \in \L^p(\R,\L^{p^*}(\R))$ implies that $M_f(\i+D)^{-1}$ is a Hille-Tamarkin operator, hence compact on $\L^p(\R)$ by \cite[Theorem p.~446]{Jor82} or \cite[Theorem 11.6 p.~275]{HiT34}.
\end{proof}

\begin{thm}
\label{thm-spectral-triple-R}
Suppose that $1 < p < \infty$. Then $(\C_0(\R),\L^p(\R),\frac{1}{\i}\frac{\d}{\d x})$ is a locally compact spectral triple which is $1^+$-summable, hence $q$-summable for any $q > 1$. More precisely, we have
$$
\dim \big(\C_0(\R),\L^p(\R),\tfrac{1}{\i}\tfrac{\d}{\d x}\big) 
= 1. 
$$
\end{thm}

\begin{proof}
By Proposition \ref{prop-fuctional-calculus-R}, the unbounded operator $D\ov{\mathrm{def}}{=}\frac{1}{\i}\frac{\d}{\d x}$ is bisectorial and admits a bounded $\H^\infty(\Sigma_\theta^\bi)$ functional calculus for any angle $\theta > 0$ on the Banach space $\L^p(\R)$. 

By Proposition \ref{prop:Lip-alg-R}, the Lipschitz algebra $\Lip_D(\C_0(\R))=\Lip_D(\L^\infty(\R)) \cap \C_0(\R)$ contains $\C_c^\infty(\R)$. Hence $\Lip_D(\C_0(\R))$ is dense in $\C_0(\R)$. Finally, Proposition \ref{prop:Ma-iD-compact} gives the compactness condition. The spectral triple is $1^+$-summable by Proposition \ref{prop:Sappq-Dirac-Rd}.

Let $f \in \C_c^\infty(\R)$ such that $f = 1$ on a closed interval $I=[a,b]$. Consider the restriction operator $R_I\co \L^p(\R) \to \L^p([a,b])$ and the extension by zero operator $E_I \co \L^p([a,b]) \to \L^p(\R)$. 
For almost all $x \in \R$ and any $g \in \L^p(\R)$, we have seen that
\[
\big((\i+D)^{-1}g\big)(x)
\ov{\eqref{inter-AZERTY-55}}{=} -\i \int_x^\infty \e^{x-y} g(y) \d y,\quad x \in \R.
\]
Hence for almost all $x \in [a,b]$ and $g \in \L^p([a,b])$,
\[
\big(R_I M_f(\i+D)^{-1}E_I g\big)(x)
= -\i \int_x^{b} e^{x-y} g(y)\d y.
\]
Consider the Hardy operator $H \co \L^p([a,b]) \to \L^p([a,b])$ defined by $
(Vh)(x)
\ov{\mathrm{def}}{=}\int_x^{b} h(y) \d y$ and the multiplication operators $U \co \L^p([a,b]) \to \L^p([a,b])$, $g \mapsto( y \mapsto \e^{-y}g(y))$, and $W \co \L^p([a,b]) \to \L^p([a,b])$, $x \mapsto (x \mapsto \e^{x}h(x))$. Then
\begin{equation}
\label{inter-OPJJ89}
R_I M_f(\i+D)^{-1}E_I
= -\i WHU.
\end{equation}
Since $U$ and $W$ are bounded isomorphisms on the Banach space $\L^p([a,b])$, the ideal property of approximation numbers yields
\[
a_n(H)
\ov{\eqref{majo-sn-2}\eqref{inter-OPJJ89}}{\leq} \bnorm{W^{-1}} \bnorm{U^{-1}} a_n\big(R_I M_f(D+\i)^{-1}E_I\big)
\lesssim a_n\big(M_f(D+\i)^{-1}\big).
\]
Thus it suffices to use the estimate $a_n(H) \gtrsim \frac{1}{n}$ of \cite[Theorem 3.3 p.~25]{EdL07}. By \eqref{def-quasi-norm-Spq}, we conclude that $(\C_0(\R),\L^p(\R),\frac{1}{\i}\frac{\d}{\d x})$ is not 1-summable.
%
\end{proof}

Now, we consider the Hilbert transform $\cal{H} \co \L^p(\R) \to \L^p(\R)$. This transformation is a bounded operator by \cite[Theorem  5.1.1 p.~374]{HvNVW16} defined by the principal value
\begin{equation}
\label{Hilbert-transform-Riesz}
(\cal{H}f)(x)
\ov{\mathrm{def}}{=} \frac{1}{\pi}\pv \int_\R \frac{f(y)}{x-y} \d y, \quad f \in \cal{S}(\R), \text{ for almost every } x \in \R,
\end{equation}
and can be seen by \cite[Proposition 5.2.2 p.~389]{HvNVW16} as a Fourier multiplier with symbol $-\i\,\sign(\xi)$, i.e., we have
\begin{equation}
\label{Hilbert-trans-Fourier}
\widehat{\cal{H}(f)}(\xi)
=-\i\, \sign(\xi)\hat{f}(\xi), \quad f \in \cal{S}(\R), \xi \in \R.
\end{equation}
Note that it was observed in \cite[p.~314]{Con94} (see also \cite[p.~330]{GVF01} and \cite[p.~349]{LMSZ23}), that $F \ov{\mathrm{def}}{=} \i \cal{H}$ is a Fredholm module on the Hilbert space $\L^2(\R)$. As observed in \cite[p.~473]{GaR85}, the commutator is defined by the formula \eqref{commutator-Hilbert-R}.
\begin{remark} \normalfont
A famous result of Coifman, Rochberg and Weiss \cite{CRW76}, described in \cite[p.~473]{GaR85} and \cite[Theorem p.~2]{Wic20}, says that the formula \eqref{commutator-Hilbert-R} induces a bounded operator on the Banach space $\L^p(\R)$ if and only if the function $f$ belongs to the space $\BMO(\R)$.
\end{remark}

\begin{remark} \normalfont
If $0 < q < \infty$, a classical result of Peller (and Semmes for the case $q < 1$), described in \cite[p.~314]{Con94}, says that the commutator $[F, M_f]$ is in the Schatten class $S^q(\L^2(\R))$ if and only if the function $f$ belongs to the Besov space $\B^q_{q,q}(\R)$. 
\end{remark} 

%
%
%

For the computation of the index pairing, we need some information on Hardy spaces. Recall the notation $\mathbb{C}_+ \ov{\mathrm{def}}{=} \{z \in \mathbb{C} : \Im z > 0\}$ for the open upper half plane.  Following \cite[Definition 6.3.2 p.~145]{Nik02} and \cite[p.~247]{Mas09}, we denote by $\H^p(\mathbb{C}_+)$ the space of  all functions $F$ which are holomorphic in the upper half plane $\mathbb{C}_+$ such that 
\[
\norm{F}_{\H^p(\mathbb{C}_+)} 
\ov{\mathrm{def}}{=} \sup_{y >0} \left (\int_\R |F(x+\i y)|^p \d x \right )^{\frac{1}{p}} 
<\infty,
\]
and $\sup_{z \in \mathbb{C}_+} |F(z)| < \infty$ for $p=\infty$. We also introduce the closed subspace 
$$  
\H^{p}(\R)
=\left\{ f \in \L^{p}(\R) : \int_{\R}\frac{f(t)}{t-\overline{z}}\d t=0 \text{ for all }z \in \mathbb{C}_{+} \right\}
$$
of the Banach space $\L^{p}(\R)$, considered in \cite[(13.4) p.~282]{Mas09}. Cauchy's representation theorem in Hardy spaces, e.g., \cite[Chapter 13]{Mas09}, can be formulated as follows. For any holomorphic function $F \co \mathbb{C}_{+} \to \mathbb{C}$, the following assertions are equivalent:
\begin{itemize}
\item[(i)]  the function $F$ belongs to the space $\H^{p}(\mathbb{C}_{+})$,
\item[(ii)] there exists a function $f \in \H^{p}(\R)$ such that
\begin{equation}
\label{eq:luxboqaurg}
F(z)
= \frac{1}{2\pi \i}\int_{\R}\frac{f(t)}{t-z} \d t, \quad z \in \mathbb{C_{+}}.
\end{equation}
\end{itemize}
In this case, $f$ is unique, we have  
$  
\norm{F}_{\H^{p}(\mathbb{C}_{+})}  
=\norm{f}_{\L^{p}(\R)}    
$ 
and the non-tangential boundary function of $F$ is equal to $f$, i.e., $f(x)
=\lim_{z \to x,z \in \mathbb{C}_+} F(z)$. Consequently, we have an isometric isomorphism of the Hardy space $\H^{p}(\mathbb{C}_{+})$ onto the space $\H^{p}(\R)$. Thus the Hardy space $\H^{p}(\mathbb{C}_{+})$ can be identified with a closed subspace of the Banach space $\L^{p}(\R)$. Recall that $P \ov{\mathrm{def}}{=} \frac{\Id+F}{2} \co \L^{p}(\R) \to \L^{p}(\R)$ is a bounded projection on the subspace $\H^p(\R)$. 

Let $f \co \R \cup \{\infty\} \to \mathbb{C}$ be a continuous function satisfying
\begin{equation}
\label{condition-de-positivite}
\inf_{x \in \R \cup \{\infty\}} |f(x)| > 0.
\end{equation}
Using the identification between one-point compactification $\R \cup \{\infty\}$ of $\R$ with the unit circle $\T$ via the standard homeomorphism $
\Phi \colon \R \cup \{\infty\} \to \T$, $x \mapsto \frac{x - \i}{x + \i}$, we then define the winding number of $f$ by $\wind f \, \ov{\mathrm{def}}{=}\, \wind (f \circ \Phi^{-1})$. Note that condition \eqref{condition-de-positivite} means that $f$ is invertible in the commutative Banach algebra $\C(\R \cup \{\infty\})$.





We will use the following result \cite[Theorem 6.2 p.~141]{Cam17}.
\begin{thm}
\label{th-index-LpR}
Let $f \in \C(\R \cup\{\infty\})$. The Toeplitz operator $T_f \co \H^p(\R) \to \H^p(\R)$, $g \mapsto P(fg)$  with symbol $f$ has closed range if and only if \eqref{condition-de-positivite} is satisfied. In this case, the operator $T_f$ is Fredholm and its Fredholm index is $\Index T_f=-\wind f$.
\end{thm}


%


The following observation is new even if $p=2$. Recall that $\K_1(\C_0(\R))=\Z$, see for instance \cite[p.~123]{WeO93}. 

\begin{prop}
Suppose that $1 < p < \infty$. The triple $(\L^p(\R),\pi,\i\cal{H})$ is an odd Banach Fredholm module over the algebra $\C_0(\R)$.  
For any function $f \in \C_0^\infty(\R)$ satisfying \eqref{condition-de-positivite}, we have
 \begin{equation}
\label{}
\big\la [f], (\L^p(\R),\pi,\i\cal{H}) \big\ra_{\K_1(\C_0(\R)),\K^1(\C_0(\R),\scr{L}^p)}
=-\wind f.
\end{equation}
\end{prop}

\begin{proof}
We let $F \ov{\mathrm{def}}{=} \i\cal{H}$. We have $F^2=\Id$. It is known \cite[Theorem 2 p.~17]{Uch78} that a function $f \in \L^\infty(\R)$ belongs to the space $\VMO(\R)$ of functions of vanishing mean oscillation if and only if the commutator $[F,M_f]$ is a compact operator acting on the Banach space $\L^p(\R)$. Since $\C_0(\R)$ is a subspace of $\VMO(\R)$ by \cite[6.8 (a) p.~180]{Ste93}, we infer that $(\L^p(\R),\pi,F)$ is a odd Fredholm module over the algebra $\C_0(\R)$ of continuous functions that vanish at infinity.

Let $f \in \C_0(\R)$  satisfying \eqref{condition-de-positivite}. This means that $f$ is an invertible element of the unitization $\C(X \cup\{\infty\})$ of the algebra $\C_0(X)$. The operator $PM_fP \co P(\L^p(\R)) \to P(\L^p(\R))$ of \eqref{pairing-odd-2} (with $n=1$) identifies to the Toeplitz operator $T_f \co \H^p(\R) \to \H^p(\R)$, $g \mapsto P(fg)$ with symbol $f$, where we identify the map $P$ with its corestriction $P|^{\H^p(\R)}$. So the operator $T_f$ is Fredholm. Using Theorem \ref{th-index-LpR}, we deduce that
\begin{align*}
\MoveEqLeft
\big\la [f], (\L^p(\R),\pi,\i\cal{H}) \big\ra_{\K_1(\C_0(\R)),\K^1(\C_0(\R),\scr{L}^p)}         
\ov{\eqref{pairing-odd-2}}{=} \Index P M_f P
=\Index T_f 
=-\wind f.
\end{align*}
\end{proof}

\begin{remark} \normalfont 
Our approach is flexible. We can state variations with a possible weight \cite{HMW73}, matricial versions and $\UMD$-vector-valued variants. 
\end{remark}

\begin{remark} \normalfont
It should be possible to prove the case $p=2$ of Theorem \ref{th-index-LpR} using Connes character formula for locally compact spectral triples. It suffices to find the result in the literature on the noncommutative measure (using the Dixmier trace) and to check the assumptions of \cite[Theorem 1.2.5]{SuZ23}.
\end{remark}


The following result is an analogue of Proposition \ref{prop-|D|-T}.

\begin{prop}
\label{prop-|D|-commutator-R}
Suppose that $1 < p < \infty$. Consider the unbounded operator $D \ov{\mathrm{def}}{=} \frac{1}{\i}\frac{\d}{\d x}$ on $\L^p(\R)$. We have $|D|=\cal{H}\,\frac{\d}{\d x}$. Moreover, for any $f \in \C^\infty_c(\R)$, the commutator $[|D|,f]$ induces a bounded operator on $\L^p(\R)$ and
\begin{equation}
\label{eq:R-kernel-absD}
\big([|D|,f]g\big)(x)
=\frac{1}{\pi} \pv \int_\R \frac{f(x)-f(y)}{(x-y)^2}\,g(y)\,\mathrm dy,
\qquad g\in\cal{S}(\R).
\end{equation}
Finally, we have
\begin{equation}
\label{eq:absD-commutator-bound-R}
\big\|[|D|,f]\big\|_{\L^p(\R)\to\L^p(\R)}
\lesssim_p \|f'\|_{\L^\infty(\R)}.
\end{equation}
\end{prop}

\begin{proof}
\textit{First proof.} The operator $|D|$ is the unbounded Fourier multiplier defined by the symbol $|\xi|$. The first assertion is a consequence of \eqref{Hilbert-trans-Fourier}. 
Fix $f \in \C^\infty_c(\R)$ and $g\in\cal{S}(\R)$. Using $|D|=\cal{H}\,\frac{\mathrm d}{\mathrm dx}$,
\[
[|D|,M_f]g
=\cal{H}\,(fg)'-f\,\cal{H}\,g'
=\cal{H}(f'g)+\big(\cal{H}(fg')-f\,\cal{H}g'\big)
=\cal{H}(f'g)+[\cal{H},M_f]g'.
\]
Write both terms using the principal value representation of $\cal{H}$ and integrate by parts in the second one. More precisely, we have
\begin{align*}
\MoveEqLeft
\big([\cal{H},M_f]g'\big)(x)
\ov{\eqref{Hilbert-transform-Riesz}}{=} \frac{1}{\pi} \pv\int_\R \frac{f(y)g'(y)}{x-y} \d y
-\frac{f(x)}{\pi} \pv\int_\R \frac{g'(y)}{x-y} \d y
\\
&=\frac{1}{\pi} \pv\int_\R \frac{f(y)-f(x)}{x-y}g'(y) \d y 
=-\frac{1}{\pi} \pv \int_\R\Big(\frac{f'(y)}{x-y}+\frac{f(y)-f(x)}{(x-y)^2}\Big)g(y) \d y, 
\end{align*}
where we used that $\frac{\d}{\d y} \big(\frac{1}{x-y}\big)=\frac{1}{(x-y)^2}$. Adding $\cal{H}(f'g) \ov{\eqref{Hilbert-transform-Riesz}}{=} \frac{1}{\pi} \pv\int_\R \frac{f'(y)g(y)}{x-y} \d y$ cancels the $\frac{f'(y)}{x-y}$ part and yields the kernel formula \eqref{eq:R-kernel-absD}. Set
\begin{equation}
\label{def-inter-AZERT}
K_f(x,y)
\ov{\mathrm{def}}{=}\frac{1}{\pi} \frac{f(x)-f(y)}{(x-y)^2}, \quad x \ne y.
\end{equation}
By the mean value theorem, we see that
\[
|K_f(x,y)|
\ov{\eqref{def-inter-AZERT}}{=} \frac{1}{\pi} \frac{|f(x)-f(y)|}{(x-y)^2}
\leq \frac{1}{\pi} \frac{\norm{f'}_{\L^\infty(\R)}|x-y|}{|x-y|^2}
\lesssim \frac{\norm{f'}_{\L^\infty(\R)}}{|x-y|}.
\]
Differentiating with respect to $x$ and using $\partial_x (x-y)^{-2}=-2(x-y)^{-3}$, we get
\[
\partial_x K_f(x,y)
=\frac{1}{\pi}\left(\frac{f'(x)}{(x-y)^2}
-2\,\frac{f(x)-f(y)}{(x-y)^3}\right).
\]
Similarly, since $\partial_y (x-y)^{-2}=2(x-y)^{-3}$, we have
\[
\partial_y K_f(x,y)
=\frac{1}{\pi}\left(-\frac{f'(y)}{(x-y)^2}
+2\,\frac{f(x)-f(y)}{(x-y)^3}\right).
\]
By the mean value theorem, we see that
\[
\left|2\,\frac{f(x)-f(y)}{(x-y)^3}\right|
\leq 2\frac{\norm{f'}_{\L^\infty(\R)}|x-y|}{|x-y|^3}
=\frac{2\norm{f'}_{\L^\infty(\R)}}{|x-y|^2}.
\]
It follows that
\[
|\partial_x K_f(x,y)|
\le \frac{1}{\pi}\left(\frac{\norm{f'}_{\L^\infty(\R)}}{|x-y|^2}
+\frac{2\norm{f'}_{\L^\infty(\R)}}{|x-y|^2}\right)
\lesssim \frac{\norm{f'}_{\L^\infty(\R)}}{|x-y|^2},
\]
and similarly
\[
|\partial_y K_f(x,y)|
\lesssim \frac{\norm{f'}_{\L^\infty(\R)}}{|x-y|^2}.
\]
Adding the two bounds yields
\[
|\partial_x K_f(x,y)|+|\partial_y K_f(x,y)|
\lesssim \frac{\norm{f'}_{\L^\infty(\R)}}{|x-y|^2}.
\]
Hence $K_f$ is a Calder\'on–Zygmund kernel with constants $\lesssim \|f'\|_{\L^\infty}$, so the principal value operator \eqref{eq:R-kernel-absD} extends boundedly to $\L^p(\R)$ with the estimate \eqref{eq:absD-commutator-bound-R} by the classical Calder\'on–Zygmund theory. 

\textit{Second proof of the first assertion.} Observe that $|D|$ is a pseudo-differential operator of order $1$. Moreover, recall that if $f \in \C_c^\infty(\R)$ then $M_f$ is a differential operator of order 0, hence a pseudo-differential operator of order 0 by \cite[Remark 23.26.5]{Dieu88}. Consequently, according to \cite[Theorem H.9 p.~241]{ADV25} the commutator $[|D|,M_f]$ is a pseudo-differential operator of order $0$. Such operator induces a bounded operator on the Banach space $\L^p(\mathbb{R})$ by \cite[Theorem 2.6.22 p.~296]{RuT10} or \cite[Theorem H.10 p.~242]{ADV25}.
\end{proof}

\subsection{The Dirac operator on the space $\L^p(\mathbb{C}) \oplus \L^p(\mathbb{C})$ and the complex Riesz transform}
%
%
%
%
%

In dimension $d=2$, we have $N=2^{\lfloor \frac{d}{2}\rfloor}=2$ and we can consider the Pauli matrices 
$
\sigma_1
\ov{\mathrm{def}}{=}
\begin{bmatrix}
0 & 1\\
1 & 0
\end{bmatrix}
$ and $
\sigma_2
\ov{\mathrm{def}}{=}
\begin{bmatrix}
0 & -\i\\
\i & 0
\end{bmatrix}$. Then we can write the Euclidean Dirac operator on $\R^2$ as 
\begin{equation}
\label{Dirac-operator-complex}
D
\ov{\eqref{Dirac-operator-Rd}}{=} \frac{1}{\i} \big( \sigma_1 \partial_x + \sigma_2 \partial_y \big)
=\frac{1}{\i}
\begin{bmatrix}
0 & \partial_x - \i \partial_y\\
\partial_x + \i \partial_y & 0
\end{bmatrix}
=\frac{2}{\i}
\begin{bmatrix}
0 &  \partial_{z}\\
\partial_{\ovl{z}} & 0
\end{bmatrix},
\end{equation}
where we use the classical operators
$$
\partial_{\ovl{z}}
\ov{\mathrm{def}}{=}\frac{1}{2}\big(\partial_x+\i \partial_y\big) 
\quad \text{and} \quad
\partial_{z}
\ov{\mathrm{def}}{=}\frac{1}{2}(\partial_x-\i \partial_y).
$$
See \cite[p.~119]{LaM89} for more information on this operator. In the sequel, we consider the domain $\dom D=\W^{1,p}(\mathbb{C})\oplus \W^{1,p}(\mathbb{C})$ of the Banach space $\L^p(\mathbb{C}) \oplus \L^p(\mathbb{C})$.  The square of this operator is given by 
\begin{equation}
\label{square-Dirac-complex}
D^2
\ov{\eqref{square-of-Dirac}}{=}
\begin{bmatrix}
  -\Delta   &  0 \\
  0   &  -\Delta \\
\end{bmatrix}.
\end{equation}
Following \eqref{isometric-homomorphism}, we can introduce the isometric homomorphism $\pi \co \L^\infty(\mathbb{C}) \to \B(\L^p(\mathbb{C}) \oplus \L^p(\mathbb{C}))$, $f \mapsto \begin{bmatrix}
   M_f  & 0  \\
   0  &  M_f \\
\end{bmatrix}$.

\begin{remark} \normalfont
\label{derivation-en Fourier-multipliers}
Recall that $\Delta=4\partial_{\ovl{z}}\partial_{z}$ 
The operators $\partial_{\ovl{z}}$, $\partial_{z}$ and $\Delta$ can be seen as unbounded Fourier multipliers on $\L^p(\mathbb{C})$ with symbol $-\i\pi \zeta$, $-\i\pi \ovl{\zeta}$ and $|\zeta|^2$, see \cite[p.~99]{AIM09}.
\end{remark}

Proposition \ref{prop:Lip-alg-Rd-Dirac} gives the following result.

\begin{prop}
\label{prop:Lip-alg-C}
Suppose that $1<p<\infty$. Then
\begin{equation}
\label{eq:Lip-equals-W1infty-C}
\Lip_{D}(\L^\infty(\mathbb{C}))
=\W^{1,\infty}(\mathbb{C}).
\end{equation}
Moreover, for any function $f \in \W^{1,\infty}(\mathbb{C})$ we have
\begin{equation}
\label{eq:commutator-C-Linfty}
[D,f]
=\frac{2}{\i}\begin{bmatrix}
0 & - M_{\partial_{\ovl z} f}\\
 M_{\partial_z f} & 0
\end{bmatrix}
\end{equation}
and
\begin{equation}
\label{eq:norm-comm-C-Linfty}
\bnorm{[D,\pi(f)]}_{\L^p(\mathbb{C})\oplus \L^p(\mathbb{C}) \to \L^p(\mathbb{C})\oplus \L^p(\mathbb{C})} 
= 2 \max\Big\{\norm{\partial_z f}_{\L^\infty(\mathbb{C})},\norm{\partial_{\ovl z} f}_{\L^\infty(\mathbb{C})}\Big\}.
\end{equation}
\end{prop}

\begin{proof}
We have
$$
[D,f]
\ov{\eqref{eq:commutator-Rd-Dirac}}{=} \frac{1}{\i}\big(\gamma_1 \ot M_{\partial_x f}+\gamma_2 \ot M_{\partial_y f}\big)
=\frac{1}{\i}
\begin{bmatrix}
0 & \partial_x - \i \partial_y\\
\partial_x + \i \partial_y & 0
\end{bmatrix}
=\frac{2}{\i}\begin{bmatrix}
0 & - M_{\partial_{\ovl z} f}\\
 M_{\partial_z f} & 0
\end{bmatrix}.
$$
The inequality $\bnorm{[D,f]}
\leq 2 \max\Big\{\norm{\partial_z f}_{\L^\infty(\mathbb{C})},\norm{\partial_{\ovl z} f}_{\L^\infty(\mathbb{C})}\Big\}$ follows from a simple computation. Indeed, for any $u,v \in \L^p(\mathbb{C})$ we have
\begin{align*}
\MoveEqLeft
\norm{\big(-2(\partial_{\ovl z} f) v,2(\partial_z f)u\big)}_{\L^p(\mathbb{C})\oplus \L^p(\mathbb{C})}
=\big(\norm{2(\partial_{\ovl z} f) v}_{\L^p(\mathbb{C})}^p+\norm{2(\partial_z f)u}_{\L^p(\mathbb{C})}^p\big)^{\frac{1}{p}} \\
&\leq 2 \max\Big\{\norm{\partial_z f}_{\L^\infty(\mathbb{C})},\norm{\partial_{\ovl z} f}_{\L^\infty(\mathbb{C})}\Big\} \big(\norm{v}_{\L^p(\mathbb{C})}^p+\norm{u}_{\L^p(\mathbb{C})}^p\big)^{\frac{1}{p}} \\
&= 2 \max\Big\{\norm{\partial_z f}_{\L^\infty(\mathbb{C})},\norm{\partial_{\ovl z} f}_{\L^\infty(\mathbb{C})}\Big\} \norm{(u,v)}_{\L^p(\mathbb{C})\oplus \L^p(\mathbb{C})}.        
\end{align*}
By considering $(u,0)$ or $(0,u)$, we obtain $\bnorm{[D,f]} \geq 2 \max\Big\{\norm{\partial_z f}_{\L^\infty(\mathbb{C})},\norm{\partial_{\ovl z} f}_{\L^\infty(\mathbb{C})}\Big\}$.
\end{proof}

From Theorem \ref{thm-spectral-triple-Rd}, we obtain the following result.

\begin{thm}
\label{thm-spectral-triple-Dirac-C}
Suppose that $1 < p < \infty$. Then $\bigg(\C_0(\mathbb{C}),\L^p(\mathbb{C})\oplus \L^p(\mathbb{C}),\begin{bmatrix}
   0  &  -2\partial_{\ovl{z}} \\
  2\partial_{z}   &  0 \\
\end{bmatrix}\bigg)$ is a Banach locally compact spectral triple, which is $2^+$-summable, in particular $q$-summable for any $q > 2$.
\end{thm}

Note that we have
\begin{equation}
\label{sign-D-Complex}
\sgn D
\ov{\eqref{Dirac-operator-complex} \eqref{square-Dirac-complex}}{=} D|D|^{-1}
=\begin{bmatrix}
   0  &  -2 \partial_{\ovl{z}} \\
 2\partial_{z}   &  0 \\
\end{bmatrix}\begin{bmatrix}
  -\Delta   & 0  \\
   0  &  -\Delta \\
\end{bmatrix}^{-\frac{1}{2}}
=2\begin{bmatrix}
  0   &  \partial_{\ovl{z}}(-\Delta)^{-\frac{1}{2}} \\
  \partial_{z}(-\Delta)^{-\frac{1}{2}}  &  0 \\
\end{bmatrix}.
\end{equation}
Note that the operator $\partial_{z}(-\Delta)^{-\frac{1}{2}}$ is the complex Riesz transform $R_{\mathbb{C}} \co \L^p(\mathbb{C}) \to \L^p(\mathbb{C})$ (or complex Hilbert transform) considered in \cite[pp.~102-103]{AIM09}, \cite{IwG96} and \cite{CDK23}. It is a Fourier multiplier with symbol $\frac{\ovl{\zeta}}{|\zeta|}$. By \cite[Theorem 4.2.1 p.~103]{AIM09}, this operator admits the following integral representation
\begin{equation}
\label{complex-Riesz-transform}
(R_{\mathbb{C}}f)(z)
= \frac{1}{2\pi\i} \pv \int_{\mathbb{C}} \frac{\ovl{z}-\ovl{w}}{|z-w|^3} f(w) \d w
=\lim_{\epsi \to 0^{+}}\frac{1}{2 \pi \i}
\int_{w \in \mathbb{C} \setminus B(z,\epsi)}\frac{\ovl{z}-\ovl{w}}{|z-w|^{3}}
f(w) \d w
, \quad z \in \mathbb{C}.
\end{equation}
It is known \cite[Exercise 1.16 p.~7]{Dra21} that 
\begin{equation}
\label{Riesz-complex-sum}
R_{\mathbb{C}}
=R_2+\i R_1,
\end{equation}
where $R_1,R_2 \co \L^p(\R^2) \to \L^p(\R^2)$ are the classical Riesz transforms on $\R^2$ described in \cite[Section 5.1.4]{Gra14a}, which are Fourier multipliers by $\frac{-\i \xi_1}{|\xi|}$ and $\frac{-\i \xi_2}{|\xi|}$,  according to \cite[Proposition 5.1.14]{Gra14a}.  It is worth noting that by \cite[p.~102]{AIM09}, the (planar) Beurling-Ahlfors operator $\cal{S} \co \L^p(\mathbb{C}) \to \L^p(\mathbb{C})$ defined by
\begin{equation}
\label{}
(\cal{S}f)(z)
\ov{\mathrm{def}}{=} -\frac{1}{\pi} \pv\int_{\mathbb{C} }\frac{f(w)}{(w-z)^2} \d w
=-\lim_{\epsi \to 0^{+}}\frac{1}{\pi}
\int\limits_{w \in \mathbb{C} \setminus B(z,\epsi)}
\frac{f(w)}{(w-z)^2} \d w, \quad z \in \mathbb{C},
\end{equation}
is equal by \cite[p.~102]{AIM09} to the square $R^2$ of the complex Riesz transform $R$. Moreover, according to \cite[Corollary 4.1.1 p.~102]{AIM09}, it is a Fourier multiplier with symbol $\big(\frac{\ovl{\zeta}}{\zeta}\big)^2=\frac{\ovl{\zeta}}{|\zeta|}$ and by \cite[p.~7]{Dra21} it is an isometry on the Hilbert space $\L^2(\mathbb{C})$ and invertible on the Banach space $\L^p(\mathbb{C})$ by \cite[Exercise 1.16 p.~8]{Dra21}. By \cite[Lemma 4.6.12 p.~142]{AIM09}, this operator has the fundamental property that it turns $\bar z$-derivatives into $z$-derivatives, i.e., we have 
\begin{equation}
\label{Ahlfors-intertwine}
\cal{S}(\partial_{\ovl{z}} f) 
=\partial_z f.
\end{equation}
for any Schwartz function $f$. We refer to \cite{AIM09} and \cite{Dra21} for more information on this operator. The following result could be deduced from Theorem \ref{thm-spectral-triple-Dirac-C} and a locally compact analogue of Proposition \ref{prop-triple-to-Fredholm}, a statement that we have not yet developed in detail.

\begin{prop}
The quadruple $\bigg(\L^p(\mathbb{C}) \oplus \L^p(\mathbb{C}),\pi, \sgn D,\begin{bmatrix}
  -\Id   & 0  \\
    0 &  \Id \\
\end{bmatrix}\bigg)$ is an even Banach Fredholm module over the algebra $\C_0(\mathbb{C})$. 
\end{prop}

\begin{proof}
For any function $f \in \C_0(\mathbb{C})$, note that 
$
\left[\sgn D,f\right]
\ov{\eqref{sign-D-Complex}}{=}         
2\begin{bmatrix}
  0   &  \big[\partial_{\ovl{z}}(-\Delta)^{-\frac{1}{2}},M_f\big] \\
  \left[R_{\mathbb{C}},M_f\right]   &  0 \\
\end{bmatrix}$. Moreover, observe that for any function $f \in \C_0(\mathbb{C})$ we have $[R_{\mathbb{C}},M_f] \ov{\eqref{Riesz-complex-sum}}{=} [R_2,M_f]+\i [R_1,M_f]$. Recall that by \cite[p.~128]{Hao20} $R_1$ and $R_2$ are Calder\'on-Zygmund singular integral operators. Consequently, from \cite[Theorem 2 p.~92]{ClC13}, we deduce that each commutator $[R_j,M_f]$ is compact if $f \in \VMO(\R^n)$. We conclude that each commutator $\left[R_{\mathbb{C}},f\right]$ is compact, as a sum of compact operators. Note that by Remark \ref{derivation-en Fourier-multipliers}, the operator $\partial_{\ovl z}(-\Delta)^{-\frac12}$ is a Fourier multiplier with symbol $
\frac{\i \xi_1 - \xi_2}{|\xi|}$, up to a constant. We deduce that $\partial_{\ovl z}(-\Delta)^{-\frac12}= R_1 + \i R_2$. Consequently, we can prove similarly that the commutator $\big[\partial_{\ovl{z}}(-\Delta)^{-\frac{1}{2}},M_f\big]$ is compact on $\L^p(\mathbb{C})$ for any function $f \in \C_0(\mathbb{C})$. The remaining assertions are straightforward to verify.
\end{proof}

Observe that the function $p_{\Bott} \co \mathbb{C} \to \M_2(\mathbb{C})$, $z \mapsto \dsp \frac{1}{|z|^2 + 1}
\begin{bmatrix}
|z|^2 & z \\
\ovl{z} & 1
\end{bmatrix}$ satisfies $p_{\Bott}^2=p_{\Bott}$ and
$\lim_{z \to \infty} p_{\Bott}(z) = 
\begin{bmatrix}
1 & 0 \\
0 & 0
\end{bmatrix}$. That entails that $p_{\Bott}-\begin{bmatrix}
1 & 0 \\
0 & 0
\end{bmatrix}$ belongs to the algebra $\C_0(\mathbb{C},\M_2)=\M_2(\C_0(\mathbb{C}))$. It is well known that the class
\[
\beta_{\mathbb{C}} 
\ov{\mathrm{def}}{=} [p_{\Bott}] - [1]  
\]
of $\K_0(\C_0(\mathbb{C}))$ generates the infinite cyclic group $\K_0(\C_0(\mathbb{C}))  \cong \Z$, as discussed in \cite[p.~303]{Eme24}, \cite[p.~104, p.~145, p.~156, p.~177]{WeO93} and \cite[p.~122]{GVF01}.
This class is called the Bott generator for $\mathbb{C}$. The pairing is described in \cite[pp.~86-86]{CGRS14} in the case $p=2$. A similar property holds for the case $p \not=2$.

Then in this context, we also have the following result.

\begin{prop}
The quadruple $\bigg(\L^p(\mathbb{C}) \oplus \L^p(\mathbb{C}),\pi,\begin{bmatrix}
   0  &  \cal{S}^{-1} \\
   \cal{S}  &  0 \\
\end{bmatrix},\begin{bmatrix}
  -\Id   & 0  \\
    0 &  \Id \\
\end{bmatrix}\bigg)$ is an even Banach Fredholm module over the algebra $\C_0(\mathbb{C})$. 
\end{prop}

\begin{proof}
Indeed, we have $
\begin{bmatrix}
   0  &  \cal{S}^{-1} \\
   \cal{S}  &  0 \\
\end{bmatrix}^2
=\begin{bmatrix}
  \cal{S}^{-1} \cal{S}  & 0  \\
   0  &  \cal{S} \cal{S}^{-1} \\
\end{bmatrix}^2
=\begin{bmatrix}
   \Id_{\L^p(\mathbb{C})}  &  0 \\
   0  &  \Id_{\L^p(\mathbb{C})} \\
\end{bmatrix}$.  
Moreover, for any $f \in \C_0(\mathbb{C})$ we have
$$
\bigg[\begin{bmatrix}
   0  &  \cal{S}^{-1} \\
   \cal{S}  &  0 \\
\end{bmatrix},f\bigg]
=\begin{bmatrix}
   0  &  \cal{S}^{-1} \\
   \cal{S}  &  0 \\
\end{bmatrix}\begin{bmatrix}
   M_f  & 0  \\
   0  &  M_f \\
\end{bmatrix}-
\begin{bmatrix}
   M_f  & 0  \\
   0  &  M_f \\
\end{bmatrix}\begin{bmatrix}
   0  &  \cal{S}^{-1} \\
   \cal{S}  &  0 \\
\end{bmatrix}
=\begin{bmatrix}
   0  &  [\cal{S}^{-1},M_f] \\
   [\cal{S},M_f]  &  0 \\
\end{bmatrix}. 
$$
According to \cite[Theorem 4.6.14 p.~145]{AIM09}, for any function $f$ in $\VMO(\mathbb{C})$, the commutator $[\cal{S},M_f]$ is compact on $\L^p(\mathbb{C})$. Hence the commutator $[\cal{S}^{-1}, M_f]= -\cal{S}^{-1} [\cal{S}, M_f] \cal{S}^{-1}$ is also compact, as the product of bounded and compact operators. It follows that $\bigg[\begin{bmatrix}
   0  &  \cal{S}^{-1} \\
   \cal{S}  &  0 \\
\end{bmatrix},f\bigg]$ 
is also compact for any function $f \in \VMO(\mathbb{C})$. Finally, we have $
F\gamma 
=\begin{bmatrix}
   0  &  \cal{S}^{-1} \\
   \cal{S}  &  0 \\
\end{bmatrix}
\begin{bmatrix}
  -\Id   & 0  \\
    0 &  \Id \\
\end{bmatrix}
=\begin{bmatrix}
   0  &  \cal{S}^{-1} \\
   -\cal{S}  & 0  \\
\end{bmatrix}
=-\begin{bmatrix}
  -\Id   & 0  \\
    0 &  \Id \\
\end{bmatrix}\begin{bmatrix}
   0  &  \cal{S}^{-1} \\
   \cal{S}  &  0 \\
\end{bmatrix}
=-\gamma F$ and the equality $ 
[\pi(f),\gamma]=\pi(f)\gamma-\gamma\pi(f) =\begin{bmatrix}
   M_f  & 0  \\
   0  &  M_f \\
\end{bmatrix}\begin{bmatrix}
  -\Id   & 0  \\
    0 &  \Id \\
\end{bmatrix}-\begin{bmatrix}
  -\Id   & 0  \\
    0 &  \Id \\
\end{bmatrix}\begin{bmatrix}
   M_f  & 0  \\
   0  &  M_f \\
\end{bmatrix} =0$.
\end{proof}

\begin{remark} \normalfont
By the way, it is worth noting that by \cite[Theorem 4.6.13 p.~143]{AIM09} the commutators $[\cal{S},M_f]$ are bounded on the Banach space $\L^p(\mathbb{C})$ for any function $f$ belonging to the space $\BMO(\mathbb{C})$.
\end{remark}
We will also use the notations $\partial \ov{\mathrm{def}}{=}2\partial_{z}$ and $\partial^*\ov{\mathrm{def}}{=}-2\partial_{\ovl{z}}$. 

\begin{prop}
\label{Prop-R-gradient-bounds-Fourier} 
Suppose that $1 < p < \infty$. The family 
\begin{equation}
\label{R-gradient-bounds-C}
\Big\{t\partial(\Id-t^2\Delta)^{-1}: t>0 \Big\}
\end{equation} 
of operators of $\B(\L^p(\mathbb{C}))$ is $R$-bounded.
\end{prop}

\begin{proof}
Note that the Riesz transform $\partial (-\Delta)^{-\frac{1}{2}} \co \L^p(\mathbb{C}) \to \L^p(\mathbb{C})$ is a well-defined bounded operator by \cite[Corollary 4.5.1 p.~127]{AIM09}. Suppose that $t > 0$. A standard functional calculus argument gives
\begin{align}
\label{Divers-987}
\MoveEqLeft
t\partial(\Id-t^2\Delta)^{-1}            
=\partial (-\Delta)^{-\frac{1}{2}}\Big((-t^2\Delta)^{\frac{1}{2}}(\Id-t^2\Delta)^{-1} \Big).
\end{align} 
By Example \ref{Laplacian-funct}, the Laplacian $-\Delta$ admits a bounded $\H^\infty(\Sigma_\theta)$ functional calculus for any angle $\theta >0$. Moreover, the Banach space $\L^p(\mathbb{C})$ is $\UMD$ by \cite[Proposition 4.2.15 p.~291]{HvNVW16}, hence has the triangular contraction property $(\Delta)$ by \cite[Theorem 7.5.9 p.~137]{HvNVW18}. We deduce by \cite[Theorem 10.3.4 (2) p.~401]{HvNVW18} that the operator $-\Delta$ is $R$-sectorial. By \cite[Example 10.3.5 p.~402]{HvNVW18} applied with $\alpha=\frac{1}{2}$ and $\beta=1$, we infer that the set
$$
\big\{(-t^2\Delta)^{\frac{1}{2}}(\Id-t^2\Delta)^{-1}: t>0\big\}
$$
of operators of $\B(\L^p(\mathbb{C}))$ is $R$-bounded. Recalling that a singleton is $R$-bounded by \cite[Example 8.1.7 p.~170]{HvNVW18}, we obtain by composition \cite[Proposition 8.1.19 (3) p.~178]{HvNVW18} that the set
$$
\Big\{ \partial (-\Delta)^{-\frac{1}{2}}\Big((-t^2\Delta)^{\frac{1}{2}}(\Id-t^2\Delta)^{-1} \Big) : t>0 \Big\}
$$
of operators of $\B(\L^p(\mathbb{C}))$ is $R$-bounded. Hence with \eqref{Divers-987} we conclude that the subset \eqref{R-gradient-bounds-C} is $R$-bounded.
\end{proof}

\begin{thm}
Suppose that $1 < p < \infty$. The unbounded operator $D$ admits a bounded $\H^\infty(\Sigma_\theta^\bi)$ functional calculus on the Banach space $\L^p(\mathbb{C}) \oplus \L^p(\mathbb{C})$ for some angle $\theta \in (0,\frac{\pi}{2})$.
\end{thm}

\begin{proof}
We will start by showing that the set $\{\i t : t \in \R, t \not=0\}$ is contained in the resolvent set of $D$. We will do this by showing that $\Id-\i tD$ has a two-sided bounded inverse $(\Id-\i tD)^{-1}$ given by
\begin{equation}
\label{Resolvent-Fourier}
\begin{bmatrix} 
(\Id_{}-t^2\Delta)^{-1} & \i t(\Id_{}-t^2\Delta)^{-1}\partial^* \\ 
\i t\partial(\Id_{}-t^2\Delta)^{-1} &  (\Id-t^2\Delta)^{-1} 
\end{bmatrix}
\end{equation}
acting on $\L^p(\mathbb{C}) \oplus_p \L^p(\mathbb{C})$. By Proposition \ref{Prop-R-gradient-bounds-Fourier} and since the operator $-\Delta$ satisfy the property \eqref{Def-R-bisectorial} of $R$-sectoriality, the four entries are bounded. It only remains to check that this matrix defines a two-sided inverse of $\Id-\i t D$. We have the following equalities of operators acting on $\dom D$.
\begin{align*}
\MoveEqLeft
\begin{bmatrix} 
(\Id_{}-t^2\Delta)^{-1} & \i t(\Id_{\L^p}-t^2\Delta)^{-1}\partial^* \\ 
\i t(\Id_{}-t^2\Delta)^{-1}\partial &  (\Id-t^2\Delta)^{-1} 
\end{bmatrix}(\Id-\i tD)   \\       
&=
\begin{bmatrix} 
(\Id-t^2\Delta)^{-1} & \i t(\Id_{}-t^2\Delta)^{-1}\partial^* \\ 
\i t(\Id_{}-t^2\Delta)^{-1}\partial &  (\Id-t^2\Delta)^{-1} 
\end{bmatrix}
\begin{bmatrix} 
\Id_{} & -\i t \partial^* \\ 
-\i t\partial &\Id_{}
\end{bmatrix}\\
&=\left[\begin{matrix} 
(\Id-t^2\Delta)^{-1}+t^2(\Id-t^2\Delta)^{-1}\partial^*\partial&-\i t(\Id-t^2\Delta)^{-1}\partial^*+\i t(\Id-t^2\Delta)^{-1} \partial^*\\ 
\i t(\Id-t^2\Delta)^{-1}\partial-\i t (\Id-t^2\Delta)^{-1}\partial& 
t^2(\Id-t^2\Delta)^{-1}\partial\partial^*+(\Id-t^2\Delta)^{-1}
\end{matrix}\right]\\
&=
\left[\begin{matrix} 
(\Id-t^2\Delta)^{-1}+t^2(\Id-t^2\Delta)^{-1}\Delta  & 0\\ 
\i t(\Id-t^2\Delta)^{-1}\partial-\i t\big(\Id-t^2\Delta\big)^{-1}\partial &
(\Id-t^2\Delta)^{-1}(t^2\partial\partial^*+\Id_{})
\end{matrix}\right] 
= 
\begin{bmatrix} 
\Id_{} & 0 \\ 
0 & \Id_{}
\end{bmatrix}
\end{align*} 
and similarly
\begin{align*}
\MoveEqLeft
(\Id-\i tD) 
\begin{bmatrix} 
(\Id_{}-t^2\Delta)^{-1} & \i t(\Id_{}-t^2\Delta)^{-1}\partial^* \\ 
\i t(\Id_{}-t^2\Delta)^{-1}\partial &  (\Id_{}-t^2\Delta)^{-1} 
\end{bmatrix}  \\          
&=\begin{bmatrix} 
\Id_{} & -\i t\partial^* \\ 
-\i t\partial &\Id_{}
\end{bmatrix}
\begin{bmatrix} 
(\Id_{}-t^2\Delta)^{-1} & \i t(\Id_{}-t^2\Delta)^{-1}\partial^* \\ 
\i t(\Id_{}-t^2\Delta)^{-1}\partial &  (\Id_{}-t^2\Delta)^{-1} 
\end{bmatrix} 
\\
&=\left[\begin{matrix} 
(\Id_{}-t^2\Delta)^{-1}+ t^2\partial^*(\Id_{}-t^2\Delta)^{-1}\partial &  \i t(\Id_{}-t^2\Delta)^{-1} \partial^*-\i t\partial^*\big( (\Id_{}-t^2\Delta)^{-1}\big)\\ 
-\i t\partial(\Id_{}-t^2\Delta)^{-1}+\i t\partial(\Id_{}-t^2\Delta)^{-1} &
t^2\partial(\Id_{}-t^2\Delta)^{-1}\partial^*+ (\Id_{}-t^2\Delta)^{-1}
\end{matrix}\right]\\
&= \begin{bmatrix} 
\Id_{} & 0 \\ 
0 & \Id_{}
\end{bmatrix}.
\end{align*} 
It remains to show the last condition of \eqref{Def-R-bisectorial}, i.e., that the set $\{\i t(\i t-D)^{-1} : t \not=0\}=\{(\Id -\i t D)^{-1} : t \not=0\}$ is $R$-bounded. For this, observe that the diagonal entries of \eqref{Resolvent-Fourier} are $R$-bounded by the $R$-sectoriality of $\Delta$. The $R$-boundedness of the other entries follows from the $R$-gradient bounds of Proposition \ref{Prop-R-gradient-bounds-Fourier}. Since a set of operator matrices is $R$-bounded precisely when each entry is $R$-bounded, we conclude that \eqref{Def-R-bisectorial} is satisfied, i.e.~, the operator $D$ is $R$-bisectorial.
\end{proof}
\subsection{The Dirac operator on the space $\L^p(\T^2_\theta) \oplus \L^p(\T^2_\theta)$}
\label{sec-Dirac-noncommutative-tori}

\paragraph{Quantum tori} We will use standard notations and we refer to the papers \cite{CXY13}, \cite{FXZ23}, \cite{EcI18}, \cite{MSX19} and \cite{XXY18} for more information. Let $d \geq 2$. To each $d \times d$ real skew-symmetric matrix $\theta$, one may associate a 2-cocycle $\sigma_\theta \co \Z^d \times \Z^d \to \T$ of the group $\Z^d$ defined by $\sigma_\theta (m,n) \ov{\mathrm{def}}{=} \e^{\frac{\i}{2} \langle m, \theta n\rangle}$ where $m,n \in \Z^d$. We have $\sigma(m,-m) = \sigma(-m,m)$ for any $m \in \Z^d$. 

We define the $d$-dimensional noncommutative torus $\L^\infty(\T_{\theta}^d)$ as the twisted group von Neumann algebra $\VN(\Z^d,\sigma_\theta)$. One can provide a concrete realization in the following manner. 
If $(\epsi_n)_{n \in \Z^d}$ is the canonical basis of the Hilbert space $\ell^2_{\Z^d}$ and if $m \in \Z^d$, we can consider the bounded operator $U_m \co \ell^2_{\Z^d} \to \ell^2_{\Z^d}$ defined by 
\begin{equation}
\label{def-lambdas}
U_m(\epsi_n)
\ov{\mathrm{def}}{=} \sigma_\theta(m,n) \epsi_{m+n}, \quad n \in \Z.	
\end{equation}
The $d$-dimensional noncommutative torus $\L^\infty(\T_{\theta}^d)$ is the von Neumann subalgebra of $\B(\ell^2_{\Z^d})$ generated by the $*$-algebra $
\mathcal{P}_{\theta}
\ov{\mathrm{def}}{=} \mathrm{span} \big\{ U^m \ : \ m \in \Z^d \big\}$. Recall that for any $m,n \in \Z^d$ we have
\begin{equation}
\label{product-adjoint-twisted}
U_m U_n 
= \sigma_\theta(m,n) U_{m+n}
\quad \text{and} \quad 
\big(U_m \big)^* 
= \ovl{\sigma_\theta(m,-m)} U_{-m}.	
\end{equation}
The von Neumann algebra $\L^\infty(\T_{\theta}^d)$ is finite with normalized trace given by $\tau(x) \ov{\mathrm{def}}{=}\langle\epsi_{0},x(\epsi_{0})\rangle_{\ell^2_{\Z^d}}$ where $x \in \L^\infty(\T_{\theta}^d)$. In particular, we have $\tau(U_m) = \delta_{m=0}$ for any $m \in \Z^d$.

Let $\Delta$ be the unbounded operator acting on $\L^\infty(\T_{\theta}^d)$ defined on the weak* dense subspace $\mathcal{P}_{\theta}$ by $\Delta(U_m) \ov{\mathrm{def}}{=} 4\pi^2 |m|^2U_m$ where $|m| \ov{\mathrm{def}}{=} m_1^2+\cdots+m_d^2$. Then this operator is weak* closable and its weak* closure is the opposite of a weak* generator of a symmetric Markovian semigroup $(T_t)_{t \geq 0}$ of operators acting on $\L^\infty(\T_{\theta}^d)$, called the noncommutative heat semigroup on the noncommutative torus. 


For any $j = 1,\ldots,d$, we may define the partial differentiation operators $\partial_j$ by
\begin{equation*}
\partial_j (U^n) 
\ov{\mathrm{def}}{=} 2\pi \i n_j U^n,\quad n = (n_1,\ldots,n_d) \in \Z^d.
\end{equation*}
Every partial derivation $\partial_j$ can be viewed a densely defined closed unbounded operator acting on the Hilbert space $\L^2(\T^d_\theta)$.



\paragraph{The Dirac operator} The Dirac operator $D$ is defined in terms of $\gamma$ matrices in direct analogy to commutative tori. Define $N \ov{\mathrm{def}}{=} 2^{\lfloor \frac{d}{2}\rfloor}$ and select $N \times N$ complex selfadjoint matrices $\gamma_1,\ldots,\gamma_d$ satisfying $\gamma_j\gamma_k +\gamma_k\gamma_j = 2\delta_{j,k}1$. Following \cite[(B.6) p.~147]{EcI18} and \cite[Definition 12.14 p.~545]{GVF01}, we define the unbounded densely defined linear operator
\begin{equation}
D 
\ov{\mathrm{def}}{=} -\i\sum_{j=1}^d \gamma_j \ot \partial_j
\end{equation}
with domain $\ell^p_N(\W^{1,p}(\T_\theta^d))$ acting on the complex Hilbert space $\ell^p_N(\L^p(\T^d_\theta))$. The operator $D$ is selfadjoint if $p=2$. By \cite[p.~545]{GVF01}, we have
\begin{equation}
\label{square-of-Dirac-torus-NC}
D^2
=-\Id_{\ell^2_{N}} \ot \Delta.
\end{equation}

\begin{example} \normalfont
If $d=2$ then $N=2$ and by \cite[p.~545]{GVF01} we have
\begin{equation}
\label{Dirac-dim-2}
D
=-\i\begin{bmatrix}
   0  &  \partial_1-\i\partial_2 \\
  \partial_1+\i\partial_2  &  0 \\
\end{bmatrix}.
\end{equation}
\end{example}

For any $f \in \L^\infty(\T^d_\theta)$, we denote $L_f \co \L^p(\T^d_\theta) \to \L^p(\T^d_\theta)$, $g \mapsto fg$ the operator of left multiplication on the Banach space $\L^p(\T^d_\theta)$. Consequently, we can consider the homomorphism $\pi \co \L^\infty(\T^d_\theta) \to \B(\L^p(\T^d_\theta) \oplus \cdots \oplus\L^p(\T^2_\theta))$, $f \mapsto L_f \oplus \cdots \oplus L_f=\Id_{\ell^p_N} \ot L_f$. We define
\[
\W^{1,\infty}(\T_\theta^d)
\ov{\mathrm{def}}{=} \{a \in \cal{S}'(\T_\theta^d) : \partial_j a \in \L^\infty(\T_\theta^d)\text{ for all }j\}.
\]
We need the canonical action $\alpha$ of $\R^d$ on $\L^p(\T_\theta^d)$ satisfying in particular
\[
\alpha_t(U^k)
=\e^{2\pi \i t\cdot k} U^k,\qquad t \in \R^d, k \in \Z^d
\]
See e.g.~\cite[Section 2.2]{HLP19} and \cite[Remark 2.2]{XXY18}. 
Now, we describe the Lipschitz algebra on the noncommutative tori.

\begin{prop}
\label{prop:Lip-alg-NCT}
Let $d \geq 2$. Suppose that $1 < p < \infty$. We have 
%
\[
\Lip_D(\L^\infty(\T_\theta^d))
=\W^{1,\infty}(\T_\theta^d).
\]
Moreover, for $f \in \W^{1,\infty}(\T_\theta^d)$, we have
\begin{equation}
\label{eq:NC-comm-formula}
[D,f]
=-\i\sum_{j=1}^d \gamma_j \ot L_{\partial_j f}.
\end{equation}
Finally, there exist positive constants $c_1,c_2$ depending only on the chosen $\{\gamma_j\}$ such that
\begin{equation}
\label{eq:NC-comm-norm}
c_1 \max_{1 \leq j\leq d} \norm{\partial_j f}_{\L^\infty(\T_\theta^d)}
\leq
\norm{[D,f]}_{\ell^p_N(\L^{p}(\T_\theta^d)) \to \ell^p_N(\L^{p}(\T_\theta^d))}
\leq 
c_2\sum_{j=1}^d \norm{\partial_j f}_{\L^\infty(\T_\theta^d)}.
\end{equation}
\end{prop}

\begin{proof}
For any $f \in \W^{1,\infty}(\T_\theta^d)$ or any $f \in \Lip_D(\L^\infty(\T_\theta^d))$, we have 
\begin{align}
\MoveEqLeft
\label{123AZERY-bis}
[ D,f ]
=D\pi(f)-\pi(f)D
=\bigg(-\i\sum_{j=1}^d \gamma_j \ot \partial_j\bigg)(\Id_{\ell^p_N} \ot L_f)
-(\Id_{\ell^p_N} \ot L_f)\Big(-\i\sum_{j=1}^d \gamma_j \ot \partial_j\Big) \\
&=-\i\bigg(\sum_{j=1}^d \gamma_j \ot \partial_j L_f
-\sum_{j=1}^d \gamma_j \ot L_f\partial_j\bigg) 
=-\i\bigg(\sum_{j=1}^d \gamma_j \ot \big( \partial_j L_f -L_f\partial_j\big) \bigg). \nonumber
\end{align}
If $f \in \W^{1,\infty}(\T_\theta^d)$, we deduce that
\begin{equation}
\label{123AZERY}
[ D,f ]
=-\i\sum_{j=1}^d \gamma_j \ot L_{\partial_j f}.
\end{equation}
For any integer $1 \leq j \leq d$, we have $\partial_j f \in \L^\infty(\T_\theta^d)$. Hence each element $\partial_j f$ induces a bounded left multiplication operator $L_{\partial_j f}$ on the space $\L^p(\T_\theta^d)$. Consequently, the right-hand side of \eqref{123AZERY} extends to a bounded operator on the Banach space $\ell^p_N(\L^{p}(\T_\theta^d))$ and \eqref{eq:NC-comm-formula} holds by density. This gives $\W^{1,\infty}(\T_\theta^d) \subset \Lip_D(\L^\infty(\T_\theta^d))$. Moreover, we have
$$
\norm{[D,f]}
= \norm{-\i\sum_{j=1}^d \gamma_j \ot L_{\partial_j f}}
\leq \sum_{j=1}^d \norm{ \gamma_j} \norm{L_{\partial_j f}}_{\L^p(\T_\theta^d) \to \L^p(\T_\theta^d)}
= \sum_{j=1}^d \norm{ \gamma_j} \norm{\partial_j f}_{_{\L^\infty(\T_\theta^d)}}.
$$
Let $f \in \Lip_D(\L^\infty(\T_\theta^d))$. Since $\{\gamma_j\}$ are linearly independent in $\M_N$, choose bounded linear functionals $\omega_j$ on $\M_N$ with $\omega_j(\gamma_k)=\delta_{jk}$. Applying $\omega_j \ot \Id$ to \eqref{123AZERY-bis}, we obtain
$$
(\omega_j \ot \Id)\big([D,f]\big)
=-\i [\partial_j ,L_f].
$$
Hence each $[\partial_j,L_f]$ is bounded on the space $\L^p(\T_\theta^d)$. Note that for each $j$, the one-parameter group $t\mapsto \alpha_{t e_j}$ has generator $\partial_j$ and
\begin{equation}
\label{eq:conjugation-NCT}
\alpha_t L_f \alpha_{-t}
= L_{\alpha_t(f)}, \qquad t \in \R^d, f \in \L^\infty(\T_\theta^d).
\end{equation}
Indeed, for any $x \in \cal{P}_\theta$ we have
$$
\alpha_t L_f \alpha_{-t}(x)
=\alpha_t(f\alpha_{-t}(x))
=\alpha_t(f)x=L_{\alpha_t(f)}(x).
$$
Fix $j \in \{1,\ldots,d\}$ and for any $t \ne 0$, consider the difference quotient
\[
\Delta_t^{(j)} f
\ov{\mathrm{def}}{=} \frac{\alpha_{t e_j}(f)-f}{t} \in \L^\infty(\T_\theta^d).
\]
Now, fix $j \in \{1,\ldots,d\}$ and use  Proposition \ref{prop-group-isometries} with $R=L_f$ and the one-parameter group $(\alpha_{t e_j})_{t \in \R}$. For any small $t \ne 0$, we obtain
\begin{equation}
\label{inter-TGF4-prime}
L_{\Delta_t^{(j)} f}
\ov{\mathrm{def}}{=}
\frac{L_{\alpha_{t e_j}(f)}-L_f}{t}
=\frac{\alpha_{t e_j} L_f \alpha_{-t e_j}-L_f}{t}
\ov{\eqref{inter-formula-complex}}{=} \int_0^1 \alpha_{s t e_j} [\partial_j,L_f]\alpha_{-s t e_j} \d s.
\end{equation}
Consequently, we have
\begin{equation}
\label{eq:diffq-bound-NCT}
\norm{\Delta_t^{(j)} f}_{\L^\infty(\T_\theta^d)}
=\norm{L_{\Delta_t^{(j)} f}}_{\L^p(\T_\theta^d) \to \L^p(\T_\theta^d)}
\leq \int_0^1 \norm{\alpha_{s t e_j} [\partial_j,L_f]\alpha_{-s t e_j}} \d s
=\norm{[\partial_j,L_f]}.
\end{equation}
By Banach-Alaoglu theorem, there exist $t_\alpha\to 0$ and $g_j \in \L^\infty(\T_\theta^d)$ such that $\Delta_{t_\alpha}^{(j)} f \to g_j$ in the weak* topology of $\L^\infty(\T_\theta^d)$.
For any $h \in \mathcal P_\theta$, we have
\[
\tau\big((\Delta_t^{(j)} f) h\big)
=\tau\Big(f \frac{\alpha_{-t e_j}(h)-h}{t}\Big)
\xrightarrow[t \to 0]{}  -\tau\big(f \partial_jh\big),
\]
since $\frac{\alpha_{-t e_j}(h)-h}{t}\to -\partial_jh$ in $\L^1(\T_\theta^d)$. Passing to the weak* limit yields
\[
\tau(g_j h)
=-\tau(f\partial_jh), \qquad h \in \cal{P}_\theta,
\]
that is, $g_j$ is the distributional derivative $\partial_j f$. Thus $\partial_j f \in \L^\infty(\T_\theta^d)$.  
Hence $f \in \W^{1,\infty}(\T_\theta^d)$, proving $\Lip_D(\L^\infty(\T_\theta^d)) \subset \W^{1,\infty}(\T_\theta^d)$.

Now, we prove the lower bound in \eqref{eq:NC-comm-norm}. For the lower bound,We use again the linear functionals $\omega_j$ on $\M_N$ with $\omega_j(\gamma_k)=\delta_{jk}$. Applying $\omega_j \ot \Id$ to \eqref{eq:NC-comm-formula} gives
\begin{equation}
\label{inter-9ouy}
(\omega_j \ot \Id)\big([D,f]\big)
\ov{\eqref{eq:NC-comm-formula}}{=} -\i( \omega_j \ot \Id )\bigg(\sum_{j=1}^d \gamma_j \ot L_{\partial_j f}\bigg)
=-\i L_{\partial_j f}.
\end{equation}
We conclude that
\[
\norm{\partial_j f}_{\L^\infty(\T_\theta^d)}
=\norm{L_{\partial_j f}}_{\L^p(\T_\theta^d) \to \L^p(\T_\theta^d)}
\ov{\eqref{inter-9ouy}}{\leq} \norm{\omega_j}\norm{[D,f]}.
\]
Taking the maximum over $j$ yields the left inequality in \eqref{eq:NC-comm-norm} with $c_1=(\max_j \|\omega_j\|)^{-1}$.
\end{proof}

\paragraph{Boundedness of the functional calculus}

We will use the following result which says that some Riesz transforms are bounded. It will be proved in a companion paper \cite{Arh24d}.

\begin{thm}
\label{th-Riesz-NC-tori}
Suppose that $1 < p < \infty$. Consider an integer $d \geq 2$ and a real skew-symmetric matrix $\theta \in \M_d(\R)$. For any $1 \leq j \leq d$, the linear map $\partial_j\Delta_p^{-\frac{1}{2}}$ is bounded from the subspace $\Ran \Delta_p^{\frac{1}{2}}$ into the Banach space $\L^p(\T_\theta^d)$. 
\end{thm}

We introduce the map $
\partial
\ov{\mathrm{def}}{=} \partial_1+\i\partial_2$. First, we prove a technical result, similar to Proposition \ref{Prop-R-gradient-bounds-Fourier}. 

\begin{prop}
\label{Prop-R-gradient-bounds-NC-torus} 
Suppose that $1 < p < \infty$. The family 
\begin{equation}
\label{R-gradient-bounds-groups}
\Big\{t\partial(\Id-t^2\Delta)^{-1}: t > 0 \Big\}
\end{equation} 
of operators of $\B(\L^p(\T^2_\theta))$ is $R$-bounded.
\end{prop}

\begin{proof}
Since we have the equality $\partial (-\Delta)^{-\frac{1}{2}}=\partial_1(-\Delta)^{-\frac{1}{2}}+\i\partial_2(-\Delta)^{-\frac{1}{2}}$, the Riesz transform $\partial (-\Delta)^{-\frac{1}{2}} \co \L^p(\T^2_\theta) \to \L^p(\T^2_\theta)$ is a well-defined bounded operator by Theorem \ref{th-Riesz-NC-tori}. Suppose that $t > 0$. A standard functional calculus argument gives
\begin{align}
\label{Divers-988}
\MoveEqLeft
t\partial(\Id-t^2\Delta)^{-1}            
=\partial (-\Delta)^{-\frac{1}{2}}\Big((-t^2\Delta)^{\frac{1}{2}}(\Id-t^2\Delta)^{-1} \Big).
\end{align} 
By transference, note that the Laplacian $-\Delta$ has a bounded $\H^\infty(\Sigma_\theta)$ functional calculus for any angle $\theta > 0$. Moreover, the noncommutative $\L^p$-space $\L^p(\T^2_\theta)$ is a $\UMD$ Banach space by \cite[Corollary 7.7 p.~1494]{PiX03}, hence has the triangular contraction property $(\Delta)$ by \cite[Theorem 7.5.9 p.~137]{HvNVW18}. We deduce by \cite[Theorem 10.3.4 (2) p.~401]{HvNVW18} that the operator $-\Delta$ is $R$-sectorial. By \cite[Example 10.3.5 p.~402]{HvNVW18} applied with $\alpha=\frac{1}{2}$ and $\beta=1$, we infer that the set
$$
\big\{(-t^2\Delta)^{\frac{1}{2}}(\Id-t^2\Delta)^{-1}: t>0\big\}
$$
of operators of $\B(\L^p(\T^2_\theta))$ is $R$-bounded. Recalling that a singleton is $R$-bounded by \cite[Example 8.1.7 p.~170]{HvNVW18}, we obtain by composition \cite[Proposition 8.1.19 (3) p.~178]{HvNVW18} that the set
$$
\Big\{ \partial (-\Delta)^{-\frac{1}{2}}\Big((-t^2\Delta)^{\frac{1}{2}}(\Id-t^2\Delta)^{-1} \Big) : t>0 \Big\}
$$
of operators of $\B(\L^p(\T^2_\theta))$ is $R$-bounded. Hence with \eqref{Divers-988} we conclude that the subset \eqref{R-gradient-bounds-groups} is $R$-bounded.
\end{proof}

\begin{prop}
Suppose that $1 < p < \infty$. The operator $D$ is bisectorial and admits a bounded $\H^\infty(\Sigma_\theta^\bi\big)$ functional calculus on the Banach space $\L^p(\T^2_\theta) \oplus_p \L^p(\T^2_\theta)$ for some angle $0 < \theta <\frac{\pi}{2}$.
\end{prop}

\begin{proof}
We will start by showing that the set $\{\i t : t \in \R, t \not=0\}$ is contained in the resolvent set of $D$. We will do this by showing that $\Id-\i tD$ has a two-sided bounded inverse $(\Id-\i tD)^{-1}$ given by
\begin{equation}
\label{Resolvent-NC-torus}
\begin{bmatrix} 
(\Id_{}-t^2\Delta)^{-1} & \i t(\Id_{}-t^2\Delta)^{-1}\partial^* \\ 
\i t\partial(\Id_{}-t^2\Delta)^{-1} &  (\Id-t^2\Delta)^{-1} 
\end{bmatrix}
\end{equation}
acting on the Banach space $\L^p(\T^2_\theta) \oplus_p \L^p(\T^2_\theta)$. By Proposition \ref{Prop-R-gradient-bounds-Fourier} and since the operator $-\Delta$ satisfy the property \eqref{Def-R-bisectorial} of $R$-sectoriality, the four entries are bounded. It only remains to check that this matrix defines a two-sided inverse of $\Id-\i t D$. We have the following equalities of operators acting on $\dom D$.
\begin{align*}
\MoveEqLeft
\begin{bmatrix} 
(\Id_{}-t^2\Delta)^{-1} & \i t(\Id_{\L^p}-t^2\Delta)^{-1}\partial^* \\ 
\i t(\Id_{}-t^2\Delta)^{-1}\partial &  (\Id-t^2\Delta)^{-1} 
\end{bmatrix}(\Id-\i tD)   \\       
&\ov{\eqref{Dirac-dim-2}}{=}
\begin{bmatrix} 
(\Id-t^2\Delta)^{-1} & \i t(\Id_{}-t^2\Delta)^{-1}\partial^* \\ 
\i t(\Id_{}-t^2\Delta)^{-1}\partial &  (\Id-t^2\Delta)^{-1} 
\end{bmatrix}
\begin{bmatrix} 
\Id_{} & -\i t \partial^* \\ 
-\i t\partial &\Id_{}
\end{bmatrix}\\
&=\left[\begin{matrix} 
(\Id-t^2\Delta)^{-1}+t^2(\Id-t^2\Delta)^{-1}\partial^*\partial&-\i t(\Id-t^2\Delta)^{-1}\partial^*+\i t(\Id-t^2\Delta)^{-1} \partial^*\\ 
\i t(\Id-t^2\Delta)^{-1}\partial-\i t (\Id-t^2\Delta)^{-1}\partial& 
t^2(\Id-t^2\Delta)^{-1}\partial\partial^*+(\Id-t^2\Delta)^{-1}
\end{matrix}\right]\\
&=
\left[\begin{matrix} 
(\Id-t^2\Delta)^{-1}+t^2(\Id-t^2\Delta)^{-1}\Delta  & 0\\ 
\i t(\Id-t^2\Delta)^{-1}\partial-\i t\big(\Id-t^2\Delta\big)^{-1}\partial &
(\Id-t^2\Delta)^{-1}(t^2\partial\partial^*+\Id_{})
\end{matrix}\right] 
= 
\begin{bmatrix} 
\Id_{} & 0 \\ 
0 & \Id_{}
\end{bmatrix}
\end{align*} 
and similarly
\begin{align*}
\MoveEqLeft
(\Id-\i tD) 
\begin{bmatrix} 
(\Id_{}-t^2\Delta)^{-1} & \i t(\Id_{}-t^2\Delta)^{-1}\partial^* \\ 
\i t(\Id_{}-t^2\Delta)^{-1}\partial &  (\Id_{}-t^2\Delta)^{-1} 
\end{bmatrix}          
= \begin{bmatrix} 
\Id_{} & 0 \\ 
0 & \Id_{}
\end{bmatrix}.
\end{align*} 
It remains to show that the set $\{\i t(\i t-D)^{-1} : t \not=0\}=\{(\Id -\i t D)^{-1} : t \not=0\}$ is $R$-bounded. For this, observe that the diagonal entries of \eqref{Resolvent-Fourier} are $R$-bounded by the $R$-sectoriality of $\Delta$. The $R$-boundedness of the other entries follows from the $R$-gradient bounds of Proposition \ref{Prop-R-gradient-bounds-Fourier}. Since a set of operator matrices is $R$-bounded precisely when each entry is $R$-bounded, we conclude that \eqref{Def-R-bisectorial} is satisfied, i.e.~that the operator $D$ is $R$-bisectorial.
\end{proof}

\begin{thm}
\label{th-NC-torus-1}
Suppose that $1 < p < \infty$. Consider an integer $d \geq 2$ and a real skew-symmetric matrix $\theta \in \M_d(\R)$. The triple $(\cal{A}_\theta,\L^p(\T^2_\theta) \oplus \L^p(\T^2_\theta), D)$ is a compact Banach spectral triple.
\end{thm}

\begin{proof}
For any function $f \in \cal{A}_\theta$, an elementary computation reveals that $M_f \cdot \dom D \subset \dom D$ and that
\begin{equation}
\label{commutator-3}
[D,f]
=\frac{1}{\i}\begin{bmatrix} 
0 & M_{\partial f} \\ 
M_{\partial f} & 0
\end{bmatrix} .
\end{equation}
So the commutator $[D,f]$ defines a bounded operator on the Banach space $\L^p(\T^2_\theta) \oplus_p \L^p(\T^2_\theta)$. It is clear that $\Delta$ has compact resolvent on $\L^2(\T^2_\theta)$. By a standard interpolation argument (adapt \cite[Section 1.6]{Dav89}), we deduce that $\Delta$ has compact resolvent on $\L^p(\T^2_\theta)$ for any $1 < p < \infty$. By \eqref{square-of-Dirac}, we conclude that $D$ has compact resolvent on the space $\L^p(\T^2_\theta) \oplus \L^p(\T^2_\theta)$. So Definition \ref{Def-Banach-spectral-triple} is satisfied.
\end{proof}


\paragraph{Fredholm module}
Using the functional calculus, we can define the operator $\sgn D$ acting on the Banach space $\L^p(\T^2_\theta) \oplus \L^p(\T^2_\theta)$. This operator can be seen as <<Riesz-Clifford transform>> 
$$
\sgn(D)    
= \sum_{j=1}^2 \gamma_j \ot  \partial_j\Delta^{-\frac{1}{2}}.
$$
Using Proposition \ref{prop-triple-to-Fredholm}, we obtain the following consequence of Theorem \ref{th-NC-torus-1}.

\begin{cor}
Suppose that $1 < p < \infty$. The triple $(\cal{A}_\theta,\L^p(\T^2_\theta) \oplus \L^p(\T^2_\theta),\sgn D)$ is a Banach Fredholm module. 
\end{cor}

%
%
%
%
%

\begin{remark} \normalfont
Suppose that $p=2$. If $f \in \L^2(\T^d_\theta)$ satisfies $\qd f \in S^q$, for some $q \in (0, d]$, then by \cite[Corollary 1.5]{MSX19} $f$ is a constant.  Moreover, if $f$ belongs to the noncommutative homogeneous Sobolev space $\dot{\H}^1_d(\T^d_\theta)$, then by \cite[Theorem 1.1]{MSX19} $\qd f$ admits a bounded extension, and the extension belongs to the weak Schatten space $S^{d,\infty}(\ell^2_N(\L^2(\T^2_\theta)))$. In particular, the Fredholm module $(\cal{A}_\theta,\L^2(\T^2_\theta) \oplus \L^2(\T^2_\theta),\sgn D)$ is $q$-summable for any $q > 2$.
\end{remark}

\subsection{Perturbed Dirac operators and the Kato square root problem}
\label{Kato}

Let $A \co \R^n \to \M_n(\mathbb{C})$ be a measurable bounded function satisfying for some constant $\kappa > 0$ the ellipticity condition
$$
\Re \langle A_x\xi, \xi\rangle_{\mathbb{C}^n} 
\geq \kappa |\xi|^2
$$
for almost all $x \in \R^n$ and all $\xi \in \mathbb{C}^n$ . We can see $A$ as an element of the von Neumann algebra $\L^\infty(\R^n,\M_n(\mathbb{C}))$. We denote the angle of accretivity by $\omega \ov{\mathrm{def}}{=}
\esssup_{x,\xi} |\arg \la A_x\xi, \xi\ra|$. Consider the multiplication operator $M_A \co \L^2(\R^n, \mathbb{C}^n) \to \L^2(\R^n, \mathbb{C}^n)$, $u \mapsto A u$ and the unbounded operator
\begin{equation}
\label{}
D
\ov{\mathrm{def}}{=} \begin{bmatrix} 
0 & -\div M_A \\ 
\nabla & 0 
\end{bmatrix},
\end{equation}
acting on the complex Hilbert space $\L^2(\R) \oplus \L^2(\R, \mathbb{C}^n)$, where $\div M_A$ denotes the composition $\div \circ M_A$. This operator generalizes the one defined in \eqref{Dirac-perturbed-intro}. We can see this operator as a deformation of the <<Hodge-Dirac operator>>
\begin{equation}
\label{D-456}
\scr{D}
\ov{\mathrm{def}}{=}\begin{bmatrix} 
0 & -\div  \\ 
\nabla & 0 
\end{bmatrix}.
\end{equation}

Note the following relation between $\mathscr{D}$ and the Hodge--Dirac operator $D_g$ on $\R^n$. Let $(X,g)$ be a smooth Riemannian manifold and denote by $D_g \ov{\mathrm{def}}{=} d + d^*$ the Hodge--Dirac operator acting on the Hilbert space
$\L^2(\Omega(X,g))=\bigoplus_{k=0}^n \L^2(\Omega^k(X,g))$.
For each $k$, the exterior derivative $d$ maps $\Omega^k$ into $\Omega^{k+1}$, and
$d^*$ maps $\Omega^{k+1}$ into $\Omega^k$. In the decomposition of $\L^2(\Omega(X,g))$,
the operator $D_g$ therefore takes the block tridiagonal form
\[
D_g =
\begin{bmatrix}
0 & d^* & 0 & 0 & \cdots\\
d & 0 & d^* & 0 & \cdots\\
0 & d & 0 & d^* & \cdots\\
\vdots & & & \ddots
\end{bmatrix}.
\]
When $X=\R^n$ with the Euclidean metric $g_0$, we have $\Omega^0(\R^n)=\C^\infty(\R^n)$ and $\Omega^1(\R^n)\simeq \C^\infty(\R^n,\mathbb{C}^n)$ via the musical isomorphisms $\sharp,\flat$. Restricting $D_{g_0}$ to $\Omega^0\oplus\Omega^1$, and identifying a vector field $v$ with the $1$-form $v^\flat$, we obtain
\[
D_{g_0}\big|_{\Omega^0\oplus\Omega^1}
=
\begin{bmatrix}
0 & d^*\\[2pt]
d & 0
\end{bmatrix}.
\]
Since on $\R^n$ one has $d=\nabla$ and $d^*=-\div$, this gives
\begin{equation}
\label{Rem-Hodge-Dirac-Relation}
D_{g_0}\big|_{\Omega^0 \oplus \Omega^1}
=
\begin{bmatrix}
0 & -\div\\[2pt]
\nabla & 0
\end{bmatrix}
=\mathscr{D}
\quad\text{on } \L^2(\R^n)\oplus \L^2(\R^n,\mathbb{C}^n).
\end{equation}
%

By \cite[Theorem 3.1 (i) p.~465]{AKM06},  the operator $D$ is bisectorial and admits a bounded $\H^\infty(\Sigma^\bi_\theta)$ functional calculus for any angle $\theta \in (\omega ,\frac{\pi}{2})$. So, we can consider the bounded operator $\sgn D \co \L^2(\R^n) \oplus \L^2(\R^n, \mathbb{C}^n) \to \L^2(\R^n) \oplus \L^2(\R^n, \mathbb{C}^n)$. As explained in \cite{AKM06} the boundedness of the operator $\sgn D$ allows to obtain the Kato square root estimate
$$
\bnorm{(-\div A \nabla)^{\frac{1}{2}}f}_{\L^2(\R^n)}
\approx \norm{\nabla f}_{\L^2(\R^n,\ell^2_n)}, \quad u \in \W^{1,2}(\R^n).
$$
We refer to \cite[pp.~509-513]{HvNVW23}, \cite[Section 4.7]{Gra14b} and \cite[Chapter 8]{Ouh05} for more information on this famous estimate.

Note that $\L^2(\R^n,\mathbb{C}^n)=\ell^2_n(\L^2(\R^n))$. If $f \in \L^\infty(\R^n)$, we define the bounded operator $\pi_f \co \L^2(\R^n) \oplus \L^2(\R^n,\mathbb{C}^n) \to \L^2(\R^n) \oplus \L^2(\R^n,\mathbb{C}^n)$ by
\begin{equation}
\label{Def-pi-a}
\pi_f
\ov{\mathrm{def}}{=} \begin{bmatrix}
    \M_f & 0  \\
    0 & \tilde{\M}_f  \\
\end{bmatrix}, \quad f \in \L^\infty(\R^n)
\end{equation}
where the linear map $\M_f \co \L^2(\R^n) \to \L^2(\R^n)$, $g \mapsto fg$ is the multiplication operator by the function $f$ and where 
$$
\tilde{\M}_f \ov{\mathrm{def}}{=} \Id_{\ell^2_n} \ot \M_f  \co \ell^2_n(\L^2(\R^n)) \to \ell^2_n(\L^2(\R^n)),\, (h_1,\ldots,h_n) \mapsto (fh_1,\ldots,fh_n)
$$ 
is also a multiplication operator (by the function $(f,\ldots,f)$ of $\ell^\infty_m(\L^\infty(\R^n))$). Using \cite[Proposition 4.10 p.~31]{EnN00}, it is easy to check that $\pi \co \L^\infty(\R^n) \to \B(\L^2(\R^n) \oplus \L^2(\R^n,\mathbb{C}^n))$ is an isometric homomorphism. 

\begin{prop}
The triple $(\C_0(\R),\L^2(\R) \oplus \L^2(\R, \mathbb{C}^n),D)$ is a possibly kernel-degenerate locally compact Banach spectral triple.
\end{prop}

\begin{proof}
We already said that the operator $D$ is bisectorial and admits a bounded $\H^\infty(\Sigma^\bi_\theta)$ functional calculus. Note that $D=\begin{bmatrix} 
0 & -\div  \\ 
\nabla & 0 
\end{bmatrix}\begin{bmatrix} 
\I & 0 \\ 
0 & M_A 
\end{bmatrix} \ov{\eqref{D-456}}{=} \scr{D}B$ where $B \ov{\mathrm{def}}{=} \begin{bmatrix} 
\I & 0 \\ 
0 & M_A 
\end{bmatrix}$. It is easy to see that $\pi_fB=B\pi_f$. Consequently, we have
$$
[D,\pi_f]
=D\pi_f-\pi_fD
=\scr{D}B\pi_f-\pi_f\scr{D}B
=\scr{D}\pi_f B-\pi_f\scr{D}B
=(\scr{D}\pi_f B-\pi_f\scr{D}) B
=[\scr{D},\pi_f]B.
$$
Since it is well-known that the commutator $[\scr{D},\pi_f]$ is bounded on $\L^2(\R^n) \oplus \L^2(\R^n,\mathbb{C}^n)$, we infer that the commutator $[D,\pi_f]$ is bounded. We have 
$$
\pi_f D^{-1}
=\pi_f(\scr{D}B)^{-1}
=\pi_f B^{-1}\scr{D}^{-1}
=B^{-1}\pi_f\scr{D}^{-1}.
$$
Since it is well-known (\cite{SuZ23} and Remark \ref{Rem-Hodge-Dirac-Relation} or \cite{Arh24b}) that $\pi_f\scr{D}^{-1}$ is compact, we conclude that the operator $\pi_f D^{-1} \co \ovl{\Ran D} \to X$ is also compact.
\end{proof}

\begin{remark} \normalfont
It is very easy to introduce variants or generalization of this result. We refer to \cite{AKM06} and \cite{AMN97} for the considered operators.
\end{remark}

\subsection{Perturbed Dirac operators and the Cauchy singular integral operator}


Recall that by Rademacher's theorem, a function $g \co \R \to \R$ is Lipschitz if and only if $g$ is differentiable almost everywhere on $\R$ and $g' \in \L^\infty(\R)$. In this case, the Lipschitz constant $$
\Lip g
\ov{\mathrm{def}}{=} \sup \bigg\{\frac{|g(x) - g(y)|}{|x - y|} : x,y \in \R, x \not= y\bigg\}.
$$
is equal to $\norm{g'}_{\L^\infty(\R)}$. In the sequel, we fix a Lipschitz function $g \co \R \to \R$ and we consider the associated Lipschitz curve in the complex plane, which is the graph $\gamma \ov{\mathrm{def}}{=} \{z = x + \i g(x) : x \in \R \}$. 


Suppose that $1 < p < \infty$. Consider the unbounded operator $D_\gamma \ov{\mathrm{def}}{=} -B\frac{1}{\i}\frac{\d}{\d x}$ acting on a suitable dense domain of $\L^p(\R)$ where $B$ is the  multiplication operator by the function $(1+\i g')^{-1}$. By \cite[Theorem 3.1 (iii) p.~~465]{AKM06} and \cite[Consequence 3.2 p.~466]{AKM06}, for any angle $\theta \in (\arctan \Lip g, \frac{\pi}{2})$, this operator is bisectorial and admits a bounded $\H^\infty(\Sigma_\theta^\bi)$ functional calculus. Note that this result is also proved in \cite[Theorem 1.4.1 p.~25]{QiL19} and \cite[Theorem 7.1 p.~159]{McQ91} using a different but equivalent formulation.

Consequently, the operator $\sgn D_\gamma \co \L^p(\R) \to \L^p(\R)$ is bounded. Moreover, according to \cite[Consequence 3.2 p.~466]{AKM06}, we have
$$
\sgn(D_\gamma)
=C_\gamma,
$$
where $C_\gamma$ is the <<Cauchy singular integral operator on the real line>> defined by
\begin{equation}
\label{}
(C_\gamma f)(x) 
\ov{\mathrm{def}}{=} 
\lim_{\epsi \to 0} \frac{\i}{\pi} \int_{y\in{\mathbb{R}}\setminus(x-\epsi,x+\epsi)} \frac{f(y)}{y-x + \i (g(y) -g(x))}(1 + \i g'(y)) \d y.
\end{equation}
Here the integral exists for almost all $x \in \R$ and define an element of the Banach space $\L^p(\R)$. We refer to \cite[Theorem 1 p.~142]{McQ91} and \cite[(b) p.~157]{McQ91} for a proof of this formula stated in a different language. We direct the reader to the excellent paper \cite{Ver21} for a detailed description of numerous applications of this operator and to \cite[Section 8]{ADM96}, \cite[Section 4.6]{Gra14b} and \cite[Chapter 7]{Jou83} for more information.

\begin{remark} \normalfont
If $g=0$, then we recover the operator $\frac{1}{\i}\frac{\d}{\d x}$.
\end{remark}

\begin{remark} \normalfont
Actually, the boundedness of the operator $C_\gamma$ was first proved by Calder\'on in \cite[Theorem 1, p.~1324]{Cal77} in the case where $\Lip g$ is small enough and in the general case in the famous paper \cite[Th\'eor\`eme I]{CMM82}. The used definition is  
\begin{equation}
\label{}
(\cal{C}_\gamma u)(z)
\ov{\mathrm{def}}{=} \lim_{\epsi \to 0} \frac{1}{2\pi \i} \int\limits_{\zeta \in \gamma \setminus B(z,\epsi)}\frac{u(\zeta)}{\zeta-z} \d\zeta,\qquad z \in \gamma.
\end{equation}
Making the bi-Lipschitz change of variables $\R \to \gamma$, $y \mapsto y+\i g(y)$ and 
identifying $f$ with the function $h(y) \ov{\mathrm{def}}{=} u(y+\i g(y))$ for $x \in \R$, this 
becomes  the Cauchy singular integral operator on the real line, up to a multiplicative constant.
\end{remark}


As noted in~\cite[p.~289]{Gra14b}, the~$\L^p$-boundedness of the operator $C_{\gamma} \co \L^p(\R) \to \L^p(\R)$ is equivalent to that of the related operator $\scr{C}_{\gamma} \co \L^p(\R) \to \L^p(\R)$ defined by
\begin{equation}
\label{eq:l3}
\scr{C}_{\gamma}(f)(x)
\ov{\mathrm{def}}{=} \lim_{\epsi \to 0} \frac{\i}{\pi}\int_{y\in{\mathbb{R}}\setminus(x-\epsi,x+\epsi)} \frac{f(y)}{y-x + \i(g(y) - g(x))} \d y, \quad f \in \L^p(\R).
\end{equation}

Suppose that $1 < p < \infty$. Recall that any multiplication operator $M_f \co \L^p(\R) \to \L^p(\R)$, $fg \mapsto fg$ by a function $f$ of $\L^\infty(\R)$ is bounded, see \cite[Proposition 4.10 p.~31]{EnN00}. Consequently, we can consider the homomorphism $\pi \co \L^\infty(\R) \to \B(\L^p(\R))$, $f \mapsto M_f$.

Consider a function $f \in \cup_{1 < q < \infty} \L^q_{\textup{loc}}(\R)$ and suppose that $1 < p < \infty$. It is proved in \cite[Theorem 1]{LNWW20} that the function $f$ belongs to the space $\BMO(\R)$ if and only if the commutator $[M_f,\scr{C}_{\gamma}]$ is bounded on the Banach space $\L^p(\R)$. In this case, we have
\begin{equation}
\label{commutator-78}
\norm{[M_f,\scr{C}_{\gamma}] }_{\L^p(\R) \to \L^p(\R)} 
\approx \norm{f}_{\BMO(\R)}. 
\end{equation}
We denote by $\VMO(\R)$ the space of functions of vanishing mean oscillation, defined to be the $\BMO(\R$)-closure of the set $\C^\infty_c(\R)$ of functions of class $\C^\infty$ with compact support.
Furthermore, if $f \in \BMO(\R)$ it is proved in \cite[Theorem 2]{LNWW20} that $f$ belongs to $\VMO(\R)$ if and only if the commutator $[M_f,\scr{C}_{\gamma}]$ is a compact operator on the Banach space $\L^p(\R)$. It is obvious that this entails similar results for the operator $C_\gamma$. 



\begin{prop}
Then $(\C_0(\R),\L^2(\R),D_\gamma)$ is a possibly kernel-degenerate locally compact Banach spectral triple. 
\end{prop}

\begin{proof}
Note that the function $1 + \i g' \co \R \to \R$ is bounded above and below. Using the notation $D \ov{\mathrm{def}}{=} \frac{1}{\i}\frac{\d}{\d x}$, for any function $f \in \C_0^\infty(\R)$ we have
$$
[D_\gamma,\pi_f]
=D_\gamma\pi_f-\pi_fD_\gamma
=BD\pi_f-\pi_fBD
=BD\pi_f-B\pi_f D
=B(D\pi_f -\pi_f D)
=B[D,\pi_f].
$$
Since it is well-known that the commutator $[D,\pi_f]$ is bounded, we infer that the commutator $[D_\gamma,\pi_f]$ is bounded. We have 
$$
\pi_f D_\gamma^{-1}
=\pi_f(BD)^{-1}
=\pi_f D^{-1}B^{-1}.
$$
Since it is well-known that $\pi_f D^{-1}$ is compact, we conclude that the operator $\pi_f D_\gamma^{-1} \co \ovl{\Ran D} \to \L^2(\R)$ is also compact.
\end{proof}

\begin{prop}
Suppose that $1 < p < \infty$. Then $(\L^p(\R),\pi,C_\gamma)$ is an odd Banach Fredholm module over the algebra $\C_0(\R)$.
\end{prop}

\begin{proof}
If $p=2$, we could use a possible (easy) locally compact extension of Proposition \ref{prop-triple-to-Fredholm}. We could also use the following reasoning. Recall that the algebra $\C_0(\R)$ of continuous functions that vanish at infinity, since $\C_0(\R)$ is a subspace of $\VMO(\R)$. So if $f \in \C_0(\R)$ by \cite[Theorem 2]{LNWW20} each commutator $[M_f,\scr{C}_{\gamma}]$ is a compact operator on the Banach space $\L^p(\R)$. So Definition \ref{def-Fredholm-module-odd} is satisfied.
\end{proof}


\begin{remark} \normalfont
It is possible to make a similar analysis with Clifford-Cauchy singular integrals on Lipschitz surfaces.
\end{remark}

\begin{remark} \normalfont
If $A$ and $B$ are some bounded operators on a complex Hilbert space $H$, the previous section section and this lead us to consider <<deformed>> triples $(\cal{A},BDA,H)$ of a spectral triple $(\cal{A},D,H)$ where the selfadjoint operator $D$ is replaced by the deformed (not necessarily selfadjoint) operator $BDA$, possibly bisectorial and admitting a bounded $\H^\infty(\Sigma^\bi_\theta)$ functional calculus for some angle $\theta \in (0,\frac{\pi}{2})$.
\end{remark}

\subsection{Banach Fredholm modules arising from pseudo-differential operators}

Let $E$ be a Hermitian complex vector bundle over a smooth closed manifold $M$. Suppose that $1 \leq p < \infty$. Using a Lebesgue measure $\mu$ on $M$, we can consider the space $\L^p(M,E)$ of sections of $\L^p$-integrable measurable sections (modulo the space of negligible sections) which does not depend on $\mu$ and the Hermitian structure on $E$, and which is a Banach space for the norm
\begin{equation}
\label{norm-a-valeurs-E}
\norm{s}_{\L^p(M,E)}
\ov{\mathrm{def}}{=} \begin{cases}
\big(\int_{M} |s(x)|_{E_x}^p \d \mu(x)\big)^{\frac{1}{p}}&\text{ if } 1 \leq p<\infty\\
\esssup_{x \in M} |s(x)|_{E_x}& \text{ if } p=\infty
\end{cases}.
\end{equation}
If $p=2$, this space is a Hilbert space. For any $1 < p < \infty$, it is not difficult to prove the following isomorphic interpolation formula
\begin{equation}
\label{interpolation-formula}
\L^p(M,E)
\approx (\L^\infty(M,E),\L^1(M,E))_{\frac{1}{p}},
\end{equation}
using a suitable partition of unity and the interpolation formula described in \cite[Theorem 2.2.6 p.~91]{HvNVW16}. 

Consider a closed smooth Riemannian manifold $M$. Recall that a pseudo-differential operator $P \co \C^\infty(M,E) \to \C^\infty(M,E)$ is said to be Hermitian if for any sections $f,g \in \C^\infty(M,E)$ we have
\begin{equation}
\label{}
\la P(f), g\ra_{\L^2(M,E)}	
=\la f,P(g) \ra_{\L^2(M,E)}.
\end{equation}
We say that a pseudo-differential operator $P \co \C^\infty(M,E) \to \C^\infty(M,E)$ is elliptic if its admits a principal symbol $\sigma_P^0$ such that the map $\sigma_P^0(x,\xi) \co E_x \to E_x $ is a bijection for any $x \in M$ and any non-zero vector $\xi \in \mathrm{T}_x^*M$.  By \cite[(23.35.2)]{Dieu88}, 
an elliptic Hermitian pseudo-differential operator $P \co \C^\infty(M,E) \to \C^\infty(M,E)$ of order $m > 0$ induces an essentially selfadjoint operator on the Hilbert space $\L^2(M,E)$. Its unique closure is still denoted $P$. It is also proved in this  reference that the closed operator $P$ has compact resolvent, that its spectrum is a closed discrete infinite subset of $\R$ and that each eigenspace is finite-dimensional.

We will use the following folklore lemma (see \cite[Corollary 6.2 p.~51]{Shu01} for a version without vector bundle, with another argument).

\begin{lemma}
\label{Lemma-compact}
Let $E$ be a Hermitian complex vector bundle over a (smooth) closed manifold $M$, endowed with a Lebesgue measure. A pseudo-differential operator $P \co \C^\infty(M,E) \to \C^\infty(M,E)$ of order $m < 0$ induces a compact operator on the Hilbert space $\L^2(M,E)$
\end{lemma}

\begin{proof}
By \cite[Theorem 9.3 p.~235]{BlB13}, a pseudo-differential operator $P$ of order $m < 0$ induces a bounded operator from $\L^2(M,E)$ into the Sobolev space $\H^{-m}(M, E)$, since $-m > 0$. By Rellich's Theorem \cite[Theorem 7.15 p.~199]{BlB13}, the canonical inclusion $\H^s(M,E) \to \H^t(M,E)$ is compact for any $s > t \geq 0$. In particular, the map $\H^{-m}(M,E) \to \H^0(M,E)=\L^2(M,E)$ is compact.  
Since the composition of a bounded operator and a compact operator is compact, we conclude that $P$ is compact on the Hilbert space $\L^2(M, E)$.
\end{proof}

An obvious <<vector bundle>> generalization of \cite[p.~267]{Tay81} and \cite[Theorem 5.2.23 p.~420]{RuT10} gives the following result.

\begin{prop}
\label{Th-order-zero-bounded}
Let $E$ be a Hermitian complex vector bundle over a (smooth) closed manifold $M$, endowed with a Lebesgue measure. A pseudo-differential operator $P \co \C^\infty(M,E) \to \C^\infty(M,E)$ of order $0$ induces a bounded operator on the Banach space $\L^p(M,E)$ for any $1 < p < \infty$.
\end{prop}

Now, we explain how obtaining Fredholm modules from pseudo-differential operators. Here, we consider the homomorphism $\pi \co \C(M) \to\B(\L^p(M,E))$, $f\mapsto M_f$ defined by multiplication operators, where $(M_fs)(x) \ov{\mathrm{def}}{=} f(x)s(x)$.

\begin{thm}
Consider an elliptic Hermitian pseudo-differential operator $P \co \C^\infty(M,E) \to \C^\infty(M,E)$ of order $m > 0$ on a closed Riemannian manifold $M$. Suppose that $1 < p < \infty$. Assume that $P$ induces a bisectorial operator on the Banach space $\L^p(M,E)$ admitting a bounded $\H^\infty(\Sigma^\bi_\theta)$ functional calculus for some angle $0 < \theta <\frac{\pi}{2}$. Then $(\L^p(M,E),\sign P)$ is an odd Banach Fredholm module over the algebra $\C(M)$ on the Banach space $\L^p(M,E)$.
\end{thm}

\begin{proof}
By \cite[Proposition 2.4, p.~143]{BaG82b} and its proof, the operator $\sign D \co \L^2(M,E) \to \L^2(M,E)$ induces a pseudo-differential operator $\sign P \co \C^\infty(M) \to \C^\infty(M)$ of order 0. Recall that a multiplication operator $M_f \co \C^\infty(M) \to \C^\infty(M)$ is a differential operator of order 0, hence a pseudo-differential operator of order 0. It is easy to check\footnote{\thefootnote. We warn the reader that the commutator of two pseudo-differential operators on a vector bundle $E$ of order $m$ and $m'$ is not necessarily of order $m + m'- 1$.} using the formula \cite[(7.2) p.~75]{Tre80} that the commutator $[\sign P,M_f]$ is a pseudo-differential operator of order $-1$. Since it has negative order, it induces by Lemma \ref{Lemma-compact} a compact operator on the Hilbert space $\L^2(M,E)$. Moreover, it equally induces by Theorem \ref{Th-order-zero-bounded} a bounded operator on the Banach space $\L^q(M,E)$ for any $1 < q < \infty$. By interpolation \cite[Theorem 3.11 p.~58]{KZPP76} relying on the formula \eqref{interpolation-formula}, we conclude that the commutator $[\sign P,M_f]$ induces a compact operator on the Banach space $\L^p(M,E)$. By approximation, we see that it is true more generally for any function $f \in \C(M)$.

Recall that by \cite[(23.35.2)]{Dieu88} the unbounded selfadjoint operator $P$ has compact resolvent on $\L^2(M,E)$. Consequently, for any continuous function $\psi \in \C_0(\R)$ the operator $\psi(P) \co \L^2(M,E) \to \L^2(M)$ is compact by \cite[Exercise 9.5.4 p.~396]{Eme24}. Since the spectrum of $P$ on $\L^2(M,E)$ is discrete, the operator $(\sign P)^2-\Id_{\L^2(M,E)}$ can be identified to such an operator $\psi(D)$. 
Consequently, it is a compact operator on $\L^2(M,E)$. Again by interpolation, we obtain that $(\sign P)^2-\Id_{\L^p(M,E)}$ is a compact operator on the space $\L^p(M,E)$.
\end{proof}

\begin{example} \normalfont
The previous result applies to the Dirac operator $\frac{1}{\i}\frac{\d}{\d \theta}$ of Section \ref{sec-dtheta-T}. This operator is a first-order elliptic Hermitian differential operator inducing a bisectorial operator on $\L^p(\T)$ which admits a bounded functional calculus by Theorem \ref{Th-dtheta}.
\end{example}

\begin{remark} \normalfont
If $c > 0$ and if $k \geq 1$ is an integer, we denote by $\mathbb{H}_c^{k+1}$ the $(k+1)$-dimensional hyperbolic space scaled such that its scalar curvature is $-c^2k(k+1)$. 
Suppose that $1< p < \infty$ with $p \not=2$. By \cite[Theorem 1.1]{AmG16}, the $\L^p$-spectrum of the Dirac operator $D$ on $\L^p(\mathbb{H}_0^{k+1},S)$ is not contained in a bisector. Consequently, this Dirac operator $D$ is not bisectorial with sharp contrast with the case of $\mathbb{T}$.
\end{remark}

\section{Banach Fredholm modules arising from noncommutative Hardy spaces}
\label{Examples-Hardy}

\subsection{Noncommutative Hardy spaces associated to subdiagonal algebras}
\label{sec-noncommutative-Hilbert-transform}

First, we explain how introduce Hilbert transform associated to subdiagonal algebras of von Neumann algebras. Such algebras were introduced by Arveson in \cite{Arv67} in the case of finite von Neumann algebras and in \cite{Bek15} for the semifinite case. We refer also to  \cite{BlL07}, \cite{Ji14}, \cite{LaX13}, \cite{MaW98}, \cite{Ran98}, \cite{Ran02} and \cite{PiX03} for more information on Hardy spaces associated with subdiagonal algebras of von Neumann algebras. Let $\cal{M}$ be a von Neumann algebra endowed with a normal semifinite faithful trace and let $\mathbb{E} \co \cal{M} \to \cal{M}$ be a trace preserving conditional expectation onto a von Neumann subalgebra $\D$ of $\cal{M}$. Following \cite[Definition 2.1 p.~1350]{Bek15}, a subdiagonal algebra $\H^\infty(\cal{M})$ with respect to $\mathbb{E}$ is a weak* closed subalgebra of $\cal{M}$ such that
\begin{enumerate}
	\item $\H^\infty(\cal{M})  \cap (\H^\infty(\cal{M}) )^* = \D$,
	 \item the trace preserving conditional expectation $\mathbb{E} \co \L^\infty(\cal{M}) \to \L^\infty(\cal{M})$ is multiplicative on $\H^\infty(\cal{M})$, i.e.~we have $\mathbb{E}(f  g) = \mathbb{E}(f) \mathbb{E}(g)$ for any $f,g \in \H^\infty(\cal{M})$,
  \item the sum $\H^\infty(\cal{M})+ (\H^\infty(\cal{M}) )^*$ is weak* dense in $\L^\infty(\cal{M})$.
\end{enumerate}
Here, we warn the reader that $(\H^\infty(\cal{M}) )^*$ is not the dual of $\H^\infty(\cal{M})$, the family of the adjoints of the elements of $\H^\infty(\cal{M})$. In this case, $\D$ is called the <<diagonal>> of the algebra $\H^\infty(\cal{M})$. We also define $\H_0^\infty(\cal{M}) \ov{\mathrm{def}}{=}  \{f \in \H^\infty(\cal{M}) : \mathbb{E}(f)=0 \}$. 

Suppose that $1 < p < \infty$. The closures of the subspaces $\H^\infty(\cal{M}) \cap \L^p(\cal{M})$ and $\H_0^\infty(\cal{M}) \cap \L^p(\cal{M})$ in $\L^p(\cal{M})$ will be denoted by $\H^p(\cal{M})$ and $\H^p_0(\cal{M})$. Note that $\mathbb{E}$ extends to a contractive projection $\mathbb{E} \co \L^p(\cal{M}) \to \L^p(\cal{M})$ onto the subspace $\L^p(\D)$ with kernel $\H^p_0(\cal{M})$. By \cite[Theorem 4.2 p.~1356]{Bek15}, we have a canonical topological direct sum decomposition
\begin{equation}
\label{}
\L^p(\cal{M})
=(\H_0^p(\cal{M}))^*\oplus \L^p(\D) \oplus \H_0^p(\cal{M}).
\end{equation}
In this setting any $f \in \L^p(\cal{M})$ admits a unique decomposition
\[
f 
=h^*+\delta+g, \quad \text{ with } g, h \in {\H_0^p(\cal{M})},  \delta \in \H^p(\D).
\] 
The Riesz projection $P \co \L^p(\cal{M}) \to \L^p(\cal{M})$, defined by
\begin{equation}
\label{}
P(h^*+\delta+g)
=\delta+g,
\end{equation}
is a bounded projection on the subspace $\H^p(\cal{M})$ for any $1 < p < \infty$. The Hilbert transform $\cal{H} \co \L^p(\cal{M}) \to \L^p(\cal{M})$ associated to the subdiagonal algebra $\H^\infty(\cal{M})$ is the bounded operator defined by
\begin{equation}
\label{Hilbert-transform-NC-Hardy}
\cal{H}(f) 
\ov{\mathrm{def}}{=} -\i(-h^*+g).
\end{equation}

\paragraph{Toeplitz operators} Given $a \in \L^\infty(\cal{M})$, the Toeplitz operator $T_a$  with symbol $a$ is defined  to be the map
\begin{equation}
\label{Toeplitz}
T_a \co \H^p(\cal{M})\to \H^p(\cal{M}), b \mapsto P_+(ab),
\end{equation}
where $P_+ \co \L^p(\cal{M}) \to \H^p(\cal{M})$ is the corestriction of the Riesz projection $P \co \L^p(\cal{M}) \to \L^p(\cal{M})$ onto the subspace $\H^p(\cal{M})$.

\paragraph{Banach Fredholm module and pairing with the $\K$-theory} 
Let $\cal{A}$ be a unital subalgebra of $\cal{M}$. Under suitable assumption, we introduce a Banach Fredholm module and we describe the pairing. We consider the homomorphism $\pi \co \cal{A} \to \B(\L^p(\cal{M}))$, $a \mapsto M_a$, where $M_a \co \L^p(\cal{M}) \to \L^p(\cal{M})$, $y \mapsto ay$ is the multiplication operator by $a$. We consider the bounded operator 
\begin{equation}
\label{def-de-F}
F \ov{\mathrm{def}}{=} \i\cal{H}+\mathbb{E} \co \L^p(\cal{M}) \to \L^p(\cal{M}).
\end{equation}

\begin{prop}
\label{prop-Hilbert-discrete-1}
Suppose that $1 < p < \infty$. Assume that the commutator $[\cal{H},M_a]$ is a compact operator acting on $\L^p(\cal{M})$ for any $a \in \cal{A}$ and that the algebra $\cal{D}$ is finite-dimensional. Then $(\L^p(\cal{M}),\pi,F)$ is an odd Banach Fredholm module. For any invertible element $a \in \cal{A}$, the Toeplitz operator $T_a$ is Fredholm and we have
\begin{equation}
\label{index-with-Toeplitz}
\big\la [a], (\L^p(\cal{M}),F) \big\ra_{\K_1(\cal{A}),\K^1(\cal{A},\scr{L}^p_{\nc})}
=\Index T_a.
\end{equation}
\end{prop}

\begin{proof}
Since the map $\mathbb{E} \co \L^p(\cal{M}) \to \L^p(\cal{M})$ is a projection on the subspace $\L^p(\cal{D})$, we have with \eqref{Hilbert-transform-NC-Hardy}
$$
F^2
=(\i\cal{H}+\mathbb{E})^2
=\Id.
$$
Moreover, for any $a \in \cal{A}$ we see that
$$
[F,M_a]
=[\i\cal{H}+\mathbb{E},M_a]
=\i[\cal{H},M_a]+[\mathbb{E},M_a]
$$
is a compact operator, since the latter commutator is finite-rank operator. Hence $(\L^p(\cal{M}), F)$ is an odd Banach Fredholm module. 
Now, observe that the map $P \ov{\mathrm{def}}{=} \frac{\Id+F}{2}\co \L^p(\cal{M}) \to \L^p(\cal{M})$ is the Riesz projection. Let $a \in \cal{A}$ be an invertible element. We deduce that $P a P \co P(\L^p(\cal{M})) \to P(\L^p(\cal{M}))$ is a Fredholm operator. This operator identifies to the Toeplitz operator $T_a \co \H^p(\cal{M}) \to \H^p(\cal{M})$ and that 
\begin{align*}
\MoveEqLeft
  \big\la [a], (\L^p(\cal{M}),F) \big\ra_{\K_1(\cal{A}),\K^1(\cal{A},\scr{L}^p_{\nc})}       
\ov{\eqref{pairing-odd-2}}{=} \Index P a P \co P(\L^p(\cal{M})) \to P(\L^p(\cal{M}))
\ov{\eqref{Toeplitz}}{=} T_a.
\end{align*}
\end{proof}


%
%
%


\subsection{Classical Hardy spaces}
\label{sec-classical-Hardy}

The fundamental example of subdiagonal algebras is the algebra $\H^\infty(\T)$ on the unit circle. In this case the index pairing is described in Proposition \ref{prop-pairing-T}. We can also state a similar result for the matrix-valued Hardy space $\H^\infty(\T,\M_n)$, which be of interest for the multivariate prediction theory. This algebra is a subalgebra of the von Neumann algebra $\L^\infty(\T,\M_n)$. 

\subsection{Fredholm module associated to the discrete Schur-Hilbert transform on $S^p_n$}
\label{sec-noncommutative-Hilbert-transform}



Consider the trace preserving conditional expectation $\E \co \M_n \to \M_n$ onto the diagonal subalgebra $\D=\ell^\infty_n$ of $\M_n$, where $\M_n$ is equipped with its canonical trace. The <<upper triangle subalgebra>> $\mathrm{T}_n$ of all upper triangular matrices is a finite subdiagonal algebra of $\M_n$.
Suppose that $1 < p < \infty$. In this context, the Hilbert transform $\cal{H} \co S^p_n \to S^p_n$ is the <<discrete Schur-Hilbert transform>> defined as being the Schur multiplier with symbol $[-\i\, \sign(i-j)]$. The Riesz projection $P \co S^p_n \to S^p_n$ is the triangular projection. 
We consider the representation $\pi \co \M_{n} \to \B(S^p_n)$, $a \mapsto M_a$ of the matrix algebra $\M_{n}$, where $M_a \co S^p_n \to S^p_n$, $ y\mapsto ay$ is the multiplication operator by the matrix $a$. The couple $(S^p_n, F)$ is obviously an odd Banach Fredholm module. The assumptions of Proposition \ref{prop-Hilbert-discrete-1} are obviously satisfied for the algebra $\cal{A}=\M_n$. Since $\K_1(\M_n)=0$, we obtain the following result.

%
%

\begin{prop}
\label{prop-Hilbert-discrete}
Suppose that $1 < p < \infty$. For any invertible matrix $a \in \M_n$, we have
 \begin{equation}
\label{}
\big\la [a], (S^p_n,F) \big\ra_{\K_1(\M_n),\K^1(\M_n,\scr{L}^p_{\nc})}
=0.
\end{equation}
\end{prop}



\begin{remark} \normalfont
It is possible to state in the same spirit more complicated statements for nest algebras.
\end{remark}

\subsection{Banach Fredholm modules on reduced group $\mathrm{C}^*$-algebras}
\label{Section-reduced}

\paragraph{Orderable groups} An order relation $\preceq$ on a group $G$ is left-invariant (resp.~right-invariant) if for all $s, t \in G$ such that $s \preceq t$, one has $rs \preceq rt$ (resp.~$sr \preceq tr$) for all $r \in G$. The relation is bi-invariant if it is simultaneously left-invariant and right-invariant. We will use the term left-ordering (resp.~right-ordering, bi-ordering) for referring to a left-invariant (resp.~right-invariant, right-invariant and left-invariant) \textit{total} order on a group.

Following \cite[p.~2]{ClR16}, we will say that a group $G$ is left-orderable (resp.~right-orderable, bi-orderable) if it admits a total order which is invariant by the left (resp.~by the right, simultaneously by the left and right). We also refer to \cite{DNR16} for more information on these classes of groups. We will write $r \prec s$ when $r \preceq s$ with $r \neq s$. Note that by \cite[Proposition 7.5 p.~277]{KaT08} or \cite[Proposition 1.3 p.~2]{ClR16}, any orderable group $G$ is torsion-free.


It is known \cite[Proposition 1.1.8 p.~11]{DNR16} \cite[Theorem 2.23 p.~28]{ClR16} that a countable group is left-orderable if and only if it is isomorphic to a subgroup of the set of orientation-preserving homeomorphisms of $\R$.

\paragraph{Fourier multipliers} Let $G$ be a discrete group. We denote by $\VN(G)$ its group von Neumann algebra. Suppose that $1 \leq p <\infty$. We say that a complex function $\varphi \co G \to \mathbb{C}$ induces a bounded Fourier multiplier $\M_\varphi \co \L^p(\VN(G)) \to \L^p(\VN(G))$ if the linear map $\mathbb{C}[G] \to \mathbb{C}[G]$, $\lambda_s \mapsto \varphi_s \lambda_s$ extends to a bounded map $M_\varphi$ on the noncommutative $\L^p$-space $\L^p(\VN(G))$, we use the canonical normalized normal finite faithful trace $\tau$ on the algebra $\VN(G)$ for defining the noncommutative $\L^p$-space. Here $\mathbb{C}[G]$ is the group algebra of $G$. We say that $\varphi$ is the symbol of $\M_\varphi$.

\paragraph{Hilbert transforms} Assume that $G$ is a left-orderable discrete group and endowed with a left-order $\preceq$. In this context, we can introduce the function $\sign \co G \to \mathbb{C}$ defined by
$$
\sgn s
\ov{\mathrm{def}}{=} 
\begin{cases}
  1 & \text{ if } e \prec s \\
  0 & \text{ if } s = e \\
 -1 & \text{ if } s \prec e
\end{cases}.
$$
Moreover, we can consider the subalgebra $\cal{D} \ov{\mathrm{def}}{=}\mathbb{C} 1$ of the group von Neumann algebra $\VN(G)$. The linear map $\mathbb{E} \co \VN(G) \to \VN(G)$, $f \mapsto \tau(f)1$ is a trace preserving normal faithful conditional expectation onto the <<diagonal>> subalgebra $\cal{D}$. If $G$ is left-orderable and endowed with a left-order $\preceq$ then  it is essentially\footnote{\thefootnote. The proposed definition in $\H^\infty(\VN(G))$ is slightly incorrect since the <<noncommutative Fourier series>> of elements in $\VN(G)$ does not converges in general.} observed in \cite{GPX22} that
\[
\H^\infty(\VN(G))
\ov{\mathrm{def}}{=} \ovl{\bigg\{ \text{finite sums } \sum_{e \preceq s} a_s \, \lambda_s  :  a_s \in \mathbb{C} \bigg\}}^{\w*}
\]
is a subdiagonal algebra whose Hilbert transform $\cal{H} \co \L^p(\VN(G)) \to \L^p(\VN(G))$ coincides with the bounded Fourier multiplier with symbol $-\i\, \sign$. Actually the boundedness of the last operator was already known for abelian groups \cite{ABG90}.

The map $F \ov{\eqref{def-de-F}}{=} \i\cal{H}+\mathbb{E} \co \L^p(\VN(G)) \to \L^p(\VN(G))$ is the Fourier multiplier defined by the symbol $\varphi \co G \to \mathbb{C}$ defined by
$$
\varphi_s
\ov{\mathrm{def}}{=} 
\begin{cases}
  1 & \text{ if } e \preceq s \\
 -1 & \text{ if } s \prec e
\end{cases}.
$$
Finally, the linear map $P_+ \co \L^p(\cal{M}) \to \H^p(\cal{M})$ satisfies $P_+(\lambda_s)=\lambda_s=0$ for any $s \in G$ satisfying $e \preceq s$ and $P_+(\lambda_s)=0$ otherwise.

The following result says that the assumptions of Proposition \ref{prop-Hilbert-discrete-1} are satisfied under additional  condition.

\begin{prop}
Suppose that $1 < p < \infty$. Let $G$ be a left-orderable discrete group endowed with a left-order such that the set $\{t \in G : \sgn(st) \not= \sgn(t)\}$ is finite for any $s \in G$. The commutator $[\cal{H},M_a]$ is a compact operator acting on the Banach space $\L^p(\VN(G))$ for any element $a$ of the reduced group $\C^*$-algebra $\C^*_{\red}(G)$ of the group $G$. Moreover, $(\L^p(\VN(G)),\pi,F)$ is an odd Banach Fredholm module. 
\end{prop}

\begin{proof}
For any $s,t \in G$, we have
\begin{align*}
\MoveEqLeft        
\left[\cal{H} ,M_{\lambda_s}\right](\lambda_t)
=\cal{H}M_{\lambda_s}(\lambda_t)-M_{\lambda_s}\cal{H}(\lambda_t) 
=\cal{H}(\lambda_{st})+\i \sgn(t) M_{\lambda_s} \lambda_t \\
&=-\i \sgn(st)\lambda_{st}+\i \sgn(t)  \lambda_{st}
=\i\big[- \sgn(st)+\sgn(t)\big]\lambda_{st}.
\end{align*}
If $f=\sum a_s\lambda_s$ is a finite sum belonging to the group algebra $\mathbb{C}[G]$ then we deduce by linearity that the commutator $\left[\cal{H} ,M_{f}\right]$ is a finite-rank operator. We conclude that for any $a \in \C^*_{\red}(G)$ the commutator $\left[\cal{H} ,M_{a}\right]$ is compact, as a limit of finite-rank operators, since the homomorphism $\pi \co \C^*_{\red}(G) \to \B(\L^p(\VN(G)))$, $a \mapsto M_a$ is continuous.
\end{proof}

\begin{remark} \normalfont
Suppose that $p=2$. For any $s \in G$, it is easy to check that the singular values of the operator $\left[\cal{H} ,M_{\lambda_s}\right]$ are given by
$$
s_n(\left[\cal{H} ,M_{\lambda_s}\right])
=|\sgn(t_n)- \sgn(s t_n)|,
$$
where $(t_n)$ is an enumeration of the elements of $G$. So the previous Fredholm module will rarely be finitely summable.
\end{remark}

\subsection{Banach Fredholm module associated to the free Hilbert transform}
\label{sec-free-Hilbert}

Let $\mathbb{F}_\infty$ be the free group with a countable sequence $g_1,g_2,g_3 \ldots$ of generators. We only use reduced words.  For any integer $n \geq 1$, consider the orthogonal projections $L_n^+ \co \L^2(\VN(\mathbb{F}_\infty)) \to \L^2(\VN(\mathbb{F}_\infty))$ and $L_n^- \co \L^2(\VN(\mathbb{F}_\infty)) \to \L^2(\VN(\mathbb{F}_\infty))$ defined by
$$
L_n^+(\lambda_s) 
\ov{\mathrm{def}}{=} 
\begin{cases}
\lambda_s & \text{if }s\text{ starts with the letter }g_n \\
0 & \text{otherwise}
\end{cases}, \ 
L_n^-(\lambda_s) 
\ov{\mathrm{def}}{=} 
\begin{cases}
\lambda_s & \text{if }s\text{ starts with }g_n^{-1} \\
0 & \text{otherwise}
\end{cases}.
$$
The ranges of the various projections $L_n^\pm$ are mutually orthogonal. Let us further consider signs $\epsi_n^+, \epsi_n^- \in \{-1,1\}$ for any integer $n \geq 1$. Following \cite{MeR17}, we define the free Hilbert transform associated with the family $\epsi = (\epsi_n^\pm)_{n\geq 1}$ of signs by
\begin{equation}
\label{free-Hilbert}
\cal{H}_\epsi 
\ov{\mathrm{def}}{=} \sum_{n \geq 1} \big(\epsi_n^+ L_n^+ + \epsi_n^- L_n^-\big).
\end{equation}
For any $s \neq e$, observe that
\begin{equation}
\label{free-Hilbert-action}
\cal{H}_\epsi (\lambda_s)
=\pm \lambda_s 
\quad \text{and} \quad
\cal{H}_\epsi(1)
=0.
\end{equation}
By \cite{MeR17}, the map $\cal{H}_\epsi$ extends to a completely bounded operator on the space $\L^p(\VN(\mathbb{F}_\infty))$ for any $1 < p < \infty$.

For any $a \in \VN(\mathbb{F}_\infty)$, we denote by $R_a \co \L^p(\VN(\mathbb{F}_\infty)) \to \L^p(\VN(\mathbb{F}_\infty))$, $x \mapsto xa$ the right multiplication operator. Then by H\"older's inequality we have $\norm{R_a}_{\L^p(\VN(\mathbb{F}_\infty)) \to \L^p(\VN(\mathbb{F}_\infty))} \leq \norm{a}_{\VN(\mathbb{F}_\infty)}$. So the linear map $\VN(\mathbb{F}_\infty) \to \B(\L^p(\VN(\mathbb{F}_\infty)))$, $a \mapsto R_a$ is a contractive homomorphism and restricts to a contractive homomorphism $\pi$ of $\C^*_\red(\mathbb{F}_\infty)$ into the space $\B(\L^p(\VN(\mathbb{F}_\infty)))$.

\begin{prop}
\label{prop:free-Hilbert-Fredholm-right}
Suppose that $1 < p < \infty$. Consider a family $\epsi=(\epsi_n^\pm)_{n \geq 1}$ of signs. Then $(\L^p(\VN(\mathbb{F}_\infty)),\pi,\cal{H}_\epsi)$ is a Banach Fredholm module over the reduced group $\C^*$-algebra $\C^*_{\mathrm{red}}(\mathbb{F}_\infty)$.
\end{prop}

\begin{proof}
For any $s \not =e$, we have $\cal{H}_\epsi^2(\lambda_s) \ov{\eqref{free-Hilbert-action}}{=} \lambda_s$ and $\cal{H}_\epsi^2(1) \ov{\eqref{free-Hilbert-action}}{=} 0$. Consider the rank-one bounded operator $Q \co \L^p(\VN(\mathbb{F}_\infty))\to \L^p(\VN(\mathbb{F}_\infty))$, $x \mapsto \tau(x)\,1$ (it is the $\L^p$-extension of a conditional expectation).  Note that $Q(1)=1$ and $Q(\lambda_s)=0$ for all $s \ne e$. We deduce that $\cal{H}_\epsi^2-\Id=-Q$ on the space $\mathbb{C}[\mathbb{F}_\infty]$, hence on $\L^p(\VN(\mathbb{F}_\infty))$ by density. In particular, $\cal{H}_\epsi^2-\Id=-Q$ is a rank-one operator.

Fix an element $t \in \mathbb{F}_\infty$ and write its reduced form as
\[
t
=a_1 a_2 \cdots a_m,
\]
where each $a_j$ is one of the letters $g_n$ or $g_n^{-1}$. For any integer $r \in \{0,\ldots,m\}$, set
\[
t_{[1,r]}
\ov{\mathrm{def}}{=} a_1a_2\cdots a_r,
\]
with the convention $t_{[1,0]} \ov{\mathrm{def}}{=} e$. Define the finite subset
\[
E_t
\ov{\mathrm{def}}{=} \big\{ (t_{[1,r]})^{-1} : 0 \leq r \leq m \big\}
\]
of the group $\mathbb{F}_\infty$. Let $s \in \mathbb{F}_\infty$ and consider the reduced product $st$. Cancellation between $s$ and $t$ can only occur at the junction between the last letters of $s$ and the first letters of $t$. Hence the first letter of the reduced word of $st$ coincides with the first letter of $s$, except possibly when $s$ is entirely cancelled by an initial segment of $t$. Equivalently, the first letter of $st$ differs from that of $s$ only when $s\in E_t$. In particular, if $s \notin E_t$ we have
\begin{equation}
\label{inter-577}
\cal{H}_\epsi(\lambda_s)
=\sigma(s)\lambda_s,
\qquad
\cal{H}_\epsi(\lambda_{s t})
=\sigma(s)\lambda_{s t}
\end{equation}
for the same sign $\sigma(s) \in \{-1,1\}$ determined by the first letter of $s$. Note that $e \in E_t$, hence the case $s=e$ is included in the exceptional set and does not require a sign $\sigma(s)$. We deduce that
\begin{align*}
\MoveEqLeft
\left[\cal{H}_\epsi,R_{\lambda_t}\right](\lambda_s)
=\cal{H}_\epsi(R_{\lambda_t}\lambda_{s}) - R_{\lambda_t}\cal{H}_\epsi(\lambda_s)
=\cal{H}_\epsi(\lambda_{st}) - \cal{H}_\epsi(\lambda_s)\lambda_t \\
&\ov{\eqref{inter-577}}{=} \sigma(s)\lambda_{st} - \sigma(s)\lambda_s\lambda_t
=0.   
\end{align*}
Therefore the commutator vanishes on the linear span of the subset $\{\lambda_s : s\notin E_t\}$ and its range is
contained in the finite-dimensional space
$$
\Span\{ \lambda_{st} : s \in E_t\}.
$$
Hence the commutator $[\cal{H}_\epsi,R_{\lambda_t}]$ has finite rank on the space $\mathbb{C}[\mathbb{F}_\infty]$ and extends by continuity to a finite-rank operator on the Banach space $\L^p(\VN(\mathbb{F}_\infty))$.

Let $a \in \C^*_\red(\mathbb{F}_\infty)$ and choose a sequence $(a_k)_{k \geq 1}$ in the group algebra $\mathbb{C}[\mathbb{F}_\infty]$ such that $a_k \to a$ in $\C^*_\red(\mathbb{F}_\infty)$.  We can write each $a_k$ as a \textit{finite} sum $a_k=\sum_{t \in \mathbb{F}_\infty} a_{k,t} \lambda_t$ where $a_{k,t} \in \mathbb{C}$. Consequently, we have
\[
[\cal{H}_\epsi,R_{a_k}]
=
\sum_{t \in \mathbb{F}_\infty} a_{k,t} [\cal{H}_\epsi,R_{\lambda_t}]
\]
is a finite sum of finite-rank operators, hence finite-rank and therefore compact on the Banach space $\L^p(\VN(\mathbb{F}_\infty))$. Since the map $a \mapsto R_a$ is continuous on the space $\C^*_\red(\mathbb{F}_\infty)$, we have $R_{a_k} \to R_a$. So, we obtain
\[
\norm{[\cal{H}_\epsi,R_{a_k}]-[\cal{H}_\epsi,R_a]}
\leq
2\norm{\cal{H}_\epsi}
\norm{R_{a_k}-R_a}
\xra[k \to \infty]{} 0.
\]
We infer that the commutator $[\cal{H}_\epsi,R_a]$ is a norm-limit of compact operators, hence compact.
\end{proof}

By restriction, we can obtain a Banach Fredholm module over each reduced group $\C^*$-algebra $\C^*_{\mathrm{red}}(\mathbb{F}_n)$ for any integer $n \geq 1$. Recall that by \cite[Corollary 3.2 p.~152]{PiV82}, we have $\K_0(\C^*_\red(\mathbb{F}_n))=\Z$ and the generator of $\K_0(\C^*_\red(\mathbb{F}_n))$ is $[1]$. The same reference gives an isomorphism $\K_1(\C^*_\red(\mathbb{F}_n))=\Z^n$ and the generators of the group $\K_1(\C^*_\lambda(\mathbb{F}_n))$ are $[\lambda_{g_1}],\ldots,[\lambda_{g_n}]$. 

\begin{prop}
\label{prop:pairing-free-Hilbert-generators}
Suppose that $1 < p < \infty$. Let $\epsi=(\epsi_j^\pm)_{1\le j\le n}$ be a family of signs. Consider the odd Banach Fredholm module $(\L^p(\VN(\mathbb{F}_n)),\pi,\cal{H}_\epsi)$ over the algebra $\C^*_\red(\mathbb{F}_n)$. Then, for any integer $j \in \{1,\ldots,n\}$, the index pairing of the element $[\lambda_{g_j}]\in \K_1(\C^*_\red(\mathbb{F}_n))$ with the class $[(\L^p(\VN(\mathbb{F}_n)),\pi,\mathcal H_\epsi)]\in \K^1(\C^*_\red(\mathbb{F}_n),\scr B)$ is given by
\begin{equation}
\label{eq:pairing-free-Hilbert-generator}
\bigl\la [\lambda_{g_j}], [(\L^p(\VN(\mathbb{F}_n)),\pi,\cal{H}_\epsi)]\bigr\ra
=
\delta_{\epsi_j^- = 1}
-
\delta_{\epsi_j^+ = 1}.
\end{equation}
In particular, if $\epsi_j^+=1$ and $\epsi_j^-=-1$, then $\la [\lambda_{g_j}], [(\L^p(\VN(\mathbb{F}_n)),\pi,\cal{H}_\epsi)]\ra=-1$.
\end{prop}

\begin{proof}
We use Theorem~\ref{th-pairing-odd} with $u=\lambda_{g_j}$ and $n=1$. Set $F \ov{\mathrm{def}}{=} \mathcal H_\epsi$ and $u_1 \ov{\mathrm{def}}{=} \pi(\lambda_{g_j})=R_{\lambda_{g_j}}$. The Fredholm operator of Theorem~\ref{th-pairing-odd} is
\[
T
=
P u_1 P-(\Id-P),
\qquad
P=\frac{\Id+F}{2}.
\]
Since $F(1)=0$, we have $P(1)=\frac12 1$ and $P$ is not a projection. Now, we replace $P$ by an idempotent $P_0$ differing from $P$ by a finite-rank operator.  Define the bounded operator $P_0\co \L^p(\VN(\mathbb{F}_n))\to \L^p(\VN(\mathbb{F}_n))$ by $P_0 \ov{\mathrm{def}}{=} P-\frac{1}{2}Q$. We have
\[
P_0(1)=0,
\qquad
P_0(\lambda_s)=
\begin{cases}
\lambda_s & \text{if } s\ne e \text{ and } F(\lambda_s)=+\lambda_s,\\
0 & \text{if } s\ne e \text{ and } F(\lambda_s)=-\lambda_s.
\end{cases}
\]
Then $P_0$ is an idempotent and $P-P_0$ has rank one and in particular is compact. Set
\[
T_0
\ov{\mathrm{def}}{=} 
P_0 u_1 P_0-(\Id-P_0).
\]
Since $u_1$ is bounded, we have
\[
T-T_0
=
(P-P_0)u_1P
+
P_0u_1(P-P_0)
+
(P-P_0),
\]
which is a sum of products involving the compact operator $P_0-P$ and bounded operators. Hence $T-T_0$ is compact and therefore
\begin{equation}
\label{eq:index-stable-compact-perturbation}
\Index T
\ov{\eqref{invariance-under-compact-perturbations}}{=}
\Index T_0.
\end{equation}
Let $
X_+ \ov{\mathrm{def}}{=} \Ran P_0$ and $X_- \ov{\mathrm{def}}{=} \Ran(\Id-P_0)$. Then $\L^p(\VN(\mathbb{F}_n)) = X_+\oplus X_-$ topologically and $P_0$ is the projection onto $X_+$ along $X_-$. By construction, $X_+$ is the closed linear span of
\[
\{ \lambda_s : s\ne e,\ \text{the reduced word of } s \text{ starts with } g_k^{\pm 1}
\text{ and } \epsi_k^{\pm}=+1\}.
\]
Consider the operator
\[
S
\ov{\mathrm{def}}{=}
P_0 u_1 P_0
\co
X_+\to X_+.
\]
In the decomposition $\L^p(\VN(\mathbb{F}_n)) = X_+\oplus X_-$, the operator $T_0$ acts as $S$ on $X_+$ and as
$-\Id$ on $X_-$. Hence $T_0$ is Fredholm if and only if $S$ is Fredholm and
\begin{equation}
\label{eq:index-T0-S}
\Index T_0
=\Index S.
\end{equation}
Let $x \in X_+$. Then $Sx=0$ means $P_0(x\lambda_{g_j})=0$. Since $P_0$ is diagonal on the Fourier basis, it suffices to test basis vectors. Let $s\ne e$ with $\lambda_s\in X_+$. Then $S(\lambda_s)=0$ if and only if $\lambda_{s g_j}\in X_-$, namely if the reduced word of $s g_j$ starts with a letter $g_k^{\pm 1}$ with $\epsi_k^\pm=-1$, or if $s g_j=e$.

Now $s g_j=e$ if and only if $s=g_j^{-1}$. In that case $S(\lambda_{g_j^{-1}})=P_0(1)=0$. Moreover, $\lambda_{g_j^{-1}}\in X_+$ if and only if $\epsi_j^- = +1$. If $s\ne g_j^{-1}$, then the first letter of $s g_j$ equals the first letter of $s$, so $\lambda_{s g_j}\in X_+$ whenever $\lambda_s\in X_+$, hence $S(\lambda_s)v\ne 0$. Therefore
\[
\ker S
=
\begin{cases}
\Span \lambda_{g_j^{-1}} & \text{if } \epsi_j^-=+1,\\
\{0\} & \text{if } \epsi_j^-=-1.
\end{cases}
\]
In particular, we have
\begin{equation}
\label{eq:dim-kernel-S}
\dim \ker S
=\delta_{\epsi_j^- = +1}.
\end{equation}
Now, we claim that
\[
X_+/\Ran S
\cong
\begin{cases}
\mathbb{C} \lambda_{g_j} & \text{if } \epsi_j^+=+1,\\
\{0\} & \text{if } \epsi_j^+=-1,
\end{cases}
\]
and hence
\begin{equation}
\label{eq:codim-range-S}
\mathrm{codim}\,\Ran S
=\delta_{\epsi_j^+ = +1}.
\end{equation}
Indeed, fix a basis vector $\lambda_w \in X_+$. If $w \ne g_j$, then $w$ admits a reduced form not equal to the one-letter word $g_j$.
Let $s=w g_j^{-1}$ reduced. Then $s\ne e$ and the reduced word of $s$ starts with the same first letter as $w$. Since $\lambda_w\in X_+$, this first letter has positive sign for $\epsi$, hence $\lambda_s\in X_+$. Moreover, we have
\[
S(\lambda_s)
=P_0(\lambda_{s g_j})
=P_0(\lambda_w)
=\lambda_w,
\]
so $\lambda_w\in \Ran S$. Therefore every basis vector of $X_+$ except possibly $\lambda_{g_j}$ belongs to $\Ran S$.

Finally, consider $\lambda_{g_j}$. If $\epsi_j^+=-1$, then $\lambda_{g_j}\notin X_+$ and there is nothing to prove. If $\epsi_j^+=+1$, then $\lambda_{g_j}\in X_+$. Suppose by contradiction that $\lambda_{g_j}\in \Ran S$. Then there exists $x\in X_+$ with $Sx=\lambda_{g_j}$, meaning $P_0(x\lambda_{g_j})=\lambda_{g_j}$. Writing $x=\sum_s c_s\lambda_s$ as a finite sum in $\mathbb{C}[\mathbb{F}_n]$ first, we see that the coefficient of $\lambda_{g_j}$ in $x\lambda_{g_j}$ is exactly the coefficient of $\lambda_e$ in $x$. But $x\in X_+$ implies $P_0(x)=x$ and $P_0(1)=0$, hence the coefficient of $\lambda_e$ in $x$ is $0$. Therefore the coefficient of $\lambda_{g_j}$ in $x\lambda_{g_j}$ is $0$, a contradiction. By density, the same conclusion holds for general $x\in X_+$. Thus $\lambda_{g_j}\notin \Ran S$ when $\epsi_j^+=+1$, and \eqref{eq:codim-range-S} follows.

We conclude that
\[
\Index T
 \ov{\eqref{eq:index-stable-compact-perturbation} \eqref{eq:index-T0-S}}{=} \Index S
\ov{\eqref{Fredholm-Index}}{=}
\dim\ker S-\mathrm{codim}\,\Ran S
\ov{\eqref{eq:dim-kernel-S} \eqref{eq:codim-range-S}}{=} \delta_{\epsi_j^- = 1}
-
\delta_{\epsi_j^+ = 1}.
\]
This equals the index pairing given by \eqref{pairing-odd-1}, and \eqref{eq:pairing-free-Hilbert-generator} follows.
\end{proof}

\begin{remark} \normalfont
This result and its proof show that it provides a generalization of Gohberg-Krein index theorem, described in \eqref{Gohberg-Krein}.
\end{remark}

\section{Semigroups, Fredholm modules and vectorial Riesz transforms}
\label{sec-triples}

\paragraph{Hilbert bimodules} 
We start by reviewing the concept of Hilbert bimodule which is crucial for defining the derivations that allow the introduction of the Hodge-Dirac operators considered in this paper. Let $\cal{M}$ be a von Neumann algebra. A Hilbert $\cal{M}$-bimodule is a Hilbert space $\cal{H}$ together with a $*$-representation $\Phi \co \cal{M} \to \B(\cal{H})$ and a $*$-anti-representation $\Psi \co \cal{M} \to \B(\cal{H})$ such that $\Phi(x)\Psi(y)=\Psi(y)\Phi(x)$ for any $x,y \in \cal{M}$. For all $x, y \in \cal{M}$ and any $\xi \in \cal{H}$, we let $x\xi y \ov{\mathrm{def}}{=} \Phi(x)\Psi(y)\xi$.  We say that the bimodule is normal if $\Phi$ and $\Psi$ are normal, i.e.~weak* continuous. The bimodule is said to be symmetric if there exists an anti-linear involution $\J \co \cal{H} \to \cal{H}$ such that $\J (x\xi y) = y^*\J (\xi)x^*$ for any $x, y \in \cal{M}$ and any $\xi \in \cal{H}$.


\paragraph{$\mathrm{W}^*$-derivations} Now, we introduce a notion of derivation that can be viewed as an abstract version of a gradient. If $(\cal{H},\Phi,\Psi)$ is a Hilbert $\cal{M}$-bimodule, then following \cite[p.~267]{Wea96} we define a $\mathrm{W}^*$-derivation to be a weak* closed densely defined unbounded operator $\partial \co \dom \partial \subset \cal{M} \to \cal{H}$ such that the domain $\dom \partial$ is a weak* dense unital $*$-subalgebra of $\cal{M}$ and 
\begin{equation}
\label{Leibniz}
\partial(fg) 
= f\partial(g) + \partial(f)g, \quad f, g \in \dom \partial.
\end{equation}
We say that a $\mathrm{W}^*$-derivation is symmetric if the bimodule $(\cal{H},\Phi,\Psi)$ is symmetric and if we have $\J(\partial(f)) = \partial(f^*)$ for any $x \in \dom \partial$. In the sequel, we let $\cal{B} \ov{\mathrm{def}}{=} \dom \partial$.

\paragraph{The triple $(\L^\infty(\cal{M}),\L^2(\cal{M}) \oplus_2 \cal{H},D)$}
From a derivation, we will now explain how to introduce a triple in the spirit of noncommutative geometry. Let $\partial \co \dom \partial \subset \cal{M} \to \cal{H}$ be a $\mathrm{W}^*$-derivation where the von Neumann algebra $\cal{M}$ is equipped with a normal faithful finite trace $\tau$. Suppose that the operator $\partial \co \dom \partial \subset \L^2(\cal{M}) \to \cal{H}$ is closable. We denote again its closure by $\partial$. Note that the subspace $\mathcal{B}$ is a core of $\partial$. Recall that it is folklore and well-known that a weak* dense subalgebra of $\cal{M}$ is dense in the space $\L^2(\cal{M})$. As the operator $\partial$ is densely defined and closed, by \cite[Theorem 5.29 p.~168]{Kat76} the adjoint operator $\partial^*$ is densely defined and closed on $\L^2(\cal{M})$ and $\partial^{**}=\partial$. 

Following essentially \cite{HiT13b} and \cite{Cip16}, we introduce the unbounded closed operator $D$ on the Hilbert space $\L^2(\cal{M}) \oplus_2 \cal{H}$ introduced in \eqref{Hodge-Dirac-I} and defined by
\begin{equation}
\label{Def-D-psi}
D(f,g)
\ov{\mathrm{def}}{=}
\big(\partial^*(g),
\partial(f)\big), \quad f \in \dom \partial,\  g \in \dom \partial^*.
\end{equation}
We call it the Hodge-Dirac operator associated to $\partial$ and it can be written as in \eqref{Hodge-Dirac-I}. It is not difficult to check that this operator is selfadjoint. If $f \in \L^\infty(\cal{M})$, we define the bounded operator $\pi_f \co \L^2(\cal{M}) \oplus_2 \cal{H} \to \L^2(\cal{M}) \oplus_2 \cal{H}$ by
\begin{equation}
\label{Def-pi-a}
\pi_f
\ov{\mathrm{def}}{=} \begin{bmatrix}
    \M_f & 0  \\
    0 & \Phi_f  \\
\end{bmatrix}, \quad f \in \L^\infty(\cal{M})
\end{equation}
where the linear map $\M_f \co \L^2(\cal{M}) \to \L^2(\cal{M})$, $g \mapsto fg$ is the multiplication operator by $f$ and where $\Phi_f \co \cal{H} \to \cal{H}$, $h \mapsto fh$ is the left bimodule action. 

Note that completely Dirichlet forms give rise to $\mathrm{W}^*$-derivations, see \cite{CiS03}, \cite{Cip97}, \cite{Cip08}, \cite{Cip16} and \cite{Wir22} and references therein. More precisely, if $\cal{E}$ is a completely Dirichlet form on a noncommutative $\L^2$-space $\L^2(\cal{M})$ with associated semigroup $(T_t)_{t \geq 0}$ and associated operator $A_2$, there exist a symmetric Hilbert $\cal{M}$-bimodule $(\cal{H},\Phi,\Psi,\J)$ and a $\mathrm{W}^*$-derivation $\partial \co \dom \partial \subset \cal{M} \to \cal{H}$ such that $\partial \co \dom \partial \subset \L^2(\cal{M}) \to \cal{H}$ is closable with closure also denoted by $\partial$ and satisfying $A_2=\partial^*\partial$. In other words, this means that the opposite of the infinitesimal generator $-A_2$ of a symmetric sub-Markovian semigroup of operators can always be represented as the composition of a <<divergence>> $\partial^*$ with a <<gradient>> $\partial$. Note that the symmetric Hilbert $\cal{M}$-bimodule $(\cal{H},\Phi,\Psi,\J)$ is not unique since one can always artificially expand $\cal{H}$. However, if we add an additional condition, namely that <<the bimodule is generated by the derivation $\partial$>>, we obtain uniqueness, see \cite[Theorem 6.9]{Wir22} for a precise statement. In this case, the bimodule (and its associated derivation) is called the tangent bimodule associated with the completely Dirichlet form (or the associated semigroup).


This fundamental result allows anyone to introduce a triple $(\L^\infty(\Omega),\L^2(\Omega) \oplus_2 \cal{H},D)$ associated to the semigroup in the spirit of noncommutative geometry. Here $D$ is the unbounded selfadjoint operator acting on a dense subspace of the Hilbert space $\L^2(\Omega) \oplus_2 \cal{H}$ defined by 
\begin{equation}
\label{Hodge-Dirac-I}
D
\ov{\mathrm{def}}{=}
\begin{bmatrix} 
0 & \partial^* \\ 
\partial & 0 
\end{bmatrix}.
\end{equation}
The Hodge-Dirac operator $D$ of \eqref{Hodge-Dirac-I} is related to the operator $A_2$ by
\begin{equation}
\label{carre-de-D}
D^2
\ov{\eqref{Hodge-Dirac-I}}{=}
\begin{bmatrix} 
0 & \partial^* \\ 
\partial & 0 
\end{bmatrix}^2
=\begin{bmatrix} 
\partial^*\partial & 0 \\ 
0 & \partial \partial^*
\end{bmatrix}
=\begin{bmatrix} 
A_2 & 0 \\ 
0 & \partial \partial^*
\end{bmatrix}. 
\end{equation}
Now, suppose that $1 < p < \infty$. Sometimes, the map $\partial_2$ induces a closable unbounded operator $\partial \co \dom \partial \subset \L^p(\cal{M}) \to \cal{X}_p$ for some Banach space $\cal{X}_p$. Denoting by $\partial_p$ its closure, we can consider the $\L^p$-realization of the previous operator
\begin{equation}
\label{Hodge-Dirac-I}
D_p
\ov{\mathrm{def}}{=}
\begin{bmatrix} 
0 & (\partial_{p^*})^* \\ 
\partial_p & 0 
\end{bmatrix}
\end{equation}
as acting on a dense subspace of the Banach space $\L^p(\cal{M}) \oplus_p \cal{X}_p$. If the operator $D_p$ is bisectorial and admits a bounded $\H^\infty(\Sigma_\theta^\bi)$ functional calculus for some $\theta \in (0,\frac{\pi}{2})$, we have
\begin{align*}
\MoveEqLeft
\label{}
\sgn D_p
=D_p|D_p|^{-1}
=D_p(D_p^2)^{-\frac{1}{2}}
=\begin{bmatrix} 
0 & (\partial_{p^*})^* \\ 
\partial_p & 0 
\end{bmatrix}
\left(\begin{bmatrix} 
A_p & 0 \\ 
0 & \partial_p (\partial_{p^*})^*
\end{bmatrix}\right)^{-\frac{1}{2}} 
=\begin{bmatrix} 
0 & * \\ 
\partial_p A_p^{-\frac{1}{2}} & 0 
\end{bmatrix}.
\end{align*}

In particular, the operator $\partial_p A_p^{-\frac{1}{2}} \co \L^p(\cal{M}) \to \cal{X}_p$ is bounded.

\begin{defi}
We say that $\cal{R} \ov{\mathrm{def}}{=} \partial_p A_p^{-\frac{1}{2}}$ is the (abstract) vectorial Riesz transform associated with the symmetric sub-Markovian semigroup $(T_t)_{t \geq 0}$.
\end{defi}


From the perspective of the action of the algebra or the kernel of the Dirac operator, the Banach space $\cal{X}_p$ is sometimes too large and needs to be replaced by a smaller space. In this case, we consider the closed subspace
\begin{equation}
\label{Def-Z-q-p}
\Omega_{p}
\ov{\mathrm{def}}{=} \ovl{\mathrm{span} \big\{ g\partial_{p}(f):  f \in \dom \partial_{p}, g \in \L^\infty(\cal{M}) }\big\}
\end{equation}
of the Banach space $\cal{X}_p$. And we replace the Banach space $\cal{X}_p$ by the subspace $\Omega_{p}$. The advantage of this space is the following result.

\begin{prop}
\label{prop-bimodule}
The Banach space $\Omega_{p}$ is a $\L^\infty(\cal{M})$-bimodule. 
\end{prop}

\begin{proof}
For any $f \in \dom \partial$ and any $g,h \in \L^\infty(\cal{M})$, note that 
$$
g(h\partial(f))
=gh\partial(f).
$$ 
Thus by linearity and density, $\Omega_{p}$ is a right $\L^\infty(\cal{M})$-module. Moreover, for any $g \in \L^\infty(\cal{M})$ and any $f,h \in \dom \partial$, we have
\begin{align*}
\MoveEqLeft
g\partial(f)h
\ov{\eqref{Leibniz}}{=} g\big[\partial(fh)-f\partial(h)\big]
=g\partial(hf)-gf\partial(h).            
\end{align*}
Thus $g\partial(f)h$ belongs to $\Omega_{p}$. Since $\dom \partial$ is a core for $\partial_{p}$, the same holds for $f \in \dom \partial_{p}$. If $h \in \L^\infty(\cal{M})$ is a general element, we approximate it in the strong operator topology by a bounded net in $\dom \partial$ and obtain again that the same holds for $h \in \L^\infty(\cal{M})$. By linearity and density, we deduce that $\Omega_{p}$ is a left $\L^\infty(\cal{M})$-module, so finally a $\L^\infty(\cal{M})$-bimodule.
\end{proof}

\begin{remark} \normalfont
It is interesting to observe that if $\L^\infty(\cal{M})$ is commutative, i.e.~the algebra $\L^\infty(\Omega)$ of a (localizable) measure space $\Omega$, the previous bimodule is not necessarily commutative, i.e.~we does not have $g k h=h k g$ for any $g,h \in \L^\infty(\Omega)$ and $k \in \Omega_{p}$. In our previous paper, we referred to this phenomenon as <<hidden noncommutative geometry>>.
\end{remark}

If $f \in \L^\infty(\cal{M})$, we define the bounded operator $\pi_f \co \L^p(\cal{M}) \oplus_p \Omega_p \to \L^p(\cal{M}) \oplus_p \Omega_{p}$ by
\begin{equation}
\label{Def-pi-a}
\pi_f
\ov{\mathrm{def}}{=} \begin{bmatrix}
    \M_f & 0  \\
    0 & \Phi_f  \\
\end{bmatrix}, \quad f \in \L^\infty(\cal{M})
\end{equation}
where the linear map $\M_f \co \L^p(\cal{M}) \to \L^p(\cal{M})$, $g \mapsto fg$ is the multiplication operator by $f$ and where $\Phi_f \co \Omega_{p} \to \Omega_{p}$, $h \mapsto fh$ is the left bimodule action provided by Proposition \ref{prop-bimodule}.

If $\bigg(\L^p(\cal{M}) \oplus_p \Omega_{p},\sgn D_p,\begin{bmatrix}
  -\Id   & 0  \\
    0 &  \Id \\
\end{bmatrix}\bigg)$ is a (well-defined) even Banach Fredholm module over a subalgebra $\cal{A}$ of $\L^\infty(\cal{M})$, the pairing with the K-homology group $\K_0(\cal{A})$ is given by the following formula. If $e \in \M_n(\cal{A})$ is an idempotent then
\begin{equation}
\label{pairing-even-abstract}
\big\la [e], (\L^p(\cal{M}) \oplus_p \Omega_{p},\sign D) \big\ra_{\K_0(\cal{A}),\K^0(\cal{A},\mathscr{L}^p_{\nc})}         
\ov{\eqref{pairing-even-1}}{=} \Index  e_n (\Id \ot \cal{R}) e_n.
\end{equation}

%
%

%

\begin{remark} \normalfont
Note that with Proposition \ref{prop-even-to-odd-Fredholm}, we can obtain sometimes an odd Banach Fredholm module with $\sgn D_p$, so a pairing with the group $\K_1(\cal{A})$.
\end{remark}

\section{Future directions}
\label{sec-future}

In a future update of this preliminary preprint, we will elaborate on the ideas presented in the final section. We will also expand several sections and clarify the analogies with \cite{ScS22}. Finally, we will undertake a thorough study of the groups $\K^{0}(\cal{A},\scr{B})$ and $\K^{1}(\cal{A},\scr{B})$ (and some generalizations) and their connections with \cite{Laf02} in a forthcoming paper.




\paragraph{Declaration of interest} None.

\paragraph{Competing interests} The author declares that he has no competing interests.

\paragraph{Data availability} No data sets were generated during this study.

\paragraph{Acknowledgment} The author would like to thank Hermann Schulz-Baldes for bringing his recent book \cite{ScS22} to our attention on his \textit{own} initiative and Anton Savin and Markus Haase and Edward McDonald for short discussions.



{\footnotesize

\vspace{0.2cm}

\noindent C\'edric Arhancet\\ 
\noindent 6 rue Didier Daurat, 81000 Albi, France\\
URL: \href{http://sites.google.com/site/cedricarhancet}{https://sites.google.com/site/cedricarhancet}\\
cedric.arhancet@protonmail.com\\
ORCID: 0000-0002-5179-6972 
}


\small
\begin{thebibliography}{79}


\bibitem[ADV25]{ADV25}
N. Abatangelo, S. Dipierro and E. Valdinoci.
\newblock A Gentle Invitation to the Fractional World.
\newblock Springer Nature Switzerland, 2025.


\bibitem[AbA02]{AbA02}
Y. A. Abramovich and C. D. Aliprantis.
\newblock An invitation to operator theory.
\newblock Graduate Studies in Mathematics, 50. American Mathematical Society, Providence, RI, 2002.


%
%
%
%

\bibitem[ADM96]{ADM96}
D. Albrecht, X. Duong and A. McIntosh.
\newblock Operator theory and harmonic analysis.
\newblock Instructional Workshop on Analysis and Geometry, Part III (Canberra, 1995), 77--136. Proc. Centre Math. Appl. Austral. Nat. Univ., 34. Australian National University, Centre for Mathematics and its Applications, Canberra, 1996.


%
%
%
%
%
%
%
%
%
%
%
%
%





%
%
%
%
%
%
%
%


\bibitem[AmG16]{AmG16}
B. Ammann and N. Gro\ss e.
\newblock $L^p$-spectrum of the Dirac operator on products with hyperbolic spaces.
\newblock Calc. Var. Partial Differential Equations 55 (2016), no. 5, Art. 127.





%
%
%
%
%
%
%
%



%
%
%
%
%
%
%
%
%

\bibitem[AGGG${}^+$86]{AGGG86}
W. Arendt, A. Grabosch, G. Greiner, U. Groh, H. P. Lotz, U. Moustakas, R. Nagel, F. Neubrander and U. Schlotterbeck.
\newblock One-parameter semigroups of positive operators.
\newblock Lecture Notes in Math., 1184, Springer-Verlag, Berlin, 1986.

%
%
%
%
%





%

\bibitem[ArK22]{ArK22}
C. Arhancet and C. Kriegler.
\newblock Riesz transforms, Hodge-Dirac operators and functional calculus for multipliers.
\newblock Lecture Notes in Mathematics, 2304. Springer, Cham, 2022.

\bibitem[Arh24a]{Arh24a}
C. Arhancet.
\newblock Sobolev algebras on Lie groups and noncommutative geometry.
\newblock J. Noncommut. Geom. 18 (2024), no. 2, 451--500.

\bibitem[Arh24b]{Arh24b}
C. Arhancet.
\newblock Spectral triples, Coulhon-Varopoulos dimension and heat kernel estimates.
\newblock Adv. Math. 451 (2024), Paper No. 109794.

\bibitem[Arh24c]{Arh24c}
C. Arhancet.
\newblock Curvature, Hodge-Dirac operators and Riesz transforms.
\newblock Preprint, arXiv:2401.04203.

\bibitem[Arh24d]{Arh24d}
C. Arhancet.
\newblock Riesz transforms and approximation numbers on noncommutative tori.
\newblock Preprint.

\bibitem[Arh26]{Arh26}
C. Arhancet.
\newblock The $\L^p$-index of the Hodge–Dirac operator on compact Riemannian manifolds.
\newblock Preprint.

%
%
%
%


\bibitem[ABG90]{ABG90}
N. Asmar, E. Berkson and T. A. Gillespie.
\newblock Representations of groups with ordered duals and generalized analyticity.
\newblock J. Funct. Anal. 90 (1990), no. 1, 206--235.
%
%
%
%

\bibitem[Arv67]{Arv67}
W. B. Arveson.
\newblock Analyticity in operator algebras.
\newblock Amer. J. Math. 89 (1967), 578--642.

\bibitem[AIM09]{AIM09}
K. Astala, T. Iwaniec and G. Martin.
\newblock Elliptic partial differential equations and quasiconformal mappings in the plane.
\newblock Princeton Math. Ser., 48. Princeton University Press, Princeton, NJ, 2009.

\bibitem[Ati70]{Ati70}
M. F. Atiyah.
\newblock Global theory of elliptic operators.
\newblock Proc. Internat. Conf. on Functional Analysis and Related Topics (Tokyo, 1969), pp. 21--30. University of Tokyo Press, Tokyo, 1970.





%

\bibitem[AMN97]{AMN97}
P. Auscher, A. McIntosh, A. Nahmod.
\newblock The square root problem of Kato in one dimension, and first order elliptic systems.
\newblock Indiana Univ. Math. J. 46 (1997), no. 3, 659--695.

\bibitem[AHLMT02]{AHLMT02}
P. Auscher, S. Hofmann, M. Lacey, A. McIntosh and P. Tchamitchian.
\newblock The solution of the Kato square root problem for second order elliptic operators on $\R^n$.
\newblock Ann. of Math. (2) 156 (2002), no. 2, 633--654.
%
%
%
%
%

\bibitem[AKM06]{AKM06}
A. Axelsson, S. Keith and A. McIntosh.
\newblock Quadratic estimates and functional calculi of perturbed Dirac operators.
\newblock Invent. Math. 163 (2006), no. 3, 455--497.




\bibitem[BaJ83]{BaJ83}
S. Baaj and P. Julg.
\newblock Th\'eorie bivariante de Kasparov et op\'erateurs non born\'es dans les $C^*$-modules hilbertiens.
\newblock C. R. Acad. Sci. Paris S\'er. I Math. 296 (1983), no. 21, 875--878.


%
%
%
%
%








%
%
 %
%
%
%
%




\bibitem[BaG82b]{BaG82b}
P. Baum and R. G. Douglas.
\newblock Toeplitz Operators and Poincar\'e Duality.
\newblock Gohberg, I. (eds) Toeplitz Centennial. Operator Theory: Advances and Applications, vol 4. Birkh\"auser, Basel, 1982.

\bibitem[BaS11]{BaS11}
P. F. Baum and R. J. S\'anchez-Garcia.
\newblock K-theory for group $C^*$-algebras.
\newblock Topics in algebraic and topological K-theory, 1--43. Lecture Notes in Math., 2008. Springer-Verlag, Berlin, 2011.


%
%
%
%
%




%
%
%
%

%
        %
%
%
%
%
%
%
%
%
%
%
%
%
%
%
%
%
%
%
%
%
%
%
%
%
%
%
%
%
%
%
%
%

\bibitem[Bl98]{Bl98}
B. Blackadar.
\newblock $K$-theory for operator algebras.
\newblock Math. Sci. Res. Inst. Publ., 5. Cambridge University Press, Cambridge, 1998.

%
%
%
%
%
%
%
%
%
%
%
%
%
%
%
%
%
 %
%
%
%
%
%
%
%
%
%
%
%
%
%
%
%
%
%
%
%
%

\bibitem[Bar07]{Bar07}
C. B\"ar.
\newblock Conformal structures in noncommutative geometry.
\newblock J. Noncommut. Geom. 1 (2007), no. 3, 385--395.


\bibitem[Bek15]{Bek15}
T. N. Bekjan.
\newblock Noncommutative Hardy space associated with semi-finite subdiagonal algebras.
\newblock J. Math. Anal. Appl. 429 (2015), no. 2, 1347--1369.

\bibitem[BEO13]{BeO13}
A. B\'enyi and T. Oh.
\newblock The Sobolev inequality on the torus revisited.
\newblock Publ. Math. Debrecen 83 (2013), no. 3, 359--374.



\bibitem[BeS22]{BeS22}
A. Belyaev and A. A. Shkalikov.
\newblock Multipliers in the scale of periodic Bessel potential spaces with smoothness indices of different signs.
\newblock  Preprint, arXiv:2212.05681.


\bibitem[BlL07]{BlL07}
D. Blecher and L. E. Labuschagne.
\newblock Von Neumann algebraic $H^p$ theory.
\newblock Function spaces, 89--114. Contemp. Math., 435. American Mathematical Society, Providence, RI, 2007.

\bibitem[BlB13]{BlB13}
D. D. Bleecker and B. Boo{\ss}-Bavnbek.
\newblock Index theory—with applications to mathematics and physics.
\newblock International Press, Somerville, MA, 2013.

\bibitem[BoS06]{BoS06}
A. B\"ottcher and B. Silbermann.
\newblock Analysis of Toeplitz operators.
\newblock Springer-Verlag, Berlin, 2006.



%



\bibitem[BDF77]{BDF77}
L. G. Brown, R. G. Douglas and P. A. Fillmore
\newblock Extensions of $C^*$-algebras and K-homology.
\newblock Ann. of Math. (2) 105 (1977), no. 2, 265--324.

\bibitem[Bre11]{Bre11}
H. Brezis.
\newblock Functional analysis, Sobolev spaces and partial differential equations.
\newblock Universitext. Springer, New York, 2011.


%
%
%
%
%
%
%
%
%
%
%
%
%
%
%
%
%
%
%
%
%
%
%
%
%
%
%
%
%
%
%
%
%
%
%
%
%
%
%

\bibitem[Cal77]{Cal77}
A.-P. Calder\'on.
\newblock Cauchy integrals on Lipschitz curves and related operators.
\newblock Proc. Nat. Acad. Sci. U.S.A. 74 (1977), no. 4, 1324--1327.


\bibitem[Cam17]{Cam17}
M. C. C\^amara.
\newblock Toeplitz operators and Wiener-Hopf factorisation: an introduction.
\newblock Concr. Oper. 4 (2017), no. 1, 130--145.
 

\bibitem[CDK23]{CDK23}
A. Carbonaro, O. Dragicevic and V. Kovac.
\newblock Sharp $L^p$ estimates of powers of the complex Riesz transform.
\newblock Math. Ann. 386 (2023), no. 1-2, 1081--1125.


%
%
%
%
%
  %
%
%
%


%
%
%
%



%
%
%
%


%
%
%
%
%

\bibitem[CaP98]{CaP98}
A. Carey and J. Phillips.
\newblock Unbounded Fredholm modules and spectral flow.
\newblock Canad. J. Math. 50 (1998), no. 4, 673--718.

\bibitem[CPR11]{CPR11}
A. L. Carey, J. Phillips and A. Rennie.
\newblock Spectral triples: examples and index theory.
\newblock ESI Lect. Math. Phys. European Mathematical Society (EMS), Z\"urich, 2011, 175--265.


\bibitem[CGRS14]{CGRS14}
A. L. Carey, V. Gayral, A. Rennie and F. A. Sukochev.
\newblock Index theory for locally compact noncommutative geometries.
\newblock Mem. Amer. Math. Soc. 231 (2014), no. 1085.





\bibitem[Cas22]{Cas22}
R. E. Castillo.
\newblock Fourier meets Hilbert and Riesz—an introduction to the corresponding transforms.
\newblock De Gruyter Stud. Math., 87. De Gruyter, Berlin, 2022.





%
%
%
%
%

\bibitem[CXY13]{CXY13}
Z. Chen, Q. Xu and Z. Yin.
\newblock Harmonic analysis on quantum tori.
\newblock Comm. Math. Phys. 322 (2013), no. 3, 755--805.
 
%
%
%
%
%
%
%








%
%
%

\bibitem[Cip08]{Cip08}
F. Cipriani.
\newblock Dirichlet forms on noncommutative spaces. 
\newblock Quantum potential theory, 161--276, Lecture Notes in Math., 1954, Springer, Berlin, 2008. 

\bibitem[Cip16]{Cip16}
F. Cipriani.
\newblock Noncommutative potential theory: a survey.
\newblock J. Geom. Phys. 105 (2016), 25--59.

\bibitem[CiS03]{CiS03}
F. Cipriani and J.-L. Sauvageot.
\newblock Derivations as square roots of Dirichlet forms.
\newblock  J. Funct. Anal. 201 (2003), no. 1, 78--120.

\bibitem[CGIS14]{CGIS14}
F. Cipriani, D. Guido, T. Isola and J.-L. Sauvageot.
\newblock Spectral triples for the Sierpinski gasket.
\newblock J. Funct. Anal. 266 (2014), no. 8, 4809--4869.

%
\bibitem[Cip97]{Cip97}
F. Cipriani.
\newblock Dirichlet forms and Markovian semigroups on standard forms of von Neumann algebras. 
\newblock J. Funct. Anal. 147 (1997), no. 2, 259--300. 

\bibitem[ClR16]{ClR16}
A. Clay and D. Rolfsen.
\newblock Ordered groups and topology.
\newblock Grad. Stud. Math., 176. American Mathematical Society, Providence, RI, 2016.

\bibitem[ClC13]{ClC13}
A. Clop and V. Cruz.
\newblock Weighted estimates for Beltrami equations.
\newblock Ann. Acad. Sci. Fenn. Math. 38 (2013), no. 1, 91--113.

%

\bibitem[CRW76]{CRW76}
R. R. Coifman, R. Rochberg and G. Weiss.
\newblock Factorization theorems for Hardy spaces in several variables.
\newblock Ann. of Math. (2) 103 (1976), no. 3, 611--635.


\bibitem[CMM82]{CMM82}
R. R. Coifman, A. McIntosh and Y. Meyer.
\newblock L'int\'egrale de Cauchy d\'efinit un op\'erateur born\'e sur $L^2$ pour les courbes lipschitziennes.
\newblock Ann. of Math. (2) 116 (1982), no. 2, 361–387.

%


%
%
%
%
%
%
%
%
%
%
%
%

\bibitem[Con94]{Con94}
A. Connes.
\newblock Noncommutative geometry.
\newblock Academic Press, Inc., San Diego, CA, 1994.




\bibitem[CoM95]{CoM95}
A. Connes and H. Moscovici.
\newblock The local index formula in noncommutative geometry.
\newblock Geom. Funct. Anal. 5 (1995), no. 2, 174--243.



%

\bibitem[CoM08]{CoM08}
A. Connes and M. Marcolli.
\newblock A walk in the noncommutative garden.
\newblock An invitation to noncommutative geometry, 1--128, World Sci. Publ., Hackensack, NJ, 2008.




%
%
%
%
%
%
%


%
%
\bibitem[CDMY96]{CDMY96}
M. Cowling, I. Doust, A. McIntosh, A. Yagi.
\newblock Banach space operators with a bounded $H^\infty$ functional calculus.
\newblock  J. Austral. Math. Soc. Ser. A 60 (1996), no. 1, 51--89. 
%
%
%
%
%
%
%
%
%
%
%
%
%

\bibitem[CMR07]{CMR07}
J. Cuntz, R. Meyer and J. M. Rosenberg.
\newblock Topological and bivariant K-theory.
\newblock Oberwolfach Semin., 36. Birkh\"auser Verlag, Basel, 2007.

%
%
%
%




%
%
%
%
%
%
%
%
%
%
%
%
%
%
%
%
%
%
%
%



\bibitem[Dav89]{Dav89}
E. B. Davies.
\newblock Heat kernels and spectral theory. 
\newblock Cambridge Tracts in Mathematics, 92. Cambridge University Press, Cambridge, 1989.

%
%
%
%
%
%
%
%
%
%
%
%
%
%
%
%
 %
%
  %
%
%
%
%
%
%
%
%
%

%
%
%
%
%
%
%
%
%
%
%
%
%
%
%
%
%
%
%
%
%
%
%
%
%
%
%



\bibitem[Deh94]{Deh94}
P. Dehornoy.
\newblock Braid groups and left distributive operations.
\newblock Trans. Amer. Math. Soc. 345 (1994), no. 1, 115--150.


\bibitem[DNR16]{DNR16}
B. Deroin, A. Navas, C. Rivas.
\newblock Groups, Orders, and Dynamics.
\newblock Preprint, arXiv:1408.5805.



\bibitem[Dieu88]{Dieu88}
\newblock J. Dieudonn\'e.
\newblock Treatise on analysis. Vol. VII. Translated from the French by Laura Fainsilber.
\newblock Pure Appl. Math., 10-VII. Academic Press, Inc., Boston, MA, 1988.

\bibitem[Dieu88]{Dieu88}
\newblock J. Dieudonn\'e.
\newblock Treatise on analysis. Vol. VII. Translated from the French by Laura Fainsilber.
\newblock Pure Appl. Math., 10-VII. Academic Press, Inc., Boston, MA, 1988.

\bibitem[Dou86]{Dou86}
R. G. Douglas.
\newblock Invariant theory for elliptic operators.
\newblock Proc. Roy. Irish Acad. Sect. A 86 (1986), no. 2, 161--174.

\bibitem[Dra21]{Dra21}
O. Dragicevic.
\newblock Analysis of the Ahlfors-Beurling operator (lecture notes for the summer school at the University of Seville, 2013).
\newblock Preprint, arXiv:2109.04555.



%
%
%
%
%
%
%
%






 

\bibitem[EcI18]{EcI18}
M. Eckstein and B. Iochum.
\newblock Spectral action in noncommutative geometry.
\newblock SpringerBriefs in Mathematical Physics, 27. Springer, Cham, 2018.




\bibitem[EdT86]{EdT86}
D. E. Edmunds and H.-O. Tylli.
\newblock On the entropy numbers of an operator and its adjoint.
\newblock Math. Nachr. 126 (1986), 231--239.

\bibitem[EdE23]{EdE23}
D. E. Edmunds and W. D. Evans.
\newblock Fractional Sobolev spaces and inequalities.
\newblock Cambridge Tracts in Math., 230. Cambridge University Press, Cambridge, 2023.



\bibitem[EdL07]{EdL07}
D. E. Edmunds and J. Lang.
\newblock Operators of Hardy type.
\newblock J. Comput. Appl. Math. 208 (2007), no. 1, 20--28.


%
%
%
%

\bibitem[Ege15]{Ege15}
M. Egert.
\newblock On Kato's conjecture and mixed boundary conditions.
\newblock PhD, 2015.


%
%
%
%
%
%
%
%
%
%

\bibitem[EnN00]{EnN00}
K.-J. Engel and R. Nagel.
\newblock One-parameter semigroups for linear evolution equations.
\newblock Graduate Texts in Mathematics, 194. Springer-Verlag, New York, 2000. 



\bibitem[Eme24]{Eme24}
H. Emerson.
\newblock An Introduction to $\C^*$-Algebras and Noncommutative Geometry.
\newblock Birkh\"auser Cham, 2024.

\bibitem[EmN18]{EmN18}
H. Emerson and B. Nica.
\newblock K-homological finiteness and hyperbolic group.
\newblock J. Reine Angew. Math. 745 (2018), 189--229.

\bibitem[Eva10]{Eva10}
L . C. Evans.
\newblock Partial differential equations. Second edition.
\newblock Grad. Stud. Math., 19. American Mathematical Society, Providence, RI, 2010.


%
%
%
%
%
%
%
%
%
%
%
%
%
%
%
%
%
%
%
%



%
%
%
%
%
%
%
%
%
%
%
%
%
%
%
%
%
%
%
%
%
%
%
%
%
%
%
%
%
%
%
%
%
%
%
%
%
%
%
%
 %
%
%
%
   %


	%
%
	%
	%
%
 %
%
	%


	
	
	%
%
	%
%
%
%
%
%


%
%
%

\bibitem[FGMR19]{FGMR19}
I. Forsyth, M. Goffeng, B. Mesland and A. Rennie.
\newblock Boundaries, spectral triples and K-homology.
\newblock J. Noncommut. Geom. 13 (2019), no. 2, 407--472.

\bibitem[Fre21]{Fre21}
D. S. Freed.
\newblock The Atiyah-Singer index theorem.
\newblock Bull. Amer. Math. Soc. (N.S.) 58 (2021), no. 4, 517--566.


\bibitem[GaR85]{GaR85}
J. Garcia-Cuerva and J. L. Rubio de Francia.
\newblock Weighted norm inequalities and related topics.
\newblock North-Holland Math. Stud., 116. Notas Mat., 104 [Mathematical Notes]. North-Holland Publishing Co., Amsterdam, 1985.

\bibitem[Ger25]{Ger25}
M. Gerhold.
\newblock Entropy-, Approximation- and Kolmogorov Numbers on Quasi-Banach Spaces.
\newblock Bachelor Thesis, 2025, arXiv:2508.06542.

\bibitem[Ger22]{Ger22}
D. M. Gerontogiannis.
\newblock On finitely summable Fredholm modules from Smale spaces.
\newblock Trans. Amer. Math. Soc. 375 (2022), no. 12, 8885--8944.

\bibitem[Gin09]{Gin09}
N. Ginoux.
\newblock The Dirac spectrum.
\newblock Lecture Notes in Math., 1976. Springer-Verlag, Berlin, 2009.



%
%
%
%


\bibitem[GJV19]{GJV19}
M. P. Gomez Aparicio, P. Julg and A. Valette.
\newblock The Baum-Connes conjecture: an extended survey.
\newblock Advances in noncommutative geometry-on the occasion of Alain Connes'70th birthday, 127--244. Springer, Cham, 2019.


\bibitem[GPX22]{GPX22}
A. Gonzalez-Perez, J. Parcet, R. Xia.
\newblock Noncommutative Cotlar identities for groups acting on tree-like structures.
\newblock Preprint, arXiv:2209.05298.

\bibitem[GVF01]{GVF01}
J. M. Gracia-Bondia, J. C. Varilly and H. Figueroa.
\newblock Elements of noncommutative geometry.
\newblock Birkh\"auser Advanced Texts: Basler Lehrb\"ucher. Birkh\"auser Boston, Inc., Boston, MA, 2001.

%

\bibitem[Gra14a]{Gra14a}
L. Grafakos.
\newblock Classical Fourier analysis. Third edition.
\newblock Grad. Texts in Math., 249. Springer, New York, 2014.

\bibitem[Gra14b]{Gra14b}
L. Grafakos.
\newblock Modern Fourier analysis. Third edition.
\newblock Grad. Texts in Math., 250. Springer, New York, 2014.

%



%
%
%
%











%
%




%
%
%

\bibitem[HLP19]{HLP19}
H. Ha, G. Lee and R. Ponge.
\newblock Pseudodifferential calculus on noncommutative tori, I. Oscillating integrals.
\newblock Internat. J. Math. 30 (2019), no. 8, 1950033, 74 pp..


%
%
%
%
%
%
%
%
%
%
%


\bibitem[Haa06]{Haa06}
M. Haase.
\newblock The functional calculus for sectorial operators.
\newblock Operator Theory: Advances and Applications, 169. Birkh\"auser Verlag (2006).

\bibitem[Had03]{Had03}
T. Hadfield.
\newblock The noncommutative geometry of the discrete Heisenberg group.
\newblock Houston J. Math. 29 (2003), no. 2, 453--481.

\bibitem[Had04]{Had04}
T. Hadfield.
\newblock Fredholm modules over certain group $C^*$-algebras.
\newblock J. Operator Theory 51 (2004), no. 1, 141--160.

%
%
%
%
%
%


\bibitem[Hao20]{Hao20}
C. Hao.
\newblock Lecture notes on harmonic analysis.
\newblock Chinese Academy of Science, 2020.


%
%



%
%
%


\bibitem[Her12]{Her12}
A. Hermann.
\newblock Dirac eigenspinors for generic metrics.
\newblock PhD thesis. Universit\"at Regensburg 2012.


\bibitem[HiR00]{HiR00}
N. Higson and J. Roe.
\newblock Analytic K-homology.
\newblock Oxford Math. Monogr. Oxford Sci. Publ. Oxford University Press, Oxford, 2000.

\bibitem[Hig02]{Hig02}
N. Higson.
\newblock The Local Index Formula in Noncommutative Geometry.
\newblock Lectures given at the School and Conference on Algebraic K-theory and Its Applications. Trieste, 8-26 July 2002.



%
%
%
%
%
%
%
%
%
%
%
%
%
%
%
%
%
%
%
%
%
%
%
%
%
%
%

\bibitem[HiT34]{HiT34}
E. Hille and J. D. Tamarkin.
\newblock On the theory of linear integral equations. II.
\newblock Ann. of Math. (2) 35 (1934), no. 3, 445--455.

\bibitem[HiT13b]{HiT13b}
M. Hinz and T. Teplyaev.
\newblock Vector analysis on fractals and applications. 
\newblock Fractal geometry and dynamical systems in pure and applied mathematics. II. Fractals in applied mathematics, 147--163,
Contemp. Math., 601, Amer. Math. Soc., Providence, RI, 2013.




%
%
%
%

\bibitem[HLM02]{HLM02}
S. Hofmann, M. Lacey and A. McIntosh.
\newblock The solution of the Kato problem for divergence form elliptic operators with Gaussian heat kernel bounds.
\newblock Ann. of Math. (2) 156 (2002), no. 2, 623--631.

%
%
%
%
%
%
%
%
%
%
%
%

\bibitem[HMW73]{HMW73}
R. Hunt, B. Muckenhoupt and R. Wheeden.
\newblock Weighted norm inequalities for the conjugate function and Hilbert transform.
\newblock Trans. Amer. Math. Soc. 176 (1973), 227--251.




\bibitem[HvNVW16]{HvNVW16}
T. Hyt\"onen, J. van Neerven, M. Veraar and L. Weis.
\newblock Analysis in Banach spaces, Volume~I: Martingales and Littlewood-Paley theory.
\newblock Springer, 2016. 

\bibitem[HvNVW18]{HvNVW18}
T. Hyt\"onen, J. van Neerven, M. Veraar and L. Weis.
\newblock Analysis in Banach spaces, Volume~II: Probabilistic Methods and Operator Theory. 
\newblock Springer, 2018. 

\bibitem[HvNVW23]{HvNVW23}
T. Hyt\"onen, J. van Neerven, M. Veraar and L. Weis.
\newblock Analysis in Banach spaces, Volume~III: Harmonic Analysis and Spectral Theory. 
\newblock Springer, 2023. 
%
%
%
%
%
%
%
\bibitem[HMP11]{HMP11}
T. Hyt\"onen, A. McIntosh and P. Portal.
\newblock Holomorphic functional calculus of Hodge-Dirac operators in $L^p$.
\newblock J. Evol. Equ. 11 (2011), no. 1, 71--105.


%
%
%
%



%
%
%





%
%


\bibitem[IwG96]{IwG96}
T. Iwaniec and G. Martin.
\newblock Riesz transforms and related singular integrals.
\newblock J. Reine Angew. Math. 473 (1996), 25--57.



%
%
%
%
%
%
%
%
%
%
%
%
%
%
%

%
%
%
%
%
%
%
%


\bibitem[Ji14]{Ji14}
G. Ji.
\newblock Analytic Toeplitz algebras and the Hilbert transform associated with a subdiagonal algebra.
\newblock Sci. China Math. 57 (2014), no.3, 579--588.

\bibitem[Jor82]{Jor82}
K. J\"orgens.
\newblock Linear integral operators. Translated from the German by G. F. Roach.
\newblock Surveys Reference Works Math., 7. Pitman (Advanced Publishing Program), Boston, Mass.-London, 1982.

\bibitem[Jou83]{Jou83}
J.-L. Journ\'e.
\newblock Calderón-Zygmund operators, pseudodifferential operators and the Cauchy integral of Calder\'on.
\newblock Lecture Notes in Math., 994. Springer-Verlag, Berlin, 1983.








%

\bibitem[JMP18]{JMP18}
M. Junge, T. Mei and J. Parcet.
\newblock Noncommutative Riesz transforms-dimension free bounds and Fourier multipliers.
\newblock J. Eur. Math. Soc. (JEMS) 20 (2018), no. 3, 529--595.

%
%
%
%
%

%
%
%
%
%
%
%
%



%
%
%
%
%
%
%
%
%
%
%
%
%
%
%
%
%
%
%
%
%
%
%
%
%
%
%
 %
%
%
%
%
%
%
%
%
%
%
%
%
%
%
%
%
%
%
%
%
%
%
%
%
%



\bibitem[KaT08]{KaT08}
C. Kassel and V. Turaev.
\newblock Braid groups. With the graphical assistance of Olivier Dodane.
\newblock Grad. Texts in Math., 247. Springer, New York, 2008.

\bibitem[Kat76]{Kat76}
T. Kato.
\newblock Perturbation theory for linear operators. Second edition.
\newblock Grundlehren der Mathematischen Wissenschaften, Band 132. Springer-Verlag, Berlin-New York, 1976.

%
%
%
%
%
%
%
%


\bibitem[Kha13]{Kha13}
M. Khalkhali.
\newblock Basic noncommutative geometry.
\newblock EMS Ser. Lect. Math. European Mathematical Society (EMS), Z\"urich, 2013.

\bibitem[Kin09a]{Kin09a}
F. King.
\newblock Hilbert transforms. Vol. 1. 
\newblock Encyclopedia of Mathematics and its Applications, 124. Cambridge University Press, Cambridge, 2009.

%
%
%
%
%
%
%
%
%
%
%
%
%
%


\bibitem[Kon86]{Kon86}
H. K\"onig.
\newblock Eigenvalue distribution of compact operators.
\newblock Oper. Theory Adv. Appl., 16. Birkh\"auser Verlag, Basel, 1986.

\bibitem[Kon01]{Kon01}
H. K\"onig.
\newblock Eigenvalues of operators and applications.
\newblock North-Holland Publishing Co., Amsterdam, 2001, 941--974.


%
%
 %
%
%
%
%
%
%
%
%
%
%
%
%
%


\bibitem[KZPP76]{KZPP76}
M. A. Krasnosel'skii, P. P. Zabreiko, E. I. Pustylnik, P. E. Sobolevskii.
\newblock Integral operators in spaces of summable functions. Translated from the Russian by T. Ando.
\newblock Monographs and Textbooks on Mechanics of Solids and Fluids, Mechanics: Analysis
Noordhoff International Publishing, Leiden, 1976.


\bibitem[KuW04]{KuW04}
P. C. Kunstmann and L. Weis.
\newblock Maximal $L_p$-regularity for parabolic equations, Fourier multiplier theorems and $\H^\infty$-functional calculus.
\newblock Functional analytic methods for evolution equations, 65--311, 
Lecture Notes in Math., 1855, Springer, Berlin, 2004.

%


%
%
%
%
%
%
%
%
%
%
%

%
%
%
%
%
%
%
%
%
%
%
%
%
%
%
%
%
%

\bibitem[LaX13]{LaX13}
L. E. Labuschagne and Q. Xu.
\newblock A Helson-Szegö theorem for subdiagonal subalgebras with applications to Toeplitz operators.
\newblock J. Funct. Anal. 265 (2013), no. 4, 545--561.

\bibitem[Laf02]{Laf02}
V. Lafforgue.
\newblock K-th\'eorie bivariante pour les alg\`ebres de Banach et conjecture de Baum-Connes.
\newblock Invent. Math. 149 (2002), no. 1, 1--95.

\bibitem[LaM89]{LaM89}
B. H. Lawson and M.-L. Michelsohn.
\newblock Spin geometry.
\newblock Princeton Math. Ser., 38. Princeton University Press, Princeton, NJ, 1989.










%
%
%
%
%
%
%
%
%

\bibitem[LM99]{LM99}%
C. Le Merdy.
\newblock $\H^\infty$-functional calculus and applications to maximal regularity.
\newblock Publ. Math. UFR Sci. Tech. Besan\c{c}on, 16, Univ. Franche-Comt\'e, Besan\c{c}on, 1999.




%
%
%
%
%
%
%
%
%
%
%
%
%
%
%
%
%
%
%
%
%
%








\bibitem[LNWW20]{LNWW20}
J. Li, T. T. T. Nguyen, L. A. Ward and B. D. Wick.
\newblock The Cauchy integral, bounded and compact commutators.
\newblock Studia Math. 250 (2020), no. 2, 193--216.




%
%
%
%
%
%
 %
%
%
%
%
%
%
%
%
%
%
%
%
%
%
%
%
%
%
%
%
%
%
%
%

\bibitem[LMSZ23]{LMSZ23}
S. Lord, E. McDonald, F. Sukochev and D. Zanin.
\newblock Singular traces. Vol. 2. Trace formulas.
\newblock De Gruyter Stud. Math., 46/2 De Gruyter, Berlin, 2023.


%
%
%
%
%
%
%
%
%
%
%
%
%
%
  %
%
%
%
%

 %
%
%
%
%
%
%
%
%
%
%
%
%
%
%
%
%
%



%
%
%
%
%

\bibitem[McI86]{McI86}%
A. McIntosh.
\newblock Operators which have an $H^\infty$ functional calculus.
\newblock Proc. Centre Math. Anal. Austral. Nat. Univ., 14. Australian National University, Centre for Mathematical Analysis, Canberra, 1986, 210--231.

\bibitem[McQ91]{McQ91}
A. McIntosh and T. Qian.
\newblock Convolution singular integral operators on Lipschitz curves.
\newblock Lecture Notes in Math. 1494. Springer-Verlag, Berlin, 1991, 142--162.


\bibitem[MSX19]{MSX19}
E. McDonald, F. Sukochev and X. Xiong.
\newblock Quantum differentiability on quantum tori.
\newblock Comm. Math. Phys. 371 (2019), no. 3, 1231--1260.




%
%
%
%



\bibitem[MaW98]{MaW98}
M. Marsalli and G. West.
\newblock Noncommutative $H^p$ spaces.
\newblock J. Operator Theory 40 (1998), no. 2, 339--355.




\bibitem[Mas09]{Mas09}
J. Mashreghi.
\newblock Representation theorems in Hardy spaces.
\newblock London Math. Soc. Stud. Texts, 74. Cambridge University Press, Cambridge, 2009.



%
%
%
%
%


%
%

\bibitem[MeR17]{MeR17}
T. Mei and \'E. Ricard.
\newblock Free Hilbert Transforms.
\newblock Duke Math. J. 166 (2017), no. 11, 2153--2182.




%
%
%
%
%
%
%
%
%
%
%
%
%
%
%
%
%
%
%
%
%
%
%
%
%
 %
%
%
%
%
%
%
%
%

\bibitem[Mey03]{Mey03}
Y. Meyer.
\newblock La conjecture de Kato (d'apr\`es Pascal Auscher, Steve Hofmann, Michael Lacey, John Lewis, Alan McIntosh et Philippe Tchamitchian).
\newblock S\'eminaire Bourbaki. Vol. 2001/2002. Ast\'erisque No. 290 (2003), Exp. No. 902, viii, 193--206.


%
%
%
%
%
%



\bibitem[NSSS06]{NSSS06}
V. E. Nazaikinskii, A. Y. Savin, B.-W. Schulze and B. Y. Sternin.
\newblock Elliptic theory on singular manifolds.
\newblock Differ. Integral Equ. Appl., 7. Chapman \& Hall/CRC, Boca Raton, FL, 2006.



%
%

%
%
%

\bibitem[NeV17]{NeV17}
J. van Neerven and R. Versendaal.
\newblock $L^p$-Analysis of the Hodge-Dirac Operator Associated with Witten Laplacians on Complete Riemannian Manifolds.
\newblock J. Geom. Anal. (2017), 1--30.
%
%
%
%
%
%
%
%

\bibitem[NeT11]{NeT11}
S. Neshveyev and L. Tuset.
\newblock $K$-homology class of the Dirac operator on a compact quantum group. 
\newblock Doc. Math. 16 (2011), 767--780. 

%
%
%
%
%
%
%
%

\bibitem[Nik02]{Nik02}
N. K. Nikolski.
\newblock Operators, functions, and systems: an easy reading. Vol. 1.
\newblock Math. Surveys Monogr., 92. American Mathematical Society, Providence, RI, 2002.


%
%
%
%
%


%

\bibitem[Ouh05]{Ouh05}
E. M. Ouhabaz.
\newblock Analysis of heat equations on domains.
\newblock London Mathematical Society Monographs Series, 31. Princeton University Press, Princeton, NJ, 2005.

 
\bibitem[PRR19]{PRR19}
N. S. Papageorgiou, V. D. Radulescu and D. D. Repovs.
\newblock Nonlinear analysis—theory and methods
\newblock Springer Monogr. Math. Springer, Cham, 2019.




%
%
%
%
%
%
%
%


%
%
%
%
%
%
%
%
%
%
%
%
%
%
%
%
%




\bibitem[Pel03]{Pel03}
Hankel operators and their applications.
\newblock V. V. Peller.
\newblock Springer Monogr. Math.. Springer-Verlag, New York, 2003.


%
%
%
%
%
%




%
%
%
%
%
%
%
%
%
%
%
%
%
%
%
%
%
%
%
%
%
%
%
%
%
%
%
%
%
%
%
%
%
%



%
%
%
%
%
%
%
%
%



\bibitem[Pie74]{Pie74}
A. Pietsch.
\newblock $s$-numbers of operators in Banach spaces.
\newblock Studia Math. 51 (1974), 201--223.

\bibitem[Pi80]{Pie80}
A. Pietsch.
\newblock Operator ideals.
\newblock North-Holland Math. Library, 20. North-Holland Publishing Co., Amsterdam-New York, 1980.

\bibitem[Pie87]{Pie87}
A. Pietsch.
\newblock Eigenvalues and $s$-numbers.
\newblock Math. Anwendungen Phys. Tech., 43. Akademische Verlagsgesellschaft Geest \& Portig K.-G., Leipzig, 1987.

\bibitem[Pie07]{Pie07}
A. Pietsch.
\newblock History of Banach spaces and linear operators.
\newblock Birkhäuser Boston, Inc., Boston, MA, 2007.

\bibitem[Pie23]{Pie23}
A. Pietsch.
\newblock Traces on operator ideals defined over the class of all Banach spaces and related open problems.
\newblock Integral Equations Operator Theory 95 (2023), no. 2, Paper No. 11.

\bibitem[PiV82]{PiV82}
M. Pimsner and D. Voiculescu.
\newblock K-groups of reduced crossed products by free groups.
\newblock J. Operator Theory 8 (1982), no. 1, 131--156.

\bibitem[PiX03]{PiX03}
G. Pisier and Q. Xu.
\newblock Non-commutative $L^p$-spaces.
\newblock In Handbook of the Geometry of Banach Spaces, Vol. II, edited by W.B. Johnson and J. Lindenstrauss, Elsevier, 1459--1517, 2003.




\bibitem[PSB16]{PSB16}
E. Prodan and H. Schulz-Baldes.
\newblock Bulk and boundary invariants for complex topological insulators.
\newblock Math. Phys. Stud. Springer, [Cham], 2016.

\bibitem[Pus11]{Pus11}
M. Puschnigg.
\newblock Finitely summable Fredholm modules over higher rank groups and lattices.
\newblock J. K-Theory 8 (2011), no. 2, 223--239.

\bibitem[QiL19]{QiL19}
T. Qian and P. Li.
\newblock Singular integrals and Fourier theory on Lipschitz boundaries.
\newblock Science Press Beijing, Beijing; Springer, Singapore, 2019.


%
%
%
%
%
%
%

\bibitem[Ran98]{Ran98}
N. Randrianantoanina.
\newblock Hilbert transform associated with finite maximal subdiagonal algebras. 
\newblock J. Austral. Math. Soc. Ser. A 65 (1998), no. 3, 388--404.

\bibitem[Ran02]{Ran02}
N. Randrianantoanina.
\newblock Spectral subspaces and non-commutative Hilbert transforms.
\newblock Colloq. Math. 91 (2002), no. 1, 9--27.

%
%
%
%
%
%
%

\bibitem[Ren04]{Ren04}
A. Rennie.
\newblock Summability for nonunital spectral triples.
\newblock $K$-Theory 31 (2004), no. 1, 71--100.

\bibitem[ReS82]{ReS82}
S. Rempel and B.-W. Schulze.
\newblock Index theory of elliptic boundary problems.
\newblock Akademie-Verlag, Berlin, 1982.

\bibitem[ReS24]{ReS24}
C. Rey and N. Saintier.
\newblock Non-local equations and optimal Sobolev inequalities on compact manifolds.
\newblock J. Geom. Anal. 34 (2024), no. 17.



%
%
%
%
%
%
%
%
%
%
%
%
%
%
%
%
%
%
%
%






%

\bibitem[RLL00]{RLL0}
M. R\o rdam, F. Larsen and N. Laustsen.
\newblock An introduction to $K$-theory for $C^*$-algebras.
\newblock London Math. Soc. Stud. Texts, 49. Cambridge University Press, Cambridge, 2000.

%
%
%
%
%




%
%
%
%

\bibitem[RuT10]{RuT10}
M. Ruzhansky and V. Turunen.
\newblock Pseudo-differential operators and symmetries. Background analysis and advanced topics.
\newblock Pseudo Diff. Oper., 2 Birkh\"auser Verlag, Basel, 2010.


%
%
%
%
%
%





\bibitem[SWW98]{SWW98}
E. Schrohe, M. Walze and J.-M. Warzecha.
\newblock Construction de triplets spectraux à partir de modules de Fredholm.
\newblock C. R. Acad. Sci. Paris Sér. I Math. 326 (1998), no. 10, 1195--1199.

\bibitem[ScS22]{ScS22}
H. Schulz-Baldes and T. Stoiber.
\newblock Harmonic analysis in operator algebras and its applications to index theory and topological solid state systems.
\newblock Math. Phys. Stud. Springer, Cham, 2022.

\bibitem[ScS23]{ScS23}
H. Schulz-Baldes and T. Stoiber.
\newblock The generators of the K-groups of the sphere.
\newblock Expo. Math. 41 (2023), no. 4, Paper No. 125519.

\bibitem[Shu01]{Shu01}
M. A. Shubin.
\newblock Pseudodifferential operators and spectral theory. Second edition.
\newblock Springer-Verlag, Berlin, 2001.

%
%
%

%
%
%
%
%
%
%
%
%
%
%
%
%
%
%
%
%
%
%
%
%
%
%
%
%
%
%
%
%
%
%
%
%
%
%


%



%
%
%
%
%
%
%
%
%
%
%
%
%
%
%
%
%
%
%
%
%
%
%
%
%
%
%

\bibitem[Ste93]{Ste93}
E. M. Stein.
\newblock Harmonic analysis: real-variable methods, orthogonality, and oscillatory integrals. With the assistance of Timothy S. Murphy.
\newblock Princeton Math. Ser., 43. Monogr. Harmon. Anal., III. Princeton University Press, Princeton, NJ, 1993.




%
%
%
%
%
%
%
%
%
%
%
%
%
%
%
%







\bibitem[SuZ23]{SuZ23}
F. Sukochev and D. Zanin.
\newblock The Connes character formula for locally compact spectral triples. 
\newblock Ast\'erisque (2023), no. 445.

\bibitem[FXZ23]{FXZ23}
F. Sukochev, X. Xiong and D. Zanin.
\newblock Asymptotics of singular values for quantised derivatives on noncommutative tori.
\newblock J. Funct. Anal. 285 (2023), no. 5, Paper No. 110021.

%
%
%
%
%




%
%
%
%
%
%
%
%
%

\bibitem[Tay81]{Tay81}
M. E. Taylor.
\newblock Pseudodifferential operators.
\newblock Princeton Math. Ser., No. 34. Princeton University Press, Princeton, NJ, 1981.

\bibitem[Tch01]{Tch01}
P. Tchamitchian.
\newblock The solution of Kato's conjecture (after Auscher, Hofmann, Lacey, McIntosh and Tchamitchian).
\newblock Journ\'ees ``\'Equations aux D\'eriv\'ees Partielles'' (Plestin-les-Gr\`eves, 2001), Exp. No. XIV, 14 pp., Univ. Nantes, Nantes, 2001.

%
%

%
%
%
%
%
%
%
%
%
%
%
%
%

\bibitem[Tre80]{Tre80}
F. Tr\`eves.
\newblock Introduction to pseudodifferential and Fourier integral operators. Vol. 1.
Pseudodifferential operators.
\newblock University Series in Mathematics. Plenum Press, New York-London, 1980.






\bibitem[Uch78]{Uch78}
A. Uchiyama.
\newblock On the compactness of operators of Hankel type.
\newblock Tohoku Math. J. (2) 30 (1978), no. 1, 163--171.


%
%
\bibitem[Var10]{Var10}
J. C. Varilly.
\newblock Dirac operators and spectral geometry.
\newblock Lecture notes on noncommutative geometry and quantum groups edited by P. M. Hajac. \href{https://www.mimuw.edu.pl/~pwit/toknotes/}{https://www.mimuw.edu.pl/\~{}pwit/toknotes/}

%
%



%
%
%
%
%
%
%


\bibitem[Ver21]{Ver21}
J. Verdera.
\newblock Birth and life of the $L^2$ boundedness of the Cauchy Integral on Lipschitz graph.
\newblock Preprint, arXiv:2109.06690.


%
%
%
%
%
%
%
%

\bibitem[Voi92]{Voi92}
J. Voigt.
\newblock On the convex compactness property for the strong operator topology.
\newblock Note Mat. 12 (1992), 259--269.


%
%
%
%
%
%
%
%
%


\bibitem[Wea96]{Wea96}
N. Weaver.
\newblock Lipschitz algebras and derivations of von Neumann algebras.
\newblock J. Funct. Anal. 139 (1996), no. 2, 261--300.


\bibitem[Web18]{Web18}
J. Weber.
\newblock Introduction to Sobolev Spaces.
\newblock Lecture Notes. MM692 2018-2. UNICAMP. 2018.

\bibitem[WeO93]{WeO93}
N. E. Wegge-Olsen.
\newblock K-theory and $C^*$-algebras. A friendly approach.
\newblock Oxford Sci. Publ. The Clarendon Press, Oxford University Press, New York, 1993.

%
%
%




%
%


\bibitem[Wic20]{Wic20}
B. D. Wick.
\newblock Commutators, BMO, Hardy spaces and factorization: a survey.
\newblock Real Anal. Exchange 45 (2020), no. 1, 1--28.

%

\bibitem[WiY20]{WiY20}
R. Willett and G. Yu.
\newblock Higher index theory.
\newblock Cambridge Stud. Adv. Math., 189. Cambridge University Press, Cambridge, 2020.


%
%

\bibitem[Wir22]{Wir22}
M. Wirth.
\newblock The Differential Structure of Generators of GNS-symmetric Quantum Markov Semigroups. 
\newblock Preprint, arXiv:2207.09247. 


 %
%
 %
%
%
%
%
%
%

\bibitem[XXY18]{XXY18}
X. Xiong, Q. Xu and Z. Yin.
\newblock Sobolev, Besov and Triebel-Lizorkin spaces on quantum tori.
\newblock Mem. Amer. Math. Soc. 252 (2018), no. 1203.



%
%
%
%
%
%
%
%
%
%
%



%
%
%
%
%
%
%
%

\end{thebibliography}
\end{document}